\documentclass[leqno]{memo-l}

\usepackage{hyperref}
\usepackage{enumerate,amssymb}
\usepackage[arrow,curve,matrix,tips,2cell]{xy}
  \SelectTips{eu}{10} \UseTips
  \UseAllTwocells
\usepackage{tikz}                \usetikzlibrary{calc}               \usetikzlibrary{matrix}		\usetikzlibrary{patterns}		\usetikzlibrary{decorations.pathreplacing}
\usepackage{pdfcolmk}
\usepackage{calc}
\usepackage{multirow}
\usepackage{stmaryrd}
\usepackage{lscape}
\usepackage[figuresright]{rotating}

\usepackage[lmargin=1.75in,rmargin=1.75in,tmargin=1.5in,bmargin=1.5in]{geometry}

\usepackage{chngcntr}
\counterwithout{footnote}{chapter}

\usepackage{tabto}

  \DeclareRobustCommand{\SkipTocEntry}[5]{}

  \newcommand{\extrivialconfnet}{Eg.~1.3}
  \newcommand{\lemirrelevanceofpoints}{Lem.~1.4}
  \newcommand{\lemopencoverofcirclesector}{Lem.~1.9}
  \newcommand{\cornoncanonicalvacuum}{Cor.~1.16}
  \newcommand{\defnoncanonicalvacuum}{Def.~1.17}
  \newcommand{\propHaagduality}{Prop.~1.18}
  \newcommand{\warnpoorformalproperties}{Warn.~1.22}
  \newcommand{\propAoastCBindependentofHandK}{Prop.~1.23}     
  \newcommand{\lemextendstoACB}{Lem.~1.24}   
  \newcommand{\corcalaKacts}{Cor.~1.29}
  
  \newcommand{\corvacuumvacuumvacuum}{Cor.~1.34}

  \newcommand{\subseccentraldecomposition}{\S1.4}
  \newcommand{\lemjisimplemented}{Lem.~2.5}
  \newcommand{\lemLconformalimplementaion}{Lem.~2.7} 
  \newcommand{\defconformalcircle}{Def.~2.12}
  \newcommand{\thmVaccumSector}{Thm.~2.13}
  \newcommand{\corConformalversionofvacuumvacuumvacuum}{Cor.~2.20}
  \newcommand{\secHilbertspaceforannulus}{\S3.2}
  \newcommand{\thmKLMallirreduciblesectorsarefinite}{Thm.~3.14}
  \newcommand{\thmKLM}{Thm.~3.23}     
  \newcommand{\lemequivdefofcontinuityaxiom}{Lem~4.6} 
  
  \newcommand{\secExtendtomanifolds}{Sec.~1.A}
  \newcommand{\thmextendAtohatA}{Thm.~1.3}
  \newcommand{\thmcomputebfBnet}{Thm.~1.20}

  \newcommand{\eqdualityequations}{(4.2)}
  \newcommand{\eqdualitynormalization}{(4.3)}
    
  \newcommand{\lemLemmaB}{Lem.~4.10}

  \newcommand{\eqpropertiesofmatrixofstatdim}{Eq.~(5.14)}
  \newcommand{\leminddim}{Lem.~5.16}
  \newcommand{\corindexofcommutants}{Cor.~5.17}
  \newcommand{\propstatdimsubalg}{Prop.~5.18}
  \newcommand{\propfunctorialityofLiso}{Prop.~6.22}
  \newcommand{\lemAvBAvBfinite}{Lem.~7.18}
  \newcommand{\corbeforefinal}{Cor.~7.26}
  \newcommand{\corfinal}{Cor.~7.27}
       \newcommand{\cala}{\mathcal{A}}
       \newcommand{\calb}{\mathcal{B}}
  \newcommand{\IC}{\mathbb{C}}     \newcommand{\calc}{\mathcal{C}}
  \newcommand{\ID}{\mathbb{D}}     \newcommand{\cald}{\mathcal{D}}

  \newcommand{\IN}{\mathbb{N}}

  \newcommand{\IR}{\mathbb{R}}

  \newcommand{\bfB}{{\mathbf B}}

  \newcounter{commentcounter}

  \definecolor{AHcolor}{rgb}{0.5,0.0,0.5}   % Textcolor for AB
  \definecolor{CDcolor}{rgb}{0.7,0.0,0.3}   % Textcolor for CD
  \definecolor{ABcolor}{rgb}{0.2,0.8,0.2}   % Textcolor for AH
  \newcommand{\tikzmath}[2][]
     {\vcenter{\hbox{\begin{tikzpicture}[#1]#2
                     \end{tikzpicture}}}
     }

  \newcommand{\textscale}{.03}
  \newcommand{\planscale}{.04}
  \newcommand{\squarescale}{.07}
  \newcommand{\displscale}{.05}
  \newcommand{\displscalesmall}{.03}
    
  \definecolor{vacuumcolor}{gray}{.55}
  \definecolor{spacecolor}{gray}{.85}
  \theoremstyle{plain}
  \newtheorem{theorem}{Theorem}[chapter]
  \newtheorem{maintheorem}[theorem]{Main Theorem}
  \newtheorem{lemma}[theorem]{Lemma}
  \newtheorem{corollary}[theorem]{Corollary}
  \newtheorem{proposition}[theorem]{Proposition}

  \newtheorem*{theorem*}{Theorem}
  \newtheorem{introthm}{Theorem}

  \theoremstyle{definition}
  \newtheorem{definition}[theorem]{Definition}

  \newtheorem{correction}[theorem]{Correction}

  \theoremstyle{remark}
  \newtheorem{notation}[theorem]{Notation}
  \newtheorem{remark}[theorem]{Remark}
  \newtheorem*{remark*}{Remark}
  \newtheorem{warning}[theorem]{Warning}
  \newtheorem{example}[theorem]{Example}

  \makeatletter\let\c@equation=\c@theorem\makeatother

  \newenvironment{numberlist}
     {\begin{list}{}%
      {%
       \setlength{\leftmargin}{\labelwidth+\labelsep}%
      }%
     }%
     {\end{list}}

  \DeclareMathOperator{\Ad}{Ad}
  
  \DeclareMathOperator{\Aut}{Aut}
  \DeclareMathOperator{\colim}{colim}
  \DeclareMathOperator{\Conf}{Conf}
  \DeclareMathOperator{\Diff}{Diff}

  \DeclareMathOperator{\id}{id}

  \DeclareMathOperator{\U}{U}
  
  \DeclareMathOperator{\supp}{supp}

  \newcommand{\INT}{{\mathsf{INT}}}
  \newcommand{\VN}{{\mathsf{VN}}}
  \newcommand{\CN}{{\mathsf{CN}}}
  \newcommand{\modules}[1]{{#1\text{-}\mathsf{modules}}}
  \def\halfbullet{{\tikz[scale = 0.045]{\useasboundingbox(0,-.3) circle(1);\draw[line width = 0.1] (0,0) circle(1); \filldraw[line width = 0] (84:1) -- (-84:1) arc (-84:84:1);}}}

\def\tworarrow{\hspace{.1cm}{\setlength{\unitlength}{.50mm}\linethickness{.09mm}
	\begin{picture}(8,8)(0,0)\qbezier(0,4)(4,7)(8,4)\qbezier(0,1)(4,-2)(8,1)\qbezier(3.5,4)(3.5,3)(3.5,1.5)
	\qbezier(4.5,4)(4.5,3)(4.5,1.5)\qbezier(4,0.8)(4.5,1.7)(5.5,2)\qbezier(4,0.8)(3.5,1.7)(2.5,2)
	\qbezier(8,1)(7.4,.2)(7.7,-.7)\qbezier(8,1)(7,1)(6.5,1.5)\qbezier(8,4)(7.4,4.8)(7.7,5.7)
	\qbezier(8,4)(7,4)(6.5,3.5)\end{picture}\hspace{.1cm}}}
\def\urc{{{}_{\scriptstyle\urcorner}}}
\def\ulc{{{}_{\scriptstyle\ulcorner}}}
\def\llc{{{}^{\scriptstyle\llcorner}}}
\def\lrc{{{}^{\scriptstyle\lrcorner}}}

  \newcommand{\alg}{{\mathit{alg}}}
  
  \newcommand{\op}{{\mathit{op}}}

  \newcommand{\x}{{\times}}
  \newcommand{\ox}{{\otimes}}
  \newcommand{\barox}{{\bar{\ox}}}
  \newcommand{\dd}{{\partial}}
  \newcommand{\e}{{\varepsilon}}

\renewcommand{\thesection}{{\sc\alph{section}}}
\newcommand{\ignore}[1]{}

\newcommand\blfootnote[1]{%
  \begingroup
  \renewcommand\thefootnote{}\footnote{#1}%
  \addtocounter{footnote}{-1}%
  \endgroup
}

\begin{document}

\frontmatter

        \title{Fusion of defects}       
       \author{Arthur Bartels}
%      \address{Westf\"alische Wilhelms-Universit\"at M\"unster\\
%               Mathematisches Institut\\
%               Einsteinstr.~62,
%               D-48149 M\"unster, Deutschland}
      \address{Westf\"alische Wilhelms-Universit\"at M\"unster}
        \email{bartelsa@wwu.de}
%      \urladdr{http://www.math.uni-muenster.de/u/bartelsa}
       \author{Christopher L. Douglas} 
      \address{University of Oxford}
%      \address{Mathematical Institute\\ Radcliffe Observatory Quarter\\ Woodstock Road\\ Oxford\\ OX2 6GG\\ United Kingdom}
        \email{cdouglas@maths.ox.ac.uk}
%      \urladdr{http://people.maths.ox.ac.uk/cdouglas}
       \author{Andr{\'e} Henriques}
      \address{Universiteit Utrecht}
%      \address{Mathematisch Instituut\\
%               Universiteit Utrecht, Postbus 80.010\\
%               3508 TA Utrecht, The Netherlands}
        \email{a.g.henriques@uu.nl}
%      \urladdr{http://www.staff.science.uu.nl/\!\raisebox{-1mm}{~}\!henri105}  
      
%      \date{June 29, 2016}

%!%!% \dedicatory{Dedication text (use \\[2pt] for line break if necessary)}

\maketitle

\tableofcontents

	\begin{abstract}
	 Conformal nets provide a mathematical model for conformal field theory.
We define a notion of defect between conformal nets, formalizing the idea of an interaction between two conformal field theories.
We introduce an operation of fusion of defects, and prove that the fusion of two defects is again a defect, provided the fusion occurs over a conformal net of finite index.
There is a notion of sector (or bimodule) between two defects, and operations of horizontal and vertical fusion of such sectors.
Our most difficult technical result is that the horizontal fusion of the vacuum sectors of two defects is isomorphic to the vacuum sector of the fused defect.
Equipped with this isomorphism, we construct the basic interchange isomorphism between the horizontal fusion of two vertical fusions and the vertical fusion of two horizontal fusions of sectors.\vspace*{-20pt}\blfootnote{This work first appeared as ``Conformal nets III: Fusion of defects" at arXiv:1310:8263.}
	\end{abstract}

%\vspace*{-30pt}
%\maketitle
%\tableofcontents

\chapter*{Acknowledgments}
%\addtocontents{toc}{\SkipTocEntry}
  We are grateful for the guidance and support of Michael Hopkins, Michael M\"uger, Stephan Stolz, and Peter Teichner.  AB was supported by the Sonderforschungsbereich 878, and CD was partially supported by a Miller Research Fellowship and by EPSRC grant EP/K015478/1.
  CD and AH benefited from the hospitality of MSRI (itself supported by NSF grant 0932078000) during its spring 2014 algebraic topology program.

\mainmatter

\chapter*{Introduction}

%\ignore{

  There are various different mathematical notions of field theories.
  For many of these there is also a notion of defects that 
  formalizes  interactions between different field theories.
  See for example~\cite{
     Fuchs-Schweigert-Valentino(Bicats-bound-cond-surface-defects),
     Kapustin-Saulina(Suface-ops-in-3dTFT-2dRCFT),
     Quella-Runkel-Watts(Reflection-and-transmission),
     Schweigert-Fuchs-Runkel(ICM)} and references therein.
  Depending on the context, sometimes the 
  terminology `surface operator' or `domain wall'
  is used in place of `defect'.
  Often field theories are described as functors from
  a bordism category whose objects are $(d-1)$-manifolds and
  morphisms are $d$-dimensional bordisms (usually with additional
  geometric structure) to a category of vector spaces.
  Defects allow the extension
  of such functors to a larger bordism category, where the manifolds 
  may be equipped with codimension-1 submanifolds that split the manifolds 
  into regions labeled by  field theories. 
  The codimension-1 submanifold itself is labeled by
  a defect between the field theories labeling the 
  neighboring regions.
    
  In this book we give a definition of defects for conformal nets.
  Conformal nets are often viewed as a particular model for conformal 
  field theory.\footnote{Note, however, that the precise relation of conformal nets to conformal 
    field theory in the sense of Segal~\cite{Segal(Def-CFT)} 
    is not at present clear.}
  Our main result is that under suitable finiteness assumptions
  there is a composition for defects that we call \emph{fusion}.
  We also extend the notion of representations of conformal nets, also known as sectors, to the context of defects.
  Sectors between defects are a simultaneous generalization of the notion of representations of conformal nets, and of bimodules between von Neumann algebras.
  Ultimately, this will lead to a 3-category whose objects
  are conformal nets, whose $1$-morphisms are defects, whose
  $2$-morphisms are sectors, and whose $3$-morphisms are intertwiners
  between sectors.
  The lengthy details of the construction of this $3$-category are
  postponed to~\cite{BDH(3-category)}, but the key ingredients of
  this $3$-category will all be presented here.
  In~\cite{BDH(3-category)} we will use the language of internal 
  bicategories developed in~\cite{Douglas-Henriques(Internal-bicategories)},
  but we expect that the results of the present book also provide all 
  the essential ingredients to construct a $3$-category of conformal nets,
  defects, sectors, and intertwiners in any other sufficiently weak
  model of $3$-categories.

\section{Conformal nets}
%\addtocontents{toc}{\SkipTocEntry}
  Conformal nets grew out of algebraic quantum field theory
  and have been intensively studied; see for 
  example~\cite{Buchholz-Mack-Todorov(1988current-alg), Gabbiani-Froehlich(OperatorAlg-CFT), Kawahigashi-Longo(2004classification),
     Wassermann(Operator-algebras-and-conformal-field-theory), Wassermann(ICM)}.
  In this book we will use our (non-standard) coordinate-free 
  definition of conformal nets~\cite{BDH(nets)}.
  A conformal net in this sense is a functor
  \begin{equation*}
    \cala \colon \INT \to \VN
  \end{equation*}
  from the category of compact oriented intervals 
  to the category of von Neumann algebras, subject to a number of axioms.
  The precise definition and properties of conformal nets
  are recalled in Appendix~\ref{app:nets}. 
  In contrast to the standard definition, in our coordinate-free 
  definition there is no need to fix a vacuum Hilbert space at the 
  outset---this feature will be useful in developing our definition of defects.
  Nevertheless, the vacuum Hilbert space can be reconstructed 
  from the functor $\cala$.
  The main ingredient for this reconstruction is  
  Haagerup's standard form $L^2(A)$, a bimodule that is canonically associated
  to any von Neumann algebra $A$.
  This standard form and various facts about von Neumann algebras that are used throughout this
  book are reviewed in Appendix~\ref{app:vN-algebras}.
  
\section{Defects}
%\addtocontents{toc}{\SkipTocEntry}
  To define defects we introduce the category $\INT_{\circ\bullet}$
  of bicolored intervals.
  Its objects are intervals $I$ that are equipped with a covering by two 
  subintervals $I_\circ$ and $I_\bullet$.
  If $I$ is not completely white ($I_\circ = I$, $I_\bullet = \emptyset$)
  or black ($I_\bullet = I$, $I_\circ = \emptyset$) then we require 
  that the white and the black subintervals meet in exactly one point and
  we also require the choice of a local coordinate around this point. 
  For conformal nets $\cala$ and $\calb$, a defect between them is
  a functor
  \begin{equation*}
    D \colon \INT_{\circ \bullet} \to \VN
  \end{equation*}
  such that $D$ coincides with $\cala$ on white intervals and with
  $\calb$ on black intervals and satisfies various axioms similar to those of conformal nets.
  Often we write ${}_\cala D_\calb$ to indicate that $D$ is a defect from $\cala$ to $\calb$, also called an
  $\cala$-$\calb$-defect.
  A defect from the trivial net to itself is simply a von 
  Neumann algebra (Proposition \ref{prop: CCdefects == VNalg}),
  so our notion of defect is a generalization of the 
  notion of von Neumann algebra.
  The precise definition and some basic properties of defects 
  are given in Chapter~\ref{sec:defects}. 
  
  Certain defects have already appeared in disguise in the conformal nets literature, through the notion of `solitons'
  \cite{Bockenhauer-Evans(Modular-invariants-graphs-and-alpha-induction-for-nets-of-subfactors-I),
  Kawahigashi(Generalized-Longo-Rehren-subfactors-and-alpha-induction),
  Longo-Rehren(Nets-of-subfactors), Longo-Xu(dichotomy)}.  
  % \AB{I will need some explanations about solitons.}
For a conformal net $\cala$ defined on subintervals $I$ of the real line (half-infinite intervals allowed), an endomorphism of the $C^*$-algebra $\cala_\IR := \colim \cala([a,b])$ is called a \emph{soliton} if it is localized in a half-line, that is, if it acts as the identity on elements in the image of the complementary half-line.  We bicolor any subinterval $I=[a,\infty]\cup [-\infty,b]$ of the projective line by $I_\circ=[a,\infty]$ and $I_\bullet=[-\infty,b]$.  Given a soliton $\sigma$, we can consider the von Neumann algebra $D(I):=\sigma \cala([a,\infty])\vee \sigma \cala([-\infty,b])$ generated by $\sigma\cala([a,\infty])$ and $\sigma\cala([-\infty,b])$ acting on the vacuum sector.  We believe that, under certain conditions, this construction associates an $\cala$-$\cala$-defect to a soliton.  The exact relationship between solitons and defects is, however, not yet clear.

\section{Sectors}
%\addtocontents{toc}{\SkipTocEntry}
  We will use the boundary of the square $S^1 := \dd [0,1]^2$ as our
  standard model for the circle. 
  (The unit speed parametrization gives this circle a canonical smooth structure.) 
  We equip the circle with the bicoloring defined by
  $S^1_\circ = S^1 \cap [0,\frac{1}{2}] \x [0,1]$ and
  $S^1_\bullet = S^1 \cap [\frac{1}{2},1] \x [0,1]$.
  Let $D$ and $E$ be $\cala$-$\calb$-defects.
  A $D$-$E$-sector is a Hilbert space equipped with
  compatible actions of the algebras 
  $D(I)$, for subintervals $I \subset S^1$ with $(\frac{1}{2},0) \notin I$,
  and $E(I)$, for subintervals
  $I \subset S^1$ with $(\frac{1}{2},1) \notin I$.
  Pictorially we  draw a $D$-$E$-sector as follows:
  \[
   \tikzmath[scale=\squarescale]
      {\fill[spacecolor](0,0) rectangle(12,12);
         \draw (6,0) -- (0,0) -- (0,12) -- (6,12);
         \draw[ultra thick](6,12) -- (12,12) -- (12,0) -- (6,0); 
         \draw (-3,6) node {$\cala$}(15,6) node {$\calb$}(6,15) 
             node {$D$}(6,-3) node {$E$}(6,6)node {$H$};
      } \; .  %tikzmath
  \]
  The thin line $\tikzmath[scale=\textscale]{\draw (0,0) -- (10,0);}$
  should be thought of as white and stands for the 
  conformal net $\cala$ and the thick line 
  $\tikzmath[scale=\textscale]{\draw[ultra thick] (0,0) -- (10,0);}$ 
  should be thought of as black 
  and stands for $\calb$.
  Usually we simplify the picture further by dropping the letters.
  The precise definition and some basic properties of sectors 
  are given in Chapter~\ref{sec:sectors}.

\section{The vacuum sector of a defect}
%\addtocontents{toc}{\SkipTocEntry}
  For any defect $D$ we can evaluate $D$ on the top half
  $S^1_\top := \tikzmath[scale=\textscale]
        {\draw (0,6) -- (0,12) -- (6,12);
         \draw[ultra thick] (6,12) -- (12,12) -- (12,6); }$
  of the circle.
  Applying the $L^2$ functor we obtain the Hilbert space 
  $H_0(S^1,D) := L^2(D(S^1_\top))$ and, as a consequence of the 
  vacuum sector axiom in the definition of defects, this Hilbert space is
  a sector for $D$, called the vacuum sector for $D$.
  In our $3$-category, the vacuum sector is the identity $2$-morphism for
  the $1$-morphism $D$. 
  We often draw it as
  \begin{equation*}
    \tikzmath[scale=\squarescale]{
       \fill[vacuumcolor](0,0) rectangle(12,12);
       \draw (6,0) -- (0,0) -- (0,12) -- (6,12);
       \draw[ultra thick](6,12) -- (12,12) -- (12,0) -- (6,0); 
       \draw (-3,6) node {$\cala$}(15,6) node {$\calb$}(6,15) 
            node {$D$}(6,-3) node {$D$} (6,6) node {$H_0$};
       } \; . %tikzmath
  \end{equation*}
  This darker shading is reserved for vacuum sectors.

\section{Composition of defects}
%\addtocontents{toc}{\SkipTocEntry}
  Let $D = {}_\cala D_\calb$ and $E = {}_\calb E_\calc$ be defects.
  Their composition or fusion $D \circledast_\calb E$ is defined 
  in Section~\ref{subsec:Composition of defects}.
  The definition is quite natural, but, surprisingly, it is
  not easy to see that $D \circledast_\calb E$ satisfies 
  all the axioms of defects.
  
  We outline the definition of the fusion $D \circledast_\calb E$.
  In our graphical notation, double lines 
  $\tikzmath[scale=\textscale]{\draw[thick, double] (0,0) -- (10,0);}$ 
  will now correspond to $\cala$,
  thin lines $\tikzmath[scale=\textscale]{\draw (0,0) -- (10,0);}$
  to $\calb$, and thick lines 
  $\tikzmath[scale=\textscale]{\draw[ultra thick] (0,0) -- (10,0);}$ 
  to $\calc$.
  Let us concentrate on the evaluation of $D \circledast_\calb E$ 
  on  $S^1_\top$.
  Denote by $S^1_+ := \dd([1,2] \x [0,1])$
  the translate of the standard circle and by $S^1_{+,\top}$ its top half.
  As with the vacuum sector $H_0(S^1,D) = L^2(D(S^1_\top))$ for $D$ on 
  $S^1$, we can form the vacuum sector $H_0(S^1_+,E) := L^2(E(S^1_{+,\top}))$
  for $E$ on $S^1_+$.
  Let $I = \{1\} \x [0,1]$ 
  be the intersection of the two circles $S^1$ and
  $S^1_+$, equipped with the orientation inherited from $S^1_+$.
  The Hilbert space $H_0(S^1,D)$ is a right $\calb(I)$-module, while
  $H_0(S^1_+,E)$ is a left $\calb(I)$-module.
  Consequently, we can form their Connes fusion 
  $H_0(S^1,D) \boxtimes_{\calb(I)} H_0(S^1_+,E)$, drawn as
  \begin{equation} \label{eq:H_0-H_0-intro}
    \tikzmath[scale=\squarescale]
  {   \fill[vacuumcolor]  (0,0) rectangle  (24,12);
      \draw[thick, double]  (6,0) -- (0,0) -- (0,12) -- (6,12);
      \draw (6,0) -- (18,0) (12,0) -- (12,12) 
            (6,12) -- (18,12); 
      \draw[ultra thick]  (18,12) -- (24,12) -- (24,0) -- (18,0);
       \draw (-3,6) node {$\scriptstyle\cala$} (27,6) node {$\scriptstyle\calc$}
             (6,14.5)  node {$\scriptstyle D$} (6,-2.5) node {$\scriptstyle D$} 
             (6,6) node {$H_0$}  (18,6) node {$H_0$}
             (18,14.5)  node {$\scriptstyle E$} (18,-2.5) node {$\scriptstyle E$}
             (12,14.5)  node {$\scriptstyle \calb$} (12,-2.5) node {$\scriptstyle \calb$};  
  }%tikzmath
  \; .
  \end{equation}
  In this picture the middle vertical  line $\tikzmath[scale=\textscale]
        {\draw (12,12) -- (12,0); }$ corresponds to $I$.
  By the axioms for defects the actions of 
  $D( \tikzmath[scale=\textscale]
        {\draw[thick, double] (0,6) -- (0,12) -- (6,12); 
         \draw (6,12) -- (12,12); })$
  and of $\calb(I)$ on $H_0(S^1,D)$ commute.
  Consequently, we obtain an action of 
  $D( \tikzmath[scale=\textscale]
        {\draw[thick, double] (0,6) -- (0,12) -- (6,12); 
         \draw (6,12) -- (12,12); })$
  on the Connes fusion~\eqref{eq:H_0-H_0-intro}. 
  Similarly, there is an action of 
  $E( \tikzmath[scale=\textscale]
        { \draw (12,12) -- (18,12);
          \draw[ultra thick] (18,12) -- (24,12) -- (24,6);} )$
  on~\eqref{eq:H_0-H_0-intro}. 
  Now $(D \circledast_\calb E)(S^1_\top)$ is defined to be the von Neumann
  algebra generated by 
  $D( \tikzmath[scale=\textscale]
        {\draw [thick, double] (0,6) -- (0,12) -- (6,12); 
         \draw (6,12) -- (12,12); })$
  and 
  $E( \tikzmath[scale=\textscale]
        { \draw (12,12) -- (18,12);
          \draw[ultra thick] (18,12) -- (24,12) -- (24,6);} )$
  acting on the Hilbert space~\eqref{eq:H_0-H_0-intro}. 
  A similar construction, using the local coordinate, is used to
  define the evaluation of $D \circledast_\calb E$ on arbitrary 
  bicolored intervals.
  
A main result of this book is that under the assumption that the intermediate conformal net $\calb$ has finite index (see Appendix~\ref{subsec:finite-nets}), the fusion of defects is in fact a defect:
\begin{introthm}[Existence of fusion of defects]
The fusion $D \circledast_\calb E$ of two defects
is again a defect.
\end{introthm}
\noindent This is established in the text as Theorem~\ref{thm:fusion-of-defects-is-defect}.

\section{Fusion of sectors and the interchange isomorphism}
%\addtocontents{toc}{\SkipTocEntry}

Let $\cala$, $\calb$, and $\calc$ be conformal nets; let ${}_\cala D_\calb$, ${}_\calb E_\calc$, ${}_\cala F_\calb$, and ${}_\calb G_\calc$ be defects; and let $H =
   \tikzmath[scale=\textscale]
      {\fill[spacecolor](0,0) rectangle(12,12);
         \draw[thick, double] (6,0) -- (0,0) -- (0,12) -- (6,12);
         \draw(6,12) -- (12,12) -- (12,0) -- (6,0); 
         \draw (6,6) node {$\scriptscriptstyle H$};
      }
$ be a $D$--$F$-sector, and $K =
   \tikzmath[scale=\textscale]
      {\fill[spacecolor](0,0) rectangle(12,12);
         \draw (6,0) -- (0,0) -- (0,12) -- (6,12);
         \draw[ultra thick](6,12) -- (12,12) -- (12,0) -- (6,0); 
         \draw (6,6) node {$\scriptscriptstyle K$};
      }
$ an $E$--$G$-sector.
The horizontal composition $H \boxtimes_\calb K$ of these two sectors is a $(D \circledast_\calb E)$--$(F \circledast_\calb G)$-sector whose Hilbert space is the Connes fusion $H \boxtimes_{\calb(I)} K$ and which is depicted as
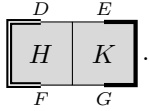
\[
    \tikzmath[scale=\squarescale]
  {   \fill[spacecolor]  (0,0) rectangle  (24,12);
      \draw[thick, double]  (6,0) -- (0,0) -- (0,12) -- (6,12);
      \draw (6,0) -- (18,0) (12,0) -- (12,12) 
            (6,12) -- (18,12); 
      \draw[ultra thick]  (18,12) -- (24,12) -- (24,0) -- (18,0);
      \draw (6,6) node {$H$} (18,6) node {$K$};
      \draw (6,14.5) node {$\scriptstyle D$} (6,-2.5) node {$\scriptstyle F$} (18,14.5) node {$\scriptstyle E$} (18,-2.5) node {$\scriptstyle G$};
  }\,.
\]
This composition operation is defined precisely in Section~\ref{ssec:Horizontal fusion}.  

Let ${}_\cala P_\calb$ be another defect and let $L =
   \tikzmath[scale=\textscale]
      {\fill[spacecolor](0,0) rectangle(12,12);
         \draw[thick, double] (6,0) -- (0,0) -- (0,12) -- (6,12);
         \draw(6,12) -- (12,12) -- (12,0) -- (6,0); 
         \draw (6,6) node {$\scriptscriptstyle L$};
      }
$ be an $F$--$P$-sector.
The vertical composition $H \boxtimes_F L$ of the sectors $H$ and $L$ is a $D$--$P$-sector whose Hilbert space is the Connes fusion $H \boxtimes_{F(S^1_\top)} L$.  Here $S^1_\top$ is the top half of the circle bounding the sector $L$; this half circle is canonically identified, by vertical reflection, with the bottom half $S^1_\bot$ of the circle bounding the sector $H$, and that identification provides the action of $F(S^1_\top)$ on $H$.  The vertical fusion operation occurs along half a circle, rather than a quarter circle as for horizontal fusion, and so is not conveniently depicted by juxtaposing squares; instead we denote this vertical composition by
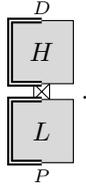
\[
\tikzmath[scale=\squarescale]
	{
	\begin{scope}[xshift=6cm]
         \draw (-1.5,0) -- (1.5,-3) (-1.5,0) -- (-1.5,-3) (1.5,0) -- (-1.5,-3) (1.5,0) -- (1.5,-3);
         \end{scope}
	\fill[spacecolor](0,0) rectangle(12,12);
	\draw[thick,double] (6,0) -- (0,0) -- (0,12) -- (6,12);
         \draw(6,12) -- (12,12) -- (12,0) -- (6,0);
                  \draw (6,6) node {$H$};
         \begin{scope}[yshift=-15cm]
	\fill[spacecolor](0,0) rectangle(12,12);
         \draw[thick,double] (6,0) -- (0,0) -- (0,12) -- (6,12);
         \draw (6,12) -- (12,12) -- (12,0) -- (6,0);
                  \draw (6,6) node {$L$};
         \end{scope}
         \draw (6,14.5) node {$\scriptstyle D$} (6,-17.5) node {$\scriptstyle P$};
         }\; .
\]
This composition operation is defined precisely in Section~\ref{sec: Vertical fusion}.  Because vertical composition is Connes fusion along the algebra associated to half a circle, and the vacuum sector is defined as an $L^2$ space for the algebra associated to half a circle, the vacuum sector serves as an identity for vertical composition; that is, $H_0(S^1,D) \boxtimes_D H \cong H$ and $H \boxtimes_F H_0(S^1,F) \cong H$ as sectors.

Let ${}_\calb Q_\calc$ be yet another defect and let $M =
   \tikzmath[scale=\textscale]
      {\fill[spacecolor](0,0) rectangle(12,12);
         \draw (6,0) -- (0,0) -- (0,12) -- (6,12);
         \draw[ultra thick](6,12) -- (12,12) -- (12,0) -- (6,0); 
         \draw (6,6) node {$\scriptscriptstyle M$};
      }
$ be a $G$--$Q$-sector.  The sectors $K$ and $M$ can be composed into an $E$--$Q$-sector $K \boxtimes_G M$.  The horizontal composition and the vertical composition of sectors ought to be compatible in the sense that there is a canonical isomorphism
\[
\tikzmath[scale=\squarescale]
	{
	\begin{scope}[xshift=12cm]
         \draw (-1.5,0) -- (1.5,-3) (-1.5,0) -- (-1.5,-3) (1.5,0) -- (-1.5,-3) (1.5,0) -- (1.5,-3);
         \end{scope}
	\fill[spacecolor]  (0,0) rectangle  (24,12);
      \draw[thick, double]  (6,0) -- (0,0) -- (0,12) -- (6,12);
      \draw (6,0) -- (18,0) (12,0) -- (12,12) 
            (6,12) -- (18,12); 
      \draw[ultra thick]  (18,12) -- (24,12) -- (24,0) -- (18,0);
      \draw (6,6) node {$H$} (18,6) node {$K$};
	\begin{scope}[yshift=-15cm]
	\fill[spacecolor]  (0,0) rectangle  (24,12);
      \draw[thick, double]  (6,0) -- (0,0) -- (0,12) -- (6,12);
      \draw (6,0) -- (18,0) (12,0) -- (12,12) 
            (6,12) -- (18,12); 
      \draw[ultra thick]  (18,12) -- (24,12) -- (24,0) -- (18,0);
      \draw (6,6) node {$L$} (18,6) node {$M$};
	\end{scope}
%      \draw (6,14.5) node {$\scriptstyle D$} (18,14.5) node {$\scriptstyle E$} (6,-17.5) node {$\scriptstyle P$} (18,-17.5) node {$\scriptstyle Q$};
      }
\;\;\cong\;\;
\tikzmath[scale=\squarescale]	
{
	\begin{scope}[xshift=6cm]
         \draw (-1.5,0) -- (1.5,-3) (-1.5,0) -- (-1.5,-3) (1.5,0) -- (-1.5,-3) (1.5,0) -- (1.5,-3);
         \end{scope}
	\fill[spacecolor](0,0) rectangle(12,12);
	\draw[thick,double] (6,0) -- (0,0) -- (0,12) -- (6,12);
         \draw(6,12) -- (12,12) -- (12,0) -- (6,0);
                  \draw (6,6) node {$H$};
         \begin{scope}[yshift=-15cm]
	\fill[spacecolor](0,0) rectangle(12,12);
         \draw[thick,double] (6,0) -- (0,0) -- (0,12) -- (6,12);
         \draw (6,12) -- (12,12) -- (12,0) -- (6,0);
                  \draw (6,6) node {$L$};
         \end{scope}
	\begin{scope}[xshift=14cm]
	\begin{scope}[xshift=6cm]
         \draw (-1.5,0) -- (1.5,-3) (-1.5,0) -- (-1.5,-3) (1.5,0) -- (-1.5,-3) (1.5,0) -- (1.5,-3);
         \end{scope}
		\fill[spacecolor](0,0) rectangle(12,12);
	\draw (6,0) -- (0,0) -- (0,12) -- (6,12);
         \draw[ultra thick](6,12) -- (12,12) -- (12,0) -- (6,0);
                  \draw (6,6) node {$K$};
         \begin{scope}[yshift=-15cm]
	\fill[spacecolor](0,0) rectangle(12,12);
         \draw (6,0) -- (0,0) -- (0,12) -- (6,12);
         \draw[ultra thick] (6,12) -- (12,12) -- (12,0) -- (6,0);
                  \draw (6,6) node {$M$};
         \end{scope}
	\end{scope}
%      \draw (6,14.5) node {$\scriptstyle D$} (20,14.5) node {$\scriptstyle E$} (6,-17.5) node {$\scriptstyle P$} (20,-17.5) node {$\scriptstyle Q$};
}\; ;
\]
here, in the left picture the horizontal compositions occur first, followed by the vertical composition, whereas in the right picture, the vertical compositions occur first, followed by the horizontal composition.  This fundamental ``interchange isomorphism" is constructed in Section~\ref{subsec:interchange}.  

The construction leverages the special case of the isomorphism where all four sectors are vacuum sectors.  That isomorphism is defined as the following composite,
\[
\tikzmath[scale=\squarescale]
	{
	\begin{scope}[xshift=12cm]
         \draw (-1.5,0) -- (1.5,-3) (-1.5,0) -- (-1.5,-3) (1.5,0) -- (-1.5,-3) (1.5,0) -- (1.5,-3);
         \end{scope}
	\fill[vacuumcolor]  (0,0) rectangle  (24,12);
      \draw[thick, double]  (6,0) -- (0,0) -- (0,12) -- (6,12);
      \draw (6,0) -- (18,0) (12,0) -- (12,12) 
            (6,12) -- (18,12); 
      \draw[ultra thick]  (18,12) -- (24,12) -- (24,0) -- (18,0);
	\begin{scope}[yshift=-15cm]
	\fill[vacuumcolor]  (0,0) rectangle  (24,12);
      \draw[thick, double]  (6,0) -- (0,0) -- (0,12) -- (6,12);
      \draw (6,0) -- (18,0) (12,0) -- (12,12) 
            (6,12) -- (18,12); 
      \draw[ultra thick]  (18,12) -- (24,12) -- (24,0) -- (18,0);
	\end{scope}
}
\;\cong\;
    \tikzmath[scale=\squarescale]
    {   
	\begin{scope}[xshift=12cm]
         \draw (-1.5,0) -- (1.5,-3) (-1.5,0) -- (-1.5,-3) (1.5,0) -- (-1.5,-3) (1.5,0) -- (1.5,-3);
         \end{scope}
    \fill[vacuumcolor]  (0,0) rectangle  (24,12);
        \draw[thick, double] (6,0) -- (0,0) -- (0,12) -- (6,12);
        \draw  (6,0) -- (18,0)  
            (6,12) -- (18,12); 
        \draw[ultra thick]  (18,12) -- (24,12) -- (24,0) -- (18,0);
	\begin{scope}[yshift=-15cm]
\fill[vacuumcolor]  (0,0) rectangle  (24,12);
        \draw[thick, double] (6,0) -- (0,0) -- (0,12) -- (6,12);
        \draw  (6,0) -- (18,0)  
            (6,12) -- (18,12); 
        \draw[ultra thick]  (18,12) -- (24,12) -- (24,0) -- (18,0);
	\end{scope}
    }
\;\cong\;
    \tikzmath[scale=\squarescale]
    {   \fill[vacuumcolor]  (0,0) rectangle  (24,12);
        \draw[thick, double] (6,0) -- (0,0) -- (0,12) -- (6,12);
        \draw  (6,0) -- (18,0)  
            (6,12) -- (18,12); 
        \draw[ultra thick]  (18,12) -- (24,12) -- (24,0) -- (18,0);
    }
\;\cong\;
\tikzmath[scale=\squarescale]
	{
	\fill[vacuumcolor]  (0,0) rectangle  (24,12);
      \draw[thick, double]  (6,0) -- (0,0) -- (0,12) -- (6,12);
      \draw (6,0) -- (18,0) (12,0) -- (12,12) 
            (6,12) -- (18,12); 
      \draw[ultra thick]  (18,12) -- (24,12) -- (24,0) -- (18,0);
}
\;\cong\;
\tikzmath[scale=\squarescale]	
{
	\begin{scope}[xshift=6cm]
         \draw (-1.5,0) -- (1.5,-3) (-1.5,0) -- (-1.5,-3) (1.5,0) -- (-1.5,-3) (1.5,0) -- (1.5,-3);
         \end{scope}
	\fill[vacuumcolor](0,0) rectangle(12,12);
	\draw[thick,double] (6,0) -- (0,0) -- (0,12) -- (6,12);
         \draw(6,12) -- (12,12) -- (12,0) -- (6,0);
         \begin{scope}[yshift=-15cm]
	\fill[vacuumcolor](0,0) rectangle(12,12);
         \draw[thick,double] (6,0) -- (0,0) -- (0,12) -- (6,12);
         \draw (6,12) -- (12,12) -- (12,0) -- (6,0);
         \end{scope}
	\begin{scope}[xshift=14cm]
	\begin{scope}[xshift=6cm]
         \draw (-1.5,0) -- (1.5,-3) (-1.5,0) -- (-1.5,-3) (1.5,0) -- (-1.5,-3) (1.5,0) -- (1.5,-3);
         \end{scope}
	\fill[vacuumcolor](0,0) rectangle(12,12);
	\draw (6,0) -- (0,0) -- (0,12) -- (6,12);
         \draw[ultra thick] (6,12) -- (12,12) -- (12,0) -- (6,0);
         \begin{scope}[yshift=-15cm]
	\fill[vacuumcolor](0,0) rectangle(12,12);
         \draw (6,0) -- (0,0) -- (0,12) -- (6,12);
         \draw[ultra thick] (6,12) -- (12,12) -- (12,0) -- (6,0);
         \end{scope}
	\end{scope}
}\; ;
\]
here the middle rectangle
  \begin{equation} \label{eq:H_0-of-fusion-intro}
    \tikzmath[scale=\squarescale]
    {   \fill[vacuumcolor]  (0,0) rectangle  (24,12);
        \draw[thick, double] (6,0) -- (0,0) -- (0,12) -- (6,12);
        \draw  (6,0) -- (18,0)  
            (6,12) -- (18,12); 
        \draw[ultra thick]  (18,12) -- (24,12) -- (24,0) -- (18,0);
        \draw (-3,6) node {$\scriptstyle\cala$} (27,6) node {$\scriptstyle\calc$}
             (12,6) node {$H_0$}  
             (12,14.5)  node {$\scriptstyle D \circledast_\calb E$} 
             (12,-2.5)  node {$\scriptstyle D \circledast_\calb E$};  
    }%tikzmath
  \end{equation}
denotes the vacuum sector $H_0(S^1, D \circledast_\calb E)$ of the composite defect $D \circledast_\calb E$.  The second and fourth isomorphisms in this composite are simply applications of the aforementioned property that the vacuum is an identify for vertical composition.  The first and third isomorphisms are instances of the much more difficult fact that the horizontal composition of two vacuum sectors on defects is itself isomorphic to the vacuum sector of the composite defect.  Because the horizontal composition is a Connes fusion, and the vacuum sector is a vertical unit, we call this isomorphism the ``one times one isomorphism".

\section{The $1 \boxtimes 1$-isomorphism}
%\addtocontents{toc}{\SkipTocEntry}

  This isomorphism provides a canonical identification
  of the Hilbert space~\eqref{eq:H_0-H_0-intro},
  used to define the defect $D \circledast_\calb E$, 
  with the vacuum sector for $D \circledast_\calb E$.
  By definition the vacuum sector is $H_0(S^1,D \circledast_\calb E) 
      := L^2 ( D \circledast_\calb E(S^1_\top) )$.

  By construction the algebra $D \circledast_\calb E(S^1_\top)$ contains 
  $D(\tikzmath[scale=\textscale]
        {\draw[thick, double] (0,6) -- (0,12) -- (6,12); 
         \draw (6,12) -- (12,12); })$
  and 
  $E( \tikzmath[scale=\textscale]
        { \draw (12,12) -- (18,12);
          \draw[ultra thick] (18,12) -- (24,12) -- (24,6);} )$
  as two commuting subalgebras and is generated by those subalgebras.
  We can think of the algebra $D \circledast_\calb E(S^1_\top)$
  as associated to the tricolored interval 
  $\tikzmath[scale=\textscale]
        { \draw[thick, double] (0,6) -- (0,12) -- (6,12);
          \draw (6,12) -- (18,12);
          \draw[ultra thick] (18,12) -- (24,12) -- (24,6);}$
  which is the upper half of the circle $\dd ([0,2] \x [0,1])$;
  it is therefore natural, as above, to draw the vacuum sector for
  $D \circledast_\calb E$ as
$
\tikzmath[scale=\textscale]
	{
	\fill[vacuumcolor]  (0,0) rectangle  (24,12);
      \draw[thick, double]  (6,0) -- (0,0) -- (0,12) -- (6,12);
      \draw (6,0) -- (18,0) 
            (6,12) -- (18,12); 
      \draw[ultra thick]  (18,12) -- (24,12) -- (24,0) -- (18,0);
}
$.
%  In the language of the $3$-category the vacuum sector for a defect
%  $D$ is the identity $2$-morphism, and a fusion along the middle vertical 
%  line as 
%  in~\eqref{eq:H_0-H_0-intro}
%  is the horizontal composition of $2$-morphisms.
%  Thus~\eqref{eq:H_0-H_0-intro} is the composition of the identities
%  for the defects $D$ and $E$,
%  while~\eqref{eq:H_0-of-fusion-intro} is the identity for the    
%  composition $D \circledast_\calb E$.
%  For this reason, we refer to the desired isomorphism
%  between~\eqref{eq:H_0-H_0-intro} and~\eqref{eq:H_0-of-fusion-intro} as
%  the ``one times one isomorphism''.
  
    Another main result of this book is the existence of an isomorphism between the fusion~\eqref{eq:H_0-H_0-intro} and the vacuum~\eqref{eq:H_0-of-fusion-intro}.  We presume as before that the net $\calb$ has finite index.
\begin{introthm}[The $1 \boxtimes 1$-isomorphism]
There is a canonical isomorphism between the vacuum sector $H_0(D \circledast_\calb E)$ of the fused defect $D \circledast_\calb E$ and the Connes fusion $H_0(D) \boxtimes_{\calb(I)} H_0(E)$ of the two vacuum sectors of the defects.
\end{introthm}
\noindent This result appears in the text as Theorem~\ref{thm: Omega is an iso}.  The construction of the $1 \boxtimes 1$-iso\-mor\-phism is quite involved
  and is carried out in Chapters \ref{sec:F-G-and-G_0}, \ref{sec: Haag duality for composition of defects}, and \ref{sec:1-box-1}.
  Chapter~\ref{sec:1-box-1} also contains a short summary, on page \pageref{THE RECAP OF OMEGA},
  collecting all the necessary ingredients in one place.
  
\section{Construction of the $1 \boxtimes 1$-isomorphism}
%\addtocontents{toc}{\SkipTocEntry}
  For any von Neumann algebra $A$ the standard form $L^2(A)$ 
  carries commuting left and right actions of $A$,
  i.e., $L^2(A)$ is an $A$--$A$-bimodule.
  In the case of the vacuum sector 
  $H_0(S^1,D) = L^2(D(\tikzmath[scale=\textscale]
        {\draw[thick, double] (0,6) -- (0,12) -- (6,12);
         \draw (6,12) -- (12,12) -- (12,6); }))$ 
  these two actions correspond to the left actions 
  of $D(\tikzmath[scale=\textscale]
        {\draw[thick, double] (0,6) -- (0,12) -- (6,12);
         \draw (6,12) -- (12,12) -- (12,6); })$
  and of $D(\tikzmath[scale=\textscale]
        {\draw[thick, double] (0,6) -- (0,0) -- (6,0);
         \draw (6,0) -- (12,0) -- (12,6); })$.\footnote{
  The reflection along the horizontal axis $\IR \x \{ \frac{1}{2} \}$
  provides an orientation reversing identification
  $\tikzmath[scale=\textscale]
        {\draw[thick, double]  (0,6) -- (0,12) -- (6,12);
         \draw (6,12) -- (12,12) -- (12,6); }
   \to \tikzmath[scale=\textscale]
        {\draw[thick, double] (0,6) -- (0,0) -- (6,0);
         \draw (6,0) -- (12,0) -- (12,6); }$
  and this accounts for the fact that the right action of
  $D(\tikzmath[scale=\textscale]
        {\draw[thick, double] (0,6) -- (0,12) -- (6,12);
         \draw (6,12) -- (12,12) -- (12,6); })$
  on $L^2(D(\tikzmath[scale=\textscale]
        {\draw[thick, double] (0,6) -- (0,12) -- (6,12);
         \draw (6,12) -- (12,12) -- (12,6); }))$
  corresponds to a left action of    
  $D(\tikzmath[scale=\textscale]
        {\draw[thick, double] (0,6) -- (0,0) -- (6,0);
         \draw (6,0) -- (12,0) -- (12,6); })$
  on $H_0(S^1,D)$.}
  One difficulty in understanding the 
  Connes fusion~\eqref{eq:H_0-H_0-intro} comes from the fact that 
  the algebra $\calb(I)$, over which the Connes fusion is taken,
  intersects both $D(\tikzmath[scale=\textscale]
        {\draw[thick, double] (0,6) -- (0,12) -- (6,12);
         \draw (6,12) -- (12,12) -- (12,6); })$
  and $D(\tikzmath[scale=\textscale]
        {\draw[thick, double] (0,6) -- (0,0) -- (6,0);
         \draw (6,0) -- (12,0) -- (12,6); })$.
  To simplify the situation we will consider a variation 
  of~\eqref{eq:H_0-H_0-intro} with a hole in the middle,
  \begin{equation}
    \label{eq:H_0-H_0-with-hole-fusion-intro}
    \tikzmath[scale=\squarescale]
      {    \fill[vacuumcolor] (0,0) rectangle (10,12)
                             (14,0) rectangle (24,12) (10,0) rectangle (14,4)  
                              (10,8) rectangle (14,12); 
           \draw (6,12) -- (10,12) -- (10,0) -- (6,0) 
                 (18,12) -- (14,12) -- (14,0) -- (18,0)
                 (10,0) -- (14,0) (10,4) -- (14,4)  
                 (10,8) -- (14,8) (10,12) -- (14,12);
           \draw[thick, double] (6,12) -- (0,12) -- (0,0) -- (6,0);
           \draw[ultra thick] (18,12) -- (24,12) -- (24,0) -- (18,0);
      } %tikzmath
      \;;
  \end{equation}
  we refer to this construction as
  keyhole fusion.
  This Hilbert space is built from vacuum sectors for $D$ and $E$ together
  with two (small) copies of the vacuum sector for $\calb$. 
  Its formal definition is given in Chapter~\ref{sec:F-G-and-G_0}; see in particular~\eqref{eq: pictorial notation 1}.
  The Connes fusion of $\calb(I)$ is now replaced by four Connes fusion
  operations along smaller algebras.
  This allows us to identify, in Theorem~\ref{thm:G_0=L^2},
  the Hilbert space~\eqref{eq:H_0-H_0-with-hole-fusion-intro}
  with the $L^2$-space of a certain von Neumann algebra that we represent by the graphical notation
  $\tikzmath[scale=\textscale]
               { \useasboundingbox (-2,0) rectangle (26,14);
                 \draw[thick, double] (0,6) -- (0,12) -- (6,12);
                 \draw (6,12) -- (18,12) 
                       (10,6) -- (10,8) -- (14,8) -- (14,6);
                 \draw[ultra thick] (18,12) -- (24,12) -- (24,6); 
                 \draw[densely dotted] (12,8) -- (12,12); 
               }$. %tikzmath
  It is generated by algebras 
  $D(\tikzmath[scale=\textscale]
               { \useasboundingbox (-2,0) rectangle (12,14);
                 \draw[thick, double] (0,6) -- (0,12) -- (6,12);
                 \draw (6,12) -- (10,12) 
                       (10,6) -- (10,8); 
               })$, %tikzmath
  $\hat\calb(\tikzmath[scale=\textscale]
               { \useasboundingbox (8,0) rectangle (16,14);
                 \draw (10,12) -- (14,12) 
                       (10,8) -- (14,8);
%                 \draw[densely dotted] (12,8) -- (12,12); 
               })$, and %tikzmath
  $E(\tikzmath[scale=\textscale]
               { \useasboundingbox (12,0) rectangle (26,14);
                 \draw (14,12) -- (18,12) 
                       (14,8) -- (14,6);
                 \draw[ultra thick] (18,12) -- (24,12) -- (24,6); 
               })$ %tikzmath
  acting on the Hilbert space 
  $\tikzmath[scale=\textscale]
      {    \useasboundingbox (-2,-1) rectangle (26,14); \fill[vacuumcolor] (0,0) rectangle (10,12) (14,0) rectangle (24,12); 
\draw[ultra thin] (6,12) -- (10,12) -- (10,0) -- (6,0) (18,12) -- (14,12) -- (14,0) -- (18,0); \draw[thick, double] (6,12) -- (0,12) -- (0,0) -- (6,0);
\draw[ultra thick] (18,12) -- (24,12) -- (24,0) -- (18,0); \filldraw[fill = vacuumcolor, ultra thin] (10,8) rectangle (14,12);}%tikzmath
$; here $\hat\calb(\tikzmath[scale=\textscale]
               { \useasboundingbox (8,0) rectangle (16,14);
                 \draw (10,12) -- (14,12) 
                       (10,8) -- (14,8);
%                 \draw[densely dotted] (12,8) -- (12,12); 
               })$ is a certain enlargement of the algebra
               $\calb(\tikzmath[scale=\textscale]
               { \useasboundingbox (8,0) rectangle (16,14);
                 \draw (10,12) -- (14,12) 
                       (10,8) -- (14,8);
%                 \draw[densely dotted] (12,8) -- (12,12); 
               })$ that we abbreviate graphically by 
  $\tikzmath[scale=\textscale]
               { \useasboundingbox (8,0) rectangle (16,14);
                 \draw (10,12) -- (14,12) 
                       (10,8) -- (14,8);
                 \draw[densely dotted] (12,8) -- (12,12); 
               }$.
We defer to (\ref{eq:  J_l  J_r  K_u  J_u}, \ref{eq: pictorial notation 2}) for the details of the definitions, and to \ref{not: Dhat} for an explanation of the notation~$\hat\calb$.
  Theorem~\ref{thm:G_0=L^2} then reads
  \begin{equation}
    \label{eq:G_0=L^2-intro}
    L^2 \left(
           \tikzmath[scale=\squarescale]
               { \useasboundingbox (-2,0) rectangle (26,14);
                 \draw[thick, double] (0,6) -- (0,12) -- (6,12);
                 \draw (6,12) -- (18,12) 
                       (10,6) -- (10,8) -- (14,8) -- (14,6);
                 \draw[ultra thick] (18,12) -- (24,12) -- (24,6); 
                 \draw[densely dotted] (12,8) -- (12,12); 
               }  
        \right)
    \; \cong \;
    \tikzmath[scale=\squarescale]
      {    \fill[vacuumcolor] (0,0) rectangle (10,12)
                             (14,0) rectangle (24,12) (10,0) rectangle (14,4)  
                              (10,8) rectangle (14,12); 
           \draw (6,12) -- (10,12) -- (10,0) -- (6,0) 
                 (18,12) -- (14,12) -- (14,0) -- (18,0)
                 (10,0) -- (14,0) (10,4) -- (14,4)  
                 (10,8) -- (14,8) (10,12) -- (14,12);
           \draw[thick, double] (6,12) -- (0,12) -- (0,0) -- (6,0);
           \draw[ultra thick] (18,12) -- (24,12) -- (24,0) -- (18,0);
      } \; . %tikzmath  
  \end{equation}
  
  At this point we blur the distinction between intervals 
  and algebras in our graphical notation
  and often draw only an interval to denote an algebra. 
  For example we abbreviate 
  $\calb( \, \tikzmath[scale=\textscale]
               { \draw  (10,6) -- (10,8) -- (14,8) -- (14,6); } \, )$ 
as simply
    $\tikzmath[scale=\textscale]
               { \draw  (10,6) -- (10,8) -- (14,8) -- (14,6); }$, and
  $D \circledast_\calb E (S^1_\top)$
  as 
    $\tikzmath[scale=\textscale]
               { \useasboundingbox (-2,0) rectangle (26,14);
                 \draw[thick, double] (0,6) -- (0,12) -- (6,12);
                 \draw (6,12) -- (18,12);
                 \draw[ultra thick] (18,12) -- (24,12) -- (24,6);
               }$.
  We therefore write, for instance
    $D \circledast_\calb E (S^1_\top)  \otimes 
     \calb(  \tikzmath[scale=\squarescale]
               { \draw  (10,6) -- (10,8) -- (14,8) -- (14,6); } )$ 
as
    $\tikzmath[scale=\textscale]
               { \useasboundingbox (-2,0) rectangle (26,14);
                 \draw[thick, double] (0,6) -- (0,12) -- (6,12);
                 \draw (6,12) -- (18,12) 
                       (10,6) -- (10,8) -- (14,8) -- (14,6);
                 \draw[ultra thick] (18,12) -- (24,12) -- (24,6); 
               }$. %tikzmath
  As the notation indicates, this tensor product is a subalgebra
  of $\tikzmath[scale=\textscale]
               { \useasboundingbox (-2,0) rectangle (26,14);
                 \draw[thick, double] (0,6) -- (0,12) -- (6,12);
                 \draw (6,12) -- (18,12) 
                       (10,6) -- (10,8) -- (14,8) -- (14,6);
                 \draw[ultra thick] (18,12) -- (24,12) -- (24,6); 
                 \draw[densely dotted] (12,8) -- (12,12); 
               }$.
  (Note the additional dotted line in the middle.)
  If $\calb$ has finite index, then we show in 
  Corollary~\ref{cor: is a finite homomorphism of von Neumann algebra}
  that this inclusion
  $\tikzmath[scale=\textscale]
               { \useasboundingbox (-2,0) rectangle (26,14);
                 \draw[thick, double] (0,6) -- (0,12) -- (6,12);
                 \draw (6,12) -- (18,12) 
                       (10,6) -- (10,8) -- (14,8) -- (14,6);
                 \draw[ultra thick] (18,12) -- (24,12) -- (24,6); 
               } \subseteq
   \tikzmath[scale=\textscale]
               { \useasboundingbox (-2,0) rectangle (26,14);
                 \draw[thick, double] (0,6) -- (0,12) -- (6,12);
                 \draw (6,12) -- (18,12) 
                       (10,6) -- (10,8) -- (14,8) -- (14,6);
                 \draw[ultra thick] (18,12) -- (24,12) -- (24,6); 
                 \draw[densely dotted] (12,8) -- (12,12); 
               }$
  is a finite homomorphism of von Neumann algebras.
  As the $L^2$-construction is functorial for such homomorphisms \cite{BDH(Dualizability+Index-of-subfactors)},
  we can apply $L^2$ to it. 
  Combining this with~\eqref{eq:G_0=L^2-intro}  we obtain a map
  \begin{equation}
    \label{eq:no-dotted-in-dotted-intro}
    L^2 \left( 
          \tikzmath[scale=\squarescale]
               { \useasboundingbox (-2,0) rectangle (26,14);
                 \draw[thick, double] (0,6) -- (0,12) -- (6,12);
                 \draw (6,12) -- (18,12) 
                       (10,6) -- (10,8) -- (14,8) -- (14,6);
                 \draw[ultra thick] (18,12) -- (24,12) -- (24,6); 
               }
        \right)  
    \; \to \;
    L^2 \left(
        \tikzmath[scale=\squarescale]
               { \useasboundingbox (-2,0) rectangle (26,14);
                 \draw[thick, double] (0,6) -- (0,12) -- (6,12);
                 \draw (6,12) -- (18,12) 
                       (10,6) -- (10,8) -- (14,8) -- (14,6);
                 \draw[ultra thick] (18,12) -- (24,12) -- (24,6); 
                 \draw[densely dotted] (12,8) -- (12,12); 
               }  %tikzmath
       \right)
    \; \cong \;
    \tikzmath[scale=\squarescale]
      {    \fill[vacuumcolor] (0,0) rectangle (10,12)
                             (14,0) rectangle (24,12) (10,0) rectangle (14,4)  
                              (10,8) rectangle (14,12); 
           \draw (6,12) -- (10,12) -- (10,0) -- (6,0) 
                 (18,12) -- (14,12) -- (14,0) -- (18,0)
                 (10,0) -- (14,0) (10,4) -- (14,4)  
                 (10,8) -- (14,8) (10,12) -- (14,12);
           \draw[thick, double] (6,12) -- (0,12) -- (0,0) -- (6,0);
           \draw[ultra thick] (18,12) -- (24,12) -- (24,0) -- (18,0);
      } \; . %tikzmath  
  \end{equation}
  
  In the next step, we need to fill the keyhole 
  in~\eqref{eq:H_0-H_0-with-hole-fusion-intro}. 
  Formally, this is done by applying Connes fusion with a further
  (small) vacuum sector for $\calb$.
  We fancy this vacuum sector as the keystone and refer to the result as
  keystone fusion.
  On the domain of~\eqref{eq:no-dotted-in-dotted-intro} the keystone 
  cancels the algebra 
  $\tikzmath[scale=\textscale]
               { \draw  (10,6) -- (10,8) -- (14,8) -- (14,6); }$.
  On the target, we simply denote the result by filling the keyhole
  with the keystone, the (small) vacuum sector for $\calb$.
  In this way we obtain, in Proposition~\ref{prop:G=L2}, an isometric
  embedding
  \begin{equation}
    \label{eq:L^2=G-intro}
    \tikzmath[scale=\squarescale]
    {   \fill[vacuumcolor]  (0,0) rectangle  (24,12);
        \draw[thick, double] (6,0) -- (0,0) -- (0,12) -- (6,12);
        \draw  (6,0) -- (18,0)  
            (6,12) -- (18,12); 
        \draw[ultra thick]  (18,12) -- (24,12) -- (24,0) -- (18,0); 
    }
    \; = \;
    L^2 \left( \tikzmath[scale=\squarescale]
	  {\useasboundingbox (-2,0) rectangle (26,14);
           \draw[thick, double] (0,6) -- (0,12) -- (6,12);
           \draw (6,12) -- (18,12);
           \draw[ultra thick] (18,12) -- (24,12) -- (24,6);} %tikzmath
	\right) 
    \;\; \rightarrow \;\;
    \tikzmath[scale=\squarescale]
	{\fill[vacuumcolor](0,0) rectangle (10,12)(14,0) rectangle (24,12); 
         \draw (6,12) -- (10,12) -- (10,0) -- (6,0) 
            (18,12) -- (14,12) -- (14,0) -- (18,0);
	 \draw[thick, double] (6,12) -- (0,12) -- (0,0) -- (6,0);
         \draw[ultra thick] (18,12) -- (24,12) -- (24,0) -- (18,0);
         \fill[vacuumcolor]  (10,0) rectangle (14,4)  
                    (10,8) rectangle (14,12);
	 \draw (10,0) rectangle (14,4) (10,8) rectangle (14,12);
         \filldraw[fill=vacuumcolor]  (10.5,4.5) rectangle (13.5,7.5);
        }\;. %tikzmath 
  \end{equation}
  The existences of this isometric embedding enables us to prove
  that $D \circledast_\calb E$ is a defect. 
  To produce the $1 \boxtimes 1$-isomorphism from~\eqref{eq:L^2=G-intro}, we construct, in Proposition \ref{prop: local-fusion}, an isomorphism
  \begin{equation}
    \label{eq:fill-hole-intro}
    \tikzmath[scale=\squarescale]
    	{\fill[vacuumcolor](0,0) rectangle (10,12)(14,0) rectangle (24,12); 
         \draw (6,12) -- (10,12) -- (10,0) -- (6,0) 
            (18,12) -- (14,12) -- (14,0) -- (18,0);
	 \draw[thick, double] (6,12) -- (0,12) -- (0,0) -- (6,0);
         \draw[ultra thick] (18,12) -- (24,12) -- (24,0) -- (18,0);
         \fill[vacuumcolor]  (10,0) rectangle (14,4)  
                    (10,8) rectangle (14,12);
	 \draw (10,0) rectangle (14,4) (10,8) rectangle (14,12);
         \filldraw[fill=vacuumcolor]  (10.5,4.5) rectangle (13.5,7.5);
        }
    \; \cong \;
    \tikzmath[scale=\squarescale]
     {   \fill[vacuumcolor]  (0,0) rectangle  (24,12);
         \draw[thick, double]  (6,0) -- (0,0) -- (0,12) -- (6,12);
         \draw (6,0) -- (18,0) (12,0) -- (12,12) 
            (6,12) -- (18,12); 
         \draw[ultra thick]  (18,12) -- (24,12) -- (24,0) -- (18,0);
     } \;  %tikzmath 
  \end{equation}
and then define the ``$1 \boxtimes 1$-isomorphism'' $\Omega$ as the composite of the two maps~\eqref{eq:L^2=G-intro} 
  and~\eqref{eq:fill-hole-intro}.
  
  It remains to prove that the composite 
  of~\eqref{eq:L^2=G-intro} and~\eqref{eq:fill-hole-intro} 
  is indeed an isomorphism.
  The proof proceeds as follows: both the domain
  $\tikzmath[scale=\textscale]
    {   \fill[vacuumcolor]  (0,0) rectangle  (24,12);
        \draw[thick, double] (6,0) -- (0,0) -- (0,12) -- (6,12);
        \draw  (6,0) -- (18,0)  
            (6,12) -- (18,12); 
        \draw[ultra thick]  (18,12) -- (24,12) -- (24,0) -- (18,0); 
    }$
  and the target 
  $\tikzmath[scale=\textscale]
     {   \fill[vacuumcolor]  (0,0) rectangle  (24,12);
         \draw[thick, double]  (6,0) -- (0,0) -- (0,12) -- (6,12);
         \draw (6,0) -- (18,0) (12,0) -- (12,12) 
            (6,12) -- (18,12); 
         \draw[ultra thick]  (18,12) -- (24,12) -- (24,0) -- (18,0);
     }$
  of $\Omega$ carry commuting actions of the algebras 
  $(D \circledast_\calb E)(S^1_\top) = 
     \tikzmath[scale=\textscale]
	  {\useasboundingbox (-2,0) rectangle (26,14);
           \draw[thick, double] (0,6) -- (0,12) -- (6,12);
           \draw (6,12) -- (18,12);
           \draw[ultra thick] (18,12) -- (24,12) -- (24,6);}$
  and 
  $(D \circledast_\calb E)(S^1_\bot) = 
     \tikzmath[scale=\textscale]
	  {\useasboundingbox (-2,0) rectangle (26,14);
           \draw[thick, double] (0,6) -- (0,0) -- (6,0);
           \draw (6,0) -- (18,0);
           \draw[ultra thick] (18,0) -- (24,0) -- (24,6);}$\,.
  On $\tikzmath[scale=\textscale]
    {   \fill[vacuumcolor]  (0,0) rectangle  (24,12);
        \draw[thick, double] (6,0) -- (0,0) -- (0,12) -- (6,12);
        \draw  (6,0) -- (18,0)  
            (6,12) -- (18,12); 
        \draw[ultra thick]  (18,12) -- (24,12) -- (24,0) -- (18,0); 
    }
    \; = \;
    L^2 \left( \tikzmath[scale=\textscale]
	  {\useasboundingbox (-2,0) rectangle (26,14);
           \draw[thick, double] (0,6) -- (0,12) -- (6,12);
           \draw (6,12) -- (18,12);
           \draw[ultra thick] (18,12) -- (24,12) -- (24,6);} %tikzmath
	\right)$ 
  these two actions are clearly each other's commutants and so
  to prove that $\Omega$ is an isomorphism it suffices to 
  show that the same holds for 
  $H_0(D) \boxtimes_{\calb(I)} H_0(E) = 
   \tikzmath[scale=\textscale]
     {   \fill[vacuumcolor]  (0,0) rectangle  (24,12);
         \draw[thick, double]  (6,0) -- (0,0) -- (0,12) -- (6,12);
         \draw (6,0) -- (18,0) (12,0) -- (12,12) 
            (6,12) -- (18,12); 
         \draw[ultra thick]  (18,12) -- (24,12) -- (24,0) -- (18,0);
     }$\,.  That these two actions are each other's commutants on this fusion of vacuum sectors, provided as before that the intermediate net has finite index, is a main technical results of this book:

  \begin{introthm}[Haag duality for fusion of defects]
    The algebras $(D \circledast_\calb E)(S^1_\top)$ and 
    $(D \circledast_\calb E)(S^1_\bot)$, associated by the defect 
    $D \circledast_\calb E$ to the two halves of the circle, 
    are each other's commutants in their action on the fusion 
    $H_0(D) \boxtimes_{\calb(I)} H_0(E)$ of the vacuum sectors of the defects.
  \end{introthm}

\noindent  This is established in the text as Theorem~\ref{thm:Haag-duality-composition-defects}; 
  (see also Corollary~\ref{cor:fiber-product-and-nets}).
  All of Chapter~\ref{sec: Haag duality for composition of defects} 
  is devoted to its proof.

\begin{remark*}  
In constructing the $3$-category of conformal nets, it is essential to know that the $1 \boxtimes 1$-isomorphism $\Omega$
satisfies certain axioms, such as associativity.  In Proposition~\ref{prop: associativity of Omega} we prove that the isomorphism is
appropriately associative, but unfortunately this is done directly by tracing through the entire construction of 
$\Omega$.  Better would be to use a \emph{characterization} of $\Omega$ (and thus of composites of multiple $\Omega$ 
maps) as the unique map satisfying certain properties.  Haagerup's standard form (that is, the $L^2$-space of a von 
Neumann algebra) does admit such a characterization: it is determined up to unique unitary isomorphism by the module
structure, the modular conjugation, and a self-dual cone.  There is a natural choice of modular conjugation on
$\tikzmath[scale=\textscale]
     {   \fill[vacuumcolor]  (0,0) rectangle  (24,12);
         \draw[thick, double]  (6,0) -- (0,0) -- (0,12) -- (6,12);
         \draw (6,0) -- (18,0) (12,0) -- (12,12) 
            (6,12) -- (18,12); 
         \draw[ultra thick]  (18,12) -- (24,12) -- (24,0) -- (18,0);
     }$.
Thus, to characterize the isomorphism $\Omega$, it suffices to specify a self-dual cone in that fusion of vacuum sectors.  Unfortunately, we do not know how to construct such a self-dual cone from the self-dual cones of
$\tikzmath[scale=\textscale]
     {   \fill[vacuumcolor]  (0,0) rectangle  (12,12);
         \draw[thick, double]  (6,0) -- (0,0) -- (0,12) -- (6,12);
         \draw (6,0) --  (12,0) -- (12,12) -- (6,12); 
     }$  
and of
  $\tikzmath[scale=\textscale]
     {   \fill[vacuumcolor]  (12,0) rectangle  (24,12);
         \draw (18,0) -- (12,0) -- (12,12) -- (18,12); 
         \draw[ultra thick]  (18,12) -- (24,12) -- (24,0) -- (18,0);
     }$.
\end{remark*}

\renewcommand{\thesection}{\arabic{chapter}.{\sc\alph{section}}}

\chapter{Defects} \label{sec:defects}

\section{Bicolored intervals and circles}

An interval is a smooth oriented $1$-manifold diffeomorphic to $[0,1]$.
We write $\Diff(I)$ for the group of diffeomorphisms of $I$ and $\Diff_0(I)$ for the subgroup that fixes a neighborhood of $\dd I$.
A \emph{bicolored interval} is an interval $I$ (always oriented)
equipped with a cover by two closed, connected, possibly empty subsets $I_{\circ}, I_{\bullet}\subset I$ with disjoint interiors, 
along with a local coordinate (that is, an embedding $(-\e,\e) \hookrightarrow I$) at $I_{\circ} \cap I_{\bullet}$.  We disallow the cases when $I_{\circ}$ or $I_{\bullet}$ consist of a single point.
The local coordinate does {\it not} need to preserve the orientation, but is required to send $(-\e,0]$ into $I_{\circ}$ and $[0,\e)$ into $I_{\bullet}$.
A bicolored interval necessarily falls in one of the following three classes:

\begin{numberlist}
\item[\label{nl:bicolored:bicolored}] $I_{\circ}$, $I_{\bullet}$ are intervals and $I_{\circ}\cap I_{\bullet}$ is a point;
the local coordinate is a smooth embedding $(-\e,\e) \hookrightarrow I$ that sends $(-\e,0]$ to $I_{\circ}$ and $[0,\e)$ to $I_{\bullet}$;
\item[\label{nl:bicolored:source}] $I_{\circ}=I$, $I_{\bullet}=\emptyset$, and there is no data of   local coordinate;
\item[\label{nl:bicolored:target}] $I_{\bullet}=I$, $I_{\circ}=\emptyset$, and there is no data of local coordinate.
\end{numberlist}
An embedding $f \colon J\hookrightarrow I$ from one bicolored interval to another is called color preserving if $f^{-1}(I_{\circ})= J_{\circ}$ and $f^{-1}(I_{\bullet}) = J_{\bullet}$.
The bicolored intervals form a category $\INT_{\circ\bullet}$, whose morphisms are the color preserving embeddings that respect the local coordinates (that is, such that the embedding intertwines the local coordinates on a sufficiently small neighborhood of 0).
We let $\INT_{\circ}$ and $\INT_{\bullet}$ be the full subcategories on the objects of the form \eqref{nl:bicolored:source} and \eqref{nl:bicolored:target}, respectively.
Both of them are canonically isomorphic to $\INT$, the category of uncolored intervals.
Elements of $\INT_{\circ}$ and $\INT_{\bullet}$ are called \emph{white} intervals and \emph{black} intervals, respectively.
Those of type \eqref{nl:bicolored:bicolored} are called \emph{genuinely bicolored} intervals.
The full subcategory of genuinely bicolored intervals is denoted $\INT_{\halfbullet}$.

Similarly, a \emph{bicolored circle} $S$ is a circle (always oriented) equipped with a cover by two closed, connected, possibly empty subsets with disjoint interiors $S_{\circ}, S_{\bullet}\subset S$, 
along with local coordinates in the neighborhood of $S_{\circ} \cap S_{\bullet}$.
We disallow the cases when $S_{\circ}$ or $S_{\bullet}$ consists of a single point.
A bicolored circle necessarily falls in one of the following three categories:
\begin{numberlist}
\item[\label{nl: circle genuine}] $S_{\circ}$ and $S_{\bullet} \subset S$ are intervals.
The intersection $S_{\circ} \cap S_{\bullet}$ consist of two points
and the two local coordinates are embeddings $(-\e,\e) \hookrightarrow S$ sending $(-\e,0]$ to $S_{\circ}$ and $[0,\e)$ to $S_{\bullet}$.
\item[\label{nl: circle white}] $S_\circ=S$, $S_\bullet =\emptyset$, and there are no local coordinates;
\item[\label{nl: circe black}] $S_\bullet=S$, $S_\circ =\emptyset$, and there are no local coordinates.
\end{numberlist}
Bicolored circles of type \eqref{nl: circle genuine} are called \emph{genuinely bicolored}.

\section{Definition of defects}\label{sec: def defects}

  Let $\VN$ be the category whose objects are 
  von Neumann algebras with separable preduals, 
  and whose morphisms are $\IC$-linear homomorphisms, and 
  $\IC$-linear antihomomorphisms\footnote{An antihomomorphism is a  
               map satisfying $f(1)=1$ and $f(ab) = f(b)f(a)$.}.

Recall our definition of conformal nets (see Appendix \ref{app:nets}).
For the following definition of defect, we do not require that the conformal nets $\cala$ and $\calb$ are irreducible:

  \begin{definition}\label{def:Defect}
    Let $\cala$ and $\calb$ be two conformal nets.
    A defect from $\cala$ to $\calb$ is a functor
    \[
      D \colon \INT_{\circ\bullet}\to \VN
    \]
    that assigns to each bicolored interval $I$ a von 
    Neumann algebra $\cala(I)$,
    and whose restrictions to $\INT_{\circ}$ and $\INT_{\bullet}$ 
    are given by $\cala$ and $\calb$, respectively.
    It sends orientation-preserving embeddings to $\IC$-linear 
    homomorphisms, and orientation-reversing embeddings to 
    $\IC$-linear antihomomorphisms.
    The functor $D$ is subject to the following axioms:
    \begin{enumerate}
      \item \emph{Isotony:} 
         If $I$ and $J$ are genuinely bicolored intervals and 
         $f:J\hookrightarrow I$ is an embedding, 
         then $D(f):D(J)\rightarrow D(I)$ is injective.
      \item \emph{Locality:} 
         If $J \subset I$ and $K\subset I$ 
         have disjoint interiors, 
         then the images of $D(J)$ and $D(K)$ are commuting 
         subalgebras of $D(I)$.
      \item \emph{Strong additivity:} 
         If $I = J \cup K$, then the images of 
         $D(J)$ and $D(K)$ topologically generate $D(I)$.
      \item \label{def:Defect:action-on-L2}
         \emph{Vacuum sector:}
         Let $S$ be a genuinely bicolored circle, 
         $I\subset S$ a genuinely bicolored interval, 
         and $j \colon S\to S$ a color preserving
         orientation reversing involution that fixes $\partial I$.
         Equip $I':=j(I)$ with the orientation induced 
         from $S$, and consider the following two maps of algebras:
        \begin{equation}       \label{eq:actions-on-L2}
        \begin{split} 
         \alpha \quad \colon \quad& D(I)  
            \xrightarrow{} \bfB(L^2D(I))\\
         \beta \quad \colon \quad& D(I') \xrightarrow{D(j)} 
           D(I)^\op \xrightarrow{} \bfB(L^2D(I))
        \end{split}
        \end{equation}
        (Here $\alpha$ is the left action of $D(I)$ on $L^2 D(I)$, and in $\beta$, the map $D(I)^\op \rightarrow \bfB(L^2 D(I))$ is the right action of $D(I)$ on $L^2 D(I)$.)  Let $J\in \INT_{\circ}\cup\INT_{\bullet}$ be a subinterval of $I$
        such that $J \cap \partial I$ consists of a single point, and
        equip $\bar J:=j(J)$ with the orientation induced from $S$.
        We then require that the action
        \begin{equation} 
           \label{eq: l otimes r : D otimes D --> B(L2(D))}
          \alpha\otimes 
          \beta \,\colon D(J) \,\ox_{\alg}\, D(\bar J) \longrightarrow 
          \bfB\big(L^2D(I)\big)
        \end{equation}
        of the algebraic tensor product extends to an action of $D(J\cup \bar J)$.
    \end{enumerate}
    The defect $D$ is said to be \emph{irreducible} if for every genuinely bicolored interval $I$, 
    the algebra $D(I)$ is a factor.
    We will write $_{\cala}D_{\calb}$ to indicate that 
    $D$ is a defect from $\cala$ to $\calb$.

  \end{definition}

Note that in the above definition of a defect $D$, for an embedding $I \to J$ of a white or black interval $I$ into a genuinely bicolored interval $J$, the induced map of von Neumann algebras $D(I) \to D(J)$ is not required to be an injection.
 
The following properties are consequences of the listed axioms and the corresponding properties of conformal nets: inner covariance (Proposition \ref{prop:inncovdef}), the split property (Proposition \ref{prop:split-property-defects}), Haag duality (Proposition \ref{prop: [Haag duality for defects]}), and continuity (Proposition \ref{prop:Continuity for defects}).

\subsection*{\hspace*{-18pt}Inner covariance and the split property}
Recall that $\Diff_0(I)$ is the subgroup of diffeomorphisms of $I$ that fix some neighborhood of $\partial I$.

\begin{proposition}[Inner covariance for defects] \label{prop:inncovdef}
Let $I$ be a genuinely bicolored interval, and let $\varphi\in\Diff_0(I)$ be a diffeomorphism that preserves the bicoloring and the local coordinate.
Then $D(\varphi)$ is an inner automorphism of $D(I)$.
\end{proposition}

\begin{proof}
Write $\varphi=\varphi_\circ\circ\varphi_\bullet$ with $\supp(\varphi_\circ)\subset I_\circ$ and $\supp(\varphi_\bullet)\subset I_\bullet$.
Let $\{J, K, L\}$ be a cover of $I$ such that $J$ is a white interval, $K$ is a genuinely bicolored interval, $L$ is a black interval, $\supp(\varphi_\circ)$ is contained in the interior of $J$,
$\supp(\varphi_\bullet)$ is contained in the interior of $L$, and $\varphi$ acts as the identity on $K$.
By inner covariance for the nets $\cala$ and $\calb$ (see Appendix \ref{subsec:defnets}), there are unitaries $u\in \cala(J)$ and $v\in \calb(L)$ that implement $\varphi_\circ$ and $\varphi_\bullet$.
Let $w$ be their product in $D(I)$.
Then $waw^* = D(\varphi)a$ holds for every $a\in D(I)$ that is in the image of $\cala(J)$, of $D(K)$, or of $\calb(L)$.
By strong additivity, it therefore holds for every element of $D(I)$.
\end{proof}

\begin{proposition}[Split property for defects] \label{prop:split-property-defects}
If $J \subset I$ and $K \subset I$ are disjoint, then the map $D(J) \,\ox_{\alg}\, D(K) \to D(I)$ extends to the spatial tensor product $D(J) \, \bar{\ox} \, D(K)$.
\end{proposition}
   
\begin{proof}
We assume without loss of generality that the interval $J$ is entirely white
and that it does not meet the boundary of $I$ (otherwise, replace $I$ by a slightly larger interval).
Let $J^+\subset I$ be a white interval that contains $J$ in its interior and that does not intersect $K$.
Finally, let $\iota:\cala(J^+)\to D(I)$ be the map induced by the inclusion $J^+\hookrightarrow I$.
By the split property and Haag duality for conformal nets, the inclusion $\iota\,\cala(J) \subseteq \iota\,\cala(J^+)$ is split in the sense of Definition~\ref{def:split}. 
%\AB{Why does $\iota$ preserve split?}
As $\cala(J^+)$ commutes with $D(K)$, the inclusion $\iota\,\cala(J) \to D(K)'$ is then also split, where the commutant is taken in any faithful representation of $D(I)$.
Thus,
\[
D(J) \otimes_{\alg} D(K)=(\iota\,\cala(J)\oplus \ker\iota)\otimes_\alg D(K) \to D(I)
\]
extends to the spatial tensor product $D(J) \, \bar{\ox} \, D(K)$.
\end{proof}

\subsection*{\hspace*{-18pt}Vacuum properties}
Let $S$ be a genuinely bicolored circle, along with an orientation reversing diffeomorphism $j \colon S \to S$, compatible with the bicoloring and with the local coordinates.
Let $I\subset S$ be an interval whose boundary is fixed by $j$ and let $I':=j(I)$.
The Hilbert space $H_0:=L^2(D(I))$ is called the {\em vacuum sector} of $D$ associated to $S$, $I$, and $j$.
It is endowed with actions of $D(J)$ for every bicolored intervals $J\subset S$, as follows. (Recall that bicolored intervals contain at most one color-change point.)
The maps \eqref{eq:actions-on-L2} provide natural actions of $D(J)$ on $H_0$ for all subintervals $J\subset I$ and $J\subset I'$.
By the vacuum sector axiom for defects, these extend to the algebras $D(J)$ associated to white and to black subintervals of $S$. 
To define the action $\rho_J:D(J)\to\bfB(H_0)$ of an arbitrary genuinely bicolored interval $J\subset S$,
pick a white interval $K_1 \subset S$, a black interval $K_2 \subset S$, and diffeomorphisms $\varphi_i\in\Diff_0(K_i)$ such that $\varphi_1\varphi_2(J)$ does not cross $\dd I$.
If $u_1\in \cala(K_1)$ and $u_2\in \calb(K_2)$ are unitaries implementing $\varphi_1$ and $\varphi_2$, then the action on $H_0$ of an element $a\in D(J)$ is defined by
\begin{equation}\label{eq: diff trick for L^2(D)}
  \rho_J(a) := u_2^*u_1^*\,
  \rho_{\varphi_1\varphi_2(J)}\big(D(\varphi_1\varphi_2)(a)\big)\,
  u_1u_2.
\end{equation}
This action is compatible with the actions associated to other intervals, and is independent of the choices of $\varphi_1$, $\varphi_2$ and $u_1$, $u_2$
(see Lemma~\ref{lem: open cover of circle => sector -- BIS} for a similar construction in a more general context).

The following result, constructing isomorphisms between different vacuum sectors, is analogous to~\cite[\cornoncanonicalvacuum]{BDH(nets)}:

\begin{lemma}\label{lem: non canonical vacuum -- defects}
Let $S$ be a genuinely bicolored circle.
Let $I_1$ and $I_2$ be genuinely bicolored subintervals and let $j_1$ and $j_2$ be involutions fixing $\partial I_1$ and $\partial I_2$.
Then the corresponding vacuum sectors $L^2D(I_1)$ and $L^2D(I_2)$ are non-canonically isomorphic
as representations of the algebras $D(J)$ for $J\subset S$.
\end{lemma}

\begin{proof}
If $I_1$ and $I_2$ contain the same color-change point, then
let $\varphi\in\Diff(S)$ be a diffeomorphism that sends $I_1$ to $I_2$, that intertwines $j_1$ and $j_2$,
and that can be written as $\varphi=\varphi_\circ\circ\varphi_\bullet$ where $\varphi_\circ$ acts on the white part only and $\varphi_\bullet$ acts on the black part only.
Let $K$ be a white interval that contains $\supp(\varphi_\circ)$ in its interior and let $L$ be a black interval that contains $\supp(\varphi_\bullet)$ in its interior.
Finally, let $u\in\cala(K)$ and $v\in \calb(L)$ be unitaries implementing $\varphi_\circ$ and $\varphi_\bullet$.
Then
\[
L^2(D(I_1))\xrightarrow{L^2(D(\varphi))} L^2(D(I_2))\xrightarrow{u^*v^*} L^2(D(I_2))
\]
is the desired isomorphism.

If $I_1$ and $I_2$ contain opposite color-change points, then we may assume without loss of generality that $j_1=j_2$ and $I_2=j_1(I_1)$.
The isomorphism from $L^2(D(I_1))$ to $L^2(D(I_2))$ is then given by $L^2(D(j_1))$.
\end{proof}

\begin{notation}
Given a genuinely bicolored circle $S$ and a defect ${}_\cala D_\calb$, we denote by $H_0(S,D)$ the vacuum sector associated to \emph{some} interval $I\subset S$ and 
\emph{some} involution $j$ fixing $\partial I$.
By the previous lemma, that representation of the algebras $D(J)$ (for $J\subset S$) is well defined up to non-canonical unitary isomorphism.
\end{notation}

%\begin{remark}
%If $S$ is a circle that is either entirely white (or entirely black),
%then the above description of $H_0(S,D)$ still makes sense and recovers the notion of vacuum sector of a conformal net $H_0(S,\cala)$ (or $H_0(S,\calb)$); 
%see Appendix~\ref{subsec:vacuum-sector-net}. 
%\end{remark}

Our next result, concerning the gluing of vacuum sectors, is a straightforward generalization of~\cite[\corvacuumvacuumvacuum]{BDH(nets)} in the presence of defects (compare Appendix~\ref{subsec:glueing}).
Let $S_1$ and $S_2$ be bicolored circles, let $I_i\subset S_i$ be bicolored intervals (whose boundaries do not touch the color change points), and let $I_i'$ be the closure of $S_i\setminus I_i$.
Assume that there exists an orientation reversing diffeomorphism $\varphi:I_2\to I_1$ compatible with the bicolorings, and let $S_3:=I_1'\cup_{\partial I_2} I_2'$.
Assume that $(S_3)_\circ$ and $(S_3)_\bullet$ are connected and non-empty.
Then, up to exchanging $S_1$ and $S_2$, we are in one of the following three situations:
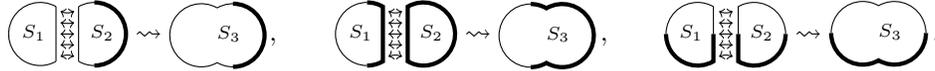
\def\inbetweenarrows{\draw[<->] (-3,0) -- (3,0);
\draw[<->] (-3,5) to[in=190, out=-10] (3,5);
\draw[<->] (-3,-5) to[in=170, out=10] (3,-5);
\draw[<->] (-3,10) to[in=200, out=-20] (3,10);
\draw[<->] (-3,-10) to[in=160, out=20] (3,-10);
%\node at (0,16) {$\scriptstyle\varphi$};
}
\[
\tikzmath[scale=\textscale]{\useasboundingbox (-27,-15) rectangle (27,15);\def\h{12}
\node[yshift=-2, xshift=-2] at ($(-\h,0)+(100:3)$) {$\scriptstyle S_1$};\draw (-\h,0) +(60:14) arc (60:300:14);\draw (-\h,0) +(60:14) -- +(300:14);
\node[yshift=-2, xshift=2.5] at ($(\h,0)+(80:3)$) {$\scriptstyle S_2$}; \draw (\h,0) +(120:14) arc (120:-120:14);\draw (\h,0) +(120:14) -- +(-120:14);
\draw[ultra thick] (\h,0)+(90:14) arc (90:-90:14);\inbetweenarrows}%tikzmath
\rightsquigarrow\!\tikzmath[scale=\textscale]{\node[xshift=-2, yshift=-2] at ($(14,0)+(80:3)$) {$\scriptstyle S_3$};
\draw (60:14) arc (60:300:14);\draw (14,0) +(120:14) arc (120:-120:14);\draw[ultra thick] (14,0) +(90:14) arc (90:-90:14);},%tikzmath
\qquad
\tikzmath[scale=\textscale]{\useasboundingbox (-27,-15) rectangle (27.5,15);\def\h{12.5}
\node[yshift=-2, xshift=-2] at ($(-\h,0)+(100:3)$) {$\scriptstyle S_1$};\draw (-\h,0) +(90:14) arc (90:270:14);\draw[ultra thick] (-\h,0) +(90:14) arc (90:60:14) -- ($(-\h,0) +(-60:14)$) arc (-60:-90:14);
\node[yshift=-2, xshift=2.5] at ($(\h,0)+(80:3)$) {$\scriptstyle S_2$}; \draw[ultra thick] (\h,0) +(120:14) arc (120:-120:14) -- cycle;\inbetweenarrows}%tikzmath
\rightsquigarrow\tikzmath[scale=\textscale]{\node[xshift=-2, yshift=-2] at ($(14,0)+(80:3)$) {$\scriptstyle S_3$};\draw (90:14) arc (90:270:14);\draw[ultra thick] (90:14) arc (90:60:14) arc (120:-120:14) arc (-60:-90:14);
},%tikzmath
\qquad
\tikzmath[scale=\textscale]{\useasboundingbox (-27,-15) rectangle (27.5,15);\def\h{12.5}
\node[yshift=-2, xshift=-2] at ($(-\h,0)+(100:3)$) {$\scriptstyle S_1$};
\node[yshift=-2, xshift=2.5] at ($(\h,0)+(80:3)$) {$\scriptstyle S_2$}; \draw (-\h,0) +(180:14) arc (180:60:14) -- ($(-\h,0)+(7,0)$);\draw[ultra thick] (-\h,0) +(0:7) -- +(300:14) arc (300:180:14);
\draw (\h,0) +(0:14) arc (0:120:14) -- ($(\h,0)-(7,0)$);\draw[ultra thick] (\h,0) +(0:-7) -- +(240:14) arc (240:360:14);\inbetweenarrows}%tikzmath
\rightsquigarrow
\tikzmath[scale=\textscale]{\node[xshift=-2, yshift=-2] at ($(14,0)+(80:3)$) {$\scriptstyle S_3$};\draw (180:14) arc (180:60:14) arc (120:0:14);\draw[ultra thick] (180:14) arc (-180:-60:14) arc (-120:0:14);}\,.%tikzmath
\]
Equip $S_1\cup_{I_2}S_2$ with a smooth structure that is compatible with the given smooth structures on $S_1$ and $S_2$ in the sense 
of \cite[Def. 1.4]{BDH(modularity)}.
That is, provide smooth structures on $S_1$, $S_2$, and $S_3$ such that there exists an action of the symmetric group $\mathfrak S_3$ on $S_1\cup_{I_2}S_2$ (with no compatibility with the bicoloring)
that permutes the three circles and has $\pi|_{S_a}$ smooth for every $\pi\in \mathfrak S_3$ and $a\in \{1,2,3\}$.

When ${}_\cala D_\calb$ is a defect, it will be convenient to write $H_0(S,D):=H_0(S,\cala)$ if $S$ is entirely white and $H_0(S,D):=H_0(S,\calb)$ if $S$ is entirely black.

\begin{lemma}\label{lem: vacuum * vacuum = vacuum -- with defects}
Let $S_1$, $S_2$, $S_3$, and $\varphi$ be as above, and let $D$ be a defect.
Use the map $D(\varphi)$ to equip $H_0(S_1,D)$ with the structure of a right $D(I_2)$-module.
Then there exists a non-canonical isomorphism
\begin{equation*}
H_0(S_1,D) \,\boxtimes_{D(I_2)} H_0(S_2,D)\,\,\cong\,\, H_0(S_3,D),
\end{equation*}
compatible with the actions of $D(J)$ for $J\subset S_3$.
\end{lemma}

\begin{proof}
Depending on the topology of the bicoloring, we can either identify $H_0(S_1,D)$ with $L^2(D(I_1))$ or identify $H_0(S_2,D)$ with $L^2(D(I_2))$.
We assume without loss of generality that we are in the first case.

Let $j\in\Diff_-(S_1)$ be an involution that is compatible with the bicoloring and that fixes $\partial I_1$,
and let $H_0(S_1,D)=L^2(D(I_1))$ be the vacuum sector associated to $S_1$, $I_1$, and $j$.
We then have
\[
L^2(D(I_1))\!\underset{D(I_2)}\boxtimes\! H_0(S_2,D)\,\cong\, L^2(D(I_2))\!\underset{D(I_2)}\boxtimes\! H_0(S_2,D) \,\cong\, H_0(S_2,D) \,\cong\, H_0(S_3,D),
\]
where the first isomorphism uses $L^2(D(\varphi)):L^2(D(I_2))\to L^2(D(I_1)^\op)= L^2(D(I_1))$
and the third one is induced by the map $(j\circ \varphi) \cup \mathrm{Id}_{I_2'}:S_2\to S_3$.
\end{proof}

\subsection*{\hspace*{-18pt}Haag duality} 
In certain cases, the geometric operation of complementation corresponds to the algebraic operation of relative commutant:

\begin{proposition}[Haag duality] \label{prop: [Haag duality for defects]}
(1) Let $S$ be a genuinely bicolored circle, let $I\subset S$ be a genuinely bicolored interval, and let $I'$ be the closure of the complement of $I$ in $S$.
Then the algebras $D(I)$ and $D(I')$ are each other's commutants on $H_0(S,D)$.

(2) Let ${}_\cala D_\calb$ be a defect and let $J\in\INT_{\halfbullet}$ and $K\in\INT_\circ\cup\INT_\bullet$ be subintervals of $I\in\INT_{\halfbullet}$.
Assume that $J \cup K=I$ and that $J \cap K$ is a point.
Then $D(J)$ is the relative commutant of (the image of) $D(K)$ in $D(I)$.
\end{proposition}

\begin{proof}
(1) Let $j\in\Diff_-(S)$ be an involution that exchanges $I$ and $I'$ and that is compatible with the bicoloring and the local coordinates.
By definition, we may take $H_0(S,D)=L^2(D(I))$ with the actions of $D(I)$ and $D(I')$ provided by \eqref{eq:actions-on-L2}.
The result follows, as the left and right actions of $D(I)$ on $L^2(D(I))$ are each other's commutants.

(2) We assume without loss of generality that $K\in\INT_\circ$. 
Let $S:=I\cup_{\partial I} (\bar I)$ be a circle 
formed by gluing two copies of $I$ along their boundary, 
such that there is a smooth involution $j$ that exchanges them:
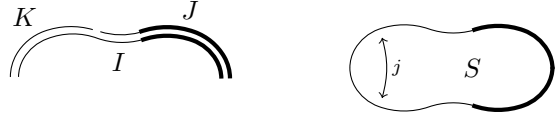
\[
\tikzmath[scale = \displscale]{
\useasboundingbox (-35,-25) rectangle (25,10);
\draw (.1,0) arc (-75:-109:20) arc (69.5:111.5:18.7) arc (110:150:12) arc (150:179.2:10);
\draw[ultra thick] (0,-.1) arc (110:70:20) arc (70:30:12) arc (30:.5:10);
\node at (-6,-6) {$I$};
\draw (-.4,2) arc (-75:-107:19) ++(-2,0.55) arc (74:110:21) arc (110:150:14) arc (150:180:12);
\draw[ultra thick] (-.5,2.1) arc (110:70:21.5) arc (70:30:15) arc (30:0:11.5);
\node at (13,8) {$J$};
\node at (-31,6) {$K$};
} %tikzmath
\qquad\qquad
\tikzmath[scale = \displscale]{
\useasboundingbox (-35,-18.5) rectangle (25,18.5);
\draw (.1,9.97) arc (-75:-109:20) arc (71:110:20) arc (110:150:12) arc (150:180:10);
\draw[ultra thick] (0,9.97) arc (110:70:20) arc (70:30:12) arc (30:0:10);
\node at (0,0) {$S$};
\node at (-20,0) {$\scriptstyle j$};
\draw[<->] (-24,-8) to[in=-75, out=75] (-24,8);
\pgftransformyscale{-1} 
\draw (.1,9.97) arc (-75:-109:20) arc (71:110:20) arc (110:150:12) arc (150:180:10);
\draw[ultra thick] (0,9.97) arc (110:70:20) arc (70:30:12) arc (30:0:10);
} %tikzmath
\]
By strong additivity and the first part of the proposition, and considering actions on $H_0(S,D)$ we then have
\[
D(K)'\cap D(I) = D(K)'\cap D(\bar I)' = \big(D(K)\vee D(\bar I)\big)' = D(K\cup \bar I)' = D(J). \qedhere
\] 
\end{proof}

\subsection*{\hspace*{-18pt}Canonical quantization}
Let $S$ be a bicolored circle and $I\subset S$ a genuinely bicolored interval.
Let $j\in\Diff_-(S)$ be an involution that fixes $\partial I$ and that is compatible with the bicoloring and the local coordinates.
Also let $K \subset S$ be a white interval such that $j(K) = K$.
We call a diffeomorphism $\varphi\in\Diff_0(K)\subset \Diff(S)$ \emph{symmetric} if it commutes with $j$,
and set
\[
\Diff_0^\mathrm{sym}(K):=\big\{\varphi\in\Diff_0(K)\,\big|\,\varphi j=j\varphi\big\}.
\]
Given a symmetric diffeomorphism $\varphi$, we also write $\varphi_0\in\Diff(I)$ for $\varphi |_I$; to be precise, $\varphi_0:=\varphi |_{I\cap K}\cup \mathrm{id}_{I\setminus K}$.

For an irreducible defect ${}_\cala D_\calb$, we want to understand the automorphism \linebreak $L^2D(\varphi_0)$ of $H_0(S,D) := L^2D(I)$,
and its relation to the automorphism $L^2\cala(\varphi_0)$ of $H_0(S,\cala) := L^2\cala(I)$, where in these expressions involving $\cala$ the circle $S$ has now been painted all white.

By~\cite[\lemLconformalimplementaion]{BDH(nets)} the unitary $u_\varphi := L^2\cala(\varphi_0)$ on  $L^2\cala(I)$  implements $\varphi$, that is,
\begin{equation}\label{eq: music sign}
  \cala(\varphi)(a) = u_\varphi a {u_\varphi}^*
    \qquad \text{for all intervals $J \subseteq S$ and all 
                  $a \in \cala(J)$.} 
\end{equation}
Let $K'$ be the closure of the complement of $K$ in $S$.
Since $u_\varphi$ commutes with $\cala(K')$, we have $u_\varphi\in\cala(K)$ by Haag duality (Proposition~\ref{prop: [Haag duality for defects]-duality-nets}).
We call $u_\varphi \in \cala(K)$ the {\em canonical quantization of the symmetric diffeomorphism $\varphi$}.

The map $\Diff_0^\mathrm{sym}(K)\to \Diff_+(I)$ given by $\varphi \mapsto \varphi_0$ is continuous for the $\calc^\infty$-topology.
The map $\cala \colon \Diff_+(I) \to \Aut(\cala(I))$ is continuous because $\cala$ is a continuous functor\footnote{
       This refers to Haagerup's
       $u$-topology on $\Aut(\cala(I))$,
       see~\cite[Def.~3.4]{Haagerup(1975standard-form)}
       or~\cite[Appendix]{BDH(nets)}.}.
The map $\Aut(\cala(I )) \to \U(L^2\cala(I))$ given by $\psi \mapsto L^2(\psi)$ is continuous by~\cite[Prop.~3.5]{Haagerup(1975standard-form)}.
Therefore, altogether, $\varphi \mapsto u_\varphi$ defines a continuous map from the group of symmetric diffeomorphisms of $K$ to $\U(\cala(K))$. 

\begin{lemma}\label{lem: canonical quantization of the symmetric diffeomorphism}
Let $S$, $I $, $K$, $\varphi$, $\varphi_0$, and $u_\varphi$ be as above,
let ${}_\cala D_\calb$ be an irreducible defect, and let $H_0:=L^2D(I)$ be the vacuum sector of $D$ associated to $S$, $I$, and $j$.
Then, letting $\rho_K$ be the action of $\cala(K)$ on $H_0$ (given by the vacuum sector axiom), we have $L^2(D(\varphi_0))=\rho_K(u_\varphi)$.
\end{lemma}
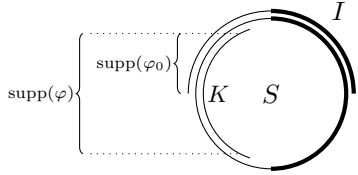
\[
\tikzmath{\draw  (-90:1) arc (-90:-270:1);\draw[ultra thick] (-90:1) arc (-90:90:1);\node at (0,0) {$S$};\draw (-118+10:.9) arc (-118+10:-242-10:.9);\node[scale=.9] at (-.7,0) {$K$};
\draw[decorate,decoration=brace] (-1.2,0) --node[left, scale=.9]{$\scriptstyle\supp(\varphi_0)$} +(0,.8);
\draw[decorate,decoration=brace] (-2.5,-.8) --node[left, scale=.9]{$\scriptstyle\supp(\varphi)$} +(0,1.6);
\draw[dotted] (-2.4,-.8) -- +(2,0)(-2.4,.8) -- +(2,0);\draw (180:1.1) arc (180:90:1.1);\draw[ultra thick] (90:1.1) arc (90:0:1.1);\useasboundingbox;\node at (.9,1.1) {$I $};}
\]

\begin{proof}
We first show that the map
\begin{equation}\label{eq: continuous symetric diffeo}
\begin{split}
\Diff_0^\mathrm{sym}(K)&\to \Aut(D(I))\\
\varphi\,\,\,\,\,&\mapsto\,\,\, D(\varphi_0)
\end{split}
\end{equation}
is continuous for the $\mathcal C^\infty$ topology on $\Diff_0^\mathrm{sym}(K)$ and the $u$-topology on $\Aut(D(I))$.
Since $\Ad(u_\varphi)=\cala(\varphi)$, the operator $\rho_K(u_\varphi)$ implements $\varphi$ on $H_0$.
In particular, $D(\varphi_0)$ is the restriction of $\Ad(\rho_K(u_\varphi))$ under the embedding $D(I) \hookrightarrow \bfB(H_0)$.
The map
\[
\Diff_0^\mathrm{sym}(K) \to \U(\cala(K)) \to \U(H_0),\quad\,\,\,\,\varphi \,\mapsto\, u_\varphi \,\mapsto\, \rho_K(u_\varphi)
\]
is continuous and lands in the subgroup $\mathrm{N}:=\{u \in \U(H_0)\,|\, uD(I)u^*=D(I)\}$.
Since $D(\varphi_0)=\mathrm{Ad}(\rho_K(u_\varphi))$ and $\mathrm{Ad}:\mathrm{N}\to \Aut(D(I))$ is continuous~\cite[A.18]{BDH(nets)},
the map \eqref{eq: continuous symetric diffeo} is therefore also continuous.
Recalling~\cite[Prop.~3.5]{Haagerup(1975standard-form)} that $L^2:\Aut(D(I))\to \U(L^2D(I))$ is continuous,
we have therefore shown that
\[
\Diff_0^\mathrm{sym}(K)\to \bfB(H_0),\quad\,\,\,\,
\varphi \mapsto L^2(D(\varphi_0))
\]
is a continuous homomorphism.

Recall that $\rho_K(u_\varphi)$ implements $\varphi$.
By the same argument as in~\cite[\lemLconformalimplementaion]{BDH(nets)}, $L^2(D(\varphi_0))$ also implements $\varphi$.
It follows that
\[
L^2(D(\varphi_0))=\lambda_\varphi\rho_K(u_\varphi)
\]
for some scalar $\lambda_\varphi\in S^1$.
Thus, we get a continuous map
\(
\varphi\mapsto \lambda_\varphi
\)
from the group of symmetric diffeomorphisms of $K$ into $\U(1)$.
Our goal is to show that $\lambda_\varphi=1$.

Let $J_\cala$ and $J_D$ be the modular conjugations on $L^2(\cala(I))$ and $L^2(D(I))$,
and let $\pi_K$ be the natural action of $\cala(K)$ on $L^2(\cala(I))$.
Since $\pi_K(u_\varphi)=L^2(\cala(\varphi_0))$ commutes with $J_\cala$, we have $J_\cala \pi_K(u_\varphi) J_\cala = \pi_K(u_\varphi)$.
Combined with the fact that $J_\cala$ implements $j$~\cite[\lemjisimplemented]{BDH(nets)}, 
this implies the equation
\begin{equation} \label{eq: A(j)(u_phi^*)=u_phi}
\cala(j)(u_\varphi^*)=u_\varphi.
\end{equation}
Applying $\rho_K$ to \eqref{eq: A(j)(u_phi^*)=u_phi}, then, by a straightforward analog of~\cite[\lemjisimplemented]{BDH(nets)}, 
we learn that $J_D\, \rho_K(u_\varphi)\, J_D = \rho_K(u_\varphi)$.
Since both $L^2(D(\varphi_0))$ and $\rho_K(u_\varphi)$ commute with $J_D$ and since the latter is antilinear,
the phase factor $\lambda_\varphi$ must be real. It follows that $\lambda_\varphi\in\{\pm1\}$.

To finish the argument, note that $\Diff_0^\mathrm{sym}(K)$ is connected and that $\{\pm1\}$ is discrete.
The map $\varphi\mapsto \lambda_\varphi$ being continuous, it must therefore be constant.
\end{proof} %!%\CDhmarg{CDTD}

\subsection*{\hspace*{-18pt}Continuity}
Given genuinely bicolored intervals $I$ and $J$, and a neighborhood $N$ of $I_\circ\cap I_\bullet$,
let $\mathrm{Hom}^{(N)}(I,J)$ denote the set of embeddings $I\to J$ that preserve the local coordinate on the whole of $N$
(this only makes sense if $N$ is contained in the domain of definition of the local coordinate).
We equip
$
\mathrm{Hom}_{\INT_{\circ\bullet}}(I,J)=
\bigcup_N \mathrm{Hom}^{(N)}(I,J)
$
with the colimit of the $\mathcal C^\infty$ topologies on $\mathrm{Hom}^{(N)}(I,J)$.

Given two von Neumann algebras $A$ and $B$, the Haagerup $u$-topology on $\mathrm{Hom}_{\VN}(A,B)$
is the topology of pointwise convergence for the induced map on preduals~\cite[Appendix]{BDH(nets)}.

\begin{proposition}[Continuity for defects]\label{prop:Continuity for defects}
Let $D \colon \INT_{\circ\bullet} \to \VN$ be a defect.
Then $D$ is a continuous functor: for bicolored intervals $I$ and $J$ the map
\begin{equation*}
\mathrm{Hom}_{\INT_{\circ\bullet}}(I,J) \to \mathrm{Hom}_{\VN}(D(I),D(J))
\end{equation*}
is continuous with the above topology on $\mathrm{Hom}_{\INT_{\circ\bullet}}(I,J)$ and with Haagerup's $u$-topology on $\mathrm{Hom}_{\VN}(D(I),D(J))$.
\end{proposition}

\begin{proof}
For every $N$ as above, we need to show that the map $D:\mathrm{Hom}^{(N)}(I,J)\to \mathrm{Hom}_{\VN}(D(I),D(J))$ is continuous. 
We argue as in~\cite[\lemequivdefofcontinuityaxiom]{BDH(nets)}.
Pick a bicolored interval $K$,  and identify $I$ and $J$ with subintervals of $K$ via some fixed embeddings into its interior.
Given a generalized sequence  $\varphi_i \in \mathrm{Hom}^{(N)}(I,J)$, $i\in\mathcal I$, with  limit $\varphi$, and given a vector $\xi$ in the predual of $D(J)$, 
we need to show that $D(\varphi_i)_*(\xi)$ converges to  $D(\varphi)_*(\xi)$ in $D(I)_*$.

Let $\Diff^{(N)}_0(K)$ be the subgroup of diffeomorphisms of $K$ that fix $N$ and also fix a neighborhood of $\partial K$.
Pick an extension $\hat \varphi \in \Diff^{(N)}_0(K)$ of $\varphi$,
and let $\hat\varphi_{n,i}\in \Diff^{(N)}_0(K)$, $n \in \IN$, be extensions of $\varphi_i$ such that
$\|\hat\varphi_{n,i}-\hat \varphi\|_{\mathcal C^n}<\|\varphi_{i}- \varphi\|_{\mathcal C^n}$, where $\|\,\,\|_{\mathcal C^n}$ is any norm that induces the $\mathcal C^n$ topology.
Letting $F$ be the filter on $\mathbb N\times \mathcal I$ generated by the sets $\{(n,i)\in \IN\times \mathcal I\,|\,n\ge n_0, i\ge i_0(n)\}$ (see \cite[\lemequivdefofcontinuityaxiom]{BDH(nets)}),
then $F\text{-lim}\,\,\hat\varphi_{n,i}=\hat\varphi$ in the $\calc^\infty$-topology.

Write $\Diff^{(N)}_0(K)$ as $\Diff_\circ\times \Diff_\bullet$, where $\Diff_\circ$ is the subgroup of $\Diff^{(N)}_0(K)$ consisting of diffeomorphisms whose support is contained in the white part,
and $\Diff_\bullet$ is the subgroup of diffeomorphisms whose support is contained in the black part.
The continuity of \eqref{eq: continuous symetric diffeo} shows that the map $\Diff^{(N)}_0(K)\to\Aut(D(K)):\psi \mapsto D(\psi)$ is continuous when restricted to either $\Diff_\circ$ or $\Diff_\bullet$.
The composite
\[
\Diff^{(N)}_0(K)=\Diff_\circ\times \Diff_\bullet\rightarrow \Aut(D(K))\times \Aut(D(K))\xrightarrow{\mathit{mult.}\!} \Aut(D(K))
\] 
is therefore also continuous.
It follows that $F\text{-lim}\,\,D(\hat \varphi_{n,i})=D(\hat \varphi)$ in the $u$-topology on $\Aut(D(K))$.
Given a lift $\hat\xi\in D(K)_*$ of $\xi$, the vectors $D(\hat\varphi_{n,i})_*(\hat\xi)$ therefore converge to $D(\hat\varphi)_*(\hat\xi)$.
Composing with the projection  $\pi:D(K)_*\twoheadrightarrow D(I)_*$, it follows that  $D(\varphi_i)_*(\xi)=\pi(D(\hat\varphi_{n,i})_*(\hat\xi))$  converges to $\pi(D(\hat\varphi)_*(\hat\xi))=D(\varphi)_*(\xi)$.  
\end{proof} %!%\CDhmarg{CDTD}

\section{Examples of defects}\label{sec: examples of defects}

\subsection*{\hspace*{-18pt}Von Neumann algebras as defects, free boundaries, and defects coming from conformal embeddings}
The trivial conformal net $\underline \IC$ evaluates to $\IC$ on every interval~\cite[\extrivialconfnet]{BDH(nets)}.

\begin{proposition} \label{prop: CCdefects == VNalg}
  There is a one-to-one correspondence (really an equivalence of categories) between 
  $\underline \IC$-$\underline \IC$-defects and von Neumann algebras.
\end{proposition}

\begin{proof}
Given a von Neumann algebra $A$, the associated defect is
\[
\underline A(I) :=
\begin{cases}
\,\IC&\quad\text{if }\, I\in\INT_\circ\,\text{ or }\,I\in\INT_\bullet\\
A &\quad\text{if }\, I\in\INT_\halfbullet\,\text{ and the local coordinate is orientation preserving}\\
A^\op &\quad\text{if }\, I\in\INT_\halfbullet\,\text{ and the local coordinate is orientation reversing}\\
\end{cases}
\] 
where $A^\op$ denotes the opposite of $A$.

Conversely, let $D$ be a $\underline \IC$-$\underline \IC$-defect.
Given a bicolored interval $I$, the orientation reversing map $\mathrm{Id}_I:I\to -I$ identifies $D(-I)$ with ${D(I)}^\op$, where $-I$ denotes $I$ with opposite orientation.
So we just need to show is that the restriction of $D$ to the subcategory of genuinely bicolored intervals with orientation preserving maps (compatible with the local coordinates) is equivalent to a constant functor.
By applying Proposition \ref{prop: [Haag duality for defects]}, we see that every embedding $J\rightarrow I$ between two such intervals induces an isomorphism $D(J)\to D(I)$.

To finish the proof, we need to check that $D(\phi)=\mathrm{Id}_{D(I)}$ for any $\phi:I\rightarrow I$.
Pick a neighborhood $J\subset I$ of $I_\circ\cap I_\bullet$ on which $\phi$ is the identity. 
Then the two arrows $D(J)\to D(I)$ in the commutative diagram
\[
\def\hh{2}\def\vv{1}
\tikzmath{
\draw[->] 
(-\hh,0) node (a) {$D(I)$} 
(\hh,0) node (b) {$D(I)$} 
(0,-\vv) node (c) {$D(J)$}
(a.east) --node[above]{$\scriptstyle D(\phi)$} (b.west);
\draw[->] (c) --node[below]{$\scriptstyle \cong$} (a);
\draw[->] (c) --node[below]{$\scriptstyle \cong$} (b);
}
\]
are equal to each other, showing that $D(\phi)=\mathrm{Id}$.
\end{proof} 

\begin{proposition} \label{prop: free boundary condition}
\def\overharp{\overset{\raisebox{-1mm}{$\,\scriptstyle\leftharpoonup$}}} 
Let $\cala$ be a conformal net.
Then the functor $\overharp \cala:\INT_{\circ\bullet}\to \VN$ given by
\[
\overharp \cala(I) :=
\begin{cases}
\cala(I_\circ)&\quad\text{\rm if }\, I_\circ\not = \emptyset,\\ \,\,\IC &\quad\text{\rm otherwise}
\end{cases}
\]
is an $\cala$-$\underline{\IC}$-defect.
\end{proposition}

\begin{proof}
The axioms for defects follow immediately from the corresponding axioms for $\cala$ (Appendix \ref{subsec:defnets}).
\end{proof}

Conformal embeddings provide examples of defects.
Recall that a morphism of conformal nets $\tau:\cala\to \calb$ is a called a \emph{conformal embedding} \cite[\S1.5]{BDH(nets)} if
\[
\mathrm{Ad}(u) = \cala(\varphi)\quad\Rightarrow\quad \mathrm{Ad} (\tau_I(u)) = \calb(\varphi)
\]
for every diffeomorphism $\varphi\in\Diff_0(I)$ and unitary $u\in \cala(I)$.

\begin{proposition}\label{prop: conformal inclusion defect}
Let $\cala$ and $\calb$ be conformal nets
and let $\tau:\cala\to\calb$ be a conformal embedding. Then
\begin{equation}\label{eq: D_tau}
D_\tau(I):=\begin{cases}
\cala(I)\quad\text{\rm for }\, I\in\INT_\circ\\
\calb(I)\quad\text{\rm for }\, I\in\INT_\halfbullet \cup \INT_\bullet
\end{cases}
\end{equation}
is an $\cala$-$\calb$-defect.
\end{proposition}

\begin{proof}
The axioms of isotony, locality, and vacuum sector follow directly from the corresponding axioms for $\calb$.
It remains to prove strong additivity.
We need to show that
\[
\tau_{[0,1]}(\cala([0,1]))\vee\calb([1,2])=\calb([0,2]).
\]

For every point $x\in (0,1)$, pick a diffeomorphism $\varphi_x\in\Diff_0([0,2])$ sending $1$ to $x$, and let $u_x\in\cala([0,2])$ be a unitary implementing $\cala(\varphi_x)$.
Since $\tau$ is a conformal embedding, we then have $u_xbu_x^*=\calb(\varphi)(b)$ for all $b\in \calb([0,2])$.
Moreover, since
\[
u_x\in \cala([0,2])=\cala([0,1])\vee\cala([1,2])\subset\cala([0,1])\vee\calb([1,2])
\]
and since $u_x$ conjugates $\calb([1,2])$ to $\calb([x,2])$, we have $\calb([x,2])\subset \cala([0,1])\vee\calb([1,2])$.
The argument being applicable to any $x\in (0,1)$, 
it follows from \cite[\lemirrelevanceofpoints]{BDH(nets)} that
\[
\calb([0,2])=\bigvee_{x\in(0,1)}\calb([x,2])\subset \cala([0,1])\vee\calb([1,2]). \qedhere
\]
\end{proof}

\subsection*{\hspace*{-18pt}Defects from $Q$-systems}
Longo and Rehren
\cite{Longo(A-duality-for-Hopf-algebras-and-for-subfactors),
            Longo-Rehren(Nets-of-subfactors)}
showed that given a conformal net $\cala$ and a unitary Frobenius algebra 
object in the category of $\cala$-sectors,
one can construct an extension $\cala\subset \calb$, where $\calb$ is a possibly non-local conformal net.

Here, the unitary Frobenius algebra object is an 
$\cala$-sector $A$ along with unit and multiplication maps
$\eta:H_0\to A$, $\mu:A\boxtimes A\to A$
subject to the relations
\begin{alignat*}{3}
&\mu(\eta\boxtimes \id_A)=\mu(1\boxtimes \eta)=\id_A && \text{(unitality)}\\
&\mu(\mu\boxtimes \id_A)=\mu(\id_A\boxtimes \mu) && \text{(associativity)}\\
&\mu^* \mu=(\id_A\boxtimes \mu)(\mu^*\boxtimes \id_A) &\,\,\,\,& \text{(Frobenius)},
\end{alignat*}
and the normalization $\mu\mu^*=\eta^*\eta\cdot\id_A$.\footnote{This last condition is only appropriate for simple Frobenius algebra objects. The $Q$-systems considered below will all correspond to simple Frobenius algebra objects.}
Here, $H_0$ is the vacuum sector of $\cala$ (the identity sector on the identity defect; see Section~\ref{sec: The category CN2 of sectors})
and $\boxtimes$ is the operation of vertical fusion (see Section~\ref{sec: Vertical fusion}).

If one encodes, as is usually done in the literature, the $\cala$-sector $A$ by a 
localized endomorphism $\theta:\cala([0,1])\to \cala([0,1])$,\footnote{Here, `localized' means that $\theta$ restricts to the identity map on $\cala([0,\epsilon])$ and $\cala([1-\epsilon,1])$.  One recovers the sector $A$ from the endomorphism $\theta$ by letting the underlying Hilbert space of $A$ be $H_0$ and twisting the action on the top half of the circle (identified with $[0,1]$) by $\theta$.}
then a unitary Frobenius structure on $A$ can be specified by a choice of two elements $w,x\in\cala([0,1])$ ($w$ is the unit and $x^*$ is the multiplication) satisfying the relations \cite[(4.1)]{Bischoff-Kawahigashi-Longo-Rehren(Phase-boundaries-in-algebraic-conformal-QFT)}:
\begin{alignat*}{3}
\quad\,\, wa=\theta(a)w,\,       x\theta(&a)=\theta^2 (a)x,     \,\,\forall a\in\cala([0,&&1])\,\, \text{(source and target of $w$ and $x$)}\\
&w^*x=\theta(w^*)x=1       &&\text{(unitality)}\\
&xx=\theta(x)x                   &&\text{(associativity)}\\
&xx^*=\theta(x^*)x                 &&\text{(Frobenius)}\\
&w^*w=x^*x=d\cdot1     &&\text{(normalization)}
\end{alignat*}
for some scalar $d$.
A triple $(\theta,w,x)$ subject to this set of equations is called a 
$Q$-system~\cite{Longo(A-duality-for-Hopf-algebras-and-for-subfactors)}.

We now use the $Q$-system $(\theta,w,x)$ to construct an $\cala$-$\cala$-defect $D$.
Given a genuinely bicolored interval $I$, we use the local coordinate to construct
a new interval $I^+ := I_\circ \cup [0,1] \cup I_\bullet$.
As $\theta$ is localized, there is a unique extension $\theta^+$ of $\theta$ to
$\cala(I^+)$.
It is determined by requiring $\theta^+(a) = a$ for $a \in \cala(K)$ with 
$K \subset I^+ \setminus (0,1)$ and $\theta^+(a) = \theta(a)$ for $a \in \cala([0,1])$.
We define $D(I)$ as the algebra generated by $\cala(I^+)$ and one extra element $v$, 
subject to the relations
\cite[(4.3)]{Bischoff-Kawahigashi-Longo-Rehren(Phase-boundaries-in-algebraic-conformal-QFT)}:
\begin{equation}\label{eq:  va=...}
\begin{split}
va=\theta(a)v\quad  \forall a\in\cala&(I^+),\\
v^*=w^*x^*v,\quad vv=xv,\quad &w^*v=1.
\end{split}
\end{equation}
The Longo--Rehren non-local extension $\cala \subset \calb$ associated to the $Q$-system is defined such that $\calb(I^+) = D(I)$~\cite{Longo-Rehren(Nets-of-subfactors)}.

\begin{proposition} \label{prop:non-local-ext-defect}
  Let $\cala$  be a conformal net, and let $(\theta,w,x)$ be a $Q$-system in $\cala([0,1])$.
  Then $D$, as defined above, is a $\cala$-$\cala$-defect.
\end{proposition}

\begin{proof}
Isotony for $D$ follows from isotony for $\cala$ because the map $\cala(I^+)\to D(I):a\mapsto av$ is a bijection.
Indeed, $D(I)$ could alternatively be defined to be the set $\{ av | a \in \cala(I^+) \}$ with unit element $w^*v$, multiplication $av\cdot bv:=a\theta(b)x v$, and star operation $(av)^*:= w^*x^* \theta(a^*) v$.

To prove locality, consider the situation where $K \subset I$ are genuinely bicolored,
and $J$ has disjoint interior from $K$.
We need to show that $\cala(J)$ and $D(K)=\cala(K^+)\vee\{v\}$ commute inside $D(I)$.
By locality for $\cala$, the algebras $\cala(J)$ and $\cala(K^+)$ commute inside $\cala(I^+)$.
The algebra $\cala(J)$ also commutes with $v$, because $va=\theta(a)v$ and $\theta|_{\cala(J)}=\id$.

To prove strong additivity, consider the situation of a genuinely bicolored interval $I$ 
that is the union of a black or white interval $J$ and a genuinely bicolored interval $K$.
By definition the algebra $D(I)$ is generated by $\cala(I^+)$ and $v$, and similarly $D(K)$ is generated by $\cala(K^+)$ and $v$.
By the strong additivity of $\cala$, we have
\[
D(J) \vee D(K)=\cala(J) \vee \cala(K^+) \vee \{v\} =\cala(I^+) \vee \{v\} = D(I).
\]

We now address the vacuum sector axiom.
Consider the situation of a genuinely bicolored interval $I$ and 
a white subinterval $J$ that touches one of the boundary points of $I$ 
(the other case is identical).
By the vacuum sector axiom for $\cala$, it is enough to construct a unitary map
\[
u:L^2D(I) \to L^2\cala(I^+)
\]
that is equivariant with respect to the left and right actions of $\cala(J)$.
The formula $E:D(I)\to\cala(I^+)$, $E(av):=d^{-1}\cdot aw$, %\ABcomm{Check this formula.}  AH: looks good.
$a\in\cala(I^+)$ defines a conditional expectation,
and the corresponding orthogonal projection $p:L^2D(I)\to L^2\cala(I^+)$ 
satisfies $p\hspace{.2mm}b\hspace{.2mm}p^*=E(b)$ for $b\in D(I)$ \cite[Lem~3.2]{Kosaki(Extension-of-Jones-index-to-arbitrary-factors)}.
We claim that
\[
u(\xi):=\sqrt{d}\cdot p(v\xi)
\]
is the desired unitary map.
The map $u$ is both left and right $\cala(J)$ equivariant because 
\[
u(a\xi a')=\sqrt{d}\cdot p(va\xi a')=\sqrt{d}\cdot p(\theta(a)v\xi a')=\sqrt{d}\cdot p(av\xi a')=\sqrt{d}\cdot ap(v\xi)a'=au(\xi)a'
\]
for $a,a'\in\cala(J)$.
To check that it is unitary, we compute
\begin{gather*}
uu^*(\xi)=d\cdot p(vv^*(p^*\xi)) = d\cdot E(vv^*)\xi\,\,\,\,\,\,  \text{ and }\\
d\cdot E(vv^*) = d\cdot E(vw^*x^*v) = d\cdot E(\theta(w^*x^*)vv) =
d\cdot E(\theta(w^*x^*)xv)=\\= \theta(w^*x^*)xw = \theta(w^*)\theta(x^*)xw = \theta(w^*)xx^*w = 1  
\end{gather*}
and
\begin{gather*}
u^*u(\xi)= d\cdot v^*p^*p(v\xi) = d\cdot v^*ev\xi\,\,\,\,\,\,  \text{ and }\\ v^*ev = v^*v_1v_1^*v= 1/d
\end{gather*}
where $e=p^*p$ is the Jones projection, and $v_1$, as in \cite[Sec~2.5]{Longo-Rehren(Nets-of-subfactors)},
satisfies $v_1v_1^*=e$ and $v^*v_1 = 1/\sqrt d$ (this last formula holds by substituting $v/\sqrt d$ in place of $w$, and $v_1$ in place of $v$, in \cite[(2.14)]{Longo-Rehren(Nets-of-subfactors)}; see also \cite[(4.3)]{Longo(A-duality-for-Hopf-algebras-and-for-subfactors)}).
\end{proof}

\begin{remark}
  The considerations 
  in~\cite{Bischoff-Kawahigashi-Longo-Rehren(Phase-boundaries-in-algebraic-conformal-QFT)}
  suggest a generalization of Proposition~\ref{prop:non-local-ext-defect} to
  the following situation.
  Let $\cala$ be a conformal net together with three $Q$-systems $(\theta^L,w^L,x^L)$,
  $(\theta^R,w^R,x^R)$, and $(\theta,w,x)$.
  If $(\theta^L,w^L,x^L)$ and $(\theta^R,w^R,x^R)$ are in addition assumed to be 
  commutative, then~\eqref{eq:  va=...} can be used to define extensions
  $\cala \subset \calb^L$ and $\cala \subset \calb^R$, as in \cite{Longo-Rehren(Nets-of-subfactors)}.
  If $(\theta,w,x)$ is such that the corresponding non-local extension  
  $\calc$ contains $\calb^L$ and $\calb^R$ and, 
  moreover, $\calb^L$ is \emph{left-local} with respect to $\calc$ and 
  $\calb^R$ is \emph{right-local}\footnote{
    This means that $\calb^L(I)$ and $\calb^R(I)$ should commute with $\calc(J)$ 
    whenever $I$ is to the left, respectively to the right, of $J$.} with respect to $\calc$,
  then the construction $D$ as in Proposition~\ref{prop:non-local-ext-defect} should define
  a $\calb^L$-$\calb^R$-defect.
%  Since $Q$-systems classify extensions, the required properties of the extensions
%  $\calb^L$, $\calb^R$ and $\calc$ are also expected to be expressible at
%  the level of $Q$-systems.
  %In this way we expect the rich structure of $Q$-systems in the representation categories of conformal nets to also manifest itself in our $3$-category of conformal nets.
  % We hope to come back to this elsewhere.
\end{remark}

\subsection*{\hspace*{-18pt}Direct sums and direct integrals of defects}

\begin{lemma}
Let $D_1$ and $D_2$ be $\cala$-$\calb$-defects.
Then their direct sum $E:=D_1\oplus D_2$ is also an $\cala$-$\calb$-defect. Here, $E$ is defined by
\begin{gather*}
E(I) = D_1(I)\oplus D_2(I) \,\,\,\text{\rm for}\,\,\, I\in\INT_\halfbullet\\ E(I) = \cala(I) \,\,\,\text{\rm for}\,\, I\in\INT_\circ\,\,\qquad E(I) = \calb(I) \,\,\,\text{\rm for}\,\, I\in\INT_\bullet. 
\end{gather*}
\end{lemma}
\begin{proof}
The only non-trivial axiom is strong additivity.
Consider the situation where $I=K\cup J$, with $J$ genuinely bicolored and $K$ white.
Letting $\Delta:\cala\to\cala^{\oplus 2}$ denote the diagonal map, we need to show that $E(I)$ is equal to the subalgebra
\[
\Delta\cala(K)\vee E(J)\subset E(I)
\]
generated by the images of $\Delta\cala(K)$ and $E(J)$.
(Note that our notation is a little bit misleading, as the map $\Delta\cala(K)\to E(I)$ might fail to be injective).
Pick a white interval $L\subset J$ that touches $K$ in a point.
Since $\Delta$ is a conformal embedding, it follows from the previous proof that $\Delta\cala(K)\vee \cala^{\oplus 2}(L)=\cala^{\oplus 2}(K\cup L)$.
%
%By inner covariance, $\Delta(\cala(K))\vee \cala(L)^2$ contains $\cala(L_1)^2$ whenever $L\subset L_1\subsetneq K\cup L$.
%This being true for every such interval $L_1$, then 
%by~\cite[\lemirrelevanceofpoints]{BDH(nets)}, 
%it follows that $\Delta(\cala(K))\vee\cala(L)^2=\cala(K\cup L)^2$.
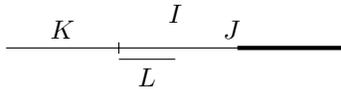
\[
  \tikzmath[scale=1.5]
  { \draw (-2,0) -- node[above]{$K$} (-1,0)(-1,.05)--(-1,-.05)(-1,0) --node[above]{$J$} (1,0)%(-1.5,-.1) -- node[below, xshift=-5]{$L_1$} (-.5,-.1)
  (-1,-.1) -- node[below]{$L$} (-.5,-.1);
    \draw[ultra thick] (.05,0) -- (1,0) (-.5,.3) node {$I$};}
\]
Thus, we have the following equalities between subalgebras of $E(I)$:
$$\hspace*{50pt}\Delta\cala(K)\vee E(J) = \Delta\cala(K)\vee \cala^{\oplus 2}(L)\vee E(J) = \cala^{\oplus 2}(K\cup L) \vee E(J) = E (I).\hspace*{50pt}\qedhere$$
\end{proof}

\begin{remark}\label{rem: direct integral of defects}
By the same argument as above, one can also show that a direct integral of $\cala$-$\calb$-defects is an $\cala$-$\calb$-defect.
\end{remark}

\subsection*{\hspace*{-18pt}Disintegrating defects}

We show that defects between semisimple nets can be disintegrated.
We warn the reader that our proof of the corresponding statement for conformal nets was incomplete:

\begin{correction} \label{corr:no-nets-disintegration}
In~\cite[\subseccentraldecomposition, Eq~1.42]{BDH(nets)}, we claimed that every conformal net decomposes as a direct integral of irreducible ones. However, the group $\Diff(I)$ is not locally compact and it is not clear that its action on $\cala(I)$ decomposes as a direct integral of actions on the irreducible components $\cala(I)_x$ of $\cala(I)$.\footnote{We thank Sebastiano Carpi for pointing this out.} 
At present, we do not know how to fill this gap.
This issue with~\cite[Eq~1.42]{BDH(nets)} does not affect any of the other results in~\cite{BDH(nets)}.
\end{correction}

\noindent The above issue with disintegrating diffeomorphism actions does not arise here when disintegrating defects, because the relevant actions are inherited from the conformal nets.

Let $D$ be a defect and 
let $f:J\to I$ be an embedding of genuinely bicolored intervals.
Then one can show as follows that
$D(f)$ induces an isomorphism between $Z(D(J))$ and $Z(D(I))$; compare the proof of \cite[Prop. 1.40]{BDH(nets)}. We may as well assume that $I$ and $J$ share a boundary point.
Let $K$ be the closure of $I\setminus J$.
The image of $Z(D(J))$ in $D(I)$ commutes with both $D(J)$ and $D(K)$, and so it is in $Z(D(I))$ by strong additivity.
Conversely, $Z(D(I))$ commutes with $D(K)$, and is therefore contained in the image of $Z(D(J))$ by Haag duality (Proposition \ref{prop: [Haag duality for defects]}).

As in the case for conformal nets~\cite[\subseccentraldecomposition]{BDH(nets)}, 
we can then introduce an algebra $Z(D)$ that only depends on $D$, and that is canonically isomorphic to $Z(D(I))$ for every genuinely bicolored interval $I$.
Disintegrating each $D(I)$ over that algebra, we can then write
\begin{equation*}
\qquad D(I) = \int_{x\in X}^\oplus D_x(I) \,\,\,\,\,\text{for every}\,\,I\in\INT_\halfbullet\,
\end{equation*}
where $X$ is any measure space with an isomorphism $L^\infty X\cong Z(D)$.

Recall that a conformal net is called semisimple if it is a finite direct sum 
of irreducible conformal nets (Appendix \ref{subsec:defnets}). 
Similarly, we call a defect \emph{semisimple} if it is a finite direct sum of irreducible defects. 

\begin{lemma} \label{lem: disintegrate}
Any $\cala$-$\calb$-defect between semisimple conformal nets\footnote{Arbitrary direct sums of irreducible conformal nets would also work.} is isomorphic to a direct integral of irreducible $\cala$-$\calb$-defects.
\end{lemma}

\begin{proof}
Fix a genuinely bicolored interval $I$.
The algebra $D(I)$ disintegrates as explained above.  We first need to show that for $K \subset I$ a white subinterval (respectively a black subinterval), the map $\cala(K)\to D(I)$ (respectively $\calb(K) \to D(I)$) similarly disintegrates.  It suffices to see that $\cala(K) \to D(I)$ induces maps $\cala(K)\to D_x(I)$ for almost every $x$.

Note that it is in general not true that a map $N\to\int^\oplus M_x$ from a von Neumann algebra $N$ (even a factor) into a direct integral induces maps $N\to M_x$ for almost every $x$.
This is however true when $N$ is a direct sum of type $I$ factors.
Indeed, letting $\mathcal K\subset N$ be the ideal of compact operators, we obtain maps $\mathcal K\to M_x$ by standard separability arguments.
One then uses the fact that a $C^*$-algebra homomorphism from $\mathcal K$ into a von Neumann algebra extends uniquely to a von Neumann algebra homomorphism from $N$.

We can leverage this observation about direct sums of type $I$ factors to construct the desired maps $\cala(K)\to D_x(I)$.
Consider a slightly larger interval $I^+$ that contains $I$, and let $K^+\subset I^+$ be a white interval that contains $K$ in its interior.
\[
  \tikzmath[scale=1.5]
  {\draw (-1.25,.2) -- (1,.2);\draw[ultra thick] (.05,.2) -- (1,.2);
  \draw(-1.25,.1) -- (-.25,.1);
  \draw(-1,0) -- (1,0) (-1,-.1) -- node[below]{$K$} (-.5,-.1);
  \draw[ultra thick] (.05,0) -- (1,0) (-.08,.4) node {$I^+$}(-.1,-.2) node {$I$} (-1.3,-.02) node {$K^+$};}
\]
By the split property and the semisimplicity of $\cala$, we can find an intermediate algebra $\cala(K)\subset N \subset \cala(K^+)$ that is a direct sum of type $I$ factors. 
The map $N\to D(I^+)=\int^\oplus D_x(I^+)$ induces maps $\tilde\iota_x: N\to D_x(I^+)$ for almost every $x$; let $\iota_x$ denote the restriction of $\tilde\iota_x$ to $\cala(K)$.
The map $\int^\oplus \iota_x:\cala(K)\to \int^\oplus D_x(I^+)$ is the composite of our original map $\cala(K)\to D(I)$ with the inclusion $D(I)\hookrightarrow D(I^+)$.
The image of $\int^\oplus \iota_x$ is contained in $\int^\oplus D_x(I)$. 
For almost every $x$ the image of $\iota_x$ is therefore contained in $D_x(I)$,
and we have our desired maps $\cala(K)\to D_x(I)$.
 
Let $f \colon I \to I$ be an embedding of genuinely bicolored intervals.
We need to know that $D(f)$ induces maps $D_x(I) \to D_x(I)$ for almost all $x$.
As $f$ fixes the local coordinate it suffices to consider the case where $f$ is the identity outside of
the interval $K$ used above. 
We can then extend $f$ to a diffeomorphism $f^+ \colon I^+ \to I^+$ that is the identity outside a small neighborhood of $K$. 
By inner covariance, $\cala(f^+ |_{K^+})$ is implemented by a unitary $u \in \cala(K^+)$.  The adjoint action $\Ad(u)$ induces the desired map $D_x(I) \to D_x(I)$.   Since any genuinely bicolored interval $J$ is isomorphic to $I$, we can transport the disintegration of the embedding $f \colon I \to I$ to a disintegration of any embedding of genuinely bicolored intervals. 

The isotony, locality, and strong additivity axioms for $D_x$ are immediate; the vacuum sector axiom requires a little bit more work, as follows. 
Let $S$, $I$, and $J$ be as in the formulation of the vacuum sector axiom, and let us assume without loss of generality that $J$ is white.
We need to show that, for almost every $x$, the representation of $\cala(J)\otimes_{alg}\cala(\bar J)$ on $H_0(S, D_x)$ extends to $\cala(J\cup \bar J)$.
We know that the corresponding representation of $\cala(J)\otimes_{alg}\cala(\bar J)$ on $H_0(S, D)=\int^\oplus H_0(S, D_x)$ does extend to $\cala(J\cup \bar J)$;
we want to see that this extension disintegrates into actions of $\cala(J\cup \bar J)$ on $H_0(S, D_x)$.
Certainly the action of $\cala(J\cup \bar J)$ on $H_0(S, D)$ commutes with that of $Z(D)$, but that is not enough to guarantee the action disintegrates into actions on the individual summands $H_0(S, D_x)$.
Pick a white interval $K\subset S$ that contains $J\cup \bar J$ in its interior, and an intermediate algebra $\cala(J\cup \bar J)\subset N \subset \cala(K)$ that is a direct sum of type $I$ factors.
By the same argument as used earlier in this proof, the action of $N$ on $\int^\oplus H_0(S, D_x)$ disintegrates into actions on $H_0(S, D_x)$, and therefore so does the action of $\cala(J\cup \bar J)$.
\end{proof}

\subsection*{\hspace*{-18pt}Irreducible defects over semisimple nets}
In Section~\ref{subsec:Composition of defects}, we will define 
the operation of fusion of defects, which is the composition of 1-morphisms in the 3-category of conformal nets.
That operation does not preserve irreducibility (even if the conformal nets are irreducible) and so, unlike for conformal nets, it is not advisable to restrict attention to irreducible defects.

%We pause to consider defects between non-irreducible conformal nets.
We call a defect $D$ \emph{faithful} if the homomorphisms $D(f)$ are injective for every embedding $f:I\to J$ of bicolored intervals. 

\begin{lemma} \label{lem:irred-defect==>irred-nets}        %--- this is a result about non-irreducible nets ---%
Let $\cala$ and $\calb$ be conformal nets, and
let $D$ be an $\cala$-$\calb$-defect that is irreducible and faithful.
Then $\cala$ and $\calb$ are irreducible.
\end{lemma}

\begin{proof}
Let $S$ be a genuinely bicolored circle, $I\subset S$ a white interval and $I'$ the closure of its complement.
Since $D$ is irreducible, the vacuum sector $H_0(S,D)$ is acted on jointly irreducibly by the algebras $D(J)$, $J\subset S$.

Since $D$ is faithful, $\cala(I)$ acts faithfully on $H_0(S,D)$.
A non-trivial central projection $p\in \cala(I)$ would thus induce a non-trivial direct sum decomposition of $H_0(S,D)$, contradicting the fact that it is irreducible.
Indeed, for a bicolored interval $J\subset S$, the projection $p$ commutes with both $D(J\cap I)$ and $D(J\cap I')$.
By strong additivity, $p$ therefore commutes with $D(J)$.
\end{proof}       

\noindent Here, as for conformal nets \cite[\S3.1]{BDH(nets)}, we have used the split property to extend the functor $D$ to disjoint unions of bicolored intervals by setting
\begin{equation}\label{eq: extend D to disconnected things}
D(I_1\cup\ldots\cup I_n):=D(I_1)\,\bar\ox\,\ldots\,\bar\ox\, D(I_n).
\end{equation}

\begin{corollary}\label{cor: irreducible defect}
Let $\cala=\bigoplus\cala_i$ and $\calb=\bigoplus\calb_j$ be semisimple conformal nets,
where $\cala_i$ and $\calb_j$ are irreducible.
Let $D$ be an irreducible $\cala$-$\calb$-defect.
Then there exist indices $i$ and $j$ such that $D$ is induced from a faithful irreducible $\cala_i$-$\calb_j$-defect
under the projections maps $\cala \rightarrow \cala_i$
and $\calb\rightarrow \calb_j$, respectively. \qed
\end{corollary}

The above discussion shows that defects between semisimple conformal nets can be entirely understood in terms of defects between irreducible conformal nets.
In the rest of this book, we will therefore mostly restrict attention to irreducible conformal nets.

\section{The category $\CN_1$ of defects}\label{sec: CN_1}
\begin{definition}  \label{def:CN_1}
  Defects form a symmetric monoidal category $\CN_1$.
  An object in that category is a triple $(\cala,\calb,D)$, where $\cala$ and $\calb$ are semisimple conformal nets, and $D$ is a defect from $\cala$ to $\calb$.
  A morphism between the objects $(\cala,\calb,D)$ and $(\cala',\calb',D')$ is triple of natural transformations 
  $\alpha \colon \cala \to \cala'$, $\beta \colon \calb\to \calb'$, $\delta \colon D \to D'$, 
  with the property that
  $\delta|_{\mathsf{INT}_{\circ}}=\alpha$ and 
  $\delta|_{\mathsf{INT}_{\bullet}}=\beta$.
  The symmetric monoidal structure on this category is given by objectwise spatial tensor product.
\end{definition}

Recall that a map between von Neumann algebras with finite-dimensional centers is said to be finite if the associated bimodule $_A{L^2B}_B$ is dualizable (Appendix~\ref{subsec:dualizability}). 

\begin{definition} \label{def:finite defects maps}
A natural transformation $\tau:D\to E$ between semisimple defects ${}_\cala D_\calb$ and ${}_{\calc} E_\cald$
is called finite if $\tau_I:D(I)\to E(I)$ is a finite homomorphism for every $I\in \INT_{\circ\bullet}$.
\end{definition}

\begin{remark}
We believe that the condition of having finite-dimensional centers is not really needed to define the notion of finite homomorphism between von Neumann algebras \cite[Conj. 6.17]{BDH(Dualizability+Index-of-subfactors)}.
If that is indeed the case, then we can extend the notion of finite natural transformations to non-semisimple defects.
\end{remark}

Let us denote by $\CN_0$ the symmetric monoidal category of semisimple conformal nets and their natural transformations (Appendix~\ref{subsec:defnets}),
and by $\CN_0^f$ the symmetric monoidal category of semisimple conformal nets all of whose irreducible summands have finite index (Appendix~\ref{subsec:finite-nets}),
together with the finite natural transformations (Appendix~\ref{subsec:defnets}).
Later on, we will denote by $\CN_1^f$ the symmetric monoidal category of semisimple defects (between semisimple conformal nets all of whose irreducible summands have finite index), together with finite natural transformations.
The category $\CN_1$ is equipped with two forgetful functors
\[
  \mathsf{source} \colon \CN_1\to \CN_0
  \qquad\quad
  \mathsf{target} \colon \CN_1\to \CN_0
\]
given by $\mathsf{source}(\cala,\calb,D) := \cala$ and 
$\mathsf{target}(\cala,\calb,D):=\calb$, respectively,
and a functor
\begin{equation}\label{eq: identity functor}
\mathsf{identity} \colon \CN_0 \to \CN_1
\end{equation}
given by  $\mathsf{identity} (\cala) (I) := \cala(I)$, where we forget the bicoloring of $I$ in order to evaluate $\cala$.
We sometimes abbreviate $\mathsf{identity} (\cala)$ by $1_\cala$.
Note that the above functors are all compatible with the symmetric monoidal structure.

\begin{remark}\label{rem : weak identity} A conformal net $\cala$ also has a \emph{weak identity} given on genuinely bicolored intervals $I$ by $I\mapsto \cala(I_\circ\cup [0,1]\cup I_\bullet)$.
That defect is not isomorphic to $1_\cala$ in the category $\CN_1$.
It is nevertheless equivalent to $1_\cala$ in the sense that there is an invertible sector between them; see Example \ref{ex: 1=/=1*1}.
\end{remark}

%----------------------------------------------------------------------
%\comment{
\section{Composition of defects}\label{subsec:Composition of defects}

Given conformal nets $\cala$, $\calb$, $\calc$, and defects $_{\cala}D_{\calb}$ and $_{\calb}E_{\calc}$, we will now define their fusion $D \circledast_{\calb} E$,
which is an $\cala$-$\calc$-defect if the conformal net $\calb$ has finite index.
If $\calb$ does not have finite index, then $D \circledast_{\calb} E$ might still be a defect, but we do not know how to prove this.

If $I$ is in $\INT_\circ$ or $\INT_\bullet$, then $(D\circledast_{\calb}E)(I)$ is given by $\cala(I)$ or $\calc(I)$, respectively.
If $I$ is genuinely bicolored, then  we use the local coordinate to construct intervals
\begin{equation*}
I^+:=I_{\circ}\cup [0,{\textstyle\frac{1}{2}}],\quad
I^{++}:=I_{\circ}\cup [0,{\textstyle\frac{3}{2}}],\quad
{}^+\hspace{-.15mm}I:=[-{\textstyle\frac{1}{2}},0]\cup I_{\bullet},\quad
{}^{++}\hspace{-.15mm}I:=[-{\textstyle\frac{3}{2}},0]\cup I_{\bullet},
\end{equation*}
bicolored by
\[
\begin{split}
I^{+}_{\circ}=I^{++}_{\circ} = I_{\circ},\quad \,\,
I^{+}_{\bullet} &= [0,{\textstyle\frac{1}{2}}],\qquad\,
I^{++}_{\bullet} = \,[0,{\textstyle\frac{3}{2}}],\\
{}^{+}\hspace{-.15mm}I_{\bullet} = {}^{++}\hspace{-.15mm}I_{\bullet} = I_{\bullet},\quad 
{}^{+}\hspace{-.15mm}I_{\circ} &= [-{\textstyle\frac{1}{2}},0],\quad
{}^{++}\hspace{-.15mm}I_{\circ} = [-{\textstyle\frac{3}{2}},0].
\end{split}
\]
Let $J:=[0,1]$ and consider the maps $J \to I^{++}_{\bullet}\hookrightarrow I^{++}$ and $J \to {}^{++}\hspace{-.15mm}I_\circ\hookrightarrow {}^{++}\hspace{-.15mm}I$ 
given by $x\mapsto \frac{3}{2}-x$ and $x\mapsto x-\frac{3}{2}$, respectively.
These embeddings induce homomorphisms 
$D(I^{++})\leftarrow \calb(J)^{op}$ and $\calb(J) \to E({}^{++}\hspace{-.15mm}I)$
that we use to form the fusion of the von Neumann algebras $D(I^{++})$ and $E({}^{++}\hspace{-.15mm}I)$
(Definition~\ref{def: Fusion of vN alg}).
We define 
\begin{equation}\label{eq:def-fusion-Defect}
\big(D \circledast_{\calb} E\big)(I):=
\begin{cases}
\,\cala(I)&\text{for }\, I\in\INT_\circ\\
D(I^{++}) \circledast_{\calb(J)} E({}^{++}\hspace{-.15mm}I)&\text{for }\, I\in\INT_\halfbullet\\
\,\calc(I)&\text{for }\, I\in\INT_\bullet\\
\end{cases}
\end{equation}
Pictorially, this is
\begin{equation}\label{eq:def-fusion-Defect (pictorial)}
\big( D\circledast_{\calb}E \big) \left( \tikzmath[scale=.4] {\useasboundingbox (-2.1,-1.1) rectangle (2.1,1.1);\draw (-2,-.7) [rounded corners=7pt]-- (-2,-.1) -- 
(-1.5,1) [rounded corners=4pt]-- (-.6,0) -- (0,0); \draw[ultra thick] (0,0) [rounded corners=3pt]-- (.5,0) -- (1,-.3) -- (1.7,.5) -- (2.2,.2);} %tikzmath
\right)  :=    D \left( \tikzmath[scale=.4] {\useasboundingbox (-2.1,-1.1) rectangle (.9,1.05); \draw (-2,-.7) [rounded corners=7pt]-- (-2,-.1) -- 
(-1.5,1) [rounded corners=4pt]-- (-.6,0) -- (0,0); \draw[thick, double, rounded corners=1.5](0,0) -- (.55,0) -- (.55,-1);} %tikzmath
\right) \circledast_{\calb \left(\tikzmath[scale=.3]{\draw[thick, double] (0,0) -- (0,-1);} %tikzmath
\right)}   E \left( \tikzmath[scale=.4] {\useasboundingbox (-.75,-1.1) rectangle (2.1,1.1);  \draw[thick, double, rounded corners=1.5] (-.5,-1) -- (-.5,0) -- (0,0);
\draw[ultra thick] (0,0) [rounded corners=3pt]-- (.5,0) -- (1,-.3) -- (1.7,.5) -- (2.2,.2);} %tikzmath
\right)
\end{equation}
If $I$ is genuinely bicolored then, by Proposition \ref{prop: [Haag duality for defects]}, we have
\[
\big(D \circledast_{\calb} E\big)(I)=\big(D(I^{++})\cap{\calb(J)^\op\hspace{.2mm}}'\hspace{.2mm}\big)\vee\big(E({}^{++}\hspace{-.15mm}I)\cap\calb(J)'\big)=
D(I^+)\vee E({}^{+}\hspace{-.15mm}I),
\]
where the algebras act on $H\boxtimes_{\calb(J)}K$
for some faithful $D(I^{++})$-module $H$ and some faithful $E({}^{++}\hspace{-.15mm}I)$-module $K$.
Therefore, we obtain the following equivalent definition of composition of defects:\medskip
\[
\tikzmath{
\node[draw, inner ysep=6, inner xsep=10] at (0,0) {\parbox{11cm}{
\addtocounter{theorem}{1}
{\bf Definition \arabic{chapter}.\arabic{theorem}.}
The algebra $(D \circledast_{\calb} E)(I)$ is the completion of the algebraic tensor product
$D\big(I^+\big)\otimes_\alg E\big({}^{+}\hspace{-.15mm}I\big)$
inside $\bfB(H\boxtimes_{\calb(J)}K)$, where $H$ is a faithful 
$D(I^{++})$-module, and $K$ is a faithful $E({}^{++}\hspace{-.15mm}I)$-module.}};
}\medskip
\]
We conjecture that that $D \circledast_\calb E$ is always an $\cala$-$\calc$-defect.
Our first main theorem says that this holds when $\calb$ has finite index.

\begin{maintheorem}
  \label{thm:fusion-of-defects-is-defect}
  Let $\cala$, $\calb$, and $\calc$ be irreducible conformal nets,
  and let us assume that $\calb$ has finite index.
  If $D$ is a defect from $\cala$ to $\calb$,
  and $E$ a defect from $\calb$ to $\calc$, then $D \circledast_\calb E$ is a
  defect from $\cala$ to $\calc$.
\end{maintheorem}

\begin{proof}
We first prove isotony.
Let $I_1\subset I_2$ be genuinely bicolored intervals, let $H$ be a faithful $D(I_2^{++})$-module and let $K$ be a faithful $E({}^{++}\hspace{-.15mm}I_2)$-module.
By the isotony property of $D$ and $E$, the actions of $D(I_1^{++})$ on $H$ and of $E({}^{++}\hspace{-.15mm}I_1)$ on $K$ are faithful.
Therefore, both $(D \circledast_{\calb} E)(I_1)$ and $(D \circledast_{\calb} E)(I_2)$ can be defined as subalgebras of $\bfB(H\boxtimes_{\calb(J)}K)$.
It is then clear that $(D \circledast_{\calb} E)(I_1)$ is a subalgebra of $(D \circledast_{\calb} E)(I_2)$.

We next show locality and strong additivity.
Let $J \subset I$ and $K\subset I$ be bicolored intervals whose union is $I$ and that intersect in a single point.
We assume without loss of generality that $K$ is white and that $I$ and $J$ are genuinely bicolored.
In particular, we then have ${}^{+}\hspace{-.15mm}I={}^{+}\hspace{-.15mm}J$.
By the strong additivity of $D$, we have
\[
\cala(K)\vee(D \circledast_{\calb} E)(J) = \cala(K)\vee D(J^+)\vee E({}^{+}\hspace{-.15mm}J) = D(I^+)\vee E({}^{+}\hspace{-.15mm}I) = (D \circledast_{\calb} E)(I),
\]
which proves that $D \circledast_{\calb} E$ is also strongly additive.
Since $D$ satisfies locality, the images of $\cala(K)$ and $D(J^+)$ commute in $D(I^+)$.
The algebra $D(I^+)$ commutes with $E({}^{+}\hspace{-.15mm}I) = E({}^{+}\hspace{-.15mm}J)$ by the definition of $\circledast$.
It follows that all three algebras $\cala(K)$, $D(J^+)$, and $E({}^{+}\hspace{-.15mm}J)$ commute with one another.
The algebras $\cala(K)$ and $(D \circledast_{\calb} E)(J)$ therefore also commute, as required.

The vacuum sector axiom is much harder.
Let us first assume that $D$ and $E$ are irreducible.
Let $J\subset I$ be as in the formulation of the vacuum sector axiom (Definition~\ref{def:Defect}), and let us assume without loss of generality that $J$ is white.
We need to show that the 
$\cala(J)\otimes_\alg\cala(\bar J)$-module structure
on $L^2((D \circledast_{\calb} E)(I))$ given by (\ref{eq:actions-on-L2}, \ref{eq: l otimes r : D otimes D --> B(L2(D))}) extends to an action of $\cala(J\cup\bar J)$.
This will follow from the existence of an injective homomorphism from $L^2((D \circledast_{\calb} E)(I))$ into some other $\cala(J)\otimes_\alg\cala(\bar J)$-module
that is visibly an $\cala(J\cup\bar J)$-module.
The desired homomorphism is \eqref{eq: main equation of thm G=L2} and will be constructed in Proposition~\ref{prop:G=L2}.
The fact that $\cala(J\cup\bar J)$ acts on the codomain of \eqref{eq: main equation of thm G=L2} is an immediate consequence of the vacuum sector axiom for $D$.

For general defects $D$ and $E$, write them as direct integrals
$D=\int^\oplus D_x$ and $E=\int^\oplus E_y$ of irreducible defects, and note that $D \circledast_\calb E=\iint^\oplus D_x \circledast_\calb E_y$ is a defect by Remark~\ref{rem: direct integral of defects}.
\end{proof}

In view of Corollary~\ref{cor: irreducible defect} and the fact that any defect between semisimple conformal nets can be disintegrated into irreducible defects (Lemma~\ref{lem: disintegrate}),
the above theorem generalizes in a straightforward way to the situation where $\cala$, $\calb$, and $\calc$ are not necessarily irreducible but merely semisimple: in this case,
if all the irreducible summands of $\calb$ have finite index, then the composition of an $\cala$-$\calb$-defect with a $\calb$-$\calc$-defect is an $\cala$-$\calc$-defect.

One might hope that composition of defects induces a functor
\begin{equation}\label{eq: composition that doesn't exist}
\mathsf{composition} \colon \CN_1 \times_{\CN_0} \CN_1 \to \CN_1.
\end{equation}
However, some caution is needed. 
First, we used the finite index condition on $\calb$ for our proof that $D\circledast_\calb E$ is a defect.  Second and
more important, the operation of fusion of von Neumann algebras 
is only functorial with respect to isomorphisms of von Neumann algebras:
given homomorphisms $A_1\leftarrow C_1^\op$, $C_1\to B_1$ and $A_2\leftarrow C_2^\op$, $C_2\to B_2$ it is \emph{not} true
that a triple of maps $a:A_1\to A_2$, $b:B_1\to B_2$, $c:C_1\to C_2$ (subject to the obvious compatibility conditions) induces a map
\begin{equation}\label{eq: does NOT exist}
a\circledast_c b: A_1\circledast_{C_1} B_1 \to A_2\circledast_{C_2} B_2.
\end{equation}
Moreover, requiring that the maps $a$, $b$, and $c$ be finite homomorphisms does not help to construct the map \eqref{eq: does NOT exist}.
However, unlike the fusion of von Neumann algebras, the composition of defects is functorial for more than just isomorphisms.
%Recall that all the conformal nets are assumed irreducible unless otherwise stated.

\begin{proposition}\label{prop: fusion of defects is a functor}
Let $a:\cala_1\to\cala_2$, $b:\calb_1\to\calb_2$, and $c:\calc_1\to\calc_2$ be natural transformations between irreducible conformal nets.
Let ${}_{\cala_1} {D_1}_{\,\calb_1}$, ${}_{\cala_2} {D_2}_{\,\calb_2}$, ${}_{\calb_1} {E_1}_{\,\calc_1}$, and ${}_{\calb_2} {E_2}_{\,\calc_2}$ be defects,
and let $d:D_1\to D_2$, $e:E_1\to E_2$ be natural transformations such that $d|_{\INT_\circ}=a$, $d|_{\INT_\bullet}=e|_{\INT_\circ}=b$, and $e|_{\INT_\bullet}=c$.
If the natural transformation $b$ is finite (Appendix~\ref{subsec:defnets}), %Appendix~\ref{subsec:finite-nets} and Appendix~\ref{subsec:dualizability}).
then the above maps induce a natural transformation
\[ %\begin{equation}\label{eq: BDE --> BDE}
D_1\circledast_{\calb_1} E_1 \to D_2\circledast_{\calb_2} E_2.
\] %\end{equation}

Moreover, if $\calb_i$ have finite index, $D$ and $E$ are semisimple, and $d$ and $e$ are finite,
then the defects $D_i\circledast_{\calb_i} E_i$ are semisimple and the above natural transformation is finite.
\end{proposition}

\begin{proof}
Given a genuinely bicolored interval $I$, we need to construct a homomorphism 
$(D_1\circledast_{\calb_1} E_1)(I) \to (D_2\circledast_{\calb_2} E_2)(I)$.
We assume without loss of generality that $d$ and $e$ are faithful (otherwise, their kernels are direct summands).
Let $H$ be a faithful $D_2(I^{++})$-module, and let $K$ be a faithful $E_2({}^{++}\hspace{-.15mm}I)$-module.
By ~\cite[Thm. 6.23]{BDH(Dualizability+Index-of-subfactors)},
the natural transformation $b$ induces a bounded 
linear map $H\boxtimes_{\calb_1(J)}K\to H\boxtimes_{\calb_2(J)}K$, which is surjective by construction.
That map is equivariant with respect to the homomorphism
$D_1\big(I^+\big)\otimes_\alg E_1\big({}^{+}\hspace{-.15mm}I\big)\to
D_2\big(I^+\big)\otimes_\alg E_2\big({}^{+}\hspace{-.15mm}I\big)$,
and therefore induces a map from the completion of $D_1\big(I^+\big)\otimes_\alg E_1\big({}^{+}\hspace{-.15mm}I\big)$
in $\bfB(H\boxtimes_{\calb_1(J)}K)$ to the completion of $D_2\big(I^+\big)\otimes_\alg E_2\big({}^{+}\hspace{-.15mm}I\big)$
in $\bfB(H\boxtimes_{\calb_2(J)}K)$.

Let us now assume that the conformal nets $\calb_i$ have finite index, the defects $D$ and $E$ are semisimple, and the natural transformations $d$ and $e$ are finite.
The semisimplicity of $D_i\circledast_{\calb_i} E_i$ is then the content of Theorem~\ref{thm: semi-simplicity of DoE}, and
the homomorphism $(D_1\circledast_{\calb_1} E_1)(I) \to (D_2\circledast_{\calb_2} E_2)(I)$ is finite 
by~\cite[\lemAvBAvBfinite]{BDH(Dualizability+Index-of-subfactors)}.
%\footnote{Note that there is a missing assumption in \cite[\lemAvBAvBfinite]{BDH(Dualizability+Index-of-subfactors)}: the map $h:H_1\to H_2$ should be surjective.}
\end{proof} 

We do not know whether the functor \eqref{eq: composition that doesn't exist} exists as stated. However, instead of trying to compose over the full category $\CN_0$ of semisimple conformal nets, we can restrict attention to the subcategory $\CN_0^f \subset \CN_0$ of semisimple conformal nets all of whose irreducible summands have finite index, together with their finite natural transformations.  If we let $\CN_1 \times_{\CN^f_0} \CN_1$ be a shorthand notation for 
$\CN_1 \times_{\CN_0} \CN^f_0\times_{\CN_0} \CN_1$,
%\AB{Having three terms is confusing. The reason is that we want to avoid
% an notation for the category of defects where the white or black net is
% finite?}
then the composition functor
\begin{equation}\label{eq: composition -- this time between the correct categories}
\mathsf{composition} \,\colon\, \CN_1 \times_{\CN^f_0} \CN_1 \to \CN_1
\end{equation}
exists by Theorem \ref{thm:fusion-of-defects-is-defect} and Proposition \ref{prop: fusion of defects is a functor}.

\section{Associativity of composition}
It will be convenient to work with the square model $S^1:=\partial[0,1]^2$ of the ``standard circle'' (see the beginning of Chapter \ref{sec:sectors})
and to use the following notation.

\begin{notation}\label{not:names-for-intervals}
Given real numbers $a<b$ and $c<d$ and $M=[a,b]\!\times\![c,d]$, we let
\begin{alignat*}{2}
\partial^\sqsubset M&:=(\{a\}\!\times\! [c,d]) \cup ([a,b]\!\times\!\{c,d\})\quad &\partial^{\hspace{.2mm}\tikzmath[scale=.15]{\draw(0,0)--(0,1)--(1,1);}} M :=(\{a\}\!\times\! [c,d]) \cup ([a,b]\!\times\!\{d\})\\
\partial^\sqsupset M&:=(\{b\}\!\times\! [c,d]) \cup ([a,b]\!\times\!\{c,d\})\quad &\partial^{\hspace{.2mm}\tikzmath[scale=.15]{\draw(1,0)--(0,0)--(0,1);}} M :=(\{a\}\!\times\! [c,d]) \cup ([a,b]\!\times\!\{c\})\\
\partial^\sqcap M&:=(\{a,b\}\!\times\! [c,d]) \cup ([a,b]\!\times\!\{d\})\quad &\partial^{\hspace{.2mm}\tikzmath[scale=.15]{\draw(0,1)--(1,1)--(1,0);}} M :=(\{b\}\!\times\! [c,d]) \cup ([a,b]\!\times\!\{d\})\\
\partial^\sqcup M&:=(\{a,b\}\!\times\! [c,d]) \cup ([a,b]\!\times\!\{c\})\quad &\partial^{\hspace{.2mm}\tikzmath[scale=.15]{\draw(0,0)--(1,0)--(1,1);}} M :=(\{b\}\!\times\! [c,d]) \cup ([a,b]\!\times\!\{c\})
\end{alignat*}
be the subsets of $\partial M$ hinted by the pictorial superscript.
\end{notation}

\begin{definition}\label{def: L2(D)}
The standard bicolored circle is
$S^1=\partial [0,1]^2$ with bicoloring $S^1_\circ=\partial^\sqsubset ([0,\frac12]\!\times\![0,1])$ and $S^1_\bullet=\partial^\sqsupset ([\frac12,1]\!\times\![0,1])$.  The upper half of this circle is $S^1_\top=\partial^\sqcap([0,1]\!\times\![\frac12,1])$, and the standard involution is $(x,y)\mapsto (x,1-y)$.  The vacuum sector associated to the standard bicolored circle, its upper half, and its standard involution, is
\[
H_0(D):=L^2(D(S^1_\top)).
\]
It has left actions of $D(I)$ for every interval $I\subset S^1$
such that $(\frac12,0)$ and $(\frac12,1)$ are not both in $I$ and $\{(\frac12,0),(\frac12,1)\}\cap \partial I=\emptyset$.
\end{definition}

The fiber product $\ast$ of von Neumann algebras
was studied in~\cite{Timmermann(Invitation-2-quantum-groups)}.
It is an alternative to the fusion $\circledast$ of
von Neumann algebras which remedies the formal shortcomings of the fusion operation
(see Appendix~\ref{subsec:fusion+fiber-prod}).
In view of this, one might rather have defined the composition of defects as
\begin{equation}\label{eq: alternative def of fusion}
\big(D \ast_{\calb} E\big)(I):=
\begin{cases}
\,\cala(I)&\text{for }\, I\in\INT_\circ\\
D(I^{++}) \ast_{\calb(J)} E({}^{++}\hspace{-.15mm}I)&\text{for }\, I\in\INT_\halfbullet\\
\,\calc(I)&\text{for }\, I\in\INT_\bullet
\end{cases}
\end{equation}
where $I^{++}$, ${}^{++}\hspace{-.15mm}I$, and $J$ are as in~\eqref{eq:def-fusion-Defect}.
This is related to the previous 
definition~\eqref{eq:def-fusion-Defect} as follows.
Let $S^1$ be the standard bicolored circle
 (see Definition \ref{def: L2(D)}), with upper and lower 
halves $S^1_\top$ and $S^1_\bot$. 

\begin{lemma}\label{lem: circledast versus boxast}
Let ${}_\cala D_\calb$ and ${}_\calb E_\calc$ be defects, with corresponding vacuum sectors $H:=H_0(D)$ and $K:=H_0(E)$.
Viewed as algebras acting on $H\boxtimes_\calb K$, we then have
\[
(D\ast E)(S^1_\top) = \big((D\circledast E)(S^1_\bot)\big)'\qquad\text{and}\qquad (D\ast E)(S^1_\bot) = \big((D\circledast E)(S^1_\top)\big)'.
\]
\end{lemma}
\begin{proof}
Using a graphical representation as 
in~\eqref{eq:def-fusion-Defect (pictorial)}, 
we have:
\[
\begin{split}
\hspace{1cm}(D\ast E)(S^1_\top)\,&=
D\big(\tikzmath[scale=\textscale]
  {   \draw (0,6) -- (0,12) -- (6,12);
      \draw[thick, double]  (6,12) -- (12,12) -- (12,0); }\big)%
\ast_{\calb \left(\tikzmath[scale=.3]{\draw[thick, double] (0,0) -- (0,-1);}\right)}
E\big(\tikzmath[scale=\textscale]
  {   \draw[ultra thick] (12,6) -- (12,12) -- (6,12);
      \draw[thick, double]  (6,12) -- (0,12) -- (0,0); }\big)\\%
&=\Big(D\big(\tikzmath[scale=\textscale]
  {   \draw (0,6) -- (0,12) -- (6,12);
      \draw[thick, double]  (6,12) -- (12,12) -- (12,0); }\big)'%
\otimes_\alg
E\big(\tikzmath[scale=\textscale]
  {   \draw[ultra thick] (12,6) -- (12,12) -- (6,12);
      \draw[thick, double]  (6,12) -- (0,12) -- (0,0); }\big)'\Big)'\qquad\tikzmath{\draw (0,0) node[scale=.85, anchor=west] {(commutants taken on $H$,} (.1,-.35) node[scale=.85, anchor=west] {$K$, and $H\boxtimes_\calb K$, respectively)};}\\%
&=\Big(D\big(\tikzmath[scale=\textscale]
  {   \draw (0,6) -- (0,0) -- (6,0);
      \draw[thick, double]  (6,0) -- (12,0); }\big)%
\otimes_\alg
E\big(\tikzmath[scale=\textscale]
  {   \draw[ultra thick] (12,6) -- (12,0) -- (6,0);
      \draw[thick, double]  (6,0) -- (0,0); }\big)\Big)'\\%
&=\Big(D\big(\tikzmath[scale=\textscale]
  {   \draw (0,6) -- (0,0) -- (6,0);
      \draw[thick, double]  (6,0) -- (12,0); }\big)%
\vee E\big(\tikzmath[scale=\textscale]
  {   \draw[ultra thick] (12,6) -- (12,0) -- (6,0);
      \draw[thick, double]  (6,0) -- (0,0); }\big)\Big)'%
=\big((D\circledast E)(S^1_\bot)\big)'.
\end{split}
\]
where the third equality follows by Haag duality (Proposition \ref{prop: [Haag duality for defects]}).
The second equation is similar.
\end{proof}

When the conformal net $\calb$ has finite index (and conjecturally even without that restriction), the two definitions of fusion \eqref{eq:def-fusion-Defect} and \eqref{eq: alternative def of fusion} actually agree:
%This will be a key ingredient in the proof of Theorem \ref{thm: Omega is an iso}.

\begin{theorem}\label{thm:fusion-of-defects-is-generated}
  Let ${}_\cala D_\calb$ and ${}_\calb E_\calc$ be defects.
  If $\calb$ has finite index, then for every bicolored interval $I$ the inclusion
  \begin{equation}\label{eq: circledast --> ast}
  \big(D \circledast_{\calb} E\big)(I) 
  % =
  %D(I^{++}) \!\underset{\calb(J)}\circledast\! E({}^{++}\hspace{-.15mm}I)\\
  \,\,\,\hookrightarrow\,\,\,  \big(D \ast_{\calb} E\big)(I)  %:= D(I^{++}) \!\underset{\calb(J)}\ast\! E({}^{++}\hspace{-.15mm}I)
  %\end{split}
  \end{equation}
  is an isomorphism.
\end{theorem}

\begin{proof}
If $I\in \INT_\circ$ or $I \in \INT_\bullet$, then there is nothing to show. 
Recall the Notation~\ref{not:names-for-intervals}. 
%\tikzmath[scale=\textscale] {\draw (12,6) -- (12,12) -- (18,12);\draw[ultra thick] (18,12) -- (24,12) -- (24,6);}:
For $I := \partial^\sqcap([0,1]  \times   [\frac{1}{2},1])$, with bicoloring $I_\circ:=I_{x\le \frac12}$ and $I_\bullet:=I_{x\ge \frac12}$,
the equality $(D \circledast_{\calb} E)(I) = (D \ast_{\calb} E)(I)$ is the content of Corollary \ref{cor:fiber-product-and-nets}.
The result follows as every genuinely bicolored interval is isomorphic to this interval.
\end{proof}

Using the above theorem, the associator
\begin{equation}\label{eq:associator for defects}
\big(D \circledast_{\calb}E\big) \circledast_{\calc}F\cong D \circledast_{\calb}\big(E \circledast_{\calc}F\big)
\end{equation}
is then induced from the associator for the operation $\ast$ of 
fiber product of von Neumann algebras.
If $I$ is a genuinely bicolored interval, 
% %(Definition \ref{def: Fusion of vN alg}). then following \eqref{eq:def-fusion-Defect},
then evaluating the two sides 
of~\eqref{eq:associator for defects} on $I$ yields
\begin{equation}\label{eq:associator for defects, expanded}
\Big(D(I^{++}) \underset{\calb(J)}\ast E(K)\Big) \underset{\calc(J')}\ast F({}^{++}\hspace{-.15mm}I)\,\,\,\,\text{and}\,\,\,\,
D(I^{++}) \underset{\calb(J)}\ast \Big( E(K) \underset{\calc(J')}\ast F({}^{++}\hspace{-.15mm}I)\Big),
\end{equation}
where
\[
I^{++}=I_{\circ}\cup [0,{\textstyle\frac{3}{2}}],\quad K=[-{\textstyle\frac{3}{2}},{\textstyle\frac{3}{2}}],\quad {}^{++}\hspace{-.15mm}I=[-{\textstyle\frac{3}{2}},0]\cup I_{\bullet},\quad J=J'=[0,1],
\]
and the embeddings $I^{++}\hookleftarrow J \hookrightarrow K\hookleftarrow J'\hookrightarrow {}^{++}\hspace{-.15mm}I$ are as in \eqref{eq:def-fusion-Defect}.
The associator relating the two sides 
of~\eqref{eq:associator for defects, expanded} 
(see \cite[Prop.~9.2.8]{Timmermann(Invitation-2-quantum-groups)} 
for a construction) is the desired natural 
isomorphisms~\eqref{eq:associator for defects}.
The properties of \eqref{eq:associator for defects} can then be 
summarized by saying that it provides a natural transformation
\begin{equation}\label{eq: associator natural transformation}
\mathsf{associator}\colon \CN_1 \times_{\CN_0^f} \CN_1\times_{\CN_0^f} \CN_1\, \tworarrow\, \CN_1
\end{equation}
that is an associator for the 
composition~\eqref{eq: composition -- this time between the correct categories}.
This associator satisfies the pentagon identity
by the corresponding pentagon identity for the operation $\ast$.

%==================================================================

\chapter{Sectors}
  \label{sec:sectors}

We will use the constant speed parametrization
to identify the standard circle $\{ z \in \IC : |z| = 1 \}$ with the boundary  of the unit square $\dd[0,1]^2$.
Under our identification, the points $1$, $i$, $-1$, and $-i$ get mapped to $(1,\frac12)$, $(\frac12,1)$, $(0,\frac12)$, and $(\frac12,0)$, respectively.

Recall the Notation~\ref{not:names-for-intervals}.
Our standard circle $S^1=\partial [0,1]^2$ has a 
standard bicoloring given by
$S^1_{\circ} := \partial^\sqsubset ([0,\frac{1}{2}] \!\times\! [0,1])$
and 
$S^1_{\bullet} := \partial^\sqsupset([\frac{1}{2},1] \!\times\! [0,1])$.
Let $\INT_{S^1}$ be the poset of subintervals of $S^1$, and 
let $\INT_{\!S^1\!,\circ\bullet}$ be the sub-poset of intervals 
$I \subset S^1$ such that $(I \cap S^1_{\circ}, I \cap S^1_{\bullet})$ is a bicoloring.
Thus, an interval $I$ is in $\INT_{\!S^1\!,\circ\bullet}$ if neither of the color-change points $(\frac12,0)$ and $(\frac12,1)$
are in its boundary, and if both $I \cap S^1_{\circ}$ and $I \cap S^1_{\bullet}$ are connected (possibly empty).
We view $\INT_{\!S^1\!,\circ\bullet}$ as a (non-full) subcategory of $\INT_{\circ\bullet}$.

\section{The category $\CN_2$ of sectors}\label{sec: The category CN2 of sectors}
The elements of $\INT_{\!S^1\!,\circ\bullet}$
naturally fall into four classes:
\begin{equation}\label{eq:4-INT-in-S1}
\begin{split}
\phantom{\big(} & \INT_{\!S^1\!,\circ}\\
\phantom{\big(} & \INT_{\!S^1\!,\bullet}\\
\phantom{\big(} & \INT_{\!S^1\!,\top}\\
\phantom{\big(} & \INT_{\!S^1\!,\bot}
\end{split}
\begin{split}
& := \{ I \subset S^1 \mid I \cap S^1_{\bullet}=\emptyset\}\\
& := \{ I \subset S^1 \mid I \cap S^1_{\circ}=\emptyset\}\\
& := \{ I \subset S^1 \mid ({\textstyle\frac{1}{2}},1)\in I\setminus\partial I 
          \,\text{ and }\, ({\textstyle\frac{1}{2}},0)\not\in I\}\\
& := \{ I \subset S^1 \mid ({\textstyle\frac{1}{2}},0)\in I\setminus\partial I 
          \,\text{ and }\, ({\textstyle\frac{1}{2}},1)\not\in I\}.
\end{split}
\end{equation}

\begin{definition}
  \label{def:sector-for-defects}
  Let $\cala$ and $\calb$ be conformal nets, and let
  $_{\cala}D_{\calb}$ and $_{\cala}E_{\calb}$ be $\cala$-$\calb$-defects.
  A $D$-$E$-sector is a Hilbert space $H$, equipped with homomorphisms
  \begin{equation} \label{eq:4rho}
   \begin{split}
    \rho_I \colon \cala(I) \to \bfB(H)\qquad 
    &\text{for}\quad I \in \INT_{\!S^1\!,\circ}\\
    \rho_I \colon D(I)\to \bfB(H)\qquad 
    &\text{for}\quad I \in \INT_{\!S^1\!,\top}\\
    \rho_I \colon \calb(I)\to \bfB(H)\qquad 
    &\text{for}\quad I \in \INT_{\!S^1\!,\bullet}\\
    \rho_I \colon E(I)\to \bfB(H)\qquad 
    &\text{for}\quad I \in \INT_{\!S^1\!,\bot}\,
   \end{split}
  \end{equation}
  subject to the condition $\rho_I|_{J} = \rho_J$ whenever $J\subset I$.
  Moreover, if $I \in \INT_{\!S^1\!,\top}$ and
  $J \in \INT_{\!S^1\!,\bot}$ 
  are intervals with disjoint interiors, then $\rho_I(D(I))$ and 
  $\rho_J(E(J))$ are required to commute with each other.
  We write $_DH_E$ to indicate that $H=(H,\rho)$ is a $D$-$E$-sector.
  If $D = E$, then we say that $H$ is a $D$-sector.
\end{definition}

Pictorially we will draw a $D$-$E$-sector as follows:
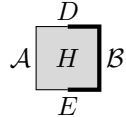
\[
\tikzmath[scale=\squarescale]
{\fill[spacecolor](0,0) rectangle(12,12);\draw (6,0) -- (0,0) -- (0,12) -- (6,12);\draw[ultra thick](6,12) -- (12,12) -- (12,0) -- (6,0); 
\draw (-3,6) node {$\cala$}(15,6) node {$\calb$}(6,15) node {$D$}(6,-3) node {$E$}(6,6)node {$H$};
}  %tikzmath
\]
The thin line stands for the conformal net $\cala$ and the thick line stands for $\calb$.

\begin{remark}
  If $I \in \INT_{\!S^1\!,\top}$ and 
  $J \in \INT_{\!S^1\!,\bot}$ 
  are disjoint intervals, we do {\em not} require that the action
  of $D(I)\ox_{\mathit{alg}}E(J)$ extends to an action of
  $D(I) \, \bar\ox \, E(J)$.
\end{remark}

Recall from Proposition \ref{prop: CCdefects == VNalg} that if $\cala$ and $\calb$ are both equal to the trivial conformal net $\underline\IC$,
then a $\underline\IC$-$\underline\IC$-defect may be viewed simply as a von Neumann algebra.
A $D$-$E$-sector between two such defects is given by a bimodule between the corresponding von Neumann algebras.

The following lemma is a straightforward analog 
of~\cite[\lemopencoverofcirclesector]{BDH(nets)}.

\begin{lemma}\label{lem: open cover of circle => sector -- BIS}
Let $S^1$ be the standard bicolored circle, and
let $\{I_i\subset S^1\}$ be a family of bicolored intervals whose interiors cover $S^1$.
Suppose that we have actions 
\begin{equation*}
\begin{split}
\rho_i \colon \cala(I_i) \to \bfB(H),\qquad 
&\text{\rm for}\quad I_i \in \INT_{\!S^1\!,\circ}\\
\rho_i \colon D(I_i)\to \bfB(H),\qquad 
&\text{\rm for}\quad I_i \in \INT_{\!S^1\!,\top}\\
\rho_i \colon\, \calb(I_i)\to \bfB(H),\qquad 
&\text{\rm for}\quad I_i \in \INT_{\!S^1\!,\bullet}\\
\rho_i \colon E(I_i)\to \bfB(H),\qquad 
&\text{\rm for}\quad I_i \in \INT_{\!S^1\!,\bot}\,,
\end{split}
\end{equation*}
subject to the following two conditions: 1. $\rho_i|_{I_i\cap I_j}=\rho_j|_{I_i\cap I_j}$,
2. if $J\subset I_i$ and $K\subset I_j$ are disjoint, then $\rho_{i}(\cala(J))$ commutes with $\rho_{j}(\cala(K))$.
%2. if $I_i\cap I_j=\emptyset$, then the images of $\rho_i$ and $\rho_j$ commute.
Then these actions endow $H$ with the structure of a $D$-$E$-sector.
\end{lemma}

\begin{proof}
Given an interval $J\subset S^1$, pick a diffeomorphism $\varphi\in \Diff_+(S^1)$ that is trivial in a neighborhood $N$ of the two color changing points, 
and such that $\varphi(J)\subset I_{i_0}$ for some $I_{i_0}$ in our cover.
Write $\varphi=\varphi_n\circ\ldots\circ \varphi_1$ for diffeomorphisms $\varphi_k$ that are trivial on $N$ and whose supports lie in elements of the cover.
Let $u_k$ be unitaries implementing $\varphi_k$ (Proposition \ref{prop:inncovdef}).
Upon identifying $u_k$ with its image (under the relevant $\rho_i$) in $\bfB(H)$, we set
\begin{equation}\label{eq:diff trick for nets ---- bis}
\qquad\rho_J(a):=u^*_1\ldots u^*_n\rho_{i_0}\big(\varphi(a)\big)u_n\ldots u_1.
\end{equation}
Here we have used $\varphi(a)$ as an abbreviation for $\cala(\varphi)(a)$, $D(\varphi)(a)$, $\calb(\varphi)(a)$, or $E(\varphi)(a)$, depending on whether $J$ is a white, top, black, or bottom interval.  Finally, as in the proof of~\cite[\lemopencoverofcirclesector]{BDH(nets)}, one checks that $\rho_J|_K=\rho_{I_\ell}|_K$ for any sufficiently small interval $K\subset J \cap I_\ell$,
and then uses strong additivity to conclude that $\rho_J|_{J\cap I_\ell}=\rho_{I_\ell}|_{J\cap I_\ell}$.
\end{proof}

As before let $S^1_{\top} = \partial^\sqcap([0,1] \! \x \!  [\frac{1}{2},1])$ and $S^1_{\bot} = \partial^\sqcup([0,1] \! \x \! [0,\frac{1}{2}])$
be the upper and lower halves of the standard bicolored circle.

\begin{definition} \label{def:CN2}
  Sectors form a category that we call $\CN_2$.
  Its objects are quintuples $(\cala, \calb,D,E,H)$, where 
  $\cala$, $\calb$ are semisimple conformal nets,
  $D$, $E$ are $\cala$-$\calb$-defects,
  and $H$ is a $D$-$E$-sector.
  A morphism from $(\cala, \calb,D,E,H)$ to 
  $(\cala', \calb',D',E',H')$
  consists of four compatible invertible natural transformations
  \[
  \alpha \colon \cala \to \cala',\quad
  \beta \colon \calb \to \calb',\quad 
  \delta \colon D \to D',\quad 
  \varepsilon \colon E \to E',
  \]
  along with a bounded linear map $h \colon H \to H'$ that is
  equivariant in the sense that
  \[
  \begin{split}
  \rho'_I(\alpha(a))\circ h=h\circ \rho_I(a),\qquad\quad
  \rho'_I(\beta(b))\circ h=h\circ \rho_I(b),\\
  \rho'_I(\delta(d))\circ h=h\circ \rho_I(d),\qquad\quad
  \rho'_I(\varepsilon(e))\circ h=h\circ \rho_I(e),
  \end{split}
  \]
  for 
  $a\in\cala(I)$, $b\in\calb(I)$, $d\in D(I)$, $e\in E(I)$,
  and $I \in 
   \INT_{\!S^1\!,\circ}$,
  $\INT_{\!S^1\!,\bullet}$,
  $\INT_{\!S^1\!,\top}$, 
  $\INT_{\!S^1\!,\bot}$
  respectively.

\end{definition}

\noindent 
There is also a symmetric monoidal structure on $\CN_2$ given by objectwise spatial tensor product for the functors $\cala$, $\calb$, $D$, $E$, and by tensor product of Hilbert spaces.

The category $\CN_2$ is equipped with two forgetful functors
\[
\mathsf{source_v} \colon \CN_2 \to \CN_1
\qquad\quad
\mathsf{target_v} \colon \CN_2\to \CN_1
\]
called `vertical source' and `vertical target', given by $\mathsf{source_v}(\cala, \calb,D,E,H) = \linebreak(\cala,\calb,D)$ and $\mathsf{target_v}(\cala,\calb,D,E,H) = (\cala,\calb,E)$.
They satisfy
\[
\mathsf{source}\circ \mathsf{source_v}  = \mathsf{source} \circ \mathsf{target_v}
\quad\text{and}\quad
\mathsf{target} \circ \mathsf{source_v}  = \mathsf{target} \circ \mathsf{target_v}.
\]
Provided we restrict to the subcategory $\CN_1^f\subset \CN_1$ whose objects are semisimple defects between semisimple conformal nets and whose morphisms are finite natural transformations (another option is to 
allow all defects between semisimple conformal nets but restrict the morphisms to be only the isomorphisms),
there is also a `vertical identity' functor  
\begin{equation}\label{eq: functor identity_v}
\mathsf{identity_v} \colon \CN_1^f\to \CN_2
\end{equation}
that sends an $\cala$-$\calb$-defect $D$ to the object $(\cala,\calb,D,D,H_0(D))$ of $\CN_2$.
Here, the \hyphenation{va-cu-um}vacuum sector $H_0(D) := L^2( D(S^1_{\top}))=L^2(D(\tikzmath[scale=0.04]{\draw (0,6) -- (0,12) -- (6,12);\draw[ultra thick] (6,12) -- (12,12) -- (12,6); }))$
is as described in Definition \ref{def: L2(D)}.
We represent it pictorially as follows:
\begin{equation}\label{pic: L^2(D)}
\mathsf{identity_v} \big( {\textstyle _\cala D_\calb} \big) 
\,\,\,\, = \,\,\,\,
\tikzmath[scale=0.12]
{\fill[vacuumcolor](0,0) rectangle(12,12);\draw (6,0) -- (0,0) -- (0,12) -- (6,12);\draw[ultra thick](6,12) -- (12,12) -- (12,0) -- (6,0); 
\draw (-2.5,6) node {$\cala$}(14.5,6) node {$\calb$}(6,14.5) node {$D$}(6,-2.5) node {$D$}(6,6)node {$H_0(D)$};
}  %tikzmath
\,\,\,\, = \,\,\,\,
\tikzmath[scale=0.18] {\fill[vacuumcolor](0,0) rectangle(12,12);\draw (6,0) -- (0,0) -- (0,12) -- (6,12);
\draw[ultra thick](6,12) -- (12,12) -- (12,0) -- (6,0); \draw (-2,6) node {$\cala$}(14,6) node {$\calb$}(6,14) node {$D$}(6,-2) node {$D$}(6,6)
node {$L^2 \left( D (\tikzmath[scale=0.04]{ \draw (0,6) -- (0,12) -- (6,12);\draw[ultra thick](6,12) -- (12,12) -- (12,6); } ) \right)$ };
}  %tikzmath
\end{equation}
We reserve this darker shading of the above squares for vacuum sectors.
Note that it is essential to restrict to the 
subcategory $\CN_1^f\subset \CN_1$ because the $L^2$-space construction 
is only functorial 
with respect to finite homomorphisms of von Neumann 
algebras~\cite{BDH(Dualizability+Index-of-subfactors)} (see also \cite[Conj. 6.17]{BDH(Dualizability+Index-of-subfactors)}).

\begin{remark}
We will see later, in Warning \ref{warning:not-natural!}, that we will have to further restrict our morphisms,
and only allow natural \emph{isomorphisms} between defects
%\AB{Maybe the paper would be less confusing if we do this right from the start, but then it is less complete.}
(even if the defects are semisimple).
This will render otiose the subtleties related to \cite[Conj. 6.17]{BDH(Dualizability+Index-of-subfactors)};
in particular, there is no need to restrict to semisimple defects.
\end{remark}
%An alternative way of addressing that issue is to only allow 
%natural \emph{isomorphisms} between defects.
%In that way, we could include arbitrary defects between 
%semisimple conformal nets in our formalism
%and would not need to restrict attention to semisimple ones.

%==================================================================

\section{Horizontal fusion}\label{ssec:Horizontal fusion}

Consider the translate $S^1_{+}
:=\partial([1,2]\times[0,1])\subset \IR^2$ of the standard circle $S^1=\partial \,[0,1]^2$, and let 
$\INT_{\!S^1_{+}, \circ}$, 
$\INT_{\!S^1_{+}, \bullet}$, 
$\INT_{\!S^1_{+}, \top}$, 
$\INT_{\!S^1_{+}, \bot}$ 
be the obvious analogs of \eqref{eq:4-INT-in-S1}.
Given conformal nets $\cala$, $\calb$, $\calc$, defects $_\cala D_\calb$, $_\cala E_\calb$,
$_\calb F_\calc$, $_\calb G_\calc$, and sectors $_D H_E$, $_F K_G$, let us replace the structure maps
\eqref{eq:4rho} of $K$ by actions
\begin{equation*}
\begin{split}
\rho_I:\calb(I)\to \bfB(K)\,\,\,\text{for}\,\,\, &I\in \INT_{\!S^1_{+}, \circ}\\
\rho_I:F(I)\to \bfB(K)\,\,\,\text{for}\,\,\, &I\in \INT_{\!S^1_{+}, \top}
\end{split}
\qquad\quad
\begin{split}
\rho_I:\calc(I)\to \bfB(K)\,\,\,\text{for}\,\,\, &I\in \INT_{\!S^1_{+}, \bullet}\\
\rho_I:G(I)\to \bfB(K)\,\,\,\text{for}\,\,\, &I\in \INT_{\!S^1_{+}, \bot}
\end{split}
\end{equation*}
by precomposing with the translation.
Let $J$ be $\{1\}\times [0,1]=S^1\cap S^1_{+}$, with the orientation inherited from $S^1_{+}$.
The algebra $\calb(J)$ has actions of opposite variance on $H$ and on $K$, so it makes sense to take the Connes fusion
\[
H\boxtimes_{\calb}K:=H\boxtimes_{\calb(J)}K.
\]
We now show that $H\boxtimes_{\calb}K$ is a $(D\circledast_{\calb}F)$-$(E\circledast_{\calb}G)$-sector.
Given $I\in \mathsf{INT}_{\!S^1\!,\circ\bullet}$, let $I_+$ be the subinterval
of $\partial([0,2]\times [0,1])$ given by
\begin{alignat*}{3}
&I_+:=I
&\text{if }\,I&\in \mathsf{INT}_{\!S^1\!,\circ},\\
&I_+:=I+(1,0)
&\text{if }\,I&\in \mathsf{INT}_{\!S^1\!,\bullet},\\
&I_+:=I_\circ\cup \big([{\textstyle\frac{1}{2},\frac{3}{2}}]\times\{1\}\big)\cup \big(I_\bullet+(1,0)\big)
&\text{if }\,I&\in \mathsf{INT}_{\!S^1\!,\top},\\
&I_+:=I_\bullet\cup \big([{\textstyle\frac{1}{2},\frac{3}{2}}]\times\{0\}\big)\cup \big(I_\bullet+(1,0)\big)
\quad 
&\text{if }\,I&\in \mathsf{INT}_{\!S^1\!,\bot}.
\end{alignat*}
For $I \in \mathsf{INT}_{\!S^1\!,\circ}$ and $I \in \mathsf{INT}_{\!S^1\!,\bullet}$, the structure maps \eqref{eq:4rho} are given by
the obvious actions of $\cala(I_+)$ and $\calc(I_+)$ on the Hilbert space $H\boxtimes_{\calb}K$.
For an interval $I \in \mathsf{INT}_{\!S^1\!,\top}$ or $I\in\mathsf{INT}_{\!S^1\!,\bot}$, 
the algebras that act on $H\boxtimes_{\calb(J)}K$ are given by
\begin{equation}\label{TBbF}
\begin{split}
D\big((I_+\cap S^1)\cup J\big)
&\circledast_{\calb(J)}
F\big(J\cup(I_+\cap S^1_+)\big)\\
\text{and}\qquad \quad E\big((I_+\cap S^1)\cup J\big)
&\circledast_{\calb(J)}
G\big(J\cup(I_+\cap S^1_+)\big)\qquad \quad 
\end{split}
\end{equation}
respectively---see Appendix~\ref{subsec:fusion+fiber-prod}.
Upon identifying the intervals 
$(I_+\cap S^1)\cup J$ and $J \cup (I_+\cap S^1_{+})$ of (\ref{TBbF})
with the intervals
$I^{++}$ and ${}^{++}\hspace{-.15mm}I$ of \eqref{eq:def-fusion-Defect}, we see that the algebras \eqref{TBbF} are equal to
$\big(D\circledast_\calb F\big)(I)$ and $\big(E\circledast_\calb G\big)(I)$, respectively.
We can now define the functor of horizontal fusion
\begin{equation}\label{eq: The functor fusion_h}
\mathsf{fusion_h}:\CN_2\times_{\CN^f_0}\CN_2\rightarrow \CN_2
\end{equation}
by
$\mathsf{fusion_h}\big((\cala,\calb,D,E,H),(\calb,\calc,F,G,K)\big):=
\big(\cala,\,\calc,\,D\circledast_\calb F,\,E\circledast_\calb G,\,H\boxtimes_\calb K\big)$.
Here, as in \eqref{eq: composition -- this time between the correct categories}, $\CN_2\times_{\CN^f_0}\CN_2$ is a shorthand notation for $\CN_2\times_{\CN_0} \CN^f_0\times_{\CN_0}\CN_2$, 
and the relevant maps $\CN_2 \rightarrow \CN_0$ are $\mathsf{target} \circ \mathsf{source_v}$ and $\mathsf{source} \circ \mathsf{source_v}$, respectively.
%where the subcategory $\CN^f_2$ of $\CN_2$ is the preimage of $\CN^f_1\times \CN^f_1\subset\CN_1\times \CN_1$ under the functor $\mathsf{source_v}\times \mathsf{target_v}$.

Pictorially, we understand the functor $\mathsf{fusion_h}$ as the operation of gluing two squares along a common edge as follows:

\[
\mathsf{fusion_h}\,\,
\Big(  \tikzmath[scale=\squarescale]
  {   \fill[spacecolor]  (0,0) rectangle  (12,12);
      \draw (6,0) -- (0,0) -- (0,12) -- (6,12);
      \draw[thick, double]  (6,12) -- (12,12) -- (12,0) -- (6,0); 
      \draw (-3,6) node {$\cala$}
            (15,6) node {$\calb$}
            (6,15) node {$D$}
            (6,-3) node {$E$}
            (6,6)  node {$H$};
  }%tikzmath
  ,
  \tikzmath[scale=\squarescale]
  {   \fill[spacecolor]  (0,0) rectangle  (12,12);
      \draw[thick, double] (6,0) -- (0,0) -- (0,12) -- (6,12);
      \draw[ultra thick]  (6,12) -- (12,12) -- (12,0) -- (6,0); 
      \draw (-3,6) node {$\calb$}
            (15,6) node {$\calc$}
            (6,15) node {$F$}
            (6,-3) node {$G$}
            (6,6)  node {$K$};
  }%tikzmath      
  \Big) = 
  \tikzmath[scale=\squarescale]
  {   \fill[spacecolor]  (0,0) rectangle  (24,12);
      \draw (6,0) -- (0,0) -- (0,12) -- (6,12);
      \draw[thick, double]  (6,0) -- (18,0)  (6,12) -- (18,12); 
      \draw[ultra thick]  (18,12) -- (24,12) -- (24,0) -- (18,0);
      \draw (-3,6) node {$\cala$}
            (27,6) node {$\calc$}
            (12,15) node {$\scriptstyle D\circledast_\calb F$}
            (12,-3) node {$\scriptstyle E\circledast_\calb G$}
            (12,5)  node {$H\!\underset{\scriptscriptstyle\calb(|\hspace{-.55mm}|)}{\boxtimes}\!K$};
  }%tikzmath
\]
\smallskip

\noindent
The associator for $\mathsf{fusion_h}$ is induced by the usual associator for Connes fusion. It consists of a natural transformation
\begin{equation}\label{eq: nt associator_h}
\mathsf{associator_h}:\CN_2\times_{\CN^f_0}\CN_2\times_{\CN^f_0}\CN_2\,\tworarrow\, \CN_2,
\end{equation}
and satisfies the pentagon identity.

% . . . . . . . . . . . . . . . . . . . . . . . . . . . . . . . . . . . . . . . . . . . . . . . . . . . . . . . . . . . . . . . . . . . . . . . . . . 
\section{Vertical fusion}\label{sec: Vertical fusion}

Unlike horizontal fusion, vertical fusion is not the operation of gluing two squares along a common edge.
Rather, it consists of gluing those two squares along half of their boundary:
\[
\mathsf{fusion_v}\,\,
\Big(\tikzmath[scale=\squarescale]{\fill[spacecolor] (0,0) -- (0,12) -- (12,12) -- (12,0);\draw (6,0) -- (0,0) -- (0,12) -- (6,12);\draw[ultra thick](6,12) -- (12,12) -- (12,0) -- (6,0); 
\draw (-3,6) node {$\cala$}(15,6) node {$\calb$}(6,15) node {$D$}(6,-3) node {$E$}(6,6)node {$H$};}%tikzmath
,\tikzmath[scale=\squarescale]{\fill[spacecolor] (0,0) -- (0,12) -- (12,12) -- (12,0);\draw (6,0) -- (0,0) -- (0,12) -- (6,12);\draw[ultra thick](6,12) -- (12,12) -- (12,0) -- (6,0); 
\draw (-3,6) node {$\cala$}(15,6) node {$\calb$}(6,15) node {$E$}(6,-3) node {$F$}(6,6)node {$K$};}%tikzmath
\Big) =\tikzmath[scale=\squarescale]
{\fill[spacecolor] (0,3) -- (0,15) -- (12,15) -- (12,3);\draw (6,3) -- (0,3) -- (0,15) -- (6,15);\draw[ultra thick](6,15) -- (12,15) -- (12,3) -- (6,3); 
\fill[spacecolor] (0,-15) -- (0,-3) -- (12,-3) -- (12,-15);\draw (6,-15) -- (0,-15) -- (0,-3) -- (6,-3);\draw[ultra thick](6,-3) -- (12,-3) -- (12,-15) -- (6,-15); 
\draw %(-9,0) node {$\cala$}(21,0) node {$\calb$}
(6,18) node {$D$}(6,9)node {$H$}(6,-18) node {$F$}(6,-9)node {$K$};
\draw[<->] (-.2,9) to[out=180,in=180] (-.2,-9); 
\draw[<->] (-.2,6) to[out=200,in=160] (-.2,-6); 
\draw[<->] (-.2,2.9) to[out=230,in=130] (-.2,-2.9); 
\draw[<->] (2.7,2.7) to[out=255,in=105] (2.7,-2.7); 
\draw[<->] (12.5,9) to[out=0,in=0] (12.5,-9); 
\draw[<->] (12.5,6) to[out=-20,in=20] (12.5,-6); 
\draw[<->] (12.5,2.6) to[out=-50,in=50] (12.5,-2.6); 
\draw[<->] (9.3,2.6) to[out=-75,in=75] (9.3,-2.6); 
\draw[<->] (6,2.6) -- (6,-2.6); 
  }\,.
\]

A sector is called \emph{dualizable} if it has a dual with respect to the operation of vertical fusion; equivalently:

\begin{definition}\label{def: finite sectors}
A $D$-$E$-sector $H$ between semisimple defects 
is called dualizable if it is dualizable (Appendix~\ref{subsec:dualizability})
as an $S^1_{\top}(D)$--$S^1_{\bot}(E)$-bimodule.
\end{definition}

We now describe in detail the functor $\mathsf{fusion_v}$ of vertical fusion.
Given conformal nets $\cala$, $\calb$, defects $_\cala D_\calb$, $_\cala E_\calb$,
$_\cala F_\calb$, and sectors $_D H_E$, $_E K_F$, 
we want to construct a $D$-$F$-sector $H\boxtimes_E K$.
Let $S^1_{\top} = \partial^\sqcap( [0,1] \!\x \! [\frac{1}{2},1])$ and
$S^1_{\bot} = \partial^\sqcup( [0,1] \!\x \! [0,\frac{1}{2}] )$ be the top and bottom halves of our standard circle $\partial [0,1]^2$,
and let $j:S^1_{\top}\xrightarrow{\scriptscriptstyle\sim} S^1_{\bot}$ be the reflection map along the horizontal symmetry axis.
The algebra $E(S^1_{\top})$ has two actions 
\[
\begin{split} 
& E(S^1_{\top})^\op  \xrightarrow{E(j)} E(S^1_{\bot}) \rightarrow \bfB(H)\\
& E(S^1_{\top})  \xrightarrow{\hspace{2.45cm}} \bfB(K)
\end{split}
\]
of opposite variance on $H$ and $K$,
and so it makes sense to take the Connes fusion
\[
H\boxtimes_{E}K:=H\boxtimes_{E(S^1_{\top})}K.
\]
To see that $H\boxtimes_{E}K$ is a $D$-$F$-sector, we have to show that the algebras
$\cala(I)$, $\calb(I)$, $D(I)$, and $F(I)$ act on it for $I \in \mathsf{INT}_{\!S^1\!,\circ}$, 
$\mathsf{INT}_{\!S^1\!,\bullet}$, 
$\mathsf{INT}_{\!S^1\!,\top}$, 
$\mathsf{INT}_{\!S^1\!,\bot}$, respectively.

We first treat the case $I \in \mathsf{INT}_{\!S^1\!,\circ}$.
If $I$ is contained in $S^1_\top$ (or $S^1_\bot$), then the action of $\cala(I)$ on $H\boxtimes_{E}K$ is induced by its action on $H$ (or $K$).
If $I$ contains the point $(0,\frac12)$ in its interior, then the algebra 
\begin{equation}\label{XI}
E(I\cup S^1_{\bot})\underset{E(S^1_{\top})}\circledast
E(I\cup S^1_{\top})
\end{equation} 
acts
on $H\boxtimes_{E(S^1_{\top})}K$, where the homomorphism $E(S^1_{\top})\to E(I\cup S^1_{\bot})^\op$ implicit in \eqref{XI} is given by $E(j)$.
We observe, as follows, that there is a canonical homomorphism (typically not an isomorphism) from $\cala(I)$ to the algebra \eqref{XI}.
In the definition of that fusion product, we are free to chose any faithful $E(I\cup S^1_{\bot})$-module and any faithful $E(I\cup S^1_{\top})$-module (see Appendix~\ref{subsec:fusion+fiber-prod}):
let us take both of them to be the vacuum $H_0(E)$.
Then, by definition, the algebra \eqref{XI} is generated on
\[
H_0(E)\boxtimes_{E(S^1_{\top})} H_0(E) \cong H_0(E)
\]
by $E(I\cup S^1_{\bot})\cap E(S^1_{\bot})'$ and $E(I\cup S^1_{\top})\cap E(S^1_{\top})'$.
By the strong additivity, vacuum, and locality axioms, we have natural homomorphisms (the first one is an isomorphism when $E$ is a faithful defect):
\[
\begin{split}
\cala(I)\,&\to E(I\cap S^1_{\top})\vee E(I\cap S^1_{\bot})\\
& \to \big(E(I\cup S^1_{\bot})\cap E(S^1_{\bot})'\big)\vee\big(E(I\cup S^1_{\top})\cap E(S^1_{\top})'\big)\\
&=E(I\cup S^1_{\bot})\circledast_{E(S^1_{\top})} E(I\cup S^1_{\top}).
\end{split}
\]
Composing this composite with the action of \eqref{XI} on $H\boxtimes_E K$ gives our desired action of $\cala(I)$.

By the same argument, we also have actions of $\calb(I)$ on $H\boxtimes_E K$ for $I\in\mathsf{INT}_{\!S^1\!,\bullet}$.
Furthermore, there are actions of $D(S^1_\top)$ and $F(S^1_\bot)$ on $H\boxtimes_E K$ coming from their respective actions on $H$ and on $K$.
We can therefore apply Lemma \ref{lem: open cover of circle => sector -- BIS} to all the actions constructed so far,
and conclude that $H\boxtimes_E K$ is a $D$-$F$-sector.

One might expect vertical fusion to be a functor $\CN_2\times_{\CN_1}\CN_2\rightarrow \CN_2$.
However, just like the vertical identity \eqref{eq: functor identity_v} which is only a functor on the smaller category $\CN_1^f$,
and the horizontal fusion which is only a functor on the restricted product $\CN_2 \times_{\CN_0^f} \CN_2$, so too
vertical fusion only gives a functor on the restricted product:
\begin{equation}\label{eq: functor of vertical fusion}
\mathsf{fusion_v}:\CN_2\times_{\CN_1^f}\CN_2\rightarrow \CN_2
\end{equation}
\[
\mathsf{fusion_v}\big((\cala,\calb,D,E,H),(\cala,\calb,E,F,K)\big):=
\big(\cala,\,\calb,\,D, F,\,H\boxtimes_E K\big).
\]
The restriction is necessary to ensure the Connes fusion $H \boxtimes_E K$ is functorial with respect to the relevant natural transformations of the defect $E$ \cite{BDH(Dualizability+Index-of-subfactors)}.

The associator for vertical fusion
\begin{equation}\label{eq: nt associator_v}
\mathsf{associator_v}:\CN_2\times_{\CN_1^f}\CN_2\times_{\CN_1^f}\CN_2\tworarrow \CN_2
\end{equation}
comes from the associator of Connes fusion and satisfies the pentagon identity.
There are also `top' and `bottom' identity natural transformations, 
\begin{equation}\label{eq: `top' and `bottom' identity nt}
\mathsf{unitor}_\mathsf{t}:\CN_2\tworarrow \CN_2,\qquad\mathsf{unitor}_\mathsf{b}:\CN_2\tworarrow \CN_2
\end{equation}
that 
describe the
way $\mathsf{fusion_v}$ and $\mathsf{identity_v}$ interact.
Given a sector $_D H_E$, they provide natural isomorphisms
\begin{equation}\label{eabd}
{}_D H_0(D)\boxtimes_D H_E \cong {}_D H_E\quad\qquad\text{and}\quad\qquad {}_D H\boxtimes_E H_0(E)_E \cong {}_D H_E
\end{equation}
subject to the usual triangle axioms.
Strictly speaking, the source functor of $\mathsf{unitor}_\mathsf{t}$ is only defined on the subcategory $\CN_1^f\times_{\CN_1}\CN_2$ of $\CN_2$,
and so the transformation $\mathsf{unitor}_\mathsf{t}$ itself is only defined on that subcategory.
Similarly, $\mathsf{unitor}_\mathsf{b}$ is only defined on the subcategory $\CN_2\times_{\CN_1}\CN_1^f$.

%==================================================================

\chapter{Properties of the composition of defects}
  \label{sec:elem-prop-of-defect-composition}

\section{Left and right units} Units are a subtle business.
One might guess that the left unit is a natural isomorphism $\CN_1 \tworarrow\, \CN_1$ whose source is
the functor $\mathsf{composition} \circ ((\mathsf{identity}\circ\mathsf{source})\times \id_{\CN_1})$ and whose target is the identity functor.  (Here $\id_{\CN_1} : \CN_1 \rightarrow \CN_1$ is the identity functor and $\mathsf{identity} : \CN_0 \rightarrow \CN_1$ takes a net to the identity defect, as in \eqref{eq: identity functor}.)  But, unfortunately, in general there is no such natural isomorphism.  Instead, we have the following `weaker' piece of data: a functor
\[
\mathsf{unitor}_\mathsf{tl}:\CN_1^f \to \CN_2
\]
(`tl' stands for top left) 
with the property that
\[\begin{split}
\mathsf{source}_\mathsf{v}\circ \mathsf{unitor}_\mathsf{tl}\,&=\mathsf{composition} \circ ((\mathsf{identity}\circ\mathsf{source})\times \id_{\CN_1^f})\\
\qquad\text{and}&\qquad
\mathsf{target}_\mathsf{v}\circ \mathsf{unitor}_\mathsf{tl}=\id_{\CN_1^f}.
\end{split}\]
%and $\mathsf{target}_\mathsf{v}\circ \mathsf{unitor}_\mathsf{tl}=1$.
This functor takes values in sectors that are invertible with respect to vertical
fusion.  Its construction is based on the following lemma:

\begin{lemma}\label{lem: unitor}
Let ${}_\cala D_\calb$ be a defect. Then $1\circledast D:=\mathsf{identity}(\cala)\circledast_\cala D$ is given on genuinely bicolored intervals $I$ by
\[
\big(1\circledast D\big)(I)=D({}^<\!I)
\]
where ${}^<\!I:=I_\circ\cup[0,1]\cup I_\bullet$ with bicoloring ${}^<\!I_\circ:=I_\circ\cup[0,1]$ and ${}^<\!I_\bullet:=I_\bullet$.

Similarly, on genuinely bicolored intervals we have
\[
(D\circledast 1)(I)=D(I^>)
\]
where $I^>:=I_\circ\cup[0,1]\cup I_\bullet$ with bicoloring $I^>_\circ:=I_\circ$ and $I^>_\bullet:=[0,1]\cup I_\bullet$.
\end{lemma}

\begin{proof}
We prove the first statement; the second one is entirely similar.
Consider the intervals $K:=\{\frac12\}\times[0,1]$, 
$J:=I_\circ\cup[0,\frac12]\times\{1\}$,
$J^+:=J\cup K$, ${}^{+}\hspace{-.15mm}I:=([\frac12,1]\times\{1\})\cup I_\bullet$, and
${}^{++}\hspace{-.15mm}I:={}^{+}\hspace{-.15mm}I\cup K$.
These intervals are bicolored by ${}^{+}\hspace{-.15mm}I_\bullet={}^{++}\hspace{-.15mm}I_\bullet=I_\bullet$ and $K_\bullet=J_\bullet=J^+_\bullet=\emptyset$.
\begin{equation*}\tikzmath[scale=\displscale]{\draw (-20,-7) [rounded corners=7pt]-- (-20,-1) -- (-15,10) [rounded corners=4pt]-- (-6,0) -- (0,0);
\draw[ultra thick] (0,0) [rounded corners=3pt]-- (5,0) -- (10,-3) -- (17,5) -- (22,2)(0,8) node {$I$}(-13,3) node {$I_\circ$}(8.8,3.3) node {$I_\bullet$};}%tikzmath
\qquad\quad\leadsto\qquad\quad
\tikzmath[scale=\displscale]{\draw (-14,5) [rounded corners=7pt]-- (-14,11) -- (-9,22) [rounded corners=4pt]-- (0,12) -- (6,12);\draw (6,12) -- (12,12) -- (12,0) (12,12) -- (18,12)
(1,4) node {$J^+$} (11,20) node {${}^<\!I$}(25,4) node {${}^{++}\hspace{-.15mm}I$}(-8.5,15) node {$J$}(25.5,15.5) node {${}^{+}\hspace{-.15mm}I$}(12,-3) node {$K$};
\draw[ultra thick] (18,12) [rounded corners=3pt]-- (23,12) -- (28,9) -- (35,17) -- (40,14);}%tikzmath
\end{equation*}
Extend the map 
$[0,\frac12]\times\{1\}\to K: (t,1)\mapsto(\frac12,t+\frac12)$ 
to an embedding $f:J\to K$ so that $K\setminus f(J)$ is non-empty. 
Using $\cala(f)$, we can then equip $L^2(\cala(K))$ with a left action of $\cala(J)$.
Combining this left action with the natural right action of $\cala(K)$, we get a faithful action of $\cala(J)\otimes_\alg \cala(K)^\op$ on $L^2(\cala(K))$, 
which extends to $\cala(J^+)$ by the vacuum sector axiom for conformal nets
(see Appendix~\ref{subsec:defnets}).
Pick a faithful $D({}^{++}\hspace{-.15mm}I)$-module $H$.
By definition,
\[
\big(1 \circledast D\big)(I)\,=\,\cala(J^+)\circledast_{\cala(K)}D({}^{++}\hspace{-.15mm}I)
\]
is the von Neumann algebra generated by
$\cala(J)$ and $D({}^{+}\hspace{-.15mm}I)$
on the Hilbert space $L^2(\cala(K))\boxtimes_{\cala(K)} H\cong H$.
This algebra is equal to $D(J \cup {}^{+}\hspace{-.15mm}I) = D({}^<\!I)$ by strong additivity.
\end{proof}

Recall that our standard circle $S^1$ is the square $\partial[0,1]^2$.
Let
\[
\begin{split}
\textstyle S^1_\ulc:=\partial^{\hspace{.2mm}\tikzmath[scale=.15]{\draw(0,0)--(0,1)--(1,1);}}\big([0,\frac12]\!\times\![\frac12,1]\big),\qquad&
\textstyle S^1_\urc:=\partial^{\hspace{.2mm}\tikzmath[scale=.15]{\draw(0,1)--(1,1)--(1,0);}}\big([\frac12,1]\!\times\![\frac12,1]\big),\\
\textstyle S^1_\llc:=\partial^{\hspace{.2mm}\tikzmath[scale=.15]{\draw(1,0)--(0,0)--(0,1);}}\big([0,\frac12]\!\times\![0,\frac12]\big),\qquad&
\textstyle S^1_\lrc:=\partial^{\hspace{.2mm}\tikzmath[scale=.15]{\draw(0,0)--(1,0)--(1,1);}}\big([\frac12,1]\!\times\![0,\frac12]\big)
\end{split}
\]
be the four ``quarter circles''. Let us also pick, once and for all, a diffeomorphism $\phi_\ulc:S^1_\ulcorner\cup[0,1]\to S^1_\ulcorner$ (here $(\frac12,1) \in S^1_\ulcorner$ is glued to $0 \in [0,1]$) whose derivative is equal to one in a neighborhood of the boundary.
The three mirror images of $\phi_\ulc$ are called
$\phi_\urc:[0,1]\cup S^1_\urcorner\to S^1_\urcorner$,
$\phi_\llcorner:S^1_\llc\cup [0,1]\to S^1_\llc$, and
$\phi_\lrcorner:[0,1]\cup S^1_\lrc\to S^1_\lrc$:
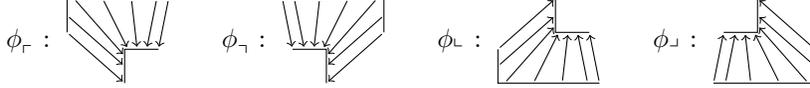
\[
\phi_\ulc:\,\,\tikzmath[scale = 0.45]{\draw (0,2) -- (0,3) -- (3,3)(1.7,0.5)--(1.7,1.5)--(2.7,1.5);\pgfsetshortenstart{1}\pgfsetshortenend{1}
\draw[->] (0,2) -- (1.7,0.5);\draw[->] (0,2.5) -- (1.7,1);\draw[->] (.2,3) -- (1.7,1.3);\draw[->] (1.05,3) -- (1.8,1.5);\draw[->] (1.9,3) -- (2.07,1.5);\draw[->] (2.65,3) -- (2.35,1.5);
\draw[->] (3,3) -- (2.7,1.5);} %tikzmath
\qquad\phi_\urc:\,\,\tikzmath[scale = 0.45]{\pgftransformxscale{-1}\draw (0,2) -- (0,3) -- (3,3)(1.7,0.5)--(1.7,1.5)--(2.7,1.5);\pgfsetshortenstart{1}\pgfsetshortenend{1}
\draw[->] (0,2) -- (1.7,0.5);\draw[->] (0,2.5) -- (1.7,1);\draw[->] (.2,3) -- (1.7,1.3);\draw[->] (1.05,3) -- (1.8,1.5);\draw[->] (1.9,3) -- (2.07,1.5);\draw[->] (2.65,3) -- (2.35,1.5);
\draw[->] (3,3) -- (2.7,1.5);} %tikzmath
\qquad\phi_\llc:\,\,\tikzmath[scale = 0.45]{\pgftransformyscale{-1}\draw (0,2) -- (0,3) -- (3,3)(1.7,0.5)--(1.7,1.5)--(2.7,1.5);\pgfsetshortenstart{1}\pgfsetshortenend{1}
\draw[->] (0,2) -- (1.7,0.5);\draw[->] (0,2.5) -- (1.7,1);\draw[->] (.2,3) -- (1.7,1.3);\draw[->] (1.05,3) -- (1.8,1.5);\draw[->] (1.9,3) -- (2.07,1.5);\draw[->] (2.65,3) -- (2.35,1.5);
\draw[->] (3,3) -- (2.7,1.5);} %tikzmath
\qquad\phi_\lrc:\,\,\tikzmath[scale = 0.45]{\pgftransformscale{-1}\draw (0,2) -- (0,3) -- (3,3)(1.7,0.5)--(1.7,1.5)--(2.7,1.5);\pgfsetshortenstart{1}\pgfsetshortenend{1}
\draw[->] (0,2) -- (1.7,0.5);\draw[->] (0,2.5) -- (1.7,1);\draw[->] (.2,3) -- (1.7,1.3);\draw[->] (1.05,3) -- (1.8,1.5);\draw[->] (1.9,3) -- (2.07,1.5);\draw[->] (2.65,3) -- (2.35,1.5);
\draw[->] (3,3) -- (2.7,1.5);} %tikzmath
\]
{
\def \Stopw {\tikz{\useasboundingbox (-.11,-.12) rectangle (.12,.19); \draw node {$S$} (-.08,.18)--(.12,.18);}}
\def \Stopb {\tikz{\useasboundingbox (-.11,-.12) rectangle (.12,.19); \draw[very thick] node {$S$} (-.08,.18)--(.12,.18);}}
\def \Sbotw {\tikz{\useasboundingbox (-.11,-.12) rectangle (.12,.19); \draw node {$S$} (-.11,-.18)--(.09,-.18);}}
\def \Sbotb {\tikz{\useasboundingbox (-.11,-.12) rectangle (.12,.19); \draw[very thick] node {$S$} (-.11,-.18)--(.09,-.18);}}

\noindent We are now ready to define the functor
\begin{equation}\label{eq: unitor_tl}
\mathsf{unitor}_\mathsf{tl}\,:\,\CN_1\to\CN_2.
\end{equation}
It assigns to every $\cala$-$\calb$-defect $D$ an invertible $(1\circledast D)$-$D$-sector.
As a Hilbert space, $\mathsf{unitor}_\mathsf{tl}(D)$ is simply the vacuum sector $H_0(D)$.
Let $\Stopw$ be the bicolored circle with white half $\Stopw_\circ:=S^1_\circ \cup_{(\frac12,1)}[0,1]$ and black half $\Stopw_\bullet:=S^1_\bullet$.
One should imagine $\Stopw$ as being the standard bicolored circle $S^1$, to which an extra white interval $[0,1]$ has been inserted at the top---see \eqref{eq:4cirleswithBAR}.
In view of Lemma \ref{lem: unitor}, 
a $(1\circledast D)$-$D$-sector is the same thing as a Hilbert space $H$ equipped with compatible actions of $D(I)$ for every bicolored interval $I\subset \Stopw$.
Let $\hat\phi_\ulc\!:\Stopw\to S^1$ be the diffeomorphism given by $\phi_\ulc$ on $S^1_\ulcorner\cup[0,1]$, and by the identity on the complement.
The $(1\circledast D)$-$D$-sector structure on $H_0(D)=\mathsf{unitor}_\mathsf{tl}(D)$ is given by letting $D(I)$ 
act by the composition of $D(\hat\phi_\ulc):D(I)\to D(\hat\phi_\ulc(I))$ with the natural action of $D(\hat\phi_\ulc(I))$ on $H_0(D)$.

We also have functors 
\begin{equation}\label{eq: unitor_tl+}
\begin{split}
\mathsf{unitor}_\mathsf{tr}\,:\,\CN_1^f \to \CN_2\\
\mathsf{unitor}_\mathsf{bl}\,:\,\CN_1^f \to \CN_2\\
\mathsf{unitor}_\mathsf{br}\,:\,\CN_1^f \to \CN_2
\end{split}
\end{equation}
that are defined in a similar fashion.
The underlying Hilbert spaces of $\mathsf{unitor}_\mathsf{tr}(D)$,
$\mathsf{unitor}_\mathsf{bl}(D)$, and
$\mathsf{unitor}_\mathsf{br}(D)$ are all $H_0(D)$, and they are equipped with the structures of 
$(D\circledast 1)$-$D$-sector,
$D$-$(1\circledast D)$-sector, and 
$D$-$(D\circledast 1)$-sector, respectively.
Let $\Stopb$, $\Sbotw$, and $\Sbotb$ be the bicolored circles given by
$\Stopb_\circ:=S^1_\circ$, $\Stopb_\bullet:= [0,1]\cup_{(\frac12,1)} S^1_\bullet$,
$\Sbotw_\circ:=S^1_\circ \cup_{(\frac12,0)}[0,1]$, $\Sbotw_\bullet:=S^1_\bullet$, and
$\Sbotb_\circ:=S^1_\circ$, $\Sbotb_\bullet:=[0,1]\cup_{(\frac12,0)} S^1_\bullet$:
\begin{equation}\label{eq:4cirleswithBAR}
\,\,\Stopw\,=\,\tikzmath[scale = 0.35]{
\draw (1,0.1)to[out = 180,in = 10] (0,0)to[out = 100,in = -45](-1,2)to[out = 36,in = 198](0,2.5)(.1,2.525) to[out = 15,in = 165](1.9,2.525);
\draw[ultra thick] (1,0.1)to[out = 0,in = 170] (2,0) to[out = 80,in = 225](3,2)to[out = 144,in = -18](2,2.5);} %tikzmath
\qquad\Stopb\,=\,\tikzmath[scale = 0.35]{
\draw (1,0.1)to[out = 180,in = 10] (0,0)to[out = 100,in = -45](-1,2)to[out = 36,in = 198](0,2.5);
\draw[ultra thick] (1,0.1)to[out = 0,in = 170] (2,0) to[out = 80,in = 225](3,2)to[out = 144,in = -18](2,2.5)(.1,2.525) to[out = 15,in = 165](1.9,2.525);} %tikzmath
\qquad\Sbotw\,=\,\tikzmath[scale = 0.35]{\pgftransformyscale{-1}
\draw (1,0.1)to[out = 180,in = 10] (0,0)to[out = 100,in = -45](-1,2)to[out = 36,in = 198](0,2.5)(.1,2.525) to[out = 15,in = 165](1.9,2.525);
\draw[ultra thick] (1,0.1)to[out = 0,in = 170] (2,0) to[out = 80,in = 225](3,2)to[out = 144,in = -18](2,2.5);} %tikzmath
\qquad\Sbotb\,=\,\tikzmath[scale = 0.35]{\pgftransformyscale{-1}
\draw (1,0.1)to[out = 180,in = 10] (0,0)to[out = 100,in = -45](-1,2)to[out = 36,in = 198](0,2.5);
\draw[ultra thick] (1,0.1)to[out = 0,in = 170] (2,0) to[out = 80,in = 225](3,2)to[out = 144,in = -18](2,2.5)(.1,2.525) to[out = 15,in = 165](1.9,2.525);}\,. %tikzmath
\end{equation}
By Lemma \ref{lem: unitor}, a $(D\circledast 1)$-$D$-sector structure on a Hilbert space is the same thing as 
a collection of compatible actions of the algebras $D(I)$ for every bicolored interval $I \subset \Stopb$.
Similarly, being a $D$-$(1\circledast D)$-sector means that there are compatible actions of $D(I)$ for every $I\subset\Sbotw$,
and being $D$-$(D\circledast 1)$-sector means that there are compatible actions of $D(I)$ for every $I\subset \Sbotb$.
We equip $H_0(D)$ with the above structures by the appropriate analogs
$\hat\phi_\urc\!:\Stopb\to S^1$,
$\hat\phi_\llcorner\!:\Sbotw\to S^1$,
$\hat\phi_\lrcorner\!:\Sbotb\to S^1$ of $\hat\phi_\ulc$,
defined as $\phi_\urc$, $\phi_\llcorner$, $\phi_\lrcorner$ on the relevant subintervals, and as the identity on the rest.
}

\begin{example}\label{ex: 1=/=1*1}
Given a non-trivial conformal net $\cala$, the identity defect $1_\cala := \mathsf{identity}(\cala):I\mapsto \cala(I)$ is not isomorphic to $1_\cala\circledast_\cala 1_\cala$.
The defect $1_\cala \circledast_\cala 1_\cala$ is the weak identity for $\cala$ discussed in Remark~\ref{rem : weak identity}.
It maps a genuinely bicolored interval $I$ to $\cala(I_\circ\cup[0,1]\cup I_\bullet)$. 
As a way of distinguishing those two defects, note that the intersection
\[
\underset{\text{bicolored}}{\underset{J\subset I, \text{ genuinely}}{\bigcap}}\big(1_\cala\circledast_\cala 1_\cala\big)(J)=\cala([0,1])
\]
is non-trivial, which is not the case if $1_\cala \circledast_\cala 1_\cala$ is replaced by $1_\cala$ in the above expression.

The invertible sector between $1_\cala$ and $1_\cala\circledast_\cala 1_\cala$ is the vacuum module of $\cala$ associated to the ``circle''
$\tikzmath[scale = 0.2]{
\draw (1,0.1)to[out = 180,in = 10] (0,0)to[out = 100,in = -45](-1,2)to[out = 36,in = 198](0,2.5)(.1,2.525) to[out = 15,in = 165](1.9,2.525);
\draw (1,0.1)to[out = 0,in = 170] (2,0) to[out = 80,in = 225](3,2)to[out = 144,in = -18](2,2.5);}$
constructed by inserting a copy of $[0,1]$ at the point $(\frac12,1) \in \partial[0,1]^2$.
\end{example}

%========================================

\section{Semisimplicity of the composite defect} 

\,\,Given two semisimple defects (finite direct sums of irreducible defects), we can ask whether their fusion is again a semisimple defect.
From now on, we always assume that our conformal nets are irreducible.
The purpose of this section is to prove the following theorem:

\begin{theorem}\label{thm: semi-simplicity of DoE}
  Let ${}_\cala D_\calb$ and  ${}_\calb E_\calc$ be semisimple defects.
  If the conformal net $\calb$ has finite index, then 
  for any genuinely bicolored interval $I$ the algebra
  $(D\circledast_\calb E)(I)$ has finite dimensional center.
\end{theorem}

\noindent In light of Theorem~\ref{thm:fusion-of-defects-is-defect} (whose proof, note well, depends on Theorem~\ref{thm: semi-simplicity of DoE}, via Corollary~\ref{cor: is a finite homomorphism of von Neumann algebra} and Proposition~\ref{prop:G=L2}), we can rephrase this result as follows:

\begin{corollary}\label{cor: composite is semisimple}
The fusion of two semisimple defects ${}_\cala D_\calb$ and  ${}_\calb E_\calc$, over a finite index conformal net $\calb$, is a semisimple $\cala$-$\calc$-defect.
\end{corollary}

\subsection*{\hspace*{-18pt}Detecting semisimplicity} %\hspace*{\fill} \vspace{5pt}

We begin with a few lemmas.

\begin{lemma}\label{lem: Lemma A}
  Let $A$, $B$ be von Neumann algebras and let $H$ be a faithful $A$--$B$-bimodule.
  If the algebra of $A$--$B$-bimodule endomorphisms of $H$
  is finite-dimensional, then $A$ is a finite direct sums of factors.
\end{lemma}

\begin{proof}
The center of $A$ acts faithfully by $A$--$B$-bimodule endomorphisms.
It is therefore finite-dimensional.
\end{proof}

From now on, we fix a faithful defect ${}_\cala D_\calb$, and denote its vacuum sector $H_0$.
Recall that our standard circle is $S^1:=\partial [0,1]^2$, and that its top and bottom halves are denoted $S^1_\top$ and $S^1_\bot$.

\begin{notation}\label{not: tildeD}
Given an interval $I\subset S^1$ that contains the two color-change points $(\frac12,0)$ and $(\frac12,1)$ in its interior,
we define an algebra $\tilde D(I)\subset \bfB(H_0)$ as follows.
It is the algebra generated by $D(I_1)$ and $D(I_2)$, where $I_1$ and $I_2$ are any two intervals covering $I$ with the property that $(\frac12,1)\not\in I_1$ and $(\frac12,0)\not\in I_2$.
By strong additivity, the algebra $\tilde D(I)$ does not depend on the choice of covering.
\end{notation}

\begin{lemma}\label{lem: tildeD}
Let $I \subset S^1$ be an interval containing both color-change points in its interior.
If $D$ is an irreducible defect, then $\tilde D(I)$ is a factor.
\end{lemma}
\begin{proof}
Let $I'$ be the closure of $S^1\!\setminus\! I$.
The center of $\tilde D(I)$ commutes with both $D(I)$ and $D(I')$.
Since $D(S^1_\top\cap I)$ and $D(S^1_\top\cap I')$ generate $D(S^1_\top)$, $Z(\tilde D(I))$ commutes with $D(S^1_\top)$.
Similarly, $Z(\tilde D(I))$ commutes with $D(S^1_\bot)$.
Therefore, $Z(\tilde D(I))$ acts on $H_0$ by $D(S^1_\top)$--$D(S^1_\bot)^\op$-bimodule automorphisms.
If $\tilde D(I)$ was not a factor, that action could be used to construct a non-trivial direct sum decomposition of $H_0=L^2(D(S^1_\top))$, contradicting the irreducibility of $D$.
\end{proof}

\subsection*{\hspace*{-18pt}Finiteness of the defect vacuum as a 4-interval bimodule\nopunct} \hspace{-.15cm} ({\it splitting
$\tikzmath[scale=.02]
{\useasboundingbox (-15,-7) rectangle (15,7);
\draw[ultra thick] (-90:15) arc (-90:-4:15) (4:15) arc (4:90:15);
\draw (90:15) arc (90:127:15) (135:15) arc (135:176:15) (-90:15) arc (-90:-127:15) (-135:15) arc (-135:-176:15);}$
}). 

\begin{notation}\label{not: Dhat}
Let $S^1=I_1\cup I_2\cup I_3\cup I_4$ be a partitioning of the standard bicolored circle into four intervals so that $I_1$ and $I_4$ are genuinely bicolored, and so that each intersection $I_i\cap I_{i+1}$ (cyclic numbering) is a single point.
For such a partition, we let $\hat D(I_1\cup I_3)$ denote the commutant of $D(I_2\cup I_4)=D(I_2)\,\bar\otimes\,D(I_4)$
acting on the vacuum sector $H_0(D)$.

Similarly, if $\cala$ is a conformal net and
$S^1=I_1\cup I_2\cup I_3\cup I_4$ is a partitioning of the standard (not bicolored) circle, we let $\hat \cala(I_1\cup I_3)$ denote the commutant of $\cala(I_2\cup I_4)=\cala(I_2)\,\bar\otimes\,\cala(I_4)$ on the vacuum sector $H_0(\cala)$.
\end{notation}

Note that the choice of ambient circle does not affect the resulting algebras $\hat D(I_1\cup I_3)$ and $\hat \cala(I_1\cup I_3)$:
they only depend (up to canonical isomorphism) on the intervals $I_1$ and $I_3$, and on the bicoloring of those intervals.

\begin{lemma}\label{lem: needed for dotted lines}
Let $I_1$, $I_2$, $I_3$, $I_4$ be as in Notation \ref{not: Dhat}.
Assume furthermore that $I_2$ and $I_3$ are white.
Write $I_1=J_1\cup J_2$, with $J_1$ genuinely bicolored, $J_2$ white, and $J_1\cap J_2$ a single point:
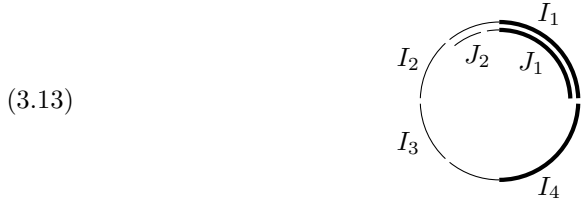
\begin{equation}\label{eq: fig of I_1 ... I_4 and J_1, J_2}
\tikzmath[scale=.07]
{\draw[ultra thick] (-90:15) arc (-90:-2:15) (2:15) arc (2:90:15)(2:13.5) arc (2:90:13.5) ;
\draw (90:15) arc (90:129:15) (133:15) arc (133:178:15) (-90:15) arc (-90:-129:15) (-133:15) arc (-133:-178:15) (90:13.5) arc (90:100:13.5)(105:13.5) arc (105:129:13.5) ;
\draw(60:19) node {$I_1$} (-60:19) node {$I_4$} (-155:19) node {$I_3$} (155:19) node {$I_2$} (116:9.5) node {$J_2$} (50:9.5) node {$J_1$} ;}
\end{equation}
%Let $\hat \cala (J_2 \cup I_3)$ be the commutant of $\cala(I_2)\,\bar\otimes\,\cala(S^1 \backslash (J_2 \cup I_2 \cup I_3))$ acting on $H_0(\cala)$.
Then there is a natural action of the algebra $\hat \cala (J_2\cup I_3)$ on the vacuum sector $H_0(D)$, and we have $\hat D(I_1\cup I_3) = D(J_1)\vee \hat\cala(J_2\cup I_3)$.
\end{lemma}

\begin{proof}
We assume that $D$ is faithful; otherwise $D=0$ (since $\cala$ and $\calb$ are irreducible) and there is nothing to show.
By Haag duality for $\cala$, the algebra $\hat \cala (J_2\cup I_3)$ is $\cala (J_2\cup I_2\cup I_3) \cap \cala(I_2)'$, where the commutant is taken in the action on $H_0(\cala)$.  The algebra $\cala(J_2 \cup I_2 \cup I_3)$ also naturally acts on $H_0(D)$, and therefore so does $\hat \cala (J_2 \cup I_3)$.  Because of the faithfulness assumption, the algebra $\hat \cala(J_2 \cup I_3)$ may equally well be expressed as $\cala (J_2\cup I_2\cup I_3) \cap \cala(I_2)'$, where the commutant is now interpreted with respect to the action on $H_0(D)$.

By Lemma~\ref{lem:commutant-spacial-vee 1},
\[
D(J_1)\vee \hat \cala (J_2\cup I_3) =  
   D(J_1)\vee \big(\cala (J_2\cup I_2\cup I_3)\cap \cala (I_2)'\big)
\]
is equal to 
\begin{align*}
\big(D(J_1)\vee \cala (J_2\cup I_2\cup I_3)\big) \cap \cala (I_2)'
&= D(I_1\cup I_2\cup I_3) \cap \cala (I_2)'\\
&= \big(D(I_1\cup I_2\cup I_3)' \vee \cala (I_2)\big)'.
\end{align*}
The last algebra is equal to $\big(D(I_4) \vee \cala (I_2)\big)' = D(I_2 \cup I_4)'$ 
by Haag duality for defects (Proposition~\ref{prop: [Haag duality for defects]}).
\end{proof}

\begin{lemma}\label{lem: B_2 is the relative commutant}
Let $I_1$, $I_2$, $I_3$, $I_4$ be arranged as in \eqref{eq: fig of I_1 ... I_4 and J_1, J_2}.
Assuming $D$ is irreducible, then $\cala(I_2)$ is the relative commutant of $\hat D(I_1\cup I_3)$ inside $D(I_1\cup I_2\cup I_3)$.
\end{lemma}

\begin{proof}
By Lemma \ref{lem:commutant-spacial-vee 2}, we have
$
\cala(I_2) = 
\big(\cala(I_2) \vee D(I_1\cup I_2\cup I_3)'\big) \cap D(I_1\cup I_2\cup I_3).
$
This algebra is equal to 
$
(\cala(I_2) \vee D(I_4)) \cap D(I_1\cup I_2\cup I_3) = \hat D(I_1\cup I_3)'\cap D(I_1\cup I_2\cup I_3). 
$
\end{proof}

In the next Lemma we will use the notion of minimal index $[A : B]$ 
of a subfactor $B \subseteq A$; 
see Appendix~\ref{subsec:stat-dim+minimal-index} for a definition.

\begin{lemma}\label{lem: [ : ] < mu(A)}
Let $\cala$ be a conformal net with finite index $\mu(\cala)$, 
and let ${}_\cala D_\calb$ be an irreducible defect.
Let $I_1$, $I_2$, $I_3$, $I_4$ be arranged as in \eqref{eq: fig of I_1 ... I_4 and J_1, J_2}.
Then $[\hat D(I_1\cup I_3) : D(I_1\cup I_3)]\le \mu(\cala)$.
\end{lemma}

\begin{proof}
Note that $\hat D(I_1 \cup I_3)$ and $D(I_1 \cup I_3) = D(I_1)\,\bar\otimes\,\cala(I_3)$ are both factors.

Let us decompose $I_1$ into intervals $J_1$, $J_2$ as in \eqref{eq: fig of I_1 ... I_4 and J_1, J_2}.
By Lemma~\ref{lem: needed for dotted lines}, we have
\[
\hat \cala(J_2\cup I_3)\vee D(J_1) = \hat D(I_1\cup I_3).
\]
We also have
\[
\cala(J_2\cup I_3)\vee D(J_1) = D(I_1\cup I_3).
\]
By definition $\mu(\cala) = [ \hat \cala(J_2\cup I_3) : \cala(J_2\cup I_3) ]$.
The result follows because the minimal index cannot increase under the operation $-\vee D(J_1)$,
see~\eqref{eq:DIS-727} in Appendix~\ref{subsec:stat-dim+minimal-index}. 
\end{proof}

\begin{remark}
We will see later, in Corollary \ref{cor: [[ : ]] = sqrt mu(A)}, that in fact we have an equality $[\hat D(I_1\cup I_3) : D(I_1\cup I_3)]= \mu(\cala)$.
\end{remark}

\subsection*{\hspace*{-18pt}Finiteness implies semisimplicity}

We can now prove the semisimplicity of the fusion of semisimple defects.

\begin{proof}[Proof of Theorem \ref{thm: semi-simplicity of DoE}]
Because the defects $D$ and $E$ are semisimple, we may write them as finite direct sums of irreducible defects: $D=\bigoplus D_i$ and $E=\bigoplus E_j$.
Fusion of defects is compatible with direct sums
\[
{{\textstyle\big(\bigoplus D_i\big)}\circledast_\calb {\textstyle\big(\bigoplus E_j\big)}} = \bigoplus_{ij} {D_i\circledast_\calb {E_j}}.
\]
It therefore suffices to assume $D$ and $E$ are irreducible, and to show that for $I$ genuinely bicolored, the von Neumann algebra $(D \circledast E)(I)$ has finite-dimensional center.

Without loss of generality, we assume that $I=S^1_\top$.
Let $H:=H_0(D)$ and $K:=H_0(E)$ be the vacuum sectors of $D$ and $E$.
The algebra $\tilde D(\partial^\sqsubset[0,1]^2)$ acts on $H$ (see notation \ref{not: tildeD}).
Similarly, the algebra $\tilde E(\partial^\sqsupset[0,1]^2)$ acts on $K$.
Let us denote those two algebras graphically by
$\tilde D(\,\tikzmath[scale=\textscale]
  {   \draw[thick, double] (6,0) -- (0,0) -- (0,12) -- (6,12);
      \draw  (6,12) -- (12,12) (12,0) -- (6,0); 
}\,)$ and
$\tilde E(\,\tikzmath[scale=\textscale]
  {   \draw (6,0) -- (0,0) (0,12) -- (6,12);
      \draw[ultra thick]  (6,12) -- (12,12) -- (12,0) -- (6,0); 
}\,)$.

The Hilbert space $H\boxtimes_\calb K$ is a faithful $(D\circledast E)(S^1_\top)$\,--\,$(D\circledast E)(S^1_\bot)$\,-bimodule.
So by Lemma \ref{lem: Lemma A}, it is enough to show that the algebra of bimodule endomorphisms of $H\boxtimes K$ is finite-dimensional.
This algebra of endomorphisms is equal to the algebra of $\tilde D(\,\tikzmath[scale=\textscale]
  {   \draw[thick, double] (6,0) -- (0,0) -- (0,12) -- (6,12);
      \draw  (6,12) -- (12,12) (12,0) -- (6,0); 
}\,)$-%
$\tilde E(\,\tikzmath[scale=\textscale]
  {   \draw (6,0) -- (0,0) (0,12) -- (6,12);
      \draw[ultra thick]  (6,12) -- (12,12) -- (12,0) -- (6,0); 
}\,)$-%
endomorphisms of $H\boxtimes K$.
Note that the algebras
$\tilde D(\,\tikzmath[scale=\textscale]
  {   \draw[thick, double] (6,0) -- (0,0) -- (0,12) -- (6,12);
      \draw  (6,12) -- (12,12) (12,0) -- (6,0); 
}\,)$ and
$\tilde E(\,\tikzmath[scale=\textscale]
  {   \draw (6,0) -- (0,0) (0,12) -- (6,12);
      \draw[ultra thick]  (6,12) -- (12,12) -- (12,0) -- (6,0); 
}\,)$ %
are factors by
Lemma \ref{lem: tildeD}.
If a bimodule has finite statistical dimension (see Appendix \ref{subsec:stat-dim+minimal-index}), then its 
algebra of bimodule endomorphisms is finite 
dimensional~\cite[\lemLemmaB]{BDH(Dualizability+Index-of-subfactors)}. 
It is therefore enough to show that the statistical dimension of 
$H\boxtimes_\calb K$ as a
$\tilde D(\,\tikzmath[scale=\textscale]
  {   \draw[thick, double] (6,0) -- (0,0) -- (0,12) -- (6,12);
      \draw  (6,12) -- (12,12) (12,0) -- (6,0); 
}\,)$-%
$\tilde E(\,\tikzmath[scale=\textscale]
  {   \draw (6,0) -- (0,0) (0,12) -- (6,12);
      \draw[ultra thick]  (6,12) -- (12,12) -- (12,0) -- (6,0); 
}\,)$-%
bimodule is finite.

Using the multiplicativity of the statistical dimension with respect to
Connes fusion~\eqref{eq:dim-and-Connes-fusion}
the dimension in question can be computed as
\[
\dim\left({}_{\tilde D(\,
\tikzmath[scale=.02]
  {   \draw[thick, double] (6,0) -- (0,0) -- (0,12) -- (6,12);
      \draw  (6,12) -- (12,12) (12,0) -- (6,0); 
}%tikzmath
\,)}H\boxtimes_\calb K_{\tilde E(\,
\tikzmath[scale=.02]
  {   \draw (6,0) -- (0,0) (0,12) -- (6,12);
      \draw[ultra thick]  (6,12) -- (12,12) -- (12,0) -- (6,0); 
}%tikzmath
\,)}\right)
=
\dim\left({}_{\tilde D(\,
\tikzmath[scale=.02]
  {   \draw[thick, double] (6,0) -- (0,0) -- (0,12) -- (6,12);
      \draw  (6,12) -- (12,12) (12,0) -- (6,0); 
}%tikzmath
\,)}H_{\raisebox{-1.5pt}{$\scriptstyle\calb(\tikzmath[scale=.02]
  {   \useasboundingbox (-3,0) rectangle (3,12);
      \draw  (0,12) -- (0,0);
}%tikzmath
)$}}\right)
\cdot
\dim\left({}_{\raisebox{-1.5pt}{$\scriptstyle\calb(\tikzmath[scale=.02]
  {   \useasboundingbox (-3,0) rectangle (3,12);
      \draw  (0,12) -- (0,0); 
}%tikzmath
)$}}K{}_{\tilde E(\,
\tikzmath[scale=.02]
  {   \draw (6,0) -- (0,0) (0,12) -- (6,12);
      \draw[ultra thick]  (6,12) -- (12,12) -- (12,0) -- (6,0); 
}%tikzmath
\,)}\right).
\]
So it suffices to argue that the dimension of $H$ as a 
$\tilde D(\,\tikzmath[scale=\textscale]
  {   \draw[thick, double] (6,0) -- (0,0) -- (0,12) -- (6,12);
      \draw  (6,12) -- (12,12) (12,0) -- (6,0); 
}\,)$-%
$\calb(\tikzmath[scale=\textscale]
  {   \useasboundingbox (-3,0) rectangle (3,12);
      \draw  (0,12) -- (0,0);
})$-%
bimodule 
and the dimension of $K$ as a 
$\calb(\tikzmath[scale=\textscale]
  {   \useasboundingbox (-3,0) rectangle (3,12);
      \draw  (0,12) -- (0,0);
})$-%
$\tilde E(\,
\tikzmath[scale=\textscale]
  {   \draw (6,0) -- (0,0) (0,12) -- (6,12);
      \draw[ultra thick]  (6,12) -- (12,12) -- (12,0) -- (6,0); 
}\,)$-%
bimodule are finite.
This is the content of Lemma \ref{lem: Dtilde H B is finite} below.
\end{proof}

Before proceeding, let us fix new names for certain subintervals of our standard circle:
\begin{alignat*}{3}
I_1 &:= \partial^{\,\tikzmath[scale=.15]{\draw(0,1)--(0,0)--(1,0);}} \big(\textstyle [0,\frac34] \times [0,\frac12]\big)&\qquad      I_2 &:= \textstyle[\frac34,1] \times \{0\}\\
I_3 &:= \{1\} \times [0,1]&\qquad      I_4 &:= \partial^{\,\tikzmath[scale=.15]{\draw(0,0)--(0,1)--(1,1);}} \big([0,1] \!\times\! \textstyle[\frac12,1]\big)
\end{alignat*}
\[
\tikzmath[scale=.15]{
\draw (6,0) -- (0,0) -- (0,5.75)(0,6.25) -- (0,12) -- (6,12);\draw[ultra thick] (6,12) -- (7,12) (6,0) -- (7,0);
\draw[ultra thick, line cap=round] (7,12) -- (11.5,12) (12,11.5) -- (12,.5) (7,0) -- (8.6,0) (9.3,0) -- (11.5,0);
\draw (-2,1) node {$I_1$};\draw (11,-2) node {$I_2$};\draw (14,6) node {$I_3$};\draw (4,14) node {$I_4$};}
\]
Given a defect ${}_\cala D_\calb$, let us also introduce the following shorthand notations:
\begin{alignat*}{4}
D_{234} &:= D(I_2 \cup I_3 \cup I_4) &\qquad D_{24}&:= \calb(I_2)\, \bar\otimes\, D(I_4)\qquad \calb_3 := \calb(I_3)\\
\hat D_{24} &:= \big( D(I_1)\,\bar\otimes\, \calb(I_3)\big)' &
\tilde D_{412} &:= D(I_4) \vee D(I_1\cup I_2).
\end{alignat*}

\begin{lemma}\label{lem: Dtilde H B is finite}
Let $\calb$ be a conformal net with finite index and let ${}_\cala D_\calb$ be an irreducible defect.
Then $H_0(D)$ has finite statistical dimension as a $\tilde D_{412}$--$\calb_3$-bimodule.
\end{lemma}

\begin{proof} 
The statistical dimension of $H_0(D)$ as a $\tilde D_{412}$--$\calb_3$-bimodule is the square root
of the minimal index of the subfactor $\calb_3 \subseteq {\tilde D_{412}}'$,
see~\eqref{eq:dim-on-H-is-dim-on-L2} and the definition of minimal index in Appendix \ref{subsec:stat-dim+minimal-index}. Therefore, we need to show that
$[{\tilde D_{412}}' : \calb_3] < \infty$.
We know by Lemma \ref{lem: B_2 is the relative commutant} (with $\cala$ and $\calb$ interchanged)
that the algebra $\calb_3$ is the relative commutant of $\hat D_{24}$ 
inside $D_{234}$. 
The algebra ${\tilde D_{412}}'$ is the relative commutant of $D_{24}$ inside $D_{234}$, as can be seen by taking the commutant of
the equation $(D_{24}'\cap D_{234})' = D_{24}\vee D_{234}' = \tilde D_{412}$.
The index is unchanged by taking commutants, and can only decrease
under the operation $-\cap D_{234}$ by~\eqref{eq:DIS-726}. 
Thus
\[
[{\tilde D_{412}}' : \calb_3] 
     = [D_{24}' \cap D_{234} : {\hat D_{24}}' \cap D_{234}] 
       \le [\hat D_{24} : D_{24}],
\]
and we have already seen in Lemma \ref{lem: [ : ] < mu(A)} that $[\hat D_{24} : D_{24}] \le \mu(\calb) < \infty$.
\end{proof}

\subsection*{\hspace*{-18pt}Finiteness of the defect vacuum as a 4-interval bimodule\nopunct} \hspace{-.15cm} ({\it splitting
$\tikzmath[scale=.02]
{\useasboundingbox (-15,-10) rectangle (15,10);
\draw[ultra thick](90:15) arc (90:45:15) (37:15) arc (37:-37:15) (-90:15) arc (-90:-45:15);
\draw (90:15) arc (90:135:15) (143:15) arc (143:217:15) (-90:15) arc (-90:-135:15);}$
\smallskip}).

\vspace*{-1pt} \noindent We record the following finiteness result, somewhat similar to Lemma \ref{lem: [ : ] < mu(A)}, for future reference. 

Let $I_1$, $I_2$, $I_3$, $I_4$ now be the four sides of our 
standard bicolored circle:
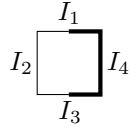
\[
\tikzmath[scale=\squarescale]
{\draw (6,0) -- (0,0) -- (0,12) -- (6,12);\draw[ultra thick] (6,0) -- (12,0) -- (12,12) -- (6,12);
\draw (-3,6) node {$I_2$}(15.5,6) node {$I_4$}(6.15,15.5) node {$I_1$}(6.3,-3.4) node {$I_3$};
}  %tikzmath
\]
The intervals $I_1$ and $I_3$ are genuinely bicolored, $I_2$ is white, and $I_4$ is black.

\begin{proposition}\label{prop: KLM finiteness for a defect's vacuum}
Let $\cala$ and $\calb$ be conformal nets with finite index, and let ${}_\cala D_\calb$ be an irreducible defect.
Then the vacuum sector $H_0(D)$ has finite statistical dimension as a $D(I_1)\vee D(I_3)$ -- $(\cala(I_2)\vee \calb(I_4))^\op$-bimodule.
\end{proposition}

\begin{proof}
Consider the following intervals:
\[\begin{split}
J_1 := \textstyle[\frac34,1] \times \{0\}
\qquad 
J_2 := I_4
\qquad
J_3 := \textstyle[\frac34,1] \times \{1\}
\qquad
J_4 := \textstyle[\frac14,\frac34] \times \{1\}\\
J_5 := \textstyle[0,\frac14] \times \{1\}
\qquad 
J_6 := I_2
\qquad
J_7 := \textstyle[0,\frac14] \times \{0\}
\qquad
J_8 := \textstyle[\frac14,\frac34] \times \{0\}
\end{split}\]
which we draw here:
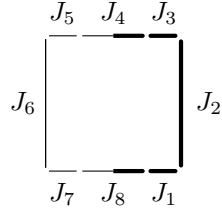
\[
\tikzmath[scale=.15]{
\draw (6,0) --(3.25,0)(2.75,0)-- (0.3,0)(0,0.3) -- (0,11.7);\draw[ultra thick] (6,12) -- (7,12) (6,0) -- (7,0);
\draw (6,12) --(3.25,12)(2.75,12)-- (0.3,12)(0,0.3) -- (0,11.7);
\draw[ultra thick, line cap=round] (12,11.5) -- (12,.5) (7,0) -- (8.6,0) (9.3,0) -- (11.5,0)(7,12) -- (8.6,12) (9.3,12) -- (11.5,12);
\draw (-2,6) node {$J_6$}(14.5,6) node {$J_2$};
\draw (1.5,-2) node {$J_7$} (6,-2) node {$J_8$}(10.5,-2) node {$J_1$};
\draw (1.5,14) node {$J_5$} (6,14) node {$J_4$}(10.5,14) node {$J_3$};}
\]
It will be convenient to introduce a graphical notation for the subalgebras of $\bfB(H_0(D))$ used in this proof:
\begin{alignat*}{3}
\tikzmath[scale=\displscale]{\useasboundingbox (-2,-2) rectangle (14,14);\draw (3,0)--(0,0)--(0,12)--(6,12);\draw[ultra thick] (9,0)--(12,0)--(12,12)--(6,12);} %tikzmath
\,\,&:=\,\,D(J_1\cup J_2\cup J_3\cup J_4\cup J_5\cup J_6\cup J_7)\quad&
\tikzmath[scale=\displscale]{\useasboundingbox (-2,-2) rectangle (14,14);\draw (3,12)--(6,12);\draw[ultra thick] (9,12)--(6,12);} %tikzmath
\,\,&:=\,\,D(J_4)\\
\tikzmath[scale=\displscale]{\useasboundingbox (-2,-2) rectangle (14,14);\draw (3,0)--(0,0)--(0,12)--(3,12);\draw[ultra thick] (9,0)--(12,0)--(12,12)--(9,12);} %tikzmath
\,\,&:=\calb(J_1\cup J_2\cup J_3)\vee \cala(J_5\cup J_6\cup J_7)\quad&
\tikzmath[scale=\displscale]{\useasboundingbox (-2,-2) rectangle (14,14);\draw (3,0)--(6,0);\draw[ultra thick] (9,0)--(6,0);} %tikzmath
\,\,&:=\,\,D(J_8)\\
\tikzmath[scale=\displscale]{\useasboundingbox (-2,-2) rectangle (14,14);\draw (3,0)--(0,0)(0,12)--(6,12);\draw[ultra thick] (9,0)--(12,0)(12,12)--(6,12);} %tikzmath
\,\,&:=\,\,\calb(J_1)\vee D(J_3\cup J_4\cup J_5)\vee \cala(J_7)\quad&
\tikzmath[scale=\displscale]{\useasboundingbox (-2,-2) rectangle (14,14);\draw (0,0)--(0,12);\draw[ultra thick] (12,0)--(12,12);} %tikzmath
\,\,&:=\,\, \calb(J_2)\vee \cala(J_6)\\\tikzmath[scale=\displscale]{\useasboundingbox (-2,-2) rectangle (14,14);
\draw (3,0)--(0,0)(0,12)--(6,12);\draw[ultra thick] (9,0)--(12,0)(12,12)--(6,12);\draw[densely dotted] (1.5,0)--(1.5,12)(10.5,0)--(10.5,12);} %tikzmath
\,\,&:=\,\,\hat\calb(J_1\cup J_3)\vee D(J_4)\vee\hat\cala(J_5\cup J_7)\quad&
\tikzmath[scale=\displscale]{\useasboundingbox (-2,-2) rectangle (14,14);\draw (0,0)--(0,12)(3,0)--(6,0);\draw[ultra thick] (12,0)--(12,12)(9,0)--(6,0);} %tikzmath
\,\,&:=\,\, \calb(J_2)\vee \cala(J_6)\vee D(J_8)\\
\tikzmath[scale=\displscale]{\useasboundingbox (-2,-2) rectangle (14,14);\draw (3,0)--(0,0)(0,12)--(3,12);
\draw[ultra thick] (9,0)--(12,0)(12,12)--(9,12);\draw[densely dotted] (1.5,0)--(1.5,12)(10.5,0)--(10.5,12);} %tikzmath
\,\,&:=\,\,\hat\calb(J_1\cup J_3)\vee\hat\cala(J_5\cup J_7)\quad&
\tikzmath[scale=\displscale]{\useasboundingbox (-2,-2) rectangle (14,14);\draw (3,0)--(0,0)(0,12)--(3,12);\draw[ultra thick] (9,0)--(12,0)(12,12)--(9,12);} %tikzmath
\,\,&:=\,\,\calb(J_1\cup J_3)\vee\cala(J_5\cup J_7),
\end{alignat*}
where, as in \ref{not: Dhat}, $\hat \cala(J_5\cup J_7)$ is the relative commutant of $\cala(J_6)$ in $\cala(J_5\cup J_6\cup J_7)$ and
$\hat \calb(J_1\cup J_3)$ is the relative commutant of $\calb(J_2)$ in $\calb(J_1\cup J_2\cup J_3)$.
Note that $\tikzmath[scale=\textscale]{\useasboundingbox (-2,-2) rectangle (14,14);
\draw (3,0)--(0,0)(0,12)--(6,12);\draw[ultra thick] (9,0)--(12,0)(12,12)--(6,12);\draw[densely dotted] (1.5,0)--(1.5,12)(10.5,0)--(10.5,12);}$ %tikzmath
is the commutant of 
$\tikzmath[scale=\textscale]{\useasboundingbox (-2,-2) rectangle (14,14);\draw (0,0)--(0,12);\draw[ultra thick] (12,0)--(12,12);\draw (3,0)--(6,0);\draw[ultra thick] (6,0)--(9,0);}$ %tikzmath
since, by Lemma \ref{lem:commutant-spacial-vee 1}, we have
\[\begin{split}
\tikzmath[scale=\displscale]{\useasboundingbox (-2,-2) rectangle (14,14);
\draw (3,0)--(0,0)(0,12)--(6,12);\draw[ultra thick] (9,0)--(12,0)(12,12)--(6,12);\draw[densely dotted] (1.5,0)--(1.5,12)(10.5,0)--(10.5,12);} %tikzmath
&=\tikzmath[scale=\displscale]{\useasboundingbox (-2,-2) rectangle (14,14);\draw (3,12)--(6,12);\draw[ultra thick] (9,12)--(6,12);} %tikzmath
\vee\Big(\tikzmath[scale=\displscale]{\useasboundingbox (-2,-2) rectangle (14,14);\draw (3,0)--(0,0)--(0,12)--(3,12);\draw[ultra thick] (9,0)--(12,0)--(12,12)--(9,12);} %tikzmath
\cap\Big(\tikzmath[scale=\displscale]{\useasboundingbox (-2,-2) rectangle (14,14);\draw (0,0)--(0,12);\draw[ultra thick] (12,0)--(12,12);} %tikzmath
\Big)'\Big)\\&=\Big(\tikzmath[scale=\displscale]{\useasboundingbox (-2,-2) rectangle (14,14);\draw (3,12)--(6,12);\draw[ultra thick] (9,12)--(6,12);} %tikzmath
\vee\tikzmath[scale=\displscale]{\useasboundingbox (-2,-2) rectangle (14,14);\draw (3,0)--(0,0)--(0,12)--(3,12);\draw[ultra thick] (9,0)--(12,0)--(12,12)--(9,12);} %tikzmath
\Big)\cap\Big(\tikzmath[scale=\displscale]{\useasboundingbox (-2,-2) rectangle (14,14);\draw (0,0)--(0,12);\draw[ultra thick] (12,0)--(12,12);} %tikzmath
\Big)'=\tikzmath[scale=\displscale]{\useasboundingbox (-2,-2) rectangle (14,14);\draw (3,0)--(0,0)--(0,12)--(6,12);\draw[ultra thick] (9,0)--(12,0)--(12,12)--(6,12);} %tikzmath
\cap\Big(\tikzmath[scale=\displscale]{\useasboundingbox (-2,-2) rectangle (14,14);\draw (0,0)--(0,12);\draw[ultra thick] (12,0)--(12,12);} %tikzmath
\Big)'=\Big(\tikzmath[scale=\displscale]{\useasboundingbox (-2,-2) rectangle (14,14);\draw (3,0)--(6,0);\draw[ultra thick] (9,0)--(6,0);} %tikzmath
\vee\tikzmath[scale=\displscale]{\useasboundingbox (-2,-2) rectangle (14,14);\draw (0,0)--(0,12);\draw[ultra thick] (12,0)--(12,12);} %tikzmath
\Big)'.
%=\Big(\tikzmath[scale=\displscale]{\useasboundingbox (-2,-2) rectangle (14,14);\draw (0,0)--(0,12);\draw[ultra thick] (12,0)--(12,12);\draw (3,0)--(6,0);\draw[ultra thick] (9,0)--(6,0);}\Big)'
\end{split}\]
In particular, the algebra $\tikzmath[scale=\textscale]{\useasboundingbox (-2,-2) rectangle (14,14);\draw (3,0)--(0,0)(0,12)--(6,12);
\draw[ultra thick] (9,0)--(12,0)(12,12)--(6,12);\draw[densely dotted] (1.5,0)--(1.5,12)(10.5,0)--(10.5,12);}$ %tikzmath
is a factor.

We have to show that 
\[
\big[\big(D(I_1)\vee D(I_3)\big)':\cala(I_2)\vee \calb(I_4)\big]<\infty.
\]
Using Haag duality and strong additivity, note that the algebra $(D(I_1)\vee D(I_3))'$ is the relative commutant of 
$\tikzmath[scale=\textscale]{\useasboundingbox (-2,-2) rectangle (14,14);\draw (3,0)--(0,0)(0,12)--(6,12);\draw[ultra thick] (9,0)--(12,0)(12,12)--(6,12);}$  %tikzmath
inside
$\tikzmath[scale=\textscale]{\useasboundingbox (-2,-2) rectangle (14,14);\draw (3,0)--(0,0)--(0,12)--(6,12);\draw[ultra thick] (9,0)--(12,0)--(12,12)--(6,12);}$.  %tikzmath
Similarly, it follows from Lemma \ref{lem:commutant-spacial-vee 2} that the algebra $\cala(I_2)\vee \calb(I_4)$ is the relative commutant of 
$\tikzmath[scale=\textscale]{\useasboundingbox (-2,-2) rectangle (14,14);
\draw (3,0)--(0,0)(0,12)--(6,12);\draw[ultra thick] (9,0)--(12,0)(12,12)--(6,12);\draw[dash pattern=on .4pt off .62pt] (1.5,0)--(1.5,12)(10.5,0)--(10.5,12);}$ %tikzmath
in
$\tikzmath[scale=\textscale]{\useasboundingbox (-2,-2) rectangle (14,14);\draw (3,0)--(0,0)--(0,12)--(6,12);\draw[ultra thick] (9,0)--(12,0)--(12,12)--(6,12);}$\,: %tikzmath
\[
\tikzmath[scale=\displscale]{\useasboundingbox (-2,-2) rectangle (14,14);\draw (0,0)--(0,12);\draw[ultra thick] (12,0)--(12,12);} \,=\, 
\Big(\tikzmath[scale=\displscale]{\useasboundingbox (-2,-2) rectangle (14,14);\draw (0,0)--(0,12);\draw[ultra thick] (12,0)--(12,12);} 
\vee \Big(\tikzmath[scale=\displscale]{\useasboundingbox (-2,-2) rectangle (14,14);\draw (3,0)--(0,0)--(0,12)--(6,12);\draw[ultra thick] (9,0)--(12,0)--(12,12)--(6,12);} %tikzmath
\Big)'\Big)\cap\tikzmath[scale=\displscale]{\useasboundingbox (-2,-2) rectangle (14,14);\draw (3,0)--(0,0)--(0,12)--(6,12);\draw[ultra thick] (9,0)--(12,0)--(12,12)--(6,12);} %tikzmath
 \,=\, \tikzmath[scale=\displscale]{\useasboundingbox (-2,-2) rectangle (14,14);\draw (0,0)--(0,12);\draw[ultra thick] (12,0)--(12,12);\draw (3,0)--(6,0);\draw[ultra thick] (6,0)--(9,0);} %tikzmath
\cap\tikzmath[scale=\displscale]{\useasboundingbox (-2,-2) rectangle (14,14);\draw (3,0)--(0,0)--(0,12)--(6,12);\draw[ultra thick] (9,0)--(12,0)--(12,12)--(6,12);} %tikzmath
\,=\,\Big(\tikzmath[scale=\displscale]{\useasboundingbox (-2,-2) rectangle (14,14);
\draw (3,0)--(0,0)(0,12)--(6,12);\draw[ultra thick] (9,0)--(12,0)(12,12)--(6,12);\draw[dash pattern=on .4pt off .62pt] (1.5,0)--(1.5,12)(10.5,0)--(10.5,12);} %tikzmath
\Big)'\cap\tikzmath[scale=\displscale]{\useasboundingbox (-2,-2) rectangle (14,14);\draw (3,0)--(0,0)--(0,12)--(6,12);\draw[ultra thick] (9,0)--(12,0)--(12,12)--(6,12);}. %tikzmath
\]
By~\eqref{eq:DIS-726} and~\eqref{eq:DIS-727}, 
we then have
\[
\begin{split}
\Big[\big(D(I_1)\vee D(I_3)\big)':\cala(I_2)\vee \calb(I_4)\Big]\,&=\,
\bigg[
\tikzmath[scale=\displscale]{\useasboundingbox (-2,-2) rectangle (14,14);
\draw (3,0)--(0,0)--(0,12)--(6,12);\draw[ultra thick] (9,0)--(12,0)--(12,12)--(6,12);} %tikzmath
\cap \Big(\tikzmath[scale=\displscale]{\useasboundingbox (-2,-2) rectangle (14,14);\draw (3,0)--(0,0)(0,12)--(6,12);\draw[ultra thick] (9,0)--(12,0)(12,12)--(6,12);} %tikzmath
\Big)':\tikzmath[scale=\displscale]{\useasboundingbox (-2,-2) rectangle (14,14);\draw (3,0)--(0,0)--(0,12)--(6,12);\draw[ultra thick] (9,0)--(12,0)--(12,12)--(6,12);} %tikzmath
\cap \Big(\tikzmath[scale=\displscale]{\useasboundingbox (-2,-2) rectangle (14,14);
\draw (3,0)--(0,0)(0,12)--(6,12);\draw[ultra thick] (9,0)--(12,0)(12,12)--(6,12);\draw[densely dotted] (1.5,0)--(1.5,12)(10.5,0)--(10.5,12);} %tikzmath
\Big)'\bigg]\\
\le\,\bigg[\tikzmath[scale=\displscale]{\useasboundingbox (-2,-2) rectangle (14,14);\draw (3,0)--(0,0)(0,12)--(6,12);
\draw[ultra thick] (9,0)--(12,0)(12,12)--(6,12);\draw[densely dotted] (1.5,0)--(1.5,12)(10.5,0)--(10.5,12);} %tikzmath
:\tikzmath[scale=\displscale]{\useasboundingbox (-2,-2) rectangle (14,14);\draw (3,0)--(0,0)(0,12)--(6,12);\draw[ultra thick] (9,0)--(12,0)(12,12)--(6,12);} %tikzmath
\bigg]\,&=\,\bigg[\tikzmath[scale=\displscale]{\useasboundingbox (-2,-2) rectangle (14,14);\draw (3,12)--(6,12);\draw[ultra thick] (9,12)--(6,12);} %tikzmath
\vee\tikzmath[scale=\displscale]{\useasboundingbox (-2,-2) rectangle (14,14);\draw (3,0)--(0,0)(0,12)--(3,12);
\draw[ultra thick] (9,0)--(12,0)(12,12)--(9,12);\draw[densely dotted] (1.5,0)--(1.5,12)(10.5,0)--(10.5,12);} %tikzmath
:\tikzmath[scale=\displscale]{\useasboundingbox (-2,-2) rectangle (14,14);\draw (3,12)--(6,12);\draw[ultra thick] (9,12)--(6,12);} %tikzmath
\vee\tikzmath[scale=\displscale]{\useasboundingbox (-2,-2) rectangle (14,14);\draw (3,0)--(0,0)(0,12)--(3,12);\draw[ultra thick] (9,0)--(12,0)(12,12)--(9,12);} %tikzmath
\bigg]\\
\le\,\bigg[\tikzmath[scale=\displscale]{\useasboundingbox (-2,-2) rectangle (14,14);\draw (3,0)--(0,0)(0,12)--(3,12);
\draw[ultra thick] (9,0)--(12,0)(12,12)--(9,12);\draw[densely dotted] (1.5,0)--(1.5,12)(10.5,0)--(10.5,12);} %tikzmath
:\tikzmath[scale=\displscale]{\useasboundingbox (-2,-2) rectangle (14,14);\draw (3,0)--(0,0)(0,12)--(3,12);\draw[ultra thick] (9,0)--(12,0)(12,12)--(9,12);} %tikzmath
\bigg]\,&=\, \mu(\cala)\mu(\calb). \qedhere
\end{split}
\]
\end{proof}

%\end{document}
%==================================================================

\chapter{A variant of horizontal fusion}  \label{sec:F-G-and-G_0}

In Section \ref{ssec:Horizontal fusion} we saw how to define the horizontal fusion of two sectors. 
We will now define a variant of the horizontal fusion, called keystone fusion, which itself depends on an intermediate construction we refer to as keyhole fusion. 
In Section \ref{sec: Comparison between F and G}, we will show that horizontal fusion and keystone fusion are in fact naturally isomorphic, and we will construct a canonical isomorphism $\Phi$ between them.
That isomorphism will be essential in our construction of the $1\boxtimes 1$-isomorphism $\Omega$~\eqref{eq: definition of Omega}.

Recall that we implicitly assume that all our conformal nets are irreducible.

\section{The keyhole and keystone fusion}
 \label{sec: def of F G0 G}
Recall the Notation~\ref{not:names-for-intervals}.
Consider the intervals $I_l :=  \partial^\sqsupset ( [2/3,1] \!\times\! [0,1] )$, $I_r :=  \partial^\sqsubset ( [1,4/3] \!\times\! [0,1] )$, 
and $I := I_l \cap I_r$.
Orient $I_l$ and $I_r$ counterclockwise, and orient $I$ 
so that the inclusion $I \hookrightarrow I_r$ is orientation preserving---see \eqref{eq: figures of intervals 1}.
The inclusion $I \hookrightarrow I_l$ is then orientation reversing.
Let $J$ be the closure of $(I_l \cup I_r) \setminus I$.
We orient $J$ so that it agrees with the orientation of $I_l$ on $J \cap I_l$.
We draw these intervals as follows:
\begin{equation}\label{eq: figures of intervals 1}
I_l = \tikzmath[scale=\displscale] {\draw[<-] (9.2,12) -- (12,12) -- (12,0) -- (8,0);\draw (9.3,12) -- (8,12);}\;, %tikzmath
\quad I_r = \tikzmath[scale=\displscale]{\draw (13.6,12) -- (12,12) -- (12,0) -- (16,0);\draw[<-] (13.5,12) -- (16,12);}\;, %tikzmath
\quad I= \, \tikzmath[scale=\displscale]{\draw (12,5.8) -- (12,0); \draw[->] (12,12) -- (12,5.7);} %tikzmath
\quad \text{and} \quad J = \tikzmath[scale=\displscale] {\draw (8,12) -- (11.5,12); \draw[->](8,0) -- (12.6,0); \draw[<-] (11.4,12) -- (16,12); \draw (12.5,0) -- (16,0);}\;. %tikzmath
\end{equation}
Given a conformal net $\cala$ with finite index, we will define three functors
\begin{equation*}
  F,\, G_0,\, G\,\, \colon\,\, \modules{\cala(I_l)} \; \x \; \modules{\cala(I_r)}\,\, \to\,\, \modules{\cala(J)}.
\end{equation*}
These operations will be called respectively the fusion, the keyhole fusion, and the keystone fusion, and will be denoted graphically as follows:
\begin{equation*}
F(H_l,H_r) =\tikzmath[scale=\displscale] {\draw (8,12) -- (16,12) (8,0) -- (16,0) (12,0) -- (12,12);
\draw[ultra thin, dash pattern=on .5pt off 1pt](8,12) -- (0,12) -- (0,0) -- (8,0)(16,12) -- (24,12) -- (24,0) -- (16,0);\draw (6,6) node {$H_l$} (18,6) node {$H_r$};} %tikzmath
\, , \,\,\,\,\,\,\quad G_{0}(H_l,H_r) = \tikzmath[scale=\displscale] {\useasboundingbox (0,-4) rectangle (24,16);\draw (8,12) -- (10,12) -- (10,0) -- (8,0) (16,12) -- (14,12) -- (14,0) -- (16,0);
\draw[ultra thin, dash pattern=on .5pt off 1pt](8,12) -- (0,12) -- (0,0) -- (8,0) (16,12) -- (24,12) -- (24,0) -- (16,0); \draw (5,6) node {$H_l$} (19,6) node {$H_r$};
\fill[vacuumcolor] (10,0) rectangle (14,4) (10,8) rectangle (14,12); \draw (10,0) rectangle (14,4) (10,8) rectangle (14,12); } %tikzmath 
\, , \,\,\,\,\,\,\quad G(H_l,H_r) =\tikzmath[scale=\displscale] {\draw (8,12) -- (10,12) -- (10,0) -- (8,0)  (16,12) -- (14,12) -- (14,0) -- (16,0);
\draw[ultra thin, dash pattern=on .5pt off 1pt] (8,12) -- (0,12) -- (0,0) -- (8,0) (16,12) -- (24,12) -- (24,0) -- (16,0); \draw (5,6) node {$H_l$} (19,6) node {$H_r$};
\fill[vacuumcolor] (10,0) rectangle (14,4) (10,8) rectangle (14,12) (10.5,4.5) rectangle (13.5,7.5); \draw (10,0) rectangle (14,4) (10,8) rectangle (14,12) (10.5,4.5) rectangle (13.5,7.5); }\,. %tikzmath
\end{equation*}
When we want to stress the dependence on the conformal net $\cala$, we will denote these functors $F_\cala$, $G_{0,\cala}$, $G_\cala$.

\subsection*{\hspace*{-18pt}The ordinary horizontal fusion}

The functor $F$ is defined by fusion over $\cala(I)$:
using the orientation preserving inclusion $I \hookrightarrow I_r$, any left $\cala(I_r)$-module is also a left $\cala(I)$-module, and
using the orientation reversing inclusion $I_0 \hookrightarrow I_l$, any left $\cala(I_l)$-module is also a right $\cala(I)$-module.
We can therefore define the horizontal fusion functor as follows:
\begin{equation*}
F(H_l,H_r) := H_l \boxtimes_{\cala(I)} H_r.
\end{equation*}
Write $J$ as $J_1\sqcup J_2$;  
we obtain actions of $\cala(J_1)$ and $\cala(J_2)$ on $H_l \boxtimes_{\cala(I)} H_r$, by~\cite[\corcalaKacts]{BDH(nets)}.
Note that in the case $H_l=L^2\cala(I_l)$ and $H_r=L^2\cala(I_r)$, the actions of $\cala(J_1)$ and $\cala(J_2)$ extend to an action of $\cala(J)=\cala(J_1)\,\bar\otimes\,\cala(J_2)$; the same therefore holds for arbitrary $H_l$ and $H_r$.
The only difference between the functor $F$ and the functor $\mathsf{fusion_h}$ from~\eqref{eq: The functor fusion_h} is that they have somewhat different source and target categories---the main construction is identical in both functors.

\subsection*{\hspace*{-18pt}The keyhole fusion}
We will need to name a few more manifolds.
Let
\[
\begin{split}
\tilde I_l  \,&:=\,  \partial^\sqsupset ( [2/3,5/6] \x [0,1] ) \\
S_u &:=\, \dd ( [5/6,7/6] \x [2/3,1] )\\
S_d &:=\, \dd ( [5/6,7/6] \x [0,1/3] )
\end{split}
\qquad\quad
\begin{split}
\tilde I_r  \,&:=\,  \partial^\sqsubset ( [7/6,4/3] \x [0,1] )\\
S_m &:=\, \dd ( [5/6,7/6] \x [1/3,2/3] )\\
K  \,&:=\,  (S_u \cup S_d) \cap (\tilde I_l \cup \tilde I_r).
\end{split}
\]
We draw these as
\begin{equation}\label{eq: figures of intervals 2}
\tilde I_l = \tikzmath[scale=\displscale]{\draw[<-] (8.2,12) -- (10,12) -- (10,0) -- (8,0);\draw (8.5,12) -- (8,12); }\;,\quad\tilde I_r = \tikzmath[scale=\displscale]
{ \draw (14.4,12) -- (14,12) -- (14,0) -- (16,0);\draw[<-] (14.3,12) -- (16,12); }\;,%tikzmath
\quad S_u= \tikzmath[scale=\displscale]{\useasboundingbox (10,0) rectangle (14,12);\draw (10,8) rectangle (14,12);\draw[->] (11.5,12) -- (11.4,12);}\;,%tikzmath
\quad S_m= \tikzmath[scale=\displscale]{\useasboundingbox (10,0) rectangle (14,12);\draw (10,4) rectangle (14,8);\draw[->] (11.5,8) -- (11.4,8);}\;,%tikzmath
\quad S_d= \tikzmath[scale=\displscale]{\useasboundingbox (10,0) rectangle (14,12);\draw (10,0) rectangle (14,4);\draw[->] (11.5,4) -- (11.4,4);}%tikzmath
\quad \text{and} \quad K = \tikzmath[scale=\displscale]{\draw (14,12) -- (14,8) (14,4) -- (14,0)(10,12) -- (10,8) (10,4) -- (10,0);\draw[->] (10,9.41) -- (10,9.4);
\draw[->] (10,1.41) -- (10,1.4);\draw[->] (14,2.59) -- (14,2.6);\draw[->] (14,10.59) -- (14,10.6);}\;.%tikzmath
\end{equation}
The intervals $\tilde I_l$ and $\tilde I_r$ are oriented counterclockwise, as were $I_l$ and $I_r$.
The manifolds $S_u$, $S_m$ and $S_d$ are conformal circles via their constant speed parametrizations and are also oriented counterclockwise.
(A conformal circle is a circle together with a homeomorphism with
$S^1$ that is only determined up to orientation preserving conformal diffeomorphisms 
of $S^1$.)
Finally, the manifold $K$ inherits its orientation from $S_u \cup S_d$.
Note that the inclusion $K \hookrightarrow \tilde I_l \cup \tilde I_r$ is orientation reversing.
We will also need the reflection $j$ along the horizontal axis $y = 1/2$.

Let us fix orientation preserving identifications $\phi_l:\tilde I_l \stackrel{\scriptscriptstyle\cong}\to I_l$ and $\phi_r:\tilde I_r \stackrel{\scriptscriptstyle\cong}\to I_r$ that are symmetric with respect to the reflection $j$,
restrict to the identity in a neighborhood of $\dd \tilde I_l = \dd I_l$ and $\dd \tilde I_r = \dd I_r$,
and satisfy $\phi_l(5/6,t)=(1,t)$ and $\phi_r(7/6,t)=(1,t)$ for all $t\in[0,1]$.
Using these identifications, any $\cala(I_l)$-module becomes an $\cala(\tilde I_l)$-module and any $\cala(I_r)$-module becomes an $\cala(\tilde I_r)$-module. 
We can now define the keyhole fusion functor as follows:
\begin{equation*}
G_0(H_l,H_r) \,:=\, \big( H_l \otimes H_r \big) \underset{\cala(K)}\boxtimes \big(H_0(S_u) \otimes H_0(S_d)\big),
\end{equation*}
where $H_0(S_u)$ and $H_0(S_d)$ are the canonical vacuum sectors;
see~\eqref{eq:vacuum-sector-functor}.
The right-hand side is an instance of what we call cyclic fusion---see Appendix~\ref{subsec:cyclic-fusion}.
In the notation of cyclic fusion, we have
\begin{equation*}
G_0(H_l,H_r) \phantom{:}=\,\tikzmath{\node (a) at (0,0) {$H_l \underset{\cala(K_1)}\boxtimes H_0(S_u) \underset{\cala(K_2)}\boxtimes H_r \underset{\cala(K_3)}\boxtimes H_0(S_d) \underset{\cala(K_4)}\boxtimes$};
\def\dd{.5}\def\ll{.35}\def\rr{.25}\draw[dashed, rounded corners = 6] (a.east)++(0,.1) -- ++(\rr,0) -- ++(0,-\dd) -- ($(a.west) + (0,.1) + (-\ll,-\dd)$) -- ++(0,\dd) -- ++(\rr,0);} % tikzmath
\end{equation*} 
where $K_1 = \tilde I_l\cap S_u$, $K_2 = \tilde I_r\cap S_u$, $K_3 = \tilde I_r\cap S_d$, and $K_4 = \tilde I_l\cap S_d$, appropriately oriented.
%We denote it pictorially as
%\[
%G_0(H_l,H_r) = \tikzmath[scale=\displscale]{\draw (8,12) -- (10,12) -- (10,0) -- (8,0) (16,12) -- (14,12) -- (14,0) -- (16,0);
%\draw[ultra thin, dash pattern=on .5pt off 1pt](8,12) -- (0,12) -- (0,0) -- (8,0)(16,12) -- (24,12) -- (24,0) -- (16,0);
%\draw (5,6) node {$H_l$} (19,6) node {$H_r$};\fill[vacuumcolor] (10,0) rectangle (14,4)(10,8) rectangle (14,12);\draw (10,0) rectangle (14,4)(10,8) rectangle (14,12);}. %tikzmath 
%\]
It follows from~\cite[\corcalaKacts]{BDH(nets)} that the algebras $\cala(J \cap (\tilde I_l \cup \tilde I_r))$ and $\cala(J \cap (S_u \cup S_d))$ generate an action of $\cala(J)$ on $G_0(H_l,H_r)$.

\subsection*{\hspace*{-18pt}The keystone fusion}
Note that the algebra
\begin{equation*}
%\cala(S_m)^\op \cong\, \cala(-S_m) \,\cong\, 
\cala \left( \;
\tikzmath[scale=\displscale]{\draw (14,12) --(14,0)(10,12) -- (10,0);\draw[->] (10,6.59) -- (10,6.6);\draw[->] (14,5.41) -- (14,5.4); } \; %tikzmath
\right) \underset { \cala \left( \;
\tikzmath[scale=\displscalesmall]{\draw (14,12) -- (14,8) (14,4) -- (14,0)(10,12) -- (10,8) (10,4) -- (10,0);
\draw[->] (10,9.41) -- (10,9.4);\draw[->] (10,1.41) -- (10,1.4);
\draw[->] (14,2.59) -- (14,2.6);\draw[->] (14,10.59) -- (14,10.6);} \; %tikzmath
\right)} {\circledast}\cala \left( \;
\tikzmath[scale=\displscale]{\draw (10,12) -- (10,8) -- (14,8) -- (14,12)(10,0) -- (10,4) -- (14,4) -- (14,0);\draw[->] (11.41,4) -- (11.4,4);\draw[->] (12.59,8) -- (12.6,8); } \; %tikzmath
\right)
\end{equation*}
has
a natural left action on $G_0(H_l,H_r)$ commuting with the action of $\cala(J)$.
This algebra can be identified
with  $\cala(S_m)^\op \cong\, \cala(-S_m)$,
where $-S_m$ denotes the circle $S_m$ equipped with the opposite 
(i.e., clockwise) orientation.
Here, we use the extension of $\cala$ from intervals to $1$-manifolds
constructed in~\cite[\secExtendtomanifolds]{BDH(modularity)}; see also Appendix~\ref{subsec:cala(1-mfd)}.
By definition,\footnote{For this, we implicitly identify the surfaces 
$\tikzmath[scale=\textscale]{\useasboundingbox (-1,-3) rectangle (25,12);\filldraw[fill=vacuumcolor] (0,0) rectangle (24,12);\filldraw[fill=white] (14,4) rectangle (10,8);}%tikzmath
$ and $S_m\times [0,1]$.} the algebra $\cala(S_m)$ is generated by all the $\cala(I)$, for $I\subset S_m$, acting on the Hilbert space
$\tikzmath[scale=\textscale]{\useasboundingbox (-2,-1) rectangle (26,14);
\fill[vacuumcolor] (10,8) rectangle (14,12)(10,0) rectangle (14,4);\filldraw[fill=vacuumcolor] (0,0) rectangle (10,12) (14,0) rectangle (24,12);
\draw (10,0) -- (14,0) (10,4) -- (14,4) (10,8) -- (14,8) (10,12) -- (14,12);}%tikzmath
$.

By Theorem~\ref{thm:cala(S)},
the algebra $\cala(S_m)$ contains a direct summand that is canonically 
isomorphic to $\bfB(H_0(S_m,\cala))$.
We can therefore define the keystone fusion functor as follows:
\begin{equation*}
  G(H_l,H_r) \,:=\, G_0(H_l,H_r)\, \boxtimes_{\cala(S_m)}\, H_0(S_m).  
\end{equation*}
Moreover, since $\bfB(H_0(S_m,\cala))$ and $\cala(J)$ commute on $G_0(H_l,H_r)$, there is a residual action of $\cala(J)$ on the Hilbert space $G(H_l,H_r)$.

\subsection*{\hspace*{-18pt}Fusion and keystone fusion are isomorphic}
We will show presently that the functors
\[
  F,\, G\,\, \colon\,\, \modules{\cala(I_l)} \; \x \; \modules{\cala(I_r)}\,\, \to\,\, \modules{\cala(J)}
\]
are naturally isomorphic to one another, and then later (in Proposition~\ref{prop: local-fusion}) construct a specific such natural isomorphism.

We use the following straightforward generalization of 
Lemma~\ref{lem: NT between module categories}.

\begin{lemma}\label{lem: Mod(A1) x Mod(A2) --> Mod(B)}
Let
\(
F,G:\modules{A_1}\; \x \; \modules{A_2} \to \modules{B}
\)
be normal functors (see Appendix~\ref{subsec: Functors between module categories}), and let $M_i$ be a faithful $A_i$-module, for $i=1,2$.
Then, in order to uniquely define a natural transformation $a:F\to G$, it is enough to specify its
value on $(M_1,M_2)$ and to check that for each pair $(r_1,r_2)$ with $r_i\in\mathrm{End}_A(M_i)$, the diagram
\[
\tikzmath{
\matrix [matrix of math nodes,column sep=1.5cm,row sep=5mm]
{ 
|(a)| F(M_1,M_2) \pgfmatrixnextcell |(b)| F(M_1,M_2)\\ 
|(c)| G(M_1,M_2) \pgfmatrixnextcell |(d)| G(M_1,M_2)\\ 
}; 
\draw[->] (a) -- node [above]	{$\scriptstyle F(r_1,r_2)$} (b);
\draw[->] (c) -- node [above]	{$\scriptstyle G(r_1,r_2)$} (d);
\draw[->] (a) -- node [left]		{$\scriptstyle a_{M_1,M_2}$} (c);
\draw[->] (b) -- node [right]	{$\scriptstyle a_{M_1,M_2}$} (d);
}
\]
commutes.\qed
\end{lemma}

Using Theorem~\ref{thm:KLM} 
and the above lemma, we prove that the two different versions of horizontal fusion are naturally 
isomorphic to each other:

\begin{proposition}\label{prop: non-canonical construction of Phi}
There exists a natural isomorphism between the fusion functor $F$ and the keystone fusion functor $G$.
\end{proposition}

\begin{proof}
Consider the circles 
\[
\begin{split}
S_l &:= \dd ([0,1] \x [0,1]),\hspace{.15cm}\qquad S_r := \dd ([1,2] \x [0,1])\qquad S_b := \dd ([0,2] \x [1,2])\\
\tilde S_l&:=\dd([0,5/6]\times[0,1]),\quad \tilde S_r:=\dd([7/6,2]\times[0,1])\qquad\\
\hat S_l&:=\partial([0,5/6]\times[0,1]\cup[5/6,7/6]\times[0,1/3]\cup[5/6,7/6]\times[2/3,1])
\end{split}
\]
which we draw as follows:
\begin{equation}\label{eq: figures of intervals 3}
\begin{matrix}
S_l = \tikzmath[scale=\displscale]
        { \useasboundingbox (0,-1.5) rectangle (24,13.5);
          \draw (0,0) rectangle (12,12);
          \draw[->] (5.6,12) -- (5.5,12);
        } %tikzmath 
  \; , &
S_r = \tikzmath[scale=\displscale]
         {\useasboundingbox (0,-1.5) rectangle (24,13.5);
           \draw (12,0) rectangle (24,12);
          \draw[->] (17.6,12) -- (17.5,12);
         }\,, %tikzmath  
   & &
S_b = \tikzmath[scale=\displscale]
        {\useasboundingbox (0,-1.5) rectangle (24,13.5);
         \draw (0,0) rectangle (24,12);
         \draw[->] (11.6,12) -- (11.5,12);
        }\,, %tikzmath
&\\
\tilde S_l = \tikzmath[scale=\displscale]
        { \useasboundingbox (0,-1.5) rectangle (24,13.5);
          \draw (0,0) rectangle (10,12);
          \draw[->] (4.6,12) -- (4.5,12);
        } %tikzmath 
  \; , &
\tilde S_r = \tikzmath[scale=\displscale]
         { \useasboundingbox (0,-1.5) rectangle (24,13.5);
           \draw (14,0) rectangle (24,12);
          \draw[->] (18.6,12) -- (18.5,12);
         }\,, %tikzmath  
  & \text{and} &
\hspace{.03cm}\hat S_l = \tikzmath[scale=\displscale]
  {\useasboundingbox (0,-1.5) rectangle (24,13.5);
  \draw[->] (6.1,12) -- (0,12) -- (0,0) -- (14,0) -- (14,4) -- (10,4) -- (10,8) -- (14,8) -- (14,12) -- (6,12);}&
\end{matrix}
\end{equation}
(`{\it b}' stands for big). The identifications $\phi_l:\tilde I_l \stackrel{\scriptscriptstyle\cong}\to I_l$ and $\phi_r:\tilde I_r \stackrel{\scriptscriptstyle\cong}\to I_r$ induce isomorphisms
$H_0(\tilde S_l)\cong H_0(S_l)$ and $H_0(\tilde S_r)\cong H_0(S_r)$ that are equivariant with respect to 
$\cala(I_l')$ and $\cala(I_r')$ (here, $I_l'$ and $I_r'$ are the closures of $S_l\!\setminus\! I_l$ and $S_r\!\setminus\! I_r$, respectively).
From the isomorphism
$H_0(\hat S_l) \cong H_0(S_d) 
  \boxtimes_{\cala(K_4)} H_0(\tilde S_l) \boxtimes_{\cala(K_1)} H_0(S_u)$
it follows that $G_0\big(H_0(S_l),H_0(S_r)\big)$ represents the
Hilbert space of an annulus; see Appendix~\ref{subsec:sectors+KLM}.
Using Theorem~\ref{thm:KLM} we therefore have
\[
\begin{split}
G_0\big(H_0(S_l),H_0(S_r)\big) &= \,
			\tikzmath{
			\node (a) at (0,0) {$H_0(\tilde S_l) \underset{\cala(K_1)}\boxtimes H_0(S_u) \underset{\cala(K_2)}\boxtimes H_0(\tilde S_r) \underset{\cala(K_3)}\boxtimes H_0(S_d) \underset{\cala(K_4)}\boxtimes$};
			\def\dd{.5}\def\ll{.35}\def\rr{.25}
			\draw[dashed, rounded corners = 6] (a.east)++(0,.1) -- ++(\rr,0) -- ++(0,-\dd) -- ($(a.west) + (0,.1) + (-\ll,-\dd)$) -- ++(0,\dd) -- ++(\rr,0);} % tikzmath
			\\
			&\cong \tikzmath{
			\node (a) at (0,0) {$H_0(\hat S_l) \underset{\cala(K_2)}\boxtimes H_0(\tilde S_r) \underset{\cala(K_3)}\boxtimes$};
			\def\dd{.5}\def\ll{.35}\def\rr{.25}
			\draw[dashed, rounded corners = 6] (a.east)++(0,.1) -- ++(\rr,0) -- ++(0,-\dd) -- ($(a.west) + (0,.1) + (-\ll,-\dd)$) -- ++(0,\dd) -- ++(\rr,0);} % tikzmath
			\\
			&\cong \bigoplus_{\lambda\in\Delta}^{\phantom .} H_\lambda(-S_m) \otimes H_{\bar \lambda}(S_b).
\end{split}
\]
We draw the above isomorphisms as follows:
\begin{equation*}
G_0\big(H_0(S_l),H_0(S_r)\big) \,= \,  \tikzmath[scale=\displscale]
  {\filldraw[fill=vacuumcolor, draw=black]  (10,0) rectangle (14,4) (10,8) rectangle (14,12) (0,0) rectangle (10,12) (14,0) rectangle (24,12);
  } %tikzmath 
  \; \xrightarrow{\cong} \;
  \tikzmath[scale=\displscale]
  {
  \filldraw[fill=vacuumcolor, draw=black] (0,12) -- (0,0) -- (14,0) -- (14,4) -- (10,4) -- (10,8) -- (14,8) -- (14,12) -- (6,12) -- cycle;  
  \filldraw[fill=vacuumcolor, draw=black]  (14,0) rectangle (24,12);
  } %tikzmath 
  \; \xrightarrow{\cong} \;
  \bigoplus_{\lambda\in\Delta}\, \;
  \tikzmath[scale=\displscale]
  { \useasboundingbox (8,0) rectangle (16,12);
    \filldraw[fill=spacecolor, draw=black]  (10,4) rectangle (14,8) (12,0) node{\tiny $\lambda$};
  } %tikzmath
  \otimes\,
  \tikzmath[scale=\displscale]
  {\filldraw[fill=spacecolor, draw=black]  (0,0) rectangle (24,12) (12,6) node{\tiny $\bar\lambda$};
  }\,. %tikzmath 
\end{equation*}
Note that the two isomorphisms intertwine the natural actions of $\{\cala(I)\}_{I\subset S_b}$ and $\{\cala(I)\}_{I\subset -S_m}$.
We can now compute
\begin{equation}\label{eq: computation of square with hole}
\begin{split}
  G\big(H_0(S_l),H_0(S_r)\big)\, &:= G_0\big(H_0(S_l),H_0(S_r)\big) \boxtimes_{\cala(S_m)} H_0(S_m)  \\
		&\phantom{:}\cong \Big(\bigoplus_{\lambda\in\Delta} H_\lambda(-S_m) \otimes H_{\bar \lambda}(S_b)\Big) \boxtimes_{\cala(S_m)} H_0(S_m) \\
		&\phantom{:}\cong \Big(H_0(-S_m) \otimes H_0(S_b)\Big) \boxtimes_{\bfB(H_0(S_m))} H_0(S_m) \\
		&\phantom{:}\cong H_0(S_b)\otimes \Big(H_0(-S_m) \boxtimes_{\bfB(H_0(S_m))} H_0(S_m)\Big) \\
		&\phantom{:}\cong H_0(S_b)\otimes \IC \,\,\cong\,\, H_0(S_b)^{\phantom{\big |}}\!\!.\\
\end{split}
\end{equation}
Combining \eqref{eq: computation of square with hole} with 
the non-canonical isomorphism $F(H_0(S_l),H_0(S_r))\cong H_0(S_b)$ 
from~\eqref{eq:non-canonical-upsilon},
we get an isomorphism
\begin{equation*}%\label{eq: G cong F}
\varphi\,:\,G(H_0(S_l),H_0(S_r))\,\xrightarrow{\scriptscriptstyle\cong}\, F(H_0(S_l),H_0(S_r)),
\end{equation*}
compatible with the actions of $\cala(I_l')$ and of $\cala(I_r')$.
%In particular, it is also an isomorphism of $\cala(J)$-sectors.

Since $H_0(S_l)$ and $H_0(S_r)$ are faithful $\cala(I_l)$- and $\cala(I_r)$-modules, we can
use Lemma \ref{lem: Mod(A1) x Mod(A2) --> Mod(B)} to finish the argument:
it remains only to check that $\varphi$ is equivariant with respect to all 
$r_1\in \mathrm{End}_{\cala(I_l)}(H_0(S_l))$ and $r_2\in \mathrm{End}_{\cala(I_r)}(H_0(S_r))$.
That equivariance follows immediately from 
Haag duality for nets (Proposition~\ref{prop: [Haag duality for defects]-duality-nets}) and the fact that 
$\varphi$ is equivariant with respect to $\cala(I_l')$ and $\cala(I_r')$.
\end{proof}

Unfortunately, the above proposition is not sufficient for our purposes: it does not construct a natural isomorphism $\Phi_\cala:F_\cala \to G_\cala$, but only proves that one exists.  This leaves unsettled, for instance, the question of whether these natural isomorphisms can be chosen so that $\Phi_{\cala\otimes \calb} = \Phi_\cala \otimes \Phi_\calb$.
In the following sections, we will construct a canonical choice of such natural isomorphisms for which the desired symmetric monoidal property is clear.

\section{The keyhole fusion of vacuum sectors of defects}\label{sec: The computation of G_0(L^2(D),L^2(E))}

Let $S_l$, $S_r$, $S_b$, $\tilde S_l$, $\tilde S_r$, $S_u$, $S_m$, $S_d$, $I_l$, $I_r$, $\tilde I_l$, $\tilde I_r$, $K$ be as in 
\eqref{eq: figures of intervals 1}, \eqref{eq: figures of intervals 2}, and \eqref{eq: figures of intervals 3}.
We bicolor $S_l$, $S_r$, $\tilde S_l$, $\tilde S_r$ by setting
\begin{equation}\label{eq: def colors of S and tildeS}
\begin{split}
(S_l)_\circ := (S_l)_{x\le \frac12}\quad
(S_l)_\bullet := (S_l)_{x\ge \frac12}&\quad
(S_r)_\circ := (S_r)_{x\le \frac32}\quad
(S_r)_\bullet := (S_r)_{x\ge \frac32}\\
(\tilde S_l)_\circ := (\tilde S_l)_{x\le \frac12}\quad
(\tilde S_l)_\bullet := (\tilde S_l)_{x\ge \frac12}&\quad
(\tilde S_r)_\circ := (\tilde S_r)_{x\le \frac32}\quad
(\tilde S_r)_\bullet := (\tilde S_r)_{x\ge \frac32}.
\end{split}
\end{equation}
Denote by $j$ the reflection across the horizontal axis $y=1/2$, 
and  let 
\[
\begin{split}
S_{l,\top}&\!:=\!(S_l)_{y\ge \frac12}\qquad S_{r,\top}\!:=\!(S_r)_{y\ge \frac12}\qquad S_{b,\top}\!:=(S_b)_{y\ge \frac12}\\
\tilde S_{l,\top}\!:=\!(\tilde S_l&)_{y\ge \frac12}\quad\tilde S_{r,\top}\!:=\!(\tilde S_r)_{y\ge \frac12}\quad \tilde I_{l,\top}\!:=\!(\tilde I_l)_{y\ge \frac12}\quad \tilde I_{r,\top}\!:=\!(\tilde I_r)_{y\ge \frac12}.
\end{split}
\]
Let $\cala$, $\calb$, $\calc$ be conformal nets, and let $_\cala D_\calb$ and $_\calb E_\calc$ be defects. 
We are interested in evaluating $G_0:=G_{0,\calb}$ on 
the vacuum sectors
\[
H_0(S_l,D):=L^2(D(S_{l,\top}))\qquad\text{and}\qquad H_0(S_r,E):=L^2(E(S_{r,\top}))
\]
from Definition \ref{def: L2(D)}.  These have compatible actions of the algebras $\{D(I)\}_{I \subset S_l}$ and $\{E(I)\}_{I \subset S_r}$ respectively.  In particular, they are respectively $\calb(I_l)$- and 
$\calb(I_r)$-modules, and so we can apply the functor $G_0$.

Let us also define
\[
H_0(\tilde S_l,D):=L^2(D(\tilde S_{l,\top}))\qquad\text{and}\qquad H_0(\tilde S_r,E):=L^2(E(\tilde S_{r,\top})).
\]
Recall that the definition of the functor $G_0$ uses identifications $\phi_l:\tilde I_l\xrightarrow{\scriptscriptstyle\cong} I_l$ and $\phi_r:\tilde I_r\xrightarrow{\scriptscriptstyle\cong} I_r$
to endow $H_0(S_l,D)$ with a $\calb(\tilde I_l)$-action,
and $H_0(S_r,D)$ with a $\calb(\tilde I_r)$-action.
We write $\phi_l^*H_0(S_l,D)$ and $\phi_r^*H_0(S_l,D)$ for the resulting 
$\tilde S_l$-sector of $D$ and $\tilde S_r$-sector of $E$.
(Here, given a bicolored circle $S$ and a defect $D$, a Hilbert space is called an $S$-sector of $D$ if it has compatible actions of $D(I)$ for all bicolored intervals $I\subset S$---compare Appendix \ref{subsec:sectors+KLM}.)
Recall that the maps $\phi_l$ and $\phi_r$ were chosen to commute with $j$, and to be the identity near to the boundary.
Let us call
\[
\phi_{l,\top}:\tilde S_{l,\top}\xrightarrow\cong S_{l,\top}\,,\qquad \phi_{r,\top}:\tilde S_{r,\top}\xrightarrow\cong S_{r,\top}
\]
the extension by the identity of the maps $\phi_{l}|_{\tilde I_{l,\top}}$ and $\phi_{r}|_{\tilde I_{r,\top}}$.
We then have canonical identifications
\[
\begin{split}
L^2(D(\phi_{l\hspace{.03cm},\top}))\,:\, H_0(\tilde S_l\hspace{.03cm},D) \to \phi_l^*\hspace{.03cm}H_0(S_l\hspace{.03cm},D) \\
L^2(D(\phi_{r,\top}))\,:\, H_0(\tilde S_r,D) \to \phi_r^*H_0(S_r,D)
\end{split}
\]
of $\tilde S_l$-sectors of $D$ and $\tilde S_r$-sectors of $E$.

We now have an isomorphism
\[
\begin{split}
G_{0,\calb}&\big(H_0(S_l,D),H_0(S_r,E)\big) \\
&\quad\cong		\, \big( H_0(\tilde S_l,D) \ox H_0(\tilde S_r,E) \big) \underset{\calb(K)}\boxtimes  \big( H_0(S_u,\calb) \ox H_0(S_d,\calb) \big)\\
&\quad =			\tikzmath{
			\node (a) at (0,0) {$H_0(\tilde S_l,D) \underset{\calb(K_1)}\boxtimes H_0(S_u,\calb) \underset{\calb(K_2)}\boxtimes H_0(\tilde S_r,E) \underset{\calb(K_3)}\boxtimes H_0(S_d,\calb) \underset{\calb(K_4)}\boxtimes$};
			\def\dd{.5}\def\ll{.35}\def\rr{.25}
			\draw[dashed, rounded corners = 6] (a.east)++(0,.1) -- ++(\rr,0) -- ++(0,-\dd) -- ($(a.west) + (0,.1) + (-\ll,-\dd)$) -- ++(0,\dd) -- ++(\rr,0);}\; % tikzmath
\end{split}
\]
which we draw as follows:
\begin{equation}\label{eq: pictorial notation 1}
G_{0,\calb}\big(H_0(S_l,D),H_0(S_r,E)\big) \,\,\,\cong\,\,\,
\tikzmath[scale=\displscale]
      {    \fill[vacuumcolor] (0,0) rectangle (10,12)
                             (14,0) rectangle (24,12); 
           \draw (6,12) -- (10,12) -- (10,0) -- (6,0) 
                 (18,12) -- (14,12) -- (14,0) -- (18,0);
           \draw[thick, double] (6,12) -- (0,12) -- (0,0) -- (6,0);
           \draw[ultra thick] (18,12) -- (24,12) -- (24,0) -- (18,0);
           \fill[vacuumcolor]  (10,0) rectangle (14,4)  
                              (10,8) rectangle (14,12);
           \draw (10,0) rectangle (14,4) (10,8) rectangle (14,12);
      }\,\,\,. %tikzmath
\end{equation}
Here, the lines \tikz{\draw[thick, double] (0,0) -- (.5,0);}, \tikz{\draw (0,0) -- (.5,0);}, and \tikz{\draw[ultra thick] (0,0) -- (.5,0);} correspond to the conformal nets 
$\cala$, $\calb$, and $\calc$, and the transition points \tikz{\draw[thick, double] (0,0) -- (.3,0);\draw (.3,0) -- (.6,0);}, and 
\tikz{\draw (0,0) -- (.3,0);\draw[ultra thick] (.3,0) -- (.6,0);} indicate the defects $D$ and $E$.

\subsection*{\hspace*{-18pt}Keyhole fusion as an $L^2$-space}

We need to introduce yet more manifolds.
We have already encountered $K_1=\tilde S_l\cap S_u$ and $K_2= \tilde S_r\cap S_u$.
We define $K_u:=K_1\cup K_2$ and 
$J_u:=J_1\cup J_2$, where $J_1:=S_b\cap S_u$ and $J_2:=S_u\cap S_m$.
We orient $K_u$ and $J_u$ compatibly with $S_u$.
Let $J_l$ be the closure of $\tilde S_{l,\top}\setminus K_1$ and,
similarly, let $J_r$ be the closure of $\tilde S_{r,\top}\setminus K_2$.
The orientations and the bicolorings of $J_l$ and $J_r$ are inherited from $\tilde S_l$ and $\tilde S_r$.
We include pictures of these manifolds:
\begin{equation}\label{eq:  J_l  J_r  K_u  J_u} 
  J_l   =  \tikzmath[scale=\displscale]
               { \useasboundingbox (-2,0) rectangle (26,14);
                 \draw[ultra thick] (6,12) -- (10,12) (10,8) -- (10,6);
                 \draw %[thick, double] 
                    (0,6) -- (0,12) -- (6,12);
                 \draw[->] (2.6,12) -- (2.5,12);
                 \draw[->, thick] (10,7.9) -- (10,8);
}   %tikzmath
        \;,\,\,\,\,
  J_r   =  \tikzmath[scale=\displscale]
               { \useasboundingbox (-2,0) rectangle (26,14);
                 \draw   (14,6) -- (14,8)  (14,12) -- (18,12);
                 \draw[ultra thick] %[ultra thick]  
                 (18,12) -- (24,12) -- (24,6);
                 \draw[->, thick] (20.1,12) -- (20,12);
                 \draw[->] (14,6.6) -- (14,6.5);
               }   %tikzmath
        \;,\,\,\,\,
  K_u   = \tikzmath[scale=\displscale]
         { \useasboundingbox (-2,0) rectangle (26,12);
         \draw (14,12) -- (14,8)  (10,12) -- (10,8);
         \draw[->] (10,9.41) -- (10,9.4);
         \draw[->] (14,10.59) -- (14,10.6);
         }    %tikzmath
        \;,\,\,\,\,
  J_u   = \tikzmath[scale=\displscale]
         {\useasboundingbox (-2,0) rectangle (26,12);
         \draw (14,12) -- (10,12)  (14,8) -- (10,8);
         \draw[->] (12.59,8) -- (12.6,8);
         \draw[->] (11.41,12) -- (11.4,12);
         }\;. %tikzmath 
\end{equation}

Following Notation \ref{not: Dhat}, we let $\hat\calb(J_u)$ denote the commutant of $\calb(K_u)$ on $H_0(S_u,\calb)$.
Our computation of the keyhole fusion will be in terms of the algebra
\[
D(J_l) \vee \hat\calb(J_u) \vee E(J_r) \subset \bfB\big(H_0(\tilde S_l,D) \underset{\calb(K_1)}\boxtimes H_0(S_u,\calb) \underset{\calb(K_2)}\boxtimes H_0(\tilde S_r,E)\big),
\]
which we denote pictorially by
\begin{equation}\label{eq: pictorial notation 2}
D(J_l) \vee \hat\calb(J_u) \vee E(J_r) \,\,=\,\, \tikzmath[scale=\displscale]
               { \useasboundingbox (-2,0) rectangle (26,14);
                 \draw[thick, double] (0,6) -- (0,12) -- (6,12);
                 \draw (6,12) -- (18,12) (10,6) -- (10,8) -- (14,8) -- (14,6);
                 \draw[ultra thick] (18,12) -- (24,12) -- (24,6); 
                 \draw[densely dotted] (12,8) -- (12,12); 
               }\;. %tikzmath
\end{equation}
The dotted line in this picture picture serves to remind us that $\hat\calb(J_u)$ was used instead of $\calb(J_u)$.
Note that the algebra \eqref{eq: pictorial notation 2} also acts on $G_{0,\calb}(H_0(S_l,D),H_0(S_r,E))$ because that fusion is obtained from
\[
\tikzmath[scale=\displscale]
      {    \useasboundingbox (-2,-1) rectangle (26,14);
           \fill[vacuumcolor] (0,0) rectangle (10,12) (14,0) rectangle (24,12); 
           \draw[ultra thin] (6,12) -- (10,12) -- (10,0) -- (6,0) (18,12) -- (14,12) -- (14,0) -- (18,0);
           \draw[thick, double] (6,12) -- (0,12) -- (0,0) -- (6,0);
           \draw[ultra thick] (18,12) -- (24,12) -- (24,0) -- (18,0);
           \filldraw[fill = vacuumcolor, ultra thin] (10,8) rectangle (14,12);
      }  %tikzmath
=H_0(\tilde S_l,D) \boxtimes_{\calb(K_1)} H_0(S_u,\calb) \boxtimes_{\calb(K_2)} H_0(\tilde S_r,E)
\] 
by fusing it over $\calb(K_3\cup K_4)$ with $H_0(S_d,\calb)$.

Let
\[
\begin{split}
S_{b,\,\tikzmath[scale=.15]{\draw(0,0)--(0,1)--(1,1);}}&:=\textstyle\{(x,y)\in S_b\,|\,x-y<\frac12\}\\
S_{b,\,\tikzmath[scale=.15]{\draw(1,0)--(0,0)--(0,1);}}&:=\textstyle\{(x,y)\in S_b\,|\,x+y<\frac32\}\\
S_{b,\,\tikzmath[scale=.15]{\draw(0,1)--(1,1)--(1,0);}}&:=\textstyle\{(x,y)\in S_b\,|\,x+y>\frac32\}\\
S_{b,\,\tikzmath[scale=.15]{\draw(0,0)--(1,0)--(1,1);}}&:=\textstyle\{(x,y)\in S_b\,|\,x-y>\frac12\},
\end{split}
\]
with orientations and bicolorings as in the following pictures
\[
S_{b,\,\tikzmath[scale=.15]{\draw(0,0)--(0,1)--(1,1);}} = 
\tikzmath[scale=\displscale]{\useasboundingbox (-2,0) rectangle (26,12);\draw(6,0)--(0,0)--(0,12)--(6,12);\draw[ultra thick](6,12)--(18,12);\draw[->](0,5.6) -- (0,5.5);}\;,\;\;
S_{b,\,\tikzmath[scale=.15]{\draw(1,0)--(0,0)--(0,1);}} = 
\tikzmath[scale=\displscale]{\useasboundingbox (-2,0) rectangle (26,12);\draw(6,0)--(0,0)--(0,12)--(6,12);\draw[ultra thick](6,0)--(18,0);\draw[->](0,5.6) -- (0,5.5);}\;,\;\;
S_{b,\,\tikzmath[scale=.15]{\draw(0,1)--(1,1)--(1,0);}} = 
\tikzmath[scale=\displscale]{\useasboundingbox (-2,0) rectangle (26,12);\draw[ultra thick](18,0)--(24,0)--(24,12)--(18,12);\draw(6,12)--(18,12);\draw[->, thick](24,6.6) -- (24,6.7);}\;,\;\;
S_{b,\,\tikzmath[scale=.15]{\draw(0,0)--(1,0)--(1,1);}} = 
\tikzmath[scale=\displscale]{\useasboundingbox (-2,0) rectangle (26,12);\draw[ultra thick](18,0)--(24,0)--(24,12)--(18,12);\draw(6,0)--(18,0);\draw[->, thick](24,6.6) -- (24,6.7);}\;.
\]
Note that these manifolds do not include their boundary points.

\begin{theorem} \label{thm:G_0=L^2}
Let $\cala$, $\calb$, $\calc$ be conformal nets, and let $_\cala D_\calb$ and $_\calb E_\calc$ be defects.
Then there is a canonical unitary isomorphism
\begin{equation}\label{eq: the map Psi_0} 
      \Psi_0 \,\,\colon\,\,
    L^2 \left( \tikzmath[scale=\displscale]
               { \useasboundingbox (-2,0) rectangle (26,14);
                 \draw[thick, double] (0,6) -- (0,12) -- (6,12);
                 \draw (6,12) -- (18,12) (10,6) -- (10,8) -- (14,8) -- (14,6);
                 \draw[ultra thick] (18,12) -- (24,12) -- (24,6); 
                 \draw[densely dotted] (12,8) -- (12,12); 
               } %tikzmath
        \right)
    \; \xrightarrow{\cong} \;
      \tikzmath[scale=\displscale]
      {    \fill[vacuumcolor] (0,0) rectangle (10,12)
                             (14,0) rectangle (24,12); 
           \draw (6,12) -- (10,12) -- (10,0) -- (6,0) 
                 (18,12) -- (14,12) -- (14,0) -- (18,0);
           \draw[thick, double] (6,12) -- (0,12) -- (0,0) -- (6,0);
           \draw[ultra thick] (18,12) -- (24,12) -- (24,0) -- (18,0);
           \fill[vacuumcolor]  (10,0) rectangle (14,4)  
                              (10,8) rectangle (14,12);
           \draw (10,0) rectangle (14,4) (10,8) rectangle (14,12);
      } %tikzmath
      \,.
  \end{equation}
In formulas, this is a map
\[
\Psi_0 = (\Psi_0)_{D,E} \colon
L^2\big(D(J_l) \vee \hat\calb(J_u) \vee E(J_r)\big)
\xrightarrow{\cong} G_{0,\calb}\big(H_0(S_l,D),H_0(S_r,E)\big),
\]
where $S_l$, $S_r$, $J_l$, $J_u$, $J_r$ are as in \eqref{eq: def colors of S and tildeS} and \eqref{eq:  J_l  J_r  K_u  J_u}.
The map $\Psi_0$ is
equivariant for the natural left actions of
$\{D(I)\}_{I\subset S_{b,\tikzmath[scale=.1]{\draw(0,0)--(0,1)--(1,1);}}}$,
$\{D(I)\}_{I\subset S_{b,\tikzmath[scale=.1]{\draw(1,0)--(0,0)--(0,1);}}}$, 
$\{E(I)\}_{I\subset S_{b,\tikzmath[scale=.1]{\draw(0,1)--(1,1)--(1,0);}}}$, and
$\{E(I)\}_{I\subset S_{b,\tikzmath[scale=.1]{\draw(0,0)--(1,0)--(1,1);}}}$,
and for the natural right actions of
$\{\calb(I)\}_{I\subset S_m}$.
\end{theorem}

\begin{proof}
%Since the Haagerup $L^2$-space is uniquely determined up to unique unitary isomorphism (see Appendix~\ref{subsec:Haagerup-L^2}),
This is the special case of Proposition~\ref{prop:hat-M} where
  $M:=D(\tilde S_{l,\top})\,\barox\,E(\tilde S_{r,\top})$, $M_0=D(J_l)\,\barox\,E(J_r)$, $A=\calb(K_u)^\op$, and $H=H_0(S_u,\calb)$.
  In pictures, these are\vspace{.2cm}
  \begin{equation}\label{eq: M, M0, A, H}
    M  :=  \tikzmath[scale=\displscale]
             { \useasboundingbox (-2,0) rectangle (26,12);
               \draw (6,12) -- (10,12) -- (10,6) (14,6) -- (14,12) -- (18,12);
               \draw[thick, double] (0,6) -- (0,12) -- (6,12);
               \draw[ultra thick]  (18,12) -- (24,12) -- (24,6);
             } %tikzmath
    , \quad
    M_0  :=  \tikzmath[scale=\displscale]
               { \useasboundingbox (-2,0) rectangle (26,12);
                 \draw (6,12) -- (10,12) (10,8) -- (10,6) 
                       (14,6) -- (14,8)  (14,12) -- (18,12);
                 \draw[thick, double] (0,6) -- (0,12) -- (6,12);
                 \draw[ultra thick]  (18,12) -- (24,12) -- (24,6);
               } %tikzmath
    , \quad
    A  :=  \tikzmath[scale=\displscale]
             { \useasboundingbox (2,0) rectangle (22,12);
               \draw  (10,12) -- (10,8) (14,8) -- (14,12);
             } %tikzmath
    ,\;\; \text{and} \quad
    H  :=  \tikzmath[scale=\displscale]
             { \useasboundingbox (2,0) rectangle (22,12);
               \fill[vacuumcolor] (10,8) rectangle (14,12);
               \draw   (10,8) rectangle (14,12);
             }%tikzmath
    .
  \end{equation}
The equivariance of $\Psi_0$ is clear for intervals $I$ that are contained in the upper half $\{(x,y)|y\ge \frac12\}$ or in the lower half $\{(x,y)|y\le \frac12\}$,
and follows by strong additivity for more general intervals.
\end{proof}

\subsection*{\hspace*{-18pt}Associativity of the $L^2$-space identification}

The isomorphism $\Psi_0$ is in an appropriate sense associative, as follows.  Suppose that we have three defects ${}_\cala D_\calb$, ${}_\calb E_\calc$, and ${}_\calc F_\cald$.
We then have various applications of $\Psi_0$ forming the square
  \def\SingleLtwoWithTwoColors#1#2{
	L^2 \hspace{-.5ex}
	\left( \tikzmath[scale=\textscale] {\draw[#1] (12,6) -- (12,12) -- (18,12);\draw[#2] (18,12) -- (24,12) -- (24,6);} \right)}
  \def\MoreComplicatedLtwoWithThreeColors#1#2#3{
	L^2 \hspace{-.5ex}
	\left(\tikzmath[scale=\textscale]
	{\draw[#1] (0,6) -- (0,12) -- (6,12);
	\draw[#2] (6,12) -- (18,12) (10,6) -- (10,8) -- (14,8) -- (14,6);
	\draw[dash pattern=on .4pt off .62pt] (12,8) -- (12,12);
	\draw[#3] (18,12) -- (24,12) -- (24,6);} \right)}
  \def\FillLeftUpperAndLowerSquares{
	\fill[vacuumcolor] (10,0) rectangle (14,4)  (10,8) rectangle (14,12);
	\draw[ultra thin] (8,0) -- (10,0) -- (10,12) -- (8,12) (16,0) -- (14,0) -- (14,12) -- (16,12) (10,0) rectangle (14,4) (10,8) rectangle (14,12);}
  \def\FillRightUpperAndLowerSquares{
	\fill[vacuumcolor](38,0) rectangle (42,4)  (38,8) rectangle (42,12);
	\draw[thick] (36,0) -- (38,0) -- (38,12) -- (36,12) (44,0) -- (42,0) -- (42,12) -- (44,12) (38,0) rectangle (42,4)  (38,8) rectangle (42,12);}
\begin{equation}\label{eq: Assoc of Psi_0}
  \begin{matrix}\xymatrix{
      %TOP LEFT 
      \tikzmath[scale=\displscale]
      {\draw[ultra thin, dash pattern=on .5pt off 1pt](-12,0) rectangle (66,12);
      	\node at (27,6) {$L^2 \hspace{-.5ex} \left(\tikzmath[scale=\textscale]
	{\draw[thick, double](0,6) -- (0,12) -- (6,12);
	\draw[ultra thin] (6,12) -- (18,12)(10,6) -- (10,8) -- (14,8) -- (14,6);
	\draw[thick] (18,12) -- (30,12)(22,6) -- (22,8) -- (26,8) -- (26,6);
	\draw[ultra thick](30,12) -- (36,12) -- (36,6); 
	\draw[dash pattern=on .4pt off .62pt] (12,8) -- (12,12)(24,8) -- (24,12);}
	\right)$};}  %tikzmath
      \ar[r]^{\cong} \ar[d]^{\cong}   &
      %TOP RIGHT
      \tikzmath[scale=\displscale]
      {\FillRightUpperAndLowerSquares
        \node at    (12,6)	{$\MoreComplicatedLtwoWithThreeColors{thick, double}{ultra thin}{thick}$};
        \node at (54.5,6)	{$\SingleLtwoWithTwoColors{thick}{ultra thick}$}; 
        \draw[ultra thin, dash pattern=on .5pt off 1pt]
                (36,0) -- (-12,0) -- (-12,12) -- (36,12)
                (44,12) -- (66,12) -- (66,0) -- (44,0);}  %tikzmath
      \ar[d]^{\cong}   \\
      %BOTTOM LEFT
      \tikzmath[scale=\displscale]
      {\FillLeftUpperAndLowerSquares 
        \node at     (-1,6)	{$\SingleLtwoWithTwoColors{thick, double}{ultra thin}$}; 
        \node at    (40,6)	{$\MoreComplicatedLtwoWithThreeColors{ultra thin}{thick}{ultra thick}$};
        \draw[ultra thin, dash pattern=on .5pt off 1pt]
                (10,0) -- (-12,0) -- (-12,12) --(10,12)
                (14,12) -- (66,12) -- (66,0) -- (14,0);}  %tikzmath
      \ar[r]^{\cong}   &
      %BOTTOM RIGHT 
      \tikzmath[scale=\displscale]
      {\FillLeftUpperAndLowerSquares\FillRightUpperAndLowerSquares
        \node at     (-1,6)	{$\SingleLtwoWithTwoColors{thick, double}{ultra thin}$}; 
        \node at    (26,6)	{$\SingleLtwoWithTwoColors{ultra thin}{thick}$}; 
        \node at (54.5,6)	{$\SingleLtwoWithTwoColors{thick}{ultra thick}$}; 
        \draw[ultra thin, dash pattern=on .5pt off 1pt]
                (14,12) -- (36,12) (14,0) -- (36,0)
                (10,0) -- (-12,0) -- (-12,12) --(10,12)
                (44,12) -- (66,12) -- (66,0) -- (44,0);} %tikzmath
    }\put(0,-50){,} %xymatrix 
  \end{matrix}\end{equation}
This diagram commutes by Proposition \ref{prop: associativity of Psi (Appendix)}.

We explain the meaning of the pictures in this square.
The conformal nets $\cala$, $\calb$, $\calc$, $\cald$ are
indicated by lines of various thickness \tikz{\draw[thick, double] (0,0) -- (.5,0);}, \tikz{\draw[ultra thin] (0,0) -- (.5,0);}, \tikz{\draw[thick] (0,0) -- (.5,0);}, \tikz{\draw[ultra thick] (0,0) -- (.5,0);},
and the defects ${}_\cala D_\calb$, ${}_\calb E_\calc$, ${}_\calc F_\cald$ are indicated by the transitions
\tikz{\draw[thick, double] (0,0) -- (.3,0);\draw[ultra thin] (.3,0) -- (.6,0);},
\tikz{\draw[ultra thin] (0,0) -- (.3,0);\draw[thick] (.3,0) -- (.6,0);}, and
\tikz{\draw[thick] (0,0) -- (.3,0);\draw[ultra thick] (.3,0) -- (.6,0);}.
The notations 
$\SingleLtwoWithTwoColors{thick, double}{ultra thin}$, 
$\SingleLtwoWithTwoColors{ultra thin}{thick}$, 
$\SingleLtwoWithTwoColors{thick}{ultra thick}$, 
stand for $L^2(D(S^1_\top))$, $L^2(E(S^1_\top))$, $L^2(F(S^1_\top))$.
The lower right corner of \eqref{eq: Assoc of Psi_0} is
\[
\begin{split}
\tikzmath[scale=\displscale]
      {\FillLeftUpperAndLowerSquares\FillRightUpperAndLowerSquares
        \node at     (-1,6)	{$\SingleLtwoWithTwoColors{thick, double}{ultra thin}$}; 
        \node at    (26,6)	{$\SingleLtwoWithTwoColors{ultra thin}{thick}$}; 
        \node at (54.5,6)	{$\SingleLtwoWithTwoColors{thick}{ultra thick}$}; 
        \draw[ultra thin, dash pattern=on .5pt off 1pt]
                (14,12) -- (36,12) (14,0) -- (36,0)
                (10,0) -- (-12,0) -- (-12,12) --(10,12)
                (44,12) -- (66,12) -- (66,0) -- (44,0);}  %tikzmath
&= G_{0,\calb}\big(L^2(D(S^1_\top)),G_{0,\calc}\big(L^2(E(S^1_\top)),L^2(F(S^1_\top))\big)\big)\\
&= G_{0,\calc}\big(G_{0,\calb}\big(L^2(D(S^1_\top)),L^2(E(S^1_\top))\big),L^2(F(S^1_\top))\big).
\end{split}
\]
Note that, following \eqref{eq: pictorial notation 1}, this Hilbert space is also denoted
$\tikzmath[scale=\textscale]
      {    \fill[vacuumcolor] (0,0) rectangle (10,12)
                             (14,0) rectangle (22,12)
                             (26,0) rectangle (36,12); 
           \draw[ultra thin] (6,12) -- (10,12) -- (10,0) -- (6,0) 
                 (18,12) -- (14,12) -- (14,0) -- (18,0);
           \draw[thick] (18,12) -- (22,12) -- (22,0) -- (18,0) 
                 (30,12) -- (26,12) -- (26,0) -- (30,0);
           \draw[thick, double] (6,12) -- (0,12) -- (0,0) -- (6,0);
           \draw[ultra thick] (30,12) -- (36,12) -- (36,0) -- (30,0);
           \filldraw[fill = vacuumcolor, ultra thin]  (10,0) rectangle (14,4) (10,8) rectangle (14,12);
           \filldraw[fill = vacuumcolor, thick] (22,0) rectangle (26,4) (22,8) rectangle (26,12);
      }$\,.

As in \eqref{eq: pictorial notation 2},
$\MoreComplicatedLtwoWithThreeColors{thick, double}{ultra thin}{thick}$ and 
$\MoreComplicatedLtwoWithThreeColors{ultra thin}{thick}{ultra thick}$
denote the Hilbert spaces $L^2(D(J_l) \vee \hat\calb(J_u) \vee E(J_r))$
and $L^2(E(J_l) \vee \hat\calc(J_u) \vee F(J_r))$, respectively.
The upper right and lower left corners of \eqref{eq: Assoc of Psi_0} are therefore given by
\[
\tikzmath[scale=\displscale]{\useasboundingbox(-12,0) rectangle (66,12);
       \FillRightUpperAndLowerSquares
        \node at    (12,6)	{$\MoreComplicatedLtwoWithThreeColors{thick, double}{ultra thin}{thick}$};
        \node at (54.5,6)	{$\SingleLtwoWithTwoColors{thick}{ultra thick}$}; 
        \draw[ultra thin, dash pattern=on .5pt off 1pt]
                (36,0) -- (-12,0) -- (-12,12) -- (36,12)
                (44,12) -- (66,12) -- (66,0) -- (44,0);}  %tikzmath
= G_{0,\calc}\big(L^2(D(J_l) \vee \hat\calb(J_u) \vee E(J_r)),L^2(F(S^1_\top))\big)\]and\[
      \tikzmath[scale=\displscale]{\useasboundingbox(-12,0) rectangle (66,12);
       \FillLeftUpperAndLowerSquares 
        \node at     (-1,6)	{$\SingleLtwoWithTwoColors{thick, double}{ultra thin}$}; 
        \node at    (40,6)	{$\MoreComplicatedLtwoWithThreeColors{ultra thin}{thick}{ultra thick}$};
        \draw[ultra thin, dash pattern=on .5pt off 1pt]
                (10,0) -- (-12,0) -- (-12,12) --(10,12)
                (14,12) -- (66,12) -- (66,0) -- (14,0);}  %tikzmath
= G_{0,\calb}\big(L^2(D(S^1_\top)),L^2(E(J_l) \vee \hat\calc(J_u) \vee F(J_r))\big).
\]
Finally, the vector space 
$L^2 \hspace{-.5ex} \left(\tikzmath[scale=\textscale]
	{\draw[thick, double](0,6) -- (0,12) -- (6,12);
	\draw[ultra thin] (6,12) -- (18,12)(10,6) -- (10,8) -- (14,8) -- (14,6);
	\draw[thick] (18,12) -- (30,12)(22,6) -- (22,8) -- (26,8) -- (26,6);
	\draw[ultra thick](30,12) -- (36,12) -- (36,6); 
	\draw[dash pattern=on .4pt off .62pt] (12,8) -- (12,12)(24,8) -- (24,12);}\right)$
that appears in the upper left corner of \eqref{eq: Assoc of Psi_0}
is the $L^2$ space of the von Neumann algebra
\[
\begin{split}
\tikzmath[scale=\displscale]{\useasboundingbox (-2,1) rectangle (38,14);
	\draw[thick, double](0,6) -- (0,12) -- (6,12);
	\draw[ultra thin] (6,12) -- (18,12)(10,6) -- (10,8) -- (14,8) -- (14,6);
	\draw[thick] (18,12) -- (30,12)(22,6) -- (22,8) -- (26,8) -- (26,6);
	\draw[ultra thick](30,12) -- (36,12) -- (36,6); 
	\draw[densely dotted] (12,8) -- (12,12)(24,8) -- (24,12);}\;\;:=\;\;
D(\tikzmath[scale=\displscale]
	{\useasboundingbox (-2,1) rectangle (38,14);
	\draw[ultra thin] (6,12) -- (10,12) (10,8) -- (10,6);
	\draw[double,thick] (0,6) -- (0,12) -- (6,12);
	\draw[->, thick] (2.6,12) -- (2.5,12);
	\draw[->] (10,7.9) -- (10,8);
	} %tikzmath
)\vee\hat\calb&(\tikzmath[scale=\displscale]
	{\useasboundingbox (-2,1) rectangle (38,14);
	\draw[ultra thin] (14,12) -- (10,12)  (14,8) -- (10,8);
	\draw[->] (12.59,8) -- (12.6,8);
	\draw[->] (11.41,12) -- (11.4,12);
	} %tikzmath
)\vee E(\tikzmath[scale=\displscale]
	{\useasboundingbox (-2,1) rectangle (38,14);
	\draw[ultra thin]   (14,6) -- (14,8)  (14,12) -- (18,12);
	\draw[thick] (18,12) -- (22,12) (22,8) -- (22,6);
	\draw[->] (18.1,12) -- (18,12);
	\draw[->] (22,7.9) -- (22,8);
	\draw[->] (14,6.6) -- (14,6.5);
	} %tikzmath
)\\\vee\,\,\hat\calc&(\tikzmath[scale=\displscale]
	{\useasboundingbox (-2,1) rectangle (38,14);
	\draw[thick] (26,12) -- (22,12)  (26,8) -- (22,8);
	\draw[->] (24.59,8) -- (24.6,8);
	\draw[->] (23.41,12) -- (23.4,12);
	} %tikzmath
)\vee F(\tikzmath[scale=\displscale]
	{\useasboundingbox (-2,1) rectangle (38,14);
	\draw[thick] (26,6) -- (26,8)  (26,12) -- (30,12);
	\draw[ultra thick](30,12) -- (36,12) -- (36,6);
	\draw[->, thick] (32.1,12) -- (32,12);
	\draw[->] (26,6.6) -- (26,6.5);
	} %tikzmath
)\,,\end{split}
\]
where the completion is taken on the Hilbert space
$\tikzmath[scale=\textscale]
      {    \fill[vacuumcolor] (0,0) rectangle (10,12) (14,0) rectangle (22,12) (26,0) rectangle (36,12); 
           \draw[ultra thin] (6,12) -- (10,12) -- (10,0) -- (6,0) (18,12) -- (14,12) -- (14,0) -- (18,0);
           \draw[thick] (18,12) -- (22,12) -- (22,0) -- (18,0) (30,12) -- (26,12) -- (26,0) -- (30,0);
           \draw[thick, double] (6,12) -- (0,12) -- (0,0) -- (6,0);
           \draw[ultra thick] (30,12) -- (36,12) -- (36,0) -- (30,0);
           \filldraw[fill = vacuumcolor, ultra thin] (10,8) rectangle (14,12);
           \filldraw[fill = vacuumcolor, thick] (22,8) rectangle (26,12);
      }$
or, equivalently, on the Hilbert space
$\tikzmath[scale=\textscale]
      {    \fill[vacuumcolor] (0,0) rectangle (10,12) (14,0) rectangle (22,12) (26,0) rectangle (36,12); 
           \draw[ultra thin] (6,12) -- (10,12) -- (10,0) -- (6,0) (18,12) -- (14,12) -- (14,0) -- (18,0);
           \draw[thick] (18,12) -- (22,12) -- (22,0) -- (18,0) (30,12) -- (26,12) -- (26,0) -- (30,0);
           \draw[thick, double] (6,12) -- (0,12) -- (0,0) -- (6,0);
           \draw[ultra thick] (30,12) -- (36,12) -- (36,0) -- (30,0);
           \filldraw[fill = vacuumcolor, ultra thin]  (10,0) rectangle (14,4) (10,8) rectangle (14,12);
           \filldraw[fill = vacuumcolor, thick] (22,0) rectangle (26,4) (22,8) rectangle (26,12);
      }$\,.

%
%-----------------------------------------------------------------------------
%

\section{The keystone fusion of vacuum sectors of defects}

In this section, the defects ${}_\cala D_\calb$ and ${}_\calb E_\calc$ are assumed to be irreducible.
As before, the conformal net $\calb$ is taken to be of finite index.

Recall the algebra
\[
\tikzmath[scale=\displscale]
{ \useasboundingbox (-2,0) rectangle (26,14); \draw[thick, double] (0,6) -- (0,12) -- (6,12); \draw (6,12) -- (18,12) (10,6) -- (10,8) -- (14,8) -- (14,6);
\draw[ultra thick] (18,12) -- (24,12) -- (24,6);  \draw[densely dotted] (12,8) -- (12,12);}\; %tikzmath
:=\;D\Big(\tikzmath[scale=\displscale] {\useasboundingbox (-2,-1) rectangle (26,14);
\draw (6,12) -- (10,12) (10,8) -- (10,6); \draw[double,thick](0,6) -- (0,12) -- (6,12); \draw[->,thick] (2.6,12) -- (2.5,12); \draw[->] (10,7.9) -- (10,8);} %tikzmath
\Big) \vee \hat\calb\Big(\tikzmath[scale=\displscale] {\useasboundingbox (-2,-1) rectangle (26,14);
\draw (14,12) -- (10,12)(14,8) -- (10,8); \draw[->] (12.59,8) -- (12.6,8); \draw[->] (11.41,12) -- (11.4,12);} %tikzmath
\Big) \vee E\Big(\tikzmath[scale=\displscale]
{\useasboundingbox (-2,-1) rectangle (26,14); \draw(14,6) -- (14,8)(14,12) -- (18,12);
\draw[ultra thick](18,12) -- (24,12) -- (24,6); \draw[->, thick] (20.1,12) -- (20,12); \draw[->] (14,6.6) -- (14,6.5);} %tikzmath
\Big) \;\subset\; \bfB\Big(\tikzmath[scale=\displscale]
{\useasboundingbox (-2,-1) rectangle (26,14); \fill[vacuumcolor] (0,0) rectangle (10,12) (14,0) rectangle (24,12); 
\draw[ultra thin] (6,12) -- (10,12) -- (10,0) -- (6,0) (18,12) -- (14,12) -- (14,0) -- (18,0); \draw[thick, double] (6,12) -- (0,12) -- (0,0) -- (6,0);
\draw[ultra thick] (18,12) -- (24,12) -- (24,0) -- (18,0); \filldraw[fill = vacuumcolor, ultra thin] (10,8) rectangle (14,12);}\Big),%tikzmath
\]
from \eqref{eq: pictorial notation 2}. Let us also introduce
\[
\tikzmath[scale=\displscale] { \useasboundingbox (-2,0) rectangle (26,14); \draw[thick, double] (0,6) -- (0,12) -- (6,12); \draw (6,12) -- (18,12) (10,6) -- (10,8) -- (14,8) -- (14,6);
\draw[ultra thick] (18,12) -- (24,12) -- (24,6);  }\; %tikzmath
:=\;D\Big(\tikzmath[scale=\displscale] {\useasboundingbox (-2,-1) rectangle (26,14); \draw (6,12) -- (10,12) (10,8) -- (10,6);
\draw[double,thick] (0,6) -- (0,12) -- (6,12);\draw[->,thick] (2.6,12) -- (2.5,12);\draw[->] (10,7.9) -- (10,8);} %tikzmath
\Big) \vee \calb\Big(\tikzmath[scale=\displscale] {\useasboundingbox (-2,-1) rectangle (26,14);
\draw (14,12) -- (10,12)(14,8) -- (10,8);\draw[->] (12.59,8) -- (12.6,8);\draw[->] (11.41,12) -- (11.4,12);} %tikzmath
\Big) \vee E\Big(\tikzmath[scale=\displscale]{\useasboundingbox (-2,-1) rectangle (26,14);\draw(14,6) -- (14,8)(14,12) -- (18,12);
\draw[ultra thick](18,12) -- (24,12) -- (24,6);\draw[->, thick] (20.1,12) -- (20,12);\draw[->] (14,6.6) -- (14,6.5);} %tikzmath
\Big) \;\subset\; \bfB\Big(\tikzmath[scale=\displscale]
{\useasboundingbox (-2,-1) rectangle (26,14); \fill[vacuumcolor] (0,0) rectangle (10,12) (14,0) rectangle (24,12);  \draw[ultra thin] (6,12) -- (10,12) -- (10,0) -- (6,0) (18,12) -- (14,12) -- (14,0) -- (18,0);
 \draw[thick, double] (6,12) -- (0,12) -- (0,0) -- (6,0); \draw[ultra thick] (18,12) -- (24,12) -- (24,0) -- (18,0); \filldraw[fill = vacuumcolor, ultra thin] (10,8) rectangle (14,12);}\Big).%tikzmath
\]
The algebra $\tikzmath[scale=\textscale]
{\draw[thick, double] (0,6) -- (0,12) -- (6,12); \draw (6,12) -- (18,12) (10,6) -- (10,8) -- (14,8) -- (14,6); \draw[dash pattern=on .4pt off .62pt] (12,8) -- (12,12); 
\draw[ultra thick] (18,12) -- (24,12) -- (24,6);}$ %tikzmath
is a factor, as can be seen by applying Lemma \ref{lem: factoriality of alg with dotted lines} in the situation of \eqref{eq: M, M0, A, H}, but its subalgebra
$\tikzmath[scale=\textscale] {\draw[thick, double] (0,6) -- (0,12) -- (6,12); \draw (6,12) -- (18,12) (10,6) -- (10,8) -- (14,8) -- (14,6); \draw[ultra thick] (18,12) -- (24,12) -- (24,6);}$ %tikzmath
will typically not be a factor.
However, since $\calb$ has finite index, we know by Theorem~\ref{thm: semi-simplicity of DoE} that the subalgebra has finite dimensional center.
%is at least semisimple.

\begin{lemma}
  \label{lem:no-dots-finite-index}
  Let $\calb$ be a conformal net with finite index $\mu(\calb)$, and let
  $_\cala D_\calb$ and $_\calb E_\calc$ be irreducible defects.
  Let $p_1,\ldots, p_n$ be the minimal central projections of the algebra 
  $\tikzmath[scale=\textscale]
	{\draw[thick, double] (0,6) -- (0,12) -- (6,12);
	\draw (6,12) -- (18,12) (10,6) -- (10,8) -- (14,8) -- (14,6);
	\draw[ultra thick] (18,12) -- (24,12) -- (24,6);}$\,.
  Then we have
  \begin{equation*}
    \sum_i \left[ \;
       p_i\, \tikzmath[scale=\displscale]
       {   \useasboundingbox (-2,0) rectangle (26,14);
           \draw[thick, double] (0,6) -- (0,12) -- (6,12);
           \draw (6,12) -- (18,12) (10,6) -- (10,8) -- (14,8) -- (14,6);
           \draw[ultra thick] (18,12) -- (24,12) -- (24,6); 
           \draw[densely dotted] (12,8) -- (12,12); 
       } \,p_i %tikzmath
       \; : \;
       p_i\, \tikzmath[scale=\displscale]
       {  \useasboundingbox (-2,0) rectangle (26,14);
          \draw[thick, double] (0,6) -- (0,12) -- (6,12);
          \draw (6,12) -- (18,12) (10,6) -- (10,8) -- (14,8) -- (14,6);
          \draw[ultra thick] (18,12) -- (24,12) -- (24,6); 
       } %tikzmath
    \; \right] \leq \mu(\calb).   
  \end{equation*}
\end{lemma}

\begin{proof}
  To simplify the notation we abbreviate
  $N:= \calb\big(\tikzmath[scale=\planscale]
	{\useasboundingbox (-2,-1) rectangle (26,14);
	\draw (14,12) -- (10,12)  (14,8) -- (10,8);
	\draw[->] (12.59,8) -- (12.6,8);
	\draw[->] (11.41,12) -- (11.4,12);
	} %tikzmath
	\big)$, 
  $M:= \hat\calb\big(\tikzmath[scale=\planscale]
	{\useasboundingbox (-2,-1) rectangle (26,14);
	\draw (14,12) -- (10,12)  (14,8) -- (10,8);
	\draw[->] (12.59,8) -- (12.6,8);
	\draw[->] (11.41,12) -- (11.4,12);
	} %tikzmath
	\big)$, and
  $A:=D\big(\tikzmath[scale=\planscale]
	{\useasboundingbox (-2,-1) rectangle (26,14);
	\draw[ultra thick] (6,12) -- (10,12) (10,8) -- (10,6);
	\draw(0,6) -- (0,12) -- (6,12);
	\draw[->] (2.6,12) -- (2.5,12);
	\draw[->] (10,7.9) -- (10,8);
	} %tikzmath
	\big) \,\bar\otimes\, E\big(\tikzmath[scale=\planscale]
	{\useasboundingbox (-2,-1) rectangle (26,14);
	\draw(14,6) -- (14,8)  (14,12) -- (18,12);
	\draw[ultra thick](18,12) -- (24,12) -- (24,6);
	\draw[->] (20.1,12) -- (20,12);
	\draw[->] (14,6.6) -- (14,6.5);
	} %tikzmath
	\big)$.
  Then $\tikzmath[scale=\planscale]
       {   \useasboundingbox (-2,0) rectangle (26,14);
           \draw[thick, double] (0,6) -- (0,12) -- (6,12);
           \draw (6,12) -- (18,12) (10,6) -- (10,8) -- (14,8) -- (14,6);
           \draw[ultra thick] (18,12) -- (24,12) -- (24,6); 
           \draw[densely dotted] (12,8) -- (12,12); 
       } = M \vee A$ and
  $\tikzmath[scale=\planscale]
       {  \useasboundingbox (-2,0) rectangle (26,14);
          \draw[thick, double] (0,6) -- (0,12) -- (6,12);
          \draw (6,12) -- (18,12) (10,6) -- (10,8) -- (14,8) -- (14,6);
          \draw[ultra thick] (18,12) -- (24,12) -- (24,6); 
       } %tikzmath
   = N \vee B$.
  By definition, $\mu(\calb) = [ M : N ]$ is the square of $\llbracket M:N\rrbracket$, the latter being our notation for the statistical dimension of $_N L^2(M)_M$.
  Also, $[p_i ( M \vee A )p_i : p_i (N \vee A)]$
  is the square of the statistical dimension
  of $_{p_i (N \vee A)} p_i L^2(M \vee A)_{M \vee A}$, as can be seen 
  using~\eqref{eq:dim-on-H-is-dim-on-L2}; cf. the proof of~\cite[\propstatdimsubalg]{BDH(Dualizability+Index-of-subfactors)}.
  The vector whose entries are the statistical dimensions of 
  $_{p_i (N \vee A)} p_i L^2(M \vee A)_{M \vee A}$ 
  is denoted by $\llbracket M \vee A : N \vee A \rrbracket$.
  By~\eqref{eq:DIS-727}
  \begin{equation*}
    \big\| \llbracket M\vee A:N\vee A \rrbracket \big\|_2\, 
         \,\,\le\,\, \llbracket M:N\rrbracket
  \end{equation*}
  and the result follows.
\end{proof}

\begin{corollary}\label{cor: is a finite homomorphism of von Neumann algebra}
The inclusion   $\tikzmath[scale=\textscale]
	{\draw[thick, double] (0,6) -- (0,12) -- (6,12);
	\draw (6,12) -- (18,12) (10,6) -- (10,8) -- (14,8) -- (14,6);
	\draw[ultra thick] (18,12) -- (24,12) -- (24,6);} \,\hookrightarrow\,
\tikzmath[scale=\textscale]
	{\draw[thick, double] (0,6) -- (0,12) -- (6,12);
	\draw (6,12) -- (18,12) (10,6) -- (10,8) -- (14,8) -- (14,6);
	\draw[dash pattern=on .4pt off .62pt] (12,8) -- (12,12);
	\draw[ultra thick] (18,12) -- (24,12) -- (24,6);}$ 
is a finite homomorphism of von Neumann algebras with finite-dimensional center
(see Appendix~\ref{subsec:dualizability}).
\end{corollary} 
\begin{proof}
Let 
$X:=\tikzmath[scale=\textscale] {\draw[thick, double] (0,6) -- (0,12) -- (6,12);\draw (6,12) -- (18,12) (10,6) -- (10,8) -- (14,8) -- (14,6);\draw[ultra thick] (18,12) -- (24,12) -- (24,6);}$\,,
with minimal central projections $p_1,\ldots, p_n$, and let
$Y:=\tikzmath[scale=\textscale] {\draw[thick, double] (0,6) -- (0,12) -- (6,12);\draw (6,12) -- (18,12) (10,6) -- (10,8) -- (14,8) -- (14,6);\draw[dash pattern=on .4pt off .62pt] (12,8) -- (12,12);
\draw[ultra thick] (18,12) -- (24,12) -- (24,6);}$\,.
Recall that $Y$ is a factor.
By definition, the inclusion $X\to Y$ is finite if and only if the bimodule ${}_XL^2Y_Y$ is dualizable, which happens if and only if its summands ${}_{p_iX}(p_iL^2Y)_Y$ are dualizable.  A bimodule between factors is dualizable if and only if its statistical dimension is finite; see Appendix~\ref{subsec:stat-dim+minimal-index}.
The commutant of $Y$ on $p_iL^2Y$ is $p_iYp_i$, and the inclusion $p_i X\hookrightarrow p_i Y p_i$ is finite by the previous lemma.
\end{proof}

\subsection*{\hspace*{-18pt}Keystone fusion contains the vacuum sector of the fused defect}

Let $\calb$ be a conformal net with finite index.
Recall from Section \ref{sec: def of F G0 G} that, given a $\calb(I_l)$-module $H_l$ and a $\calb(I_r)$-module $H_r$,
then the keystone fusion $G_\calb(H_r, H_l)$ is defined by
\begin{equation*}
  G_\calb(H_l,H_r) := 
     G_{0,\calb}(H_l,H_r) \boxtimes_{\calb(S_m)} H_0(S_m,\calb).  
\end{equation*}
This construction uses the isomorphism 
$\calb(S_m) \cong \bigoplus_{\lambda\in\Delta} \bfB(H_\lambda(S_m,\calb))$
from Theorem~\ref{thm:cala(S)}.

\begin{lemma}\label{lem: Hbar x_B(S_m) H = C}
There is a canonical isomorphism
\[
\overline{H_0(S_m,\calb)} \boxtimes_{\calb(S_m)} H_0(S_m,\calb) \cong \IC
\]
\end{lemma}
\begin{proof}
The two actions of $\calb(S_m)$ factor through its summand $\bfB(H_0(S_m,\calb))$.
The result is therefore a special case of the general isomorphism $\bar H \boxtimes_{\bfB(H)} H\cong \IC$.
\end{proof}

%In the formulation of the following Theorem we use the 
%isomorphism $\Psi_0$ from Theorem~\ref{thm:G_0=L^2}.
%\ABcomm{What follows is a Proposition and it seems to use $\Psi_0$ 
%only in the proof, not its formulation.}

\begin{proposition}\label{prop:G=L2}
Let $\calb$ be a conformal net with finite index, and let $_\cala D_\calb$ and $_\calb E_\calc$ be irreducible defects.
Then there is an isometric embedding
\begin{equation}\label{eq: main equation of thm G=L2}
	\Psi \,\,\colon\,\,L^2 \left( \tikzmath[scale=\displscale]
	{\useasboundingbox (-2,0) rectangle (26,14);\draw[thick, double] (0,6) -- (0,12) -- (6,12);\draw (6,12) -- (18,12);\draw[ultra thick] (18,12) -- (24,12) -- (24,6);} %tikzmath
	\right) \;\; \rightarrow \;\;\tikzmath[scale=\displscale]
	{\fill[vacuumcolor](0,0) rectangle (10,12)(14,0) rectangle (24,12); \draw (6,12) -- (10,12) -- (10,0) -- (6,0) (18,12) -- (14,12) -- (14,0) -- (18,0);
	\draw[thick, double] (6,12) -- (0,12) -- (0,0) -- (6,0);\draw[ultra thick] (18,12) -- (24,12) -- (24,0) -- (18,0);\fill[vacuumcolor]  (10,0) rectangle (14,4)  (10,8) rectangle (14,12);
	\draw (10,0) rectangle (14,4)(10,8) rectangle (14,12);\filldraw[fill=vacuumcolor]  (10.5,4.5) rectangle (13.5,7.5);}\;; %tikzmath 
\end{equation}
that is, there is a map
\[
\Psi = \Psi_{D,E} \colon
L^2\big((D \circledast_{\calb} E)(S^1_\top)\big)
\to G_\calb\big(H_0(S_l,D),H_0(S_r,E)\big).
\]
  As in Theorem~\ref{thm:G_0=L^2}, the map $\Psi$ is
equivariant with respect to the left actions of
$\{D(I)\}_{I\subset S_{b,\tikzmath[scale=.1]{\draw(0,0)--(0,1)--(1,1);}}}$,
$\{D(I)\}_{I\subset S_{b,\tikzmath[scale=.1]{\draw(1,0)--(0,0)--(0,1);}}}$,
$\{E(I)\}_{I\subset S_{b,\tikzmath[scale=.1]{\draw(0,1)--(1,1)--(1,0);}}}$,
and
$\{E(I)\}_{I\subset S_{b,\tikzmath[scale=.1]{\draw(0,0)--(1,0)--(1,1);}}}$.
\end{proposition}

\begin{proof}
By the split property, we can identify
$\tikzmath[scale=\textscale]
   {  \useasboundingbox (-2,4) rectangle (26,14);
      \draw[thick, double] (0,6) -- (0,12) -- (6,12);
      \draw (6,12) -- (18,12) (10,6) -- (10,8) -- (14,8) -- (14,6);
      \draw[ultra thick] (18,12) -- (24,12) -- (24,6); 
   }$ with
$\tikzmath[scale=\textscale]
   {  \draw[thick, double] (0,6) -- (0,12) -- (6,12);
      \draw (6,12) -- (18,12);
      \draw[ultra thick] (18,12) -- (24,12) -- (24,6);}
      \,\,\bar\otimes\,\,\tikzmath[scale=\textscale]
      {\useasboundingbox (22,6) rectangle (26,12);\draw (22,6) -- (22,8) -- (26,8) -- (26,6);}%tikzmath	
$\,,
and thus
   $L^2 \left(\tikzmath[scale=\textscale]
   {  \useasboundingbox (-2,0) rectangle (26,14);
      \draw[thick, double] (0,6) -- (0,12) -- (6,12);
      \draw (6,12) -- (18,12) (10,6) -- (10,8) -- (14,8) -- (14,6);
      \draw[ultra thick] (18,12) -- (24,12) -- (24,6); 
   } \right)$
with
\[\begin{split}
   L^2 \left(
   \tikzmath[scale=\displscale]
   {  \useasboundingbox (-2,0) rectangle (26,14);
      \draw[thick, double] (0,6) -- (0,12) -- (6,12);
      \draw (6,12) -- (18,12); 
      \draw[ultra thick] (18,12) -- (24,12) -- (24,6); \draw[->] (11.5,12) -- (11.4,12);
   } \right) 
   \otimes
   L^2 \left(
   \tikzmath[scale=\displscale]
   {  \useasboundingbox (8,0) rectangle (16,14);
      \draw (10,6) -- (10,8) -- (14,8) -- (14,6);\draw[->] (12.59,8) -- (12.6,8);      
   } \right) 
\,&=\, L^2\Big((D \circledast_{\calb} E)(S^1_\top)\Big)
\otimes H_0(-S_m,\calb)\\
\,&=\, L^2\Big((D \circledast_{\calb} E)(S^1_\top)\Big)
\otimes\, \overline{H_0(S_m,\calb)}.
\end{split}\]
Fusing with $H_0(S_m,\calb)$ and applying Lemma \ref{lem: Hbar x_B(S_m) H = C}, we get a canonical isomorphism
\begin{align*}
U\,:\,\, L^2 \Big(
   \tikzmath[scale=\displscale]
   {  \useasboundingbox (-2,0) rectangle (26,14);
      \draw[thick, double] (0,6) -- (0,12) -- (6,12);
      \draw (6,12) -- (18,12); 
      \draw[ultra thick] (18,12) -- (24,12) -- (24,6); \draw[->] (11.5,12) -- (11.4,12);
   } \Big)
\,=\, & L^2\Big((D \circledast_{\calb} E)(S^1_\top)\Big)\\
\xrightarrow{\cong}\, & L^2\Big((D \circledast_{\calb} E)(S^1_\top)\Big)
\otimes\, \overline{H_0(S_m,\calb)}\underset{\calb(S_m)}\boxtimes H_0(S_m,\calb)\\
\xrightarrow{\cong}\,& L^2 \left(
   \tikzmath[scale=\displscale]
   {  \useasboundingbox (-2,0) rectangle (26,14);
      \draw[thick, double] (0,6) -- (0,12) -- (6,12);
      \draw (6,12) -- (18,12)  (10,6) -- (10,8) -- (14,8) -- (14,6);
      \draw[ultra thick] (18,12) -- (24,12) -- (24,6); \draw[->] (11.5,12) -- (11.4,12);\draw[->] (12.59,8) -- (12.6,8);
   } \right) \underset{\calb(S_m)}\boxtimes H_0(S_m,\calb).
\end{align*}
Recall from Appendix~\ref{subsec:dualizability} that
the $L^2$-space construction is functorial for finite homomorphisms
between von Neumann algebras with finite-dimensional center.
By Corollary~\ref{cor: is a finite homomorphism of von Neumann algebra},
the inclusion 
  $\iota \colon 
    \tikzmath[scale=\textscale]
    {  \useasboundingbox (-2,4) rectangle (26,14);
       \draw[thick, double] (0,6) -- (0,12) -- (6,12);
       \draw (6,12) -- (18,12) (10,6) -- (10,8) -- (14,8) -- (14,6);
       \draw[ultra thick] (18,12) -- (24,12) -- (24,6); 
    } %tikzmath
    \to
    \tikzmath[scale=\textscale]
       {   \useasboundingbox (-2,4) rectangle (26,14);
           \draw[thick, double] (0,6) -- (0,12) -- (6,12);
           \draw (6,12) -- (18,12) (10,6) -- (10,8) -- (14,8) -- (14,6);
           \draw[ultra thick] (18,12) -- (24,12) -- (24,6); 
           \draw[densely dotted] (12,8) -- (12,12); 
       } %tikzmath
  $
  therefore induces a map 
  $L^2(\iota)\colon 
    L^2(\tikzmath[scale=\textscale]
    {  \useasboundingbox (-2,4) rectangle (26,14);
       \draw[thick, double] (0,6) -- (0,12) -- (6,12);
       \draw (6,12) -- (18,12) (10,6) -- (10,8) -- (14,8) -- (14,6);
       \draw[ultra thick] (18,12) -- (24,12) -- (24,6); 
    }) %tikzmath
    \to
    L^2(\tikzmath[scale=\textscale]
       {   \useasboundingbox (-2,4) rectangle (26,14);
           \draw[thick, double] (0,6) -- (0,12) -- (6,12);
           \draw (6,12) -- (18,12) (10,6) -- (10,8) -- (14,8) -- (14,6);
           \draw[ultra thick] (18,12) -- (24,12) -- (24,6); 
           \draw[densely dotted] (12,8) -- (12,12); 
       }) %tikzmath
  $.
  Let $L^2(\iota)_{\mathrm{iso}}$ be the isometry in the polar decomposition of $L^2(\iota)$.
  We set $\Psi$ to be the composite
  \begin{equation}\label{eq: def of Psi}\begin{split}
  \Psi\,\colon\, L^2 \Big(
   \tikzmath[scale=\displscale]
   {  \useasboundingbox (-2,0) rectangle (26,14);
      \draw[thick, double] (0,6) -- (0,12) -- (6,12);
      \draw (6,12) -- (18,12); 
      \draw[ultra thick] (18,12) -- (24,12) -- (24,6); %\draw[->] (11.5,12) -- (11.4,12);
   } \Big)
   \xrightarrow{\,\,U\,\,}&\,\,
   L^2 \left(
   \tikzmath[scale=\displscale]
   {  \useasboundingbox (-2,0) rectangle (26,14);
      \draw[thick, double] (0,6) -- (0,12) -- (6,12);
      \draw (6,12) -- (18,12)  (10,6) -- (10,8) -- (14,8) -- (14,6);
      \draw[ultra thick] (18,12) -- (24,12) -- (24,6); %\draw[->] (11.5,12) -- (11.4,12);\draw[->] (12.59,8) -- (12.6,8);
   } \right) \boxtimes_{\calb(S_m)} H_0(S_m,\calb)\\
   \xrightarrow{L^2(\iota)_\mathrm{iso}\otimes \mathrm{id}}&\,\,
   L^2 \left(
   \tikzmath[scale=\displscale]
   {  \useasboundingbox (-2,0) rectangle (26,14);
      \draw[thick, double] (0,6) -- (0,12) -- (6,12);
      \draw (6,12) -- (18,12)  (10,6) -- (10,8) -- (14,8) -- (14,6);
      \draw[ultra thick] (18,12) -- (24,12) -- (24,6); %\draw[->] (11.5,12) -- (11.4,12);\draw[->] (12.59,8) -- (12.6,8);
      \draw[densely dotted] (12,8) -- (12,12); 
   } \right) \boxtimes_{\calb(S_m)} H_0(S_m,\calb)\\
   \xrightarrow{\hspace{.3cm}\Psi_0\otimes \mathrm{id}\hspace{.3cm}}&\,\,\,\,
   \tikzmath[scale=\displscale]
      {    \fill[vacuumcolor] (0,0) rectangle (10,12)
                             (14,0) rectangle (24,12); 
           \draw (6,12) -- (10,12) -- (10,0) -- (6,0) 
                 (18,12) -- (14,12) -- (14,0) -- (18,0);
           \draw[thick, double] (6,12) -- (0,12) -- (0,0) -- (6,0);
           \draw[ultra thick] (18,12) -- (24,12) -- (24,0) -- (18,0);
           \fill[vacuumcolor]  (10,0) rectangle (14,4)  
                              (10,8) rectangle (14,12);
           \draw (10,0) rectangle (14,4) (10,8) rectangle (14,12);
      }\,\underset{\calb(S_m)}\boxtimes\! H_0(S_m,\calb) = G\big(H_0(D),H_0(E)\big),
  \end{split}\end{equation}
  where $\Psi_0$ is the unitary isomorphism from Theorem~\ref{thm:G_0=L^2}.
\end{proof}

We will prove later, in Theorem \ref{thm: Omega is an iso} concerning the composite map \eqref{eq: definition of Omega}, that the map $\Psi$ is actually an isomorphism.  We can already observe the following special case of that result:

\begin{lemma}\label{lem: Psi_1,1 is an iso}
If $D$ and $E$ are the identity defects of some finite-index conformal net $\cala$,
then $\Psi_{D,E}$ is a unitary isomorphism.
\end{lemma}
\begin{proof}
We need to show that the map 
\[
\big(\big(\Psi_0\circ L^2(\iota)_\mathrm{iso}\big)\otimes \mathrm{id}\big)\circ U\,:\,
H_0(S_b) =
  L^2 \left(\tikzmath[scale=\displscale]
  {\useasboundingbox (-2,0) rectangle (26,13);
   \draw (0,6) -- (0,12) -- (24,12) -- (24,6);
  }%tikzmath
   \right)
\,\,\to\,\,
   \tikzmath[scale=\displscale]
  {\filldraw[fill=vacuumcolor] (0,0) rectangle (10,12)
   (14,0) rectangle (24,12) (10,0) rectangle (14,4) (10,8) rectangle (14,12);
  }%tikzmath
  \,\underset{\cala(S_m)}\boxtimes\! H_0(S_m)
\]
is an isomorphism.
By the computation \eqref{eq: computation of square with hole}, we know that the right-hand side is isomorphic to $H_0(S_b)$, and thus is irreducible as an $S_b$-sector of $\cala$.
The above map is a homomorphism of $S_b$-sectors and is injective by the previous proposition.
It is therefore an isomorphism.
\end{proof}

\subsection*{\hspace*{-18pt}Associativity for the inclusion of the vacuum sector}

Using the isometric embedding $\Psi$ from \eqref{eq: def of Psi} in place of the unitary isomorphism $\Psi_0$ from \eqref{eq: the map Psi_0}, we can form the following diagram analogous to \eqref{eq: Assoc of Psi_0}:
  \def\SingleLtwoWithTwoColors#1#2{
	L^2 \hspace{-.5ex}
	\left( \tikzmath[scale=\textscale] {\draw[#1] (12,6) -- (12,12) -- (18,12);\draw[#2] (18,12) -- (24,12) -- (24,6);} \right)}
  \def\LtwoWithThreeColors#1#2#3#4{
	L^2 \hspace{-.5ex}
	\left(\tikzmath[scale=\textscale]
	{\draw[#1] (0,6) -- (0,12) -- (6,12);
	\draw[#2] (6,12) -- (18,12);
	\draw[#3] (18,12) -- (24,12) -- (24,6);#4} \right)}
  \def\FillLeftUpperAndLowerSquares{
	\fill[vacuumcolor] (10,0) rectangle (14,4)  (10,8) rectangle (14,12);
	\draw[ultra thin] (8,0) -- (10,0) -- (10,12) -- (8,12) (16,0) -- (14,0) -- (14,12) -- (16,12) (10,0) rectangle (14,4) (10,8) rectangle (14,12);}
  \def\FillRightUpperAndLowerSquares{
	\fill[vacuumcolor](38,0) rectangle (42,4)  (38,8) rectangle (42,12);
	\draw[thick] (36,0) -- (38,0) -- (38,12) -- (36,12) (44,0) -- (42,0) -- (42,12) -- (44,12) (38,0) rectangle (42,4)  (38,8) rectangle (42,12);}
\begin{equation}\label{eq: Assoc of Psi}
  \begin{matrix}\xymatrix{
      %TOP LEFT 
      \tikzmath[scale=\displscale]
      {\draw[ultra thin, dash pattern=on .5pt off 1pt](-12,0) rectangle (66,12);
      	\node at (27,6) {$L^2 \hspace{-.5ex} \left(\tikzmath[scale=\textscale]
	{\draw[thick, double](0,6) -- (0,12) -- (6,12);
	\draw[ultra thin] (6,12) -- (18,12);
	\draw[thick] (18,12) -- (30,12);
	\draw[ultra thick](30,12) -- (36,12) -- (36,6); 
                }  %tikzmath
	\right)$};}
      \ar[r] \ar[d]   &
      %TOP RIGHT
      \tikzmath[scale=\displscale]
      {\FillRightUpperAndLowerSquares
        \node at    (12,6)	{$\LtwoWithThreeColors{thick, double}{ultra thin}{thick}{}$};
        \node at (54.5,6)	{$\SingleLtwoWithTwoColors{thick}{ultra thick}$}; 
        \draw[ultra thin, dash pattern=on .5pt off 1pt]
                (36,0) -- (-12,0) -- (-12,12) -- (36,12)
                (44,12) -- (66,12) -- (66,0) -- (44,0);
        \filldraw[thick, fill=vacuumcolor]  (38.7,4.7) rectangle (41.3,7.3);
                }  %tikzmath
      \ar[d]   \\
      %BOTTOM LEFT
      \tikzmath[scale=\displscale]
      {\FillLeftUpperAndLowerSquares 
        \node at     (-1,6)	{$\SingleLtwoWithTwoColors{thick, double}{ultra thin}$}; 
        \node at    (40,6)	{$\LtwoWithThreeColors{ultra thin}{thick}{ultra thick}{}$};
        \draw[ultra thin, dash pattern=on .5pt off 1pt]
                (10,0) -- (-12,0) -- (-12,12) --(10,12)
                (14,12) -- (66,12) -- (66,0) -- (14,0);
        \filldraw[ultra thin, fill=vacuumcolor]  (10.4,4.4) rectangle (13.6,7.6);
                }  %tikzmath
      \ar[r]   &
      %BOTTOM RIGHT 
      \tikzmath[scale=\displscale]
      {\FillLeftUpperAndLowerSquares\FillRightUpperAndLowerSquares
        \node at     (-1,6)	{$\SingleLtwoWithTwoColors{thick, double}{ultra thin}$}; 
        \node at    (26,6)	{$\SingleLtwoWithTwoColors{ultra thin}{thick}$}; 
        \node at (54.5,6)	{$\SingleLtwoWithTwoColors{thick}{ultra thick}$}; 
        \draw[ultra thin, dash pattern=on .5pt off 1pt]
                (14,12) -- (36,12) (14,0) -- (36,0)
                (10,0) -- (-12,0) -- (-12,12) --(10,12)
                (44,12) -- (66,12) -- (66,0) -- (44,0);
        \filldraw[ultra thin, fill=vacuumcolor]  (10.4,4.4) rectangle (13.6,7.6);
        \filldraw[thick, fill=vacuumcolor]  (38.7,4.7) rectangle (41.3,7.3);
                } %tikzmath
    }\put(0,-50){.} %xymatrix 
  \end{matrix}\end{equation}

\begin{lemma}\label{lem: Associativity of Psi}
Let $\calb$ and $\calc$ be finite-index conformal nets,
and let ${}_\cala D_\calb$, ${}_\calb E_\calc$, and ${}_\calc F_\cald$ be irreducible defects.
Then the diagram \eqref{eq: Assoc of Psi} is commutative.
\end{lemma}

\begin{proof}
By definition, each side of \eqref{eq: Assoc of Psi} is a composite of three maps.
Replacing each side by its definition, that diagram can be expanded to a $4\times 4$ grid 
\[
\tikzmath{
\begin{scope}[cm={1,0,0,.5,(0,0)}]
\foreach \source in {{(0,0)},{(1,0)},{(2,0)},{(0,1)},{(1,1)},{(2,1)},{(0,2)},{(1,2)},{(2,2)},{(0,3)},{(1,3)},{(2,3)}}     {\draw[->] \source + (.1,0) -- +(.9,0); }
\foreach \source in {{(0,3)},{(1,3)},{(2,3)},{(0,1)},{(1,1)},{(2,1)},{(0,2)},{(1,2)},{(2,2)},{(3,3)},{(3,1)},{(3,2)}}     {\draw[->] \source + (0,-.1) -- +(0,-.9); }
\foreach \ite/\x/\y in {1/.5/2.5,2/1.5/2.5,3/2.5/2.5,4/.5/1.5,5/1.5/1.5,6/2.5/1.5,7/.5/.5,8/1.5/.5,9/2.5/.5}             {\node[draw, inner sep = 1.5] at (\x,\y) {\tiny\ite};}
\end{scope}}
\]
that contains $9$ squares.
The four upper left squares of that grid are given by
\def\squared#1{\tikz{\useasboundingbox (-.13,-.11) rectangle (.13,.12);\node[draw, inner sep = 1.5]{\tiny #1};}}
{
\def\thinsquare{\tikzmath[scale=\textscale]{\draw[ultra thin] (0,0) rectangle (4,4);}}
\def\thicksquare{\tikzmath[scale=\textscale]{\draw[thick] (0,0) rectangle (4,4);}}
\def\thinsqfill{\tikzmath[scale=\textscale]{\filldraw[ultra thin, fill=vacuumcolor] (0,0) rectangle (4,4);}}
\def\thicksqfill{\tikzmath[scale=\textscale]{\filldraw[thick, fill=vacuumcolor] (0,0) rectangle (4,4);}}
\[
  \begin{matrix}\xymatrix{
        %TOP LEFT 
L^2 \hspace{-.5ex} \left(\tikzmath[scale=\textscale]
	{\draw[thick, double](0,6) -- (0,12) -- (6,12);
	\draw[ultra thin] (6,12) -- (18,12);
	\draw[thick] (18,12) -- (30,12);
	\draw[ultra thick](30,12) -- (36,12) -- (36,6); }%tikzmath
	\right)
      \ar[r] \ar[d] \ar@{}[dr]|{\squared1}  &
      %TOP MID
L^2 \hspace{-.5ex} \left(\tikzmath[scale=\textscale]
	{\draw[thick, double](0,6) -- (0,12) -- (6,12);
	\draw[ultra thin] (6,12) -- (18,12);
	\draw[thick] (18,12) -- (30,12)(22,6) -- (22,8) -- (26,8) -- (26,6);
	\draw[ultra thick](30,12) -- (36,12) -- (36,6); }%tikzmath
	\right)\boxtimes_{\thicksquare}\thicksqfill
      \ar[r] \ar[d] \ar@{}[dr]|{\squared2}  &
      %TOP RIGHT
L^2 \hspace{-.5ex} \left(\tikzmath[scale=\textscale]
	{\draw[thick, double](0,6) -- (0,12) -- (6,12);
	\draw[ultra thin] (6,12) -- (18,12);
	\draw[thick] (18,12) -- (30,12)(22,6) -- (22,8) -- (26,8) -- (26,6);
	\draw[ultra thick](30,12) -- (36,12) -- (36,6);
	\draw[dash pattern=on .4pt off .62pt] (24,8) -- (24,12);}%tikzmath
	\right)\boxtimes_{\thicksquare}\thicksqfill
	\ar[d]\\
        %MID LEFT 
L^2 \hspace{-.5ex} \left(\tikzmath[scale=\textscale]
	{\draw[thick, double](0,6) -- (0,12) -- (6,12);
	\draw[ultra thin] (6,12) -- (18,12)(10,6) -- (10,8) -- (14,8) -- (14,6);
	\draw[thick] (18,12) -- (30,12);
	\draw[ultra thick](30,12) -- (36,12) -- (36,6); }%tikzmath
	\right)\boxtimes_{\thinsquare}\thinsqfill
      \ar[r] \ar[d] \ar@{}[dr]|{\squared4}  &
      %MID MID
L^2 \hspace{-.5ex} \left(\tikzmath[scale=\textscale]
	{\draw[thick, double](0,6) -- (0,12) -- (6,12);
	\draw[ultra thin] (6,12) -- (18,12)(10,6) -- (10,8) -- (14,8) -- (14,6);
	\draw[thick] (18,12) -- (30,12)(22,6) -- (22,8) -- (26,8) -- (26,6);
	\draw[ultra thick](30,12) -- (36,12) -- (36,6); }%tikzmath
	\right)\boxtimes_{\thinsquare\bar\otimes\thicksquare}(\thinsqfill\!\otimes\! \thicksqfill)
      \ar[r] \ar[d] \ar@{}[dr]|{\squared5}  &
      %MID RIGHT
L^2 \hspace{-.5ex} \left(\tikzmath[scale=\textscale]
	{\draw[thick, double](0,6) -- (0,12) -- (6,12);
	\draw[ultra thin] (6,12) -- (18,12)(10,6) -- (10,8) -- (14,8) -- (14,6);
	\draw[thick] (18,12) -- (30,12)(22,6) -- (22,8) -- (26,8) -- (26,6);
	\draw[ultra thick](30,12) -- (36,12) -- (36,6); 
	\draw[dash pattern=on .4pt off .62pt] (24,8) -- (24,12);}%tikzmath
	\right)\boxtimes_{\thinsquare\bar\otimes\thicksquare}(\thinsqfill\!\otimes\! \thicksqfill)
	\ar[d]\\
        %BOTTOM LEFT 
L^2 \hspace{-.5ex} \left(\tikzmath[scale=\textscale]
	{\draw[thick, double](0,6) -- (0,12) -- (6,12);
	\draw[ultra thin] (6,12) -- (18,12)(10,6) -- (10,8) -- (14,8) -- (14,6);
	\draw[thick] (18,12) -- (30,12);
	\draw[ultra thick](30,12) -- (36,12) -- (36,6); 
	\draw[dash pattern=on .4pt off .62pt] (12,8) -- (12,12);}%tikzmath
	\right)\boxtimes_{\thinsquare}\thinsqfill
      \ar[r]   &
      %BOTTOM MID
L^2 \hspace{-.5ex} \left(\tikzmath[scale=\textscale]
	{\draw[thick, double](0,6) -- (0,12) -- (6,12);
	\draw[ultra thin] (6,12) -- (18,12)(10,6) -- (10,8) -- (14,8) -- (14,6);
	\draw[thick] (18,12) -- (30,12)(22,6) -- (22,8) -- (26,8) -- (26,6);
	\draw[ultra thick](30,12) -- (36,12) -- (36,6); 
	\draw[dash pattern=on .4pt off .62pt] (12,8) -- (12,12);}%tikzmath
	\right)\boxtimes_{\thinsquare\bar\otimes\thicksquare}(\thinsqfill\!\otimes\! \thicksqfill)
      \ar[r]   &
      %BOTTOM RIGHT
L^2 \hspace{-.5ex} \left(\tikzmath[scale=\textscale]
	{\draw[thick, double](0,6) -- (0,12) -- (6,12);
	\draw[ultra thin] (6,12) -- (18,12)(10,6) -- (10,8) -- (14,8) -- (14,6);
	\draw[thick] (18,12) -- (30,12)(22,6) -- (22,8) -- (26,8) -- (26,6);
	\draw[ultra thick](30,12) -- (36,12) -- (36,6); 
	\draw[dash pattern=on .4pt off .62pt] (12,8) -- (12,12)(24,8) -- (24,12);}%tikzmath
	\right)\boxtimes_{\thinsquare\bar\otimes\thicksquare}(\thinsqfill\!\otimes\! \thicksqfill)
}\end{matrix}
\]
where $\boxtimes_{\thinsquare}\thinsqfill$ and $\boxtimes_{\thicksquare}\thicksqfill$ stand for 
$\boxtimes_{\calb(S_m)}H_0(S_m,\calb)$ and $\boxtimes_{\calc(S^{\mathrm{tr}}_m)}H_0(S^{\mathrm{tr}}_m,\calc)$, respectively, and
$\text{$\boxtimes_{\thinsquare\bar\otimes\thicksquare}(\thinsqfill\!\otimes\! \thicksqfill)$}$ stands for
$\boxtimes_{\calb(S_m)\bar\otimes \calc(S^{\mathrm{tr}}_m)}(H_0(S_m,\calb)\otimes H_0(S^{\mathrm{tr}}_m,\calc))$.   } %ends the "{" before the def of \thinsquare
Here, $S^{\mathrm{tr}}_m$ denotes a translated copy of the circle $S_m$.

The squares \squared1\,, \squared2\,, and \squared4 clearly commute.
To see that \squared5 commutes, note first that $\tikzmath[scale=\textscale]
	{\draw[thick, double](0,6) -- (0,12) -- (6,12);
	\draw[ultra thin] (6,12) -- (18,12)(10,6) -- (10,8) -- (14,8) -- (14,6);
	\draw[thick] (18,12) -- (30,12)(22,6) -- (22,8) -- (26,8) -- (26,6);
	\draw[ultra thick](30,12) -- (36,12) -- (36,6); 
	\draw[dash pattern=on .4pt off .62pt] (12,8) -- (12,12);
	\draw[dash pattern=on .4pt off .62pt] (24,8) -- (24,12);}$ is a factor, as can be seen by applying 
Lemma~\ref{lem: factoriality of alg with dotted lines} twice.
That square then commutes by the functoriality of $L^2(-)_\mathrm{iso}$---see~\cite[\propfunctorialityofLiso]{BDH(Dualizability+Index-of-subfactors)} and note that the necessary conditions for that functoriality are satisfied by Corollary~\ref{cor: is a finite homomorphism of von Neumann algebra} and by Lemma~\ref{lem: check conditions in order to apply functoriality of L^2_iso} below.
The upper right squares of our $4\times 4$ gird are given by
\[
\def\thinsquare{\tikzmath[scale=\textscale]{\draw[ultra thin] (0,0) rectangle (4,4);}}
\def\thicksquare{\tikzmath[scale=\textscale]{\draw[thick] (0,0) rectangle (4,4);}}
\def\thinsqfill{\tikzmath[scale=\textscale]{\filldraw[ultra thin, fill=vacuumcolor] (0,0) rectangle (4,4);}}
\def\thicksqfill{\tikzmath[scale=\textscale]{\filldraw[thick, fill=vacuumcolor] (0,0) rectangle (4,4);}}
  \begin{matrix}\xymatrix{
      %TOP LEFT
L^2 \big(\,\tikzmath[scale=\planscale]
	{\draw[thick, double](0,6) -- (0,12) -- (6,12);
	\draw[ultra thin] (6,12) -- (18,12);
	\draw[thick] (18,12) -- (30,12)(22,6) -- (22,8) -- (26,8) -- (26,6);
	\draw[ultra thick](30,12) -- (36,12) -- (36,6);
	\draw[dash pattern=on .4pt off .62pt] (24,8) -- (24,12);}%tikzmath
	\,\big)\boxtimes_{\thicksquare}\thicksqfill
	\ar[d]\ar[r] \ar@{}[dr]|{\squared3} &
	%TOP RIGHT
	      \tikzmath[scale=\displscale]
      {\FillRightUpperAndLowerSquares
        \node at    (12,6)	{$\LtwoWithThreeColors{thick, double}{ultra thin}{thick}{}$};
        \node at (54.5,6)	{$\SingleLtwoWithTwoColors{thick}{ultra thick}$}; 
        \draw[ultra thin, dash pattern=on .5pt off 1pt]
                (36,0) -- (-12,0) -- (-12,12) -- (36,12)
                (44,12) -- (66,12) -- (66,0) -- (44,0);
        \filldraw[thick, fill=vacuumcolor]  (38.7,4.7) rectangle (41.3,7.3);
                }  %tikzmath
	\ar[d]\\
      %MID LEFT
L^2 \big(\,\tikzmath[scale=\planscale]
	{\draw[thick, double](0,6) -- (0,12) -- (6,12);
	\draw[ultra thin] (6,12) -- (18,12)(10,6) -- (10,8) -- (14,8) -- (14,6);
	\draw[thick] (18,12) -- (30,12)(22,6) -- (22,8) -- (26,8) -- (26,6);
	\draw[ultra thick](30,12) -- (36,12) -- (36,6); 
	\draw[dash pattern=on .4pt off .62pt] (24,8) -- (24,12);}%tikzmath
	\,\big)\boxtimes_{\thinsquare\bar\otimes\thicksquare}(\thinsqfill\!\otimes\! \thicksqfill)
	\ar[d]\ar[r] \ar@{}[dr]|{\squared6} &
	%MID RIGHT
	      \tikzmath[scale=\displscale]
      {\FillRightUpperAndLowerSquares
        \node at    (12,6)	{$\LtwoWithThreeColors{thick, double}{ultra thin}{thick}{\draw[ultra thin] (10,6) -- (10,8) -- (14,8) -- (14,6);}$};
        \node at (54.5,6)	{$\SingleLtwoWithTwoColors{thick}{ultra thick}$}; 
        \draw[ultra thin, dash pattern=on .5pt off 1pt]
                (36,0) -- (-12,0) -- (-12,12) -- (36,12)
                (44,12) -- (66,12) -- (66,0) -- (44,0);
        \filldraw[thick, fill=vacuumcolor]  (38.7,4.7) rectangle (41.3,7.3);
                }  %tikzmath
	\boxtimes_{\thinsquare}\thinsqfill
	\ar[d]\\
      %BOTTOM LEFT
L^2 \big(\,\tikzmath[scale=\planscale]
	{\draw[thick, double](0,6) -- (0,12) -- (6,12);
	\draw[ultra thin] (6,12) -- (18,12)(10,6) -- (10,8) -- (14,8) -- (14,6);
	\draw[thick] (18,12) -- (30,12)(22,6) -- (22,8) -- (26,8) -- (26,6);
	\draw[ultra thick](30,12) -- (36,12) -- (36,6); 
	\draw[dash pattern=on .4pt off .62pt] (12,8) -- (12,12)(24,8) -- (24,12);}%tikzmath
	\,\big)\boxtimes_{\thinsquare\bar\otimes\thicksquare}(\thinsqfill\!\otimes\! \thicksqfill)
	\ar[r] &
	%BOTTOM RIGHT
	      \tikzmath[scale=\displscale]
      {\FillRightUpperAndLowerSquares
        \node at    (12,6)	{$\LtwoWithThreeColors{thick, double}{ultra thin}{thick}{\draw[ultra thin] (10,6) -- (10,8) -- (14,8) -- (14,6);
	\draw[dash pattern=on .4pt off .62pt] (12,8) -- (12,12)(24,8) -- (24,12);}$};
        \node at (54.5,6)	{$\SingleLtwoWithTwoColors{thick}{ultra thick}$}; 
        \draw[ultra thin, dash pattern=on .5pt off 1pt]
                (36,0) -- (-12,0) -- (-12,12) -- (36,12)
                (44,12) -- (66,12) -- (66,0) -- (44,0);
        \filldraw[thick, fill=vacuumcolor]  (38.7,4.7) rectangle (41.3,7.3);
                }  %tikzmath
	\boxtimes_{\thinsquare}\thinsqfill
}\end{matrix}
\]
and their commutativity is unproblematic.
We refrain from drawing the last row of the gird.
The squares \squared7 and \squared8 are similar to \squared3 and \squared6\,.
The commutativity of \squared9 follows from that of \eqref{eq: Assoc of Psi_0}.
\end{proof}

\begin{lemma}\label{lem: check conditions in order to apply functoriality of L^2_iso}
Let ${}_\cala D_\calb$, ${}_\calb E_\calc$, and ${}_\calc F_\cald$ be as in Lemma \ref{lem: Associativity of Psi}.
Let $A:= \tikzmath[scale=\textscale]
	{\draw[thick, double](0,6) -- (0,12) -- (6,12);
	\draw[ultra thin] (6,12) -- (18,12)(10,6) -- (10,8) -- (14,8) -- (14,6);
	\draw[thick] (18,12) -- (30,12)(22,6) -- (22,8) -- (26,8) -- (26,6);
	\draw[ultra thick](30,12) -- (36,12) -- (36,6); }$\,,
$B:= \tikzmath[scale=\textscale]
	{\draw[thick, double](0,6) -- (0,12) -- (6,12);
	\draw[ultra thin] (6,12) -- (18,12)(10,6) -- (10,8) -- (14,8) -- (14,6);
	\draw[thick] (18,12) -- (30,12)(22,6) -- (22,8) -- (26,8) -- (26,6);
	\draw[ultra thick](30,12) -- (36,12) -- (36,6); 
	\draw[dash pattern=on .4pt off .62pt] (12,8) -- (12,12);}$\,,
and let $\iota:A\to B$ denote the inclusion.
Then $Z(B)\subset \iota(A)$.
\end{lemma}

\begin{proof}
By the split property, we have isomorphisms
$	\tikzmath[scale=\textscale]
	{\draw[thick, double](0,6) -- (0,12) -- (6,12);
	\draw[ultra thin] (6,12) -- (18,12)(10,6) -- (10,8) -- (14,8) -- (14,6);
	\draw[thick] (18,12) -- (30,12)(22,6) -- (22,8) -- (26,8) -- (26,6);
	\draw[ultra thick](30,12) -- (36,12) -- (36,6); }%tikzmath
	\,\cong\, \tikzmath[scale=\textscale]
	{\draw[thick, double](0,6) -- (0,12) -- (6,12);
	\draw[ultra thin] (6,12) -- (18,12)(10,6) -- (10,8) -- (14,8) -- (14,6);
	\draw[thick] (18,12) -- (30,12);
	\draw[ultra thick](30,12) -- (36,12) -- (36,6); }%tikzmath
	\,\,\bar\otimes\,\,\tikzmath[scale=\textscale]
	{\useasboundingbox (22,6) rectangle (26,12);\draw[thick](22,6) -- (22,8) -- (26,8) -- (26,6);}%tikzmath	
$ and
$	\tikzmath[scale=\textscale]
	{\draw[thick, double](0,6) -- (0,12) -- (6,12);
	\draw[ultra thin] (6,12) -- (18,12)(10,6) -- (10,8) -- (14,8) -- (14,6);
	\draw[thick] (18,12) -- (30,12)(22,6) -- (22,8) -- (26,8) -- (26,6);
	\draw[ultra thick](30,12) -- (36,12) -- (36,6);
	\draw[dash pattern=on .4pt off .62pt] (12,8) -- (12,12);}%tikzmath
	\,\cong\, \tikzmath[scale=\textscale]
	{\draw[thick, double](0,6) -- (0,12) -- (6,12);
	\draw[ultra thin] (6,12) -- (18,12)(10,6) -- (10,8) -- (14,8) -- (14,6);
	\draw[thick] (18,12) -- (30,12);
	\draw[ultra thick](30,12) -- (36,12) -- (36,6);
	\draw[dash pattern=on .4pt off .62pt] (12,8) -- (12,12);}%tikzmath
	\,\,\bar\otimes\,\,\tikzmath[scale=\textscale]
	{\useasboundingbox (22,6) rectangle (26,12);\draw[thick](22,6) -- (22,8) -- (26,8) -- (26,6);}%tikzmath	
$\,.
It is therefore sufficient to show that 
\[	Z\Big(\tikzmath[scale=\displscale]
	{\useasboundingbox (-2,2) rectangle (38,14);
	\draw[thick, double](0,6) -- (0,12) -- (6,12);
	\draw[ultra thin] (6,12) -- (18,12)(10,6) -- (10,8) -- (14,8) -- (14,6);
	\draw[thick] (18,12) -- (30,12);
	\draw[ultra thick](30,12) -- (36,12) -- (36,6);
	\draw[dash pattern=on .4pt off .62pt] (12,8) -- (12,12);}%tikzmath
\Big)\,\,\subset\,\,\iota\Big(\tikzmath[scale=\displscale]
	{\useasboundingbox (-2,2) rectangle (38,14);
	\draw[thick, double](0,6) -- (0,12) -- (6,12);
	\draw[ultra thin] (6,12) -- (18,12)(10,6) -- (10,8) -- (14,8) -- (14,6);
	\draw[thick] (18,12) -- (30,12);
	\draw[ultra thick](30,12) -- (36,12) -- (36,6);}%tikzmath
\Big).
\]
Let $S_u$, $K_u$, $J_u$ be as in \eqref{eq: figures of intervals 2} and \eqref{eq:  J_l  J_r  K_u  J_u},
and let $H:=H_0(S_u,\calb)$, and $M:=\calb(K_u)$, with commutant $M'=\hat\calb(J_u)$.
Since $H$ is a faithful $M$-module, we can pick an $M$-linear isomorphism $\ell^2\otimes H\cong \ell^2\otimes L^2(M)$.\footnote{Here, $\ell^2:=\ell^2(\mathbb N)$
could be removed from this isomorphism if we knew that $M$ was a type $\mathit{III}$ factor, a fact which is likely to be true (unless $\calb$ is trivial) but which we haven't proven in our setup.}
Under the corresponding isomorphism of Hilbert spaces
$\ell^2\otimes\,
\tikzmath[scale=\textscale]
      {    \fill[vacuumcolor] (0,0) rectangle (10,12) (14,0) rectangle (22,12) (26,0) rectangle (36,12); 
           \draw[ultra thin] (6,12) -- (10,12) -- (10,0) -- (6,0) (18,12) -- (14,12) -- (14,0) -- (18,0);
           \draw[thick] (18,12) -- (22,12) -- (22,0) -- (18,0) (30,12) -- (26,12) -- (26,0) -- (30,0);
           \draw[thick, double] (6,12) -- (0,12) -- (0,0) -- (6,0);
           \draw[ultra thick] (30,12) -- (36,12) -- (36,0) -- (30,0);
           \filldraw[fill = vacuumcolor, ultra thin] (10,8) rectangle (14,12);
           \filldraw[fill = vacuumcolor, thick] (22,8) rectangle (26,12);
      } %tikzmath
\,\,\cong\, \ell^2\otimes\,
\tikzmath[scale=\textscale]
      {    \fill[vacuumcolor] (0,0) rectangle (10,12) (14,0) rectangle (22,12) (26,0) rectangle (36,12); 
           \draw[ultra thin] (6,12) -- (10,12) -- (10,0) -- (6,0) (18,12) -- (14,12) -- (14,0) -- (18,0);
           \draw[thick] (18,12) -- (22,12) -- (22,0) -- (18,0) (30,12) -- (26,12) -- (26,0) -- (30,0);
           \draw[thick, double] (6,12) -- (0,12) -- (0,0) -- (6,0);
           \draw[ultra thick] (30,12) -- (36,12) -- (36,0) -- (30,0);
           \filldraw[fill = vacuumcolor, thick] (22,8) rectangle (26,12);
      } %tikzmath
$\,, the algebra
\[	\bfB(\ell^2)\,\,\bar\otimes\,\,\tikzmath[scale=\displscale]
	{\useasboundingbox (-2,3) rectangle (38,14);
	\draw[thick, double](0,6) -- (0,12) -- (6,12);
	\draw[ultra thin] (6,12) -- (18,12)(10,6) -- (10,8) -- (14,8) -- (14,6);
	\draw[thick] (18,12) -- (30,12);
	\draw[ultra thick](30,12) -- (36,12) -- (36,6);
	\draw[dash pattern=on .4pt off .62pt] (12,8) -- (12,12);}%tikzmath
\,=\,	\big(\bfB(\ell^2) \,\bar\otimes\, M'\big) \,\vee\, \tikzmath[scale=\displscale]
	{\useasboundingbox (-2,3) rectangle (38,14);
	\draw[thick, double](0,6) -- (0,12) -- (6,12);
	\draw[ultra thin] (6,12) -- (10,12)(14,12) -- (18,12)(10,6) -- (10,8)(14,8) -- (14,6);
	\draw[thick] (18,12) -- (30,12);
	\draw[ultra thick](30,12) -- (36,12) -- (36,6);}%tikzmath
\]
corresponds to
\[
\big(\bfB(\ell^2) \,\bar\otimes\, M^\op\big) \,\vee\, \tikzmath[scale=\displscale]
	{\useasboundingbox (-2,3) rectangle (38,14);
	\draw[thick, double](0,6) -- (0,12) -- (6,12);
	\draw[ultra thin] (6,12) -- (10,12)(14,12) -- (18,12)(10,6) -- (10,8)(14,8) -- (14,6);
	\draw[thick] (18,12) -- (30,12);
	\draw[ultra thick](30,12) -- (36,12) -- (36,6);}%tikzmath
\,=\,\bfB(\ell^2)\,\,\bar\otimes\,\,\tikzmath[scale=\displscale]
	{\useasboundingbox (-2,3) rectangle (38,14);
	\draw[thick, double](0,6) -- (0,12) -- (6,12);
	\draw[ultra thin] (6,12) -- (10,12) -- (10,6)(18,12) -- (14,12) -- (14,6);
	\draw[thick] (18,12) -- (30,12);
	\draw[ultra thick](30,12) -- (36,12) -- (36,6);}%tikzmath
.\]
It follows that 
\[	Z\Big(\tikzmath[scale=\displscale]
	{\useasboundingbox (-2,2) rectangle (38,14);
	\draw[thick, double](0,6) -- (0,12) -- (6,12);
	\draw[ultra thin] (6,12) -- (18,12)(10,6) -- (10,8) -- (14,8) -- (14,6);
	\draw[thick] (18,12) -- (30,12);
	\draw[ultra thick](30,12) -- (36,12) -- (36,6);
	\draw[dash pattern=on .4pt off .62pt] (12,8) -- (12,12);}%tikzmath
\Big)\cong Z\Big(\tikzmath[scale=\displscale]
	{\useasboundingbox (-2,3) rectangle (38,14);
	\draw[thick, double](0,6) -- (0,12) -- (6,12);
	\draw[ultra thin] (6,12) -- (10,12) -- (10,6)(18,12) -- (14,12) -- (14,6);
	\draw[thick] (18,12) -- (30,12);
	\draw[ultra thick](30,12) -- (36,12) -- (36,6);}%tikzmath
\Big)\cong Z\Big(\tikzmath[scale=\displscale]
	{\useasboundingbox (-2,3) rectangle (12,14);
	\draw[thick, double](0,6) -- (0,12) -- (6,12);
	\draw[ultra thin] (6,12) -- (10,12) -- (10,6);}%tikzmath
\Big)\,\bar\otimes\, Z\Big(\tikzmath[scale=\displscale]
	{\useasboundingbox (12,3) rectangle (38,14);
	\draw[ultra thin](18,12) -- (14,12) -- (14,6);
	\draw[thick] (18,12) -- (30,12);
	\draw[ultra thick](30,12) -- (36,12) -- (36,6);}%tikzmath
\Big)\cong Z\Big(\tikzmath[scale=\displscale]
	{\useasboundingbox (12,3) rectangle (38,14);
	\draw[ultra thin](18,12) -- (14,12) -- (14,6);
	\draw[thick] (18,12) -- (30,12);
	\draw[ultra thick](30,12) -- (36,12) -- (36,6);}%tikzmath
\Big),
\]
where the last equality is because ${}_\cala D_\calb$ is irreducible.

We now argue that the natural inclusion
$\tikzmath[scale=\textscale]
	{\useasboundingbox (12,6) rectangle (38,14);
	\draw[ultra thin](18,12) -- (14,12);
	\draw[thick] (18,12) -- (30,12);
	\draw[ultra thick](30,12) -- (36,12);}%tikzmath
\hookrightarrow\tikzmath[scale=\textscale]
	{\useasboundingbox (12,3) rectangle (38,14);
	\draw[ultra thin](18,12) -- (14,12) -- (14,6);
	\draw[thick] (18,12) -- (30,12);
	\draw[ultra thick](30,12) -- (36,12) -- (36,6);}%tikzmath
$
induces an isomorphism of centers.
By Theorem \ref{thm: semi-simplicity of DoE},
the center of these algebras is finite-dimensional.
The center
$Z(\tikzmath[scale=\textscale]
	{\useasboundingbox (12,6) rectangle (38,14);
	\draw[ultra thin](18,12) -- (14,12);
	\draw[thick] (18,12) -- (30,12);
	\draw[ultra thick](30,12) -- (36,12);}%tikzmath
)$ certainly maps to the center
$Z(\tikzmath[scale=\textscale]
	{\useasboundingbox (12,3) rectangle (38,14);
	\draw[ultra thin](18,12) -- (14,12) -- (14,6);
	\draw[thick] (18,12) -- (30,12);
	\draw[ultra thick](30,12) -- (36,12) -- (36,6);}%tikzmath
)$ and that the map is injective.
It is therefore an isomorphism.
The claim now follows, as
\[	Z\Big(\tikzmath[scale=\displscale]
	{\useasboundingbox (-2,2) rectangle (38,14);
	\draw[thick, double](0,6) -- (0,12) -- (6,12);
	\draw[ultra thin] (6,12) -- (18,12)(10,6) -- (10,8) -- (14,8) -- (14,6);
	\draw[thick] (18,12) -- (30,12);
	\draw[ultra thick](30,12) -- (36,12) -- (36,6);
	\draw[dash pattern=on .4pt off .62pt] (12,8) -- (12,12);}%tikzmath
\Big)\,\cong\, Z\Big(\tikzmath[scale=\displscale]
	{\useasboundingbox (12,3) rectangle (38,14);
	\draw[ultra thin](18,12) -- (14,12) -- (14,6);
	\draw[thick] (18,12) -- (30,12);
	\draw[ultra thick](30,12) -- (36,12) -- (36,6);}%tikzmath
\Big)\,\cong\, Z\Big(\tikzmath[scale=\displscale]
	{\useasboundingbox (12,6) rectangle (38,14);
	\draw[ultra thin](18,12) -- (14,12);
	\draw[thick] (18,12) -- (30,12);
	\draw[ultra thick](30,12) -- (36,12);}%tikzmath
\Big)\,\subset\, \tikzmath[scale=\displscale]
	{\useasboundingbox (-2,5) rectangle (38,14);
	\draw[thick, double](0,6) -- (0,12) -- (6,12);
	\draw[ultra thin] (6,12) -- (18,12)(10,6) -- (10,8) -- (14,8) -- (14,6);
	\draw[thick] (18,12) -- (30,12);
	\draw[ultra thick](30,12) -- (36,12) -- (36,6);}%tikzmath
.\qedhere\]
\end{proof}

\begin{remark}\label{rem: we assumed irreducibility}
All the defects in this section were assumed to be irreducible.
However, using the compatibility of directs integrals with various operations, it is straightforward to extend 
Proposition \ref{prop:G=L2} and Lemma \ref{lem: Associativity of Psi} to arbitrary defects.
\end{remark}

\section{Comparison between fusion and keystone fusion}\label{sec: Comparison between F and G}

Let $\cala$ be a conformal net with finite index (implicitly irreducible as before).
In this section, we will define a unitary natural transformation $\Phi_\cala:F_\cala\to G_\cala$ between the functors introduced in Section \ref{sec: def of F G0 G}.
Graphically, this natural transformation is denoted
\begin{equation*}
\Phi\colon
\tikzmath[scale=\displscale]{\draw (8,12) -- (16,12) (8,0) -- (16,0) (12,0) -- (12,12);
\draw[ultra thin, dash pattern=on .5pt off 1pt](8,12) -- (0,12) -- (0,0) -- (8,0)(16,12) -- (24,12) -- (24,0) -- (16,0); \draw (6,6) node {$H_r$} (18,6) node {$H_l$};}\; %tikzmath
\rightarrow\; \tikzmath[scale=\displscale]{\draw (8,12) -- (10,12) -- (10,0) -- (8,0) (16,12) -- (14,12) -- (14,0) -- (16,0);
\draw[ultra thin, dash pattern=on .5pt off 1pt] (8,12) -- (0,12) -- (0,0) -- (8,0) (16,12) -- (24,12) -- (24,0) -- (16,0); \draw (5,6) node {$H_r$} (19,6) node {$H_l$};
\fill[vacuumcolor] (10,0) rectangle (14,4) (10,8) rectangle (14,12)(10.5,4.5) rectangle (13.5,7.5);\draw (10,0) rectangle (14,4)(10,8) rectangle (14,12)(10.5,4.5) rectangle (13.5,7.5);}\,. %tikzmath 
\end{equation*}

Recall the circles $S_l$, $S_r$, and $S_b$ introduced in \eqref{eq: figures of intervals 3}:
\begin{equation}\label{eq: S_l and S_r once again}
S_l = \tikzmath[scale=\displscale]{\useasboundingbox (0,0) rectangle (24,12); \draw (0,0) rectangle (12,12);\draw[->] (5.6,12) -- (5.5,12);} %tikzmath 
\; , \quad S_r = \tikzmath[scale=\displscale] {\useasboundingbox (0,0) rectangle (24,12); \draw (12,0) rectangle (24,12); \draw[->] (17.6,12) -- (17.5,12);}\,, %tikzmath
\quad \text{and} \quad S_b = \tikzmath[scale=\displscale] {\draw (0,0) rectangle (24,12); \draw[->] (11.6,12) -- (11.5,12);}\; .  %tikzmath
\end{equation}
As before, we let $I:= S_l\cap S_r$, with orientation inherited from $S_r$.
The circles $S_l$ and $S_r$ are given conformal structures by their unit speed parametrizations.
The circle $S_b$ is also given a conformal structure, as follows.
Let $j_l\in\Conf_-(S_l)$ and $j_r\in\Conf_-(S_r)$ be the unique conformal involutions fixing $\partial I$.
The conformal structure on $S_b$ is the one making $\epsilon_l := j_l|_I\cup \mathrm{Id}_{j_r(S_r \backslash I)}:S_r\to S_b$ into a conformal map.
Equivalently, it is the one for which $\epsilon_r := j_r|_I\cup \mathrm{Id}_{j_l(S_l \backslash I)}:S_l\to S_b$ is a conformal map.
\begin{warning}
The conformal structure on $S_b$ is \emph{not} the one 
induced by its constant speed parametrization.
Nevertheless, the reflection along the horizontal and vertical symmetry 
axes of  $S_b$ are  conformal involutions.
\end{warning}

Consider the vacuum sectors $H_0(S_l)$, $H_0(S_r)$, and $H_0(S_b)$ for the net $\cala$ on the circles~\eqref{eq: S_l and S_r once again}.  By the construction~\eqref{eq:Upsilon}, there is a canonical identification 
$\Upsilon:H_0(S_l)\boxtimes_{\cala(I)} H_0(S_r)\cong H_0(S_b)$
that is equivariant with respect to the actions of $\cala(J)$ 
for every $J\subset S_b$ (i.e., it is an isomorphism of $S_b$-sectors of $\cala$).
We denote it graphically as
\begin{equation}\label{eq: Upsilon display}
\Upsilon\;:\;\, \tikzmath[scale=\displscale] {\fill[vacuumcolor] (0,0) rectangle (24,12); \draw (0,0) rectangle (24,12) (12,0) -- (12,12);} %tikzmath
\;\; \xrightarrow{\,\,\cong\,\,} \;\; \tikzmath[scale=\displscale] {\filldraw[fill=vacuumcolor] (0,0) rectangle (24,12);}\,. %tikzmath
\end{equation}
The following proposition improves on Proposition \ref{prop: non-canonical construction of Phi}
by providing an explicit construction of the natural isomorphism $\Phi:F\to G$.

\begin{proposition}
\label{prop: local-fusion}
There is a unitary natural isomorphism $\Phi=\Phi_\cala$ between the fusion functor $F_\cala$ and the keystone fusion functor $G_\cala$:
  \begin{equation*}
     \Phi %= \Phi_{H_r,H_l} 
     \;\colon\;
     \tikzmath[scale=\displscale]
     { \draw (8,12) -- (16,12) (8,0) -- (16,0) (12,0) -- (12,12); \draw[ultra thin, dash pattern=on .5pt off 1pt]
         (8,12) -- (0,12) -- (0,0) -- (8,0) (16,12) -- (24,12) -- (24,0) -- (16,0); \draw (6,6) node {$H_r$} (18,6) node {$H_l$};
     }\; %tikzmath
     \xrightarrow{\cong}\; 
     \tikzmath[scale=\displscale]
     { \draw (8,12) -- (10,12) -- (10,0) -- (8,0) (16,12) -- (14,12) -- (14,0) -- (16,0);
       \draw[ultra thin, dash pattern=on .5pt off 1pt] (8,12) -- (0,12) -- (0,0) -- (8,0) (16,12) -- (24,12) -- (24,0) -- (16,0);
       \draw (5,6) node {$H_r$} (19,6) node {$H_l$};
       \fill[vacuumcolor] (10,0) rectangle (14,4) (10,8) rectangle (14,12) (10.5,4.5) rectangle (13.5,7.5);
       \draw (10,0) rectangle (14,4)  (10,8) rectangle (14,12) (10.5,4.5) rectangle (13.5,7.5);
     }\,. %tikzmath   
  \end{equation*} 
This natural isomorphism is symmetric monoidal in the sense that $\Phi_{\cala\otimes \calb} = \Phi_\cala \otimes \Phi_\calb$.
\end{proposition}

\begin{proof}
Let $I_l$ and $I_r$ be as in \eqref{eq: figures of intervals 1}, and let $I_l'$ and $I_r'$ be the closures of their complements in $S_l$ and $S_r$, respectively.
Since the actions of $\cala(I_l)$ and $\cala(I_r)$ on $H_0(S_l)$ and $H_0(S_r)$ are faithful, by Lemma \ref{lem: Mod(A1) x Mod(A2) --> Mod(B)}
it is enough to define the isomorphism $\Phi_{H_0(S_l),H_0(S_r)}:
\tikzmath[scale=\textscale]
{ \fill[vacuumcolor]  (0,0) rectangle (12,12) (12,0) rectangle (24,12);\draw (0,0) rectangle (12,12) (12,0) rectangle (24,12);} %tikzmath
\to\tikzmath[scale=\textscale]
{  \fill[vacuumcolor] (0,0) rectangle (10,12) (14,0) rectangle (24,12) (10,0) rectangle (14,4) (10,8) rectangle (14,12) (10.5,4.5) rectangle (13.5,7.5);
\draw  (0,0) rectangle (10,12) (14,0) rectangle (24,12)(10,0) rectangle (14,4) (10,8) rectangle (14,12)(10.5,4.5) rectangle (13.5,7.5);} %tikzmath
$\,, and to check that it commutes with the natural actions of $\cala(I_l)'=\cala(I_l')$ and $\cala(I_r)'=\cala(I_r')$.
We define this isomorphism as the composite
\begin{equation}\label{eq: Upsilon -- v_K -- Psi}
\xymatrix@C=1.2cm{
\tikzmath[scale=\displscale] {\fill[vacuumcolor]  (0,0) rectangle (12,12) (12,0) rectangle (24,12);\draw (0,0) rectangle (12,12) (12,0) rectangle (24,12);} %tikzmath
\ar[r]^{\Upsilon}_\cong   &
\tikzmath[scale=\displscale] {\fill[vacuumcolor] (0,0) rectangle  (24,12);\draw (0,0) rectangle  (24,12);} %tikzmath
\ar[r]^(.43){v_{S_{b,\top}}}_(.43)\cong   &
L^2\left(\tikzmath[scale=\displscale] {\useasboundingbox (-2,1) rectangle (26,14);\draw (0,6) -- (0,12) -- (24,12) -- (24,6);}\right) %tikzmath
\ar[r]^(.55){\Psi}_(.55)\cong   &
\tikzmath[scale=\displscale] {\fill[vacuumcolor] (0,0) rectangle (10,12) (14,0) rectangle (24,12) (10,0) rectangle (14,4) (10,8) rectangle (14,12) (10.5,4.5) rectangle (13.5,7.5);
\draw  (0,0) rectangle (10,12) (14,0) rectangle (24,12)(10,0) rectangle (14,4) (10,8) rectangle (14,12) (10.5,4.5) rectangle (13.5,7.5);} %tikzmath
}, %xymatrix 
\end{equation}
where 
$v_{S_{b,\top}} \colon H_0(S_b,\cala) \to L^2 (\cala({S_{b,\top}}))$ is 
the canonical unitary isomorphism~\eqref{eq:v_I} associated to the upper half ${S_{b,\top}}$ of the conformal circle $S_b$,
and $\Psi$ is the unitary isomorphism from 
Lemma~\ref{lem: Psi_1,1 is an iso}.  The symmetric monoidal condition is clear by construction.
\end{proof} 

Let $\cala$ and $\calc$ be conformal nets, let $\calb$ be a conformal net with finite index, and let ${}_\cala D_\calb$ and ${}_\calb E_\calc$ be defects.
Let us introduce the notation 
$\tikzmath[scale=\textscale]{\fill[vacuumcolor]  (0,0) rectangle (24,12);\draw[thick, double]  (6,0) -- (0,0) -- (0,12) -- (6,12);
\draw (6,12) -- (18,12)  (6,0) -- (18,0);\draw[ultra thick] (18,0) -- (24,0) -- (24,12) -- (18,12);}$ for the Hilbert space
$L^2(\tikzmath[scale=\textscale]{\useasboundingbox (-2,0) rectangle (26,14);\draw[thick, double] (0,6) -- (0,12) -- (6,12);\draw (6,12) -- (18,12);\draw[ultra thick] (18,12) -- (24,12) -- (24,6);})
= L^2((D \circledast_{\calb} E)(S^1_\top))$
that appears in the left-hand side of \eqref{eq: main equation of thm G=L2}.
Combining Proposition~\ref{prop:G=L2} (see also Remark \ref{rem: we assumed irreducibility}) and Proposition~\ref{prop: local-fusion}, we can construct an isometric map
\begin{equation}\label{eq: definition of Omega}
\Omega\;:\;\xymatrix@C=1.2cm{
\tikzmath[scale=\displscale]{\fill[vacuumcolor]  (0,0) rectangle (24,12);\draw[thick, double]  (6,0) -- (0,0) -- (0,12) -- (6,12);
\draw (6,12) -- (18,12)  (6,0) -- (18,0);\draw[ultra thick] (18,0) -- (24,0) -- (24,12) -- (18,12);} %tikzmath
\ar[r]^{\Psi_{D,E}}   &
\tikzmath[scale=\displscale]{\fill[vacuumcolor] (0,0) rectangle (10,12)(14,0) rectangle (24,12);\draw (6,12) -- (10,12) -- (10,0) -- (6,0)(18,12) -- (14,12) -- (14,0) -- (18,0);
\draw[thick, double] (6,12) -- (0,12) -- (0,0) -- (6,0);\draw[ultra thick]  (18,12) -- (24,12) -- (24,0) -- (18,0);\fill[vacuumcolor]  (10,0) rectangle (14,4)(10,8) rectangle (14,12);
\draw (10,0) rectangle (14,4) (10,8) rectangle (14,12);\fill[vacuumcolor]  (10.5,4.5) rectangle (13.5,7.5);\draw (10.5,4.5) rectangle (13.5,7.5);} %tikzmath
\ar[r]^{\Phi^{-1}}_{\cong}   &
\tikzmath[scale=\displscale] {\fill[vacuumcolor] (0,0) rectangle (24,12);\draw[thick, double]  (6,0) -- (0,0) -- (0,12) -- (6,12);
\draw (6,12) -- (12,12) -- (12,0) -- (6,0)(18,12) -- (12,12) -- (12,0) -- (18,0);\draw[ultra thick] (18,0) -- (24,0) -- (24,12) -- (18,12);} %tikzmath
},\hspace{.71cm} %xymatrix 
\end{equation}
where $\Phi$ stands for $\Phi_{H_0(S_l,D),H_0(S_r,E)}$.
We will show later, in Theorem \ref{thm: Omega is an iso}, that the map $\Omega=\Omega_{D,E}$ is in fact an isomorphism.  This map is the fundamental ``$1 \boxtimes 1 = 1$ isomorphism" comparing $1_{D \circledast E}$ with $1_D \boxtimes 1_E$.

\begin{proposition}\label{prop: associativity of Omega}
Let $\cala$, $\calb$, $\calc$, $\cald$ be conformal nets, of which the second and third are assumed to be finite, and
let ${}_\cala D_\calb$, ${}_\calb E_\calc$, ${}_\calc F_\cald$ be defects. 
Then the maps
\[
\tikzmath{
\matrix [matrix of math nodes,column sep=1.5cm,row sep=5mm]
{|(a)|\tikzmath[scale=\displscale]{\fill[vacuumcolor](0,0) rectangle (36,12);\draw[thick, double] (6,0)--(0,0)--(0,12)--(6,12);
\draw[ultra thin](6,12)--(18,12)(6,0)--(18,0);\draw[thick](18,12)--(30,12)(18,0)--(30,0);\draw[ultra thick] (30,0)--(36,0)--(36,12)--(30,12);}%tikzmath
\pgfmatrixnextcell |(b)|\tikzmath[scale=\displscale]{\fill[vacuumcolor](0,0) rectangle (36,12);\draw[thick, double] (6,0)--(0,0)--(0,12)--(6,12);
\draw[ultra thin](6,12)--(18,12)(6,0)--(18,0);\draw[thick](18,12)--(30,12)(18,0)--(30,0)(24,0)--(24,12);\draw[ultra thick] (30,0)--(36,0)--(36,12)--(30,12);}\\%tikzmath
|(c)|\tikzmath[scale=\displscale]{\fill[vacuumcolor](0,0) rectangle (36,12);\draw[thick, double] (6,0)--(0,0)--(0,12)--(6,12);
\draw[ultra thin](6,12)--(18,12)(6,0)--(18,0)(12,0)--(12,12);\draw[thick](18,12)--(30,12)(18,0)--(30,0);\draw[ultra thick] (30,0)--(36,0)--(36,12)--(30,12);}%tikzmath
\pgfmatrixnextcell |(d)|\tikzmath[scale=\displscale]{\fill[vacuumcolor](0,0) rectangle (36,12);\draw[thick, double] (6,0)--(0,0)--(0,12)--(6,12);
\draw[ultra thin](6,12)--(18,12)(6,0)--(18,0)(12,0)--(12,12);\draw[thick](18,12)--(30,12)(18,0)--(30,0)(24,0)--(24,12);\draw[ultra thick] (30,0)--(36,0)--(36,12)--(30,12);}\\%tikzmath
};
\draw[->](a)--node[above]{$\scriptstyle\Omega_{D\circledast E,F}$}(b);\draw[->](c)--node[above]{$\scriptstyle \id \boxtimes\Omega_{E,F}$}(d);
\draw[->](a)--node[left]{$\scriptstyle\Omega_{D,E\circledast F}$}(c);\draw[->](b)--node[right]{$\scriptstyle\Omega_{D,E}\boxtimes \id$}(d);}
\]
form a commutative diagram.
\end{proposition}
\begin{proof}
By the definition of $\Omega$, the above diagram can be expanded to
\[
\tikzmath{
\matrix [matrix of math nodes,column sep=1.5cm,row sep=5mm]
{|(a)|\tikzmath[scale=\displscale]{\fill[vacuumcolor](0,0) rectangle (36,12);\draw[thick, double] (6,0)--(0,0)--(0,12)--(6,12);
\draw[ultra thin](6,12)--(18,12)(6,0)--(18,0);\draw[thick](18,12)--(30,12)(18,0)--(30,0);\draw[ultra thick] (30,0)--(36,0)--(36,12)--(30,12);}%tikzmath
\pgfmatrixnextcell 
|(b)|\tikzmath[scale=\displscale]{\fill[vacuumcolor](0,0) rectangle (36,12);\draw[thick, double] (6,0)--(0,0)--(0,12)--(6,12);\fill[white](22,4) rectangle (26,8);
\draw[ultra thin](6,12)--(18,12)(6,0)--(18,0);\draw[thick](18,12)--(30,12)(18,0)--(30,0)(22,0)--(22,12)(26,0)--(26,12)(22,4)--(26,4)(22,8)--(26,8);
\draw[ultra thick] (30,0)--(36,0)--(36,12)--(30,12);\filldraw[thick, fill=vacuumcolor]  (22.7,4.7) rectangle (25.3,7.3);}%tikzmath
\pgfmatrixnextcell 
|(c)|\tikzmath[scale=\displscale]{\fill[vacuumcolor](0,0) rectangle (36,12);\draw[thick, double] (6,0)--(0,0)--(0,12)--(6,12);
\draw[ultra thin](6,12)--(18,12)(6,0)--(18,0);\draw[thick](18,12)--(30,12)(18,0)--(30,0)(24,0)--(24,12);\draw[ultra thick] (30,0)--(36,0)--(36,12)--(30,12);}\\%tikzmath
|(d)|\tikzmath[scale=\displscale]{\fill[vacuumcolor](0,0) rectangle (36,12);\draw[thick, double] (6,0)--(0,0)--(0,12)--(6,12);\fill[white](10,4) rectangle (14,8);
\draw[ultra thin](6,12)--(18,12)(6,0)--(18,0)(10,0)--(10,12)(14,0)--(14,12)(10,4)--(14,4)(10,8)--(14,8);\draw[thick](18,12)--(30,12)(18,0)--(30,0);
\draw[ultra thick] (30,0)--(36,0)--(36,12)--(30,12);\filldraw[ultra thin, fill=vacuumcolor]  (10.4,4.4) rectangle (13.6,7.6);}%tikzmath
\pgfmatrixnextcell 
|(e)|\tikzmath[scale=\displscale]{\fill[vacuumcolor](0,0) rectangle (36,12);\draw[thick, double] (6,0)--(0,0)--(0,12)--(6,12);\fill[white](22,4) rectangle (26,8);\fill[white](10,4) rectangle (14,8);
\draw[ultra thin](6,12)--(18,12)(6,0)--(18,0)(10,0)--(10,12)(14,0)--(14,12)(10,4)--(14,4)(10,8)--(14,8);
\draw[thick](18,12)--(30,12)(18,0)--(30,0)(22,0)--(22,12)(26,0)--(26,12)(22,4)--(26,4)(22,8)--(26,8);\draw[ultra thick] (30,0)--(36,0)--(36,12)--(30,12);
\filldraw[thick, fill=vacuumcolor]  (22.7,4.7) rectangle (25.3,7.3);\filldraw[ultra thin, fill=vacuumcolor]  (10.4,4.4) rectangle (13.6,7.6);}%tikzmath
\pgfmatrixnextcell 
|(f)|\tikzmath[scale=\displscale]{\fill[vacuumcolor](0,0) rectangle (36,12);\draw[thick, double] (6,0)--(0,0)--(0,12)--(6,12);\fill[white](10,4) rectangle (14,8);
\draw[ultra thin](6,12)--(18,12)(6,0)--(18,0)(10,0)--(10,12)(14,0)--(14,12)(10,4)--(14,4)(10,8)--(14,8);\draw[thick](18,12)--(30,12)(18,0)--(30,0)(24,0)--(24,12);
\draw[ultra thick] (30,0)--(36,0)--(36,12)--(30,12);\filldraw[ultra thin, fill=vacuumcolor]  (10.4,4.4) rectangle (13.6,7.6);}\\%tikzmath
|(g)|\tikzmath[scale=\displscale]{\fill[vacuumcolor](0,0) rectangle (36,12);\draw[thick, double] (6,0)--(0,0)--(0,12)--(6,12);
\draw[ultra thin](6,12)--(18,12)(6,0)--(18,0)(12,0)--(12,12);\draw[thick](18,12)--(30,12)(18,0)--(30,0);\draw[ultra thick] (30,0)--(36,0)--(36,12)--(30,12);}%tikzmath
\pgfmatrixnextcell
|(h)|\tikzmath[scale=\displscale]{\fill[vacuumcolor](0,0) rectangle (36,12);\draw[thick, double] (6,0)--(0,0)--(0,12)--(6,12);\fill[white](22,4) rectangle (26,8);
\draw[ultra thin](6,12)--(18,12)(6,0)--(18,0)(12,0)--(12,12);\draw[thick](18,12)--(30,12)(18,0)--(30,0)(22,0)--(22,12)(26,0)--(26,12)(22,4)--(26,4)(22,8)--(26,8);
\draw[ultra thick] (30,0)--(36,0)--(36,12)--(30,12);\filldraw[thick, fill=vacuumcolor]  (22.7,4.7) rectangle (25.3,7.3);}%tikzmath
\pgfmatrixnextcell
|(i)|\tikzmath[scale=\displscale]{\fill[vacuumcolor](0,0) rectangle (36,12);\draw[thick, double] (6,0)--(0,0)--(0,12)--(6,12);
\draw[ultra thin](6,12)--(18,12)(6,0)--(18,0)(12,0)--(12,12);\draw[thick](18,12)--(30,12)(18,0)--(30,0)(24,0)--(24,12);\draw[ultra thick] (30,0)--(36,0)--(36,12)--(30,12);}\\%tikzmath
};
\foreach \source/\target/\location in {a/b/above,d/e/above,g/h/above,a/d/right,b/e/right,c/f/right}
{\draw[->] (\source) --node[\location]{$\scriptstyle\Psi$} (\target);}
\foreach \source/\target/\location in {b/c/above,e/f/above,h/i/above,d/g/right,e/h/right,f/i/right}
{\draw[->] (\source) --node[\location]{$\scriptstyle\Phi^{-1}$} (\target);}}
\]
The upper left square commutes by Lemma \ref{lem: Associativity of Psi} (see also Remark \ref{rem: we assumed irreducibility}).
The remaining three squares commute by the naturality of $\Phi^{-1}$.
\end{proof}

%==================================================================

\chapter{Haag duality for composition of defects}\label{sec: Haag duality for composition of defects}

Throughout this section we fix conformal nets $\cala$, $\calb$, and $\calc$, always assumed to be irreducible, and irreducible defects ${}_\cala D_\calb$ and ${}_\calb E_\calc$.
In our pictures, we will use the notation \tikz{\draw[thick, double] (0,0) -- (.5,0);}  for intervals on which we evaluate $\cala$, 
we will use \tikz{\draw (0,0) -- (.5,0);}  for intervals on which we evaluate $\calb$,
and \tikz{\draw[ultra thick] (0,0) -- (.5,0);}  for intervals on which we evaluate $\calc$.
We will also use \tikz{\draw[thick, double] (0,0) -- (.3,0);\draw (.3,0) -- (.6,0);} for bicolored intervals on which we evaluate $D$, and
\tikz{\draw (0,0) -- (.3,0);\draw[ultra thick] (.3,0) -- (.6,0);} for bicolored intervals on which we evaluate $E$.

{Let $S_l$ and $S_r$ be as in \eqref{eq: figures of intervals 3} and \eqref{eq: def colors of S and tildeS},
with intersection $I$ oriented like $S_r$.
\def\SingleLtwoWithTwoColors#1#2{L^2 \hspace{-.5ex}\left( \tikzmath[scale=\textscale] {\draw[#1] (12,6) -- (12,12) -- (18,12);\draw[#2] (18,12) -- (24,12) -- (24,6);} \right)}
As before, we use the notation $\tikzmath[scale=\textscale]{\fill[vacuumcolor] (0,0) rectangle (12,12);
\draw[double, thick](6,0) -- (0,0) -- (0,12) -- (6,12); \draw (6,0) -- (12,0) -- (12,12) -- (6,12);} := H_0(S_l,D) = \SingleLtwoWithTwoColors{double, thick}{}$,
similarly $\tikzmath[scale=\textscale]{\fill[vacuumcolor] (12,0) rectangle (24,12);
\draw(18,0) -- (12,0) -- (12,12) -- (18,12); \draw[ultra thick] (18,0) -- (24,0) -- (24,12) -- (18,12);} := H_0(S_r,E) = \SingleLtwoWithTwoColors{}{ultra thick}$,
and $\tikzmath[scale=\textscale]{\fill[vacuumcolor] (0,0) rectangle (24,12);
\draw[double, thick](6,0) -- (0,0) -- (0,12) -- (6,12); \draw (6,0) -- (12,0) -- (12,12) -- (6,12)(18,0) -- (12,0)(12,12) -- (18,12);
\draw[ultra thick] (18,0) -- (24,0) -- (24,12) -- (18,12);} := H_0(S_l,D)\boxtimes_{\calb(I)}H_0(S_r,E)$.
We will again be using the Notation~\ref{not:names-for-intervals}. 
Letting}
\begin{equation}\label{eq: I_1234}
\begin{split}
I_1:=\partial^{\,\tikzmath[scale=.15]{\draw(0,0)--(0,1)--(1,1);}}([0,1]\times[\tfrac12,1]) &\qquad (I_1)_\circ :=  (I_1)_{x\le \frac12}  \quad (I_1)_\bullet :=  (I_1)_{x\ge \frac12}  \\
I_2:=\partial^{\,\tikzmath[scale=.15]{\draw(1,0)--(1,1)--(0,1);}}([1,2]\times[\tfrac12,1]) &\qquad (I_2)_\circ :=  (I_2)_{x\le \frac32}  \quad (I_2)_\bullet :=  (I_2)_{x\ge \frac32}  \\
I_3:=\partial^{\,\tikzmath[scale=.15]{\draw(0,1)--(0,0)--(1,0);}}([0,1]\times[0,\tfrac12]) &\qquad (I_3)_\circ :=  (I_3)_{x\le \frac12}  \quad (I_3)_\bullet :=  (I_3)_{x\ge \frac12}  \\
I_4:=\partial^{\,\tikzmath[scale=.15]{\draw(0,0)--(1,0)--(1,1);}}([1,2]\times[\tfrac12,1]) &\qquad (I_4)_\circ :=  (I_4)_{x\le \frac32}  \quad (I_4)_\bullet :=  (I_4)_{x\ge \frac32},
\end{split}
\end{equation}
we will write 
$D ( \tikzmath[scale=\textscale]{\useasboundingbox (-2,0) rectangle (14,14); \draw[double, thick](0,6) -- (0,12) -- (6,12);\draw  (6,12) -- (12,12);}%tikzmath
)$, 
$E ( \tikzmath[scale=\textscale]{\useasboundingbox (10,0) rectangle (26,14); \draw  (12,12) -- (18,12);\draw[ultra thick](18,12) -- (24,12) -- (24,6);}%tikzmath
)$, 
$D ( \tikzmath[scale=\textscale]{\useasboundingbox (-2,-2) rectangle (14,12); \draw[double, thick](0,6) -- (0,0) -- (6,0);\draw  (6,0) -- (12,0);}%tikzmath
)$, 
$E ( \tikzmath[scale=\textscale]{\useasboundingbox (10,-2) rectangle (26,12); \draw  (12,0) -- (18,0);\draw[ultra thick](18,0) -- (24,0) -- (24,6);}%tikzmath
)$
for $D(I_1)$, $E(I_2)$, $D(I_3)$, $E(I_4)$, respectively.

\begin{maintheorem}\label{thm:Haag-duality-composition-defects}
Assuming $\calb$ has finite index, then on the Hilbert space 
$\tikzmath[scale=\textscale]{\fill[vacuumcolor] (0,0) rectangle (24,12);
\draw[double, thick](6,0) -- (0,0) -- (0,12) -- (6,12); \draw (6,0) -- (12,0) -- (12,12) -- (6,12)(18,0) -- (12,0)(12,12) -- (18,12);
\draw[ultra thick] (18,0) -- (24,0) -- (24,12) -- (18,12);}$\,,
we have 
\begin{equation}\label{eq:thm:Haag-duality-composition-defects}
D\, \Big(\, \tikzmath[scale=\displscale]{\useasboundingbox (-2,0) rectangle (14,14);\draw[double, thick](0,6) -- (0,12) -- (6,12);\draw  (6,12) -- (12,12);}%tikzmath
\,\Big) \vee E\, \Big(\, \tikzmath[scale=\displscale]{\useasboundingbox (10,0) rectangle (26,14);\draw  (12,12) -- (18,12);\draw[ultra thick](18,12) -- (24,12) -- (24,6);}%tikzmath
\,\Big) \;=\; \bigg(\! D\, \Big(\, \tikzmath[scale=\displscale] {\useasboundingbox (-2,-2) rectangle (14,12); \draw[double, thick](0,6) -- (0,0) -- (6,0);\draw  (6,0) -- (12,0);}%tikzmath
\,\Big) \vee E\, \Big(\, \tikzmath[scale=\displscale]{\useasboundingbox (10,-2) rectangle (26,12); \draw (12,0) -- (18,0);\draw[ultra thick](18,0) -- (24,0) -- (24,6);}%tikzmath
\,\Big)\bigg)'.
\end{equation}
\end{maintheorem}

\begin{proof}
Let us introduce some notation for various algebras that act on the Hilbert space $\tikzmath[scale=\textscale]{\fill[vacuumcolor] (0,0) rectangle (24,12);
\draw[double, thick](6,0) -- (0,0) -- (0,12) -- (6,12); \draw (6,0) -- (12,0) -- (12,12) -- (6,12)(18,0) -- (12,0)(12,12) -- (18,12);
\draw[ultra thick] (18,0) -- (24,0) -- (24,12) -- (18,12);}$\,.
The main algebras of interest are
$\tikzmath[scale=\textscale]{\useasboundingbox (-2,2) rectangle (26,14);\draw[thick, double] (0,6) -- (0,12) -- (6,12);\draw (6,12) -- (18,12);\draw[ultra thick] (18,12) -- (24,12) -- (24,6);}
=(D \circledast_{\calb} E)(S^1_\top)$
and
$\tikzmath[scale=\textscale]{\useasboundingbox (-2,-2) rectangle (26,12);\draw[thick, double] (0,6)--(0,0)--(6,0); \draw (6,0)--(18,0);\draw[ultra thick] (18,0)--(24,0)--(24,6);}
=(D \circledast_{\calb} E)(S^1_\bot)$,
and our goal is to show that the inclusion
\begin{equation}\label{eq: our goal is to show that the inclusion is an isomorphism}
\tikzmath[scale=\displscale]{\useasboundingbox (-2,-2) rectangle (26,14);\draw[thick, double] (0,6) -- (0,12) -- (6,12);\draw (6,12) -- (18,12);\draw[ultra thick] (18,12) -- (24,12) -- (24,6);}%tikzmath
\;\subseteq\;\left(\tikzmath[scale=\displscale]{\useasboundingbox (-2,-2) rectangle (26,14);\draw[thick, double] (0,6)--(0,0)--(6,0); \draw (6,0)--(18,0);\draw[ultra thick] (18,0)--(24,0)--(24,6);}%tikzmath
\right)'
\end{equation}
is an isomorphism.
Let us fix once and for all a small number $\epsilon$.
Consider the 1-manifolds
\begin{equation*} %\label{eq:JJJJ1234}
\begin{split}
J_0&:=\partial^{\,\tikzmath[scale=.15]{\draw(0,0)--(0,1)--(1,1);}}\big([0,\tfrac12+\epsilon]\times[\tfrac12,1]\big),\\
%(\{0\}\times [1/2,1])\cup([0,1/2+\epsilon]\times\{1\}),\\
J_1&:=\big([\tfrac12+\epsilon,\tfrac23]\cup[\tfrac56,1]\big)\times \{1\},\\
J_2&:=\big([1,\tfrac32-\epsilon]\times \{1\}\big)\cup \big([\tfrac76,\tfrac43]\times \{0\}\big),\\
J_3&:=\partial^{\,\tikzmath[scale=.15]{\draw(1,0)--(1,1)--(0,1);}}\big([\tfrac32-\epsilon,2]\times[\tfrac12,1]\big).
%([3/2-\epsilon,2]\times\{1\})\cup(\{2\}\times [1/2,1]).
\end{split}
\end{equation*} %!%CD: include pictures?
We will use the following algebras:
\begin{gather}
\label{eq: BOTTOM of 12345} \tikzmath[scale=\displscale]{\useasboundingbox (-2,-2) rectangle (32,14);\draw[thick, double]  (0,6) -- (0,12) -- (6,12);
\draw (6,12) -- (9,12) (12,12) -- (24,12) (18,0) -- (21,0);\draw[ultra thick]  (24,12) -- (30,12) -- (30,6);} %tikzmath
:= D(J_0)\vee \calb(J_1)\vee \calb(J_2)\vee E(J_3)=D(J_0\cup J_1)\vee E(J_2\cup J_3)\\
\label{eq: LOWER LEFT of 12345} \tikzmath[scale=\displscale]{\useasboundingbox (-2,-2) rectangle (32,14);\draw[thick, double]  (0,6) -- (0,12) -- (6,12);           
\draw (6,12) -- (9,12) (12,12) -- (24,12) (18,0) -- (21,0);\draw[ultra thick]   (24,12) -- (30,12) -- (30,6);\draw[densely dotted] (7.5,12) arc (180:360:3 and 4);} %tikzmath
:= D(J_0)\vee \hat\calb(J_1)\vee \calb(J_2)\vee E(J_3)=\hat D(J_0\cup J_1)\vee E(J_2\cup J_3)\\
\label{eq: RIGHT of 12345} \tikzmath[scale=\displscale]{\useasboundingbox (-2,-2) rectangle (32,14);\draw[thick, double]  (0,6) -- (0,12) -- (6,12);
\draw (6,12) -- (9,12) (12,12) -- (24,12) (18,0) -- (21,0);\draw[ultra thick]   (24,12) -- (30,12) -- (30,6);\draw[densely dotted]  (19.5,0) -- (19.5,12);} %tikzmath
:= D(J_0)\vee \calb(J_1)\vee \hat\calb(J_2)\vee E(J_3)=D(J_0\cup J_1)\vee \hat E(J_2\cup J_3)\\
\label{eq: TOP of 12345} \tikzmath[scale=\displscale]{\useasboundingbox (-2,-2) rectangle (32,14);\draw[thick, double]  (0,6) -- (0,12) -- (6,12);
\draw (6,12) -- (9,12) (12,12) -- (24,12) (18,0) -- (21,0);\draw[ultra thick]  (24,12) -- (30,12) -- (30,6);\draw[densely dotted]   (7.5,12) arc (180:360:3 and 4)(19.5,0) -- (19.5,12);} %tikzmath
:= D(J_0)\vee \hat\calb(J_1)\vee \hat\calb(J_2)\vee E(J_3)=\hat D(J_0\cup J_1)\vee \hat E(J_2\cup J_3).
\end{gather}
Here 
$\hat \calb$, $\hat D$, and $\hat E$ are as in \ref{not: Dhat},
and $J_0$ and $J_3$ are bicolored as in \eqref{eq: def colors of S and tildeS}.
By Lemma~\ref{lem: needed for dotted lines}, the algebras $\hat \calb(J_1)$ and $\hat \calb(J_2)$ act on
$\tikzmath[scale=\textscale]{\fill[vacuumcolor] (0,0) rectangle (12,12);
\draw[double, thick](6,0) -- (0,0) -- (0,12) -- (6,12); \draw (6,0) -- (12,0) -- (12,12) -- (6,12);}$
and
$\tikzmath[scale=\textscale]{\fill[vacuumcolor] (0,0) rectangle (12,12);
\draw (6,0) -- (0,0) -- (0,12) -- (6,12);
\draw[ultra thick] (6,0) -- (12,0) -- (12,12) -- (6,12);}$
respectively, and satisfy 
$D(J_0)\vee\hat\calb(J_1)=\hat D(J_0\cup J_1)$ and $\hat\calb(J_2)\vee E(J_3)=\hat E(J_2\cup J_3)$.
The equalities in \eqref{eq: LOWER LEFT of 12345}--\eqref{eq: TOP of 12345} 
%for actions on
%$\tikzmath[scale=\textscale]{\fill[vacuumcolor] (0,0) rectangle (24,12);
%\draw[double, thick](6,0) -- (0,0) -- (0,12) -- (6,12); \draw (6,0) -- (12,0) -- (12,12) -- (6,12)(18,0) -- (12,0)(12,12) -- (18,12);
%\draw[ultra thick] (18,0) -- (24,0) -- (24,12) -- (18,12);}$
follow.

In Section~\ref{sec:haaginclusion} below we will obtain some purchase on the Haag inclusion
$\tikzmath[scale=\textscale]{\useasboundingbox (-2,2) rectangle (26,14);\draw[thick, double] (0,6) -- (0,12) -- (6,12);\draw (6,12) -- (18,12);\draw[ultra thick] (18,12) -- (24,12) -- (24,6);}
\subseteq
(\tikzmath[scale=\textscale]{\useasboundingbox (-2,-2) rectangle (26,12);\draw[thick, double] (0,6)--(0,0)--(6,0); \draw (6,0)--(18,0);\draw[ultra thick] (18,0)--(24,0)--(24,6);})'$
by showing that its statistical dimension is the same as that of the inclusion
$\tikzmath[scale=\textscale]{\useasboundingbox (-2,-2) rectangle (32,14);\draw[thick, double]  (0,6) -- (0,12) -- (6,12);           
\draw (6,12) -- (9,12) (12,12) -- (24,12) (18,0) -- (21,0);\draw[ultra thick]   (24,12) -- (30,12) -- (30,6);\draw[dash pattern=on .4pt off .62pt] (7.5,12) arc (180:360:3 and 4);}
\subseteq
(\tikzmath[scale=\textscale]{\useasboundingbox (-2,-2) rectangle (32,14);\draw[thick, double]  (0,6) -- (0,0) -- (6,0);
\draw (6,0) -- (18,0) (21,0) -- (24,0) (9,12) -- (12,12);\draw[ultra thick]  (24,0) -- (30,0) -- (30,6);\draw[dash pattern=on .4pt off .62pt] (16.5,0) arc (180:0:3 and 4);})'$.  (Here the algebra $\tikzmath[scale=\textscale]{\useasboundingbox (-2,-2) rectangle (32,14);\draw[thick, double]  (0,6) -- (0,0) -- (6,0);
\draw (6,0) -- (18,0) (21,0) -- (24,0) (9,12) -- (12,12);\draw[ultra thick]  (24,0) -- (30,0) -- (30,6);\draw[dash pattern=on .4pt off .62pt] (16.5,0) arc (180:0:3 and 4);}$
is defined similarly to $\tikzmath[scale=\textscale]{\useasboundingbox (-2,-2) rectangle (32,14);\draw[thick, double]  (0,6) -- (0,12) -- (6,12);           
\draw (6,12) -- (9,12) (12,12) -- (24,12) (18,0) -- (21,0);\draw[ultra thick]   (24,12) -- (30,12) -- (30,6);\draw[dash pattern=on .4pt off .62pt] (7.5,12) arc (180:360:3 and 4);}$.) We can compute the statistical dimension of that latter inclusion by squeezing it into a sequence of simpler inclusions of von Neumann algebras, as follows:
\begin{equation*}    %\label{eq:proof-haag-duality-defects}
%TOP
\tikzmath[scale=\displscale]{\node (a) at (0,35) {$\tikzmath[scale=\displscale]{\useasboundingbox (-2,-2) rectangle (32,14);\draw[thick, double]  (0,6) -- (0,12) -- (6,12);
\draw (6,12) -- (9,12) (12,12) -- (24,12) (18,0) -- (21,0);\draw[ultra thick]  (24,12) -- (30,12) -- (30,6);\draw[densely dotted]   (7.5,12) arc (180:360:3 and 4)(19.5,0) -- (19.5,12); }$}; %tikzmath
%LOWER LEFT
\node (b) at (-60,-18){$\tikzmath[scale=\displscale]{\useasboundingbox (-2,-2) rectangle (32,14);\draw[thick, double]  (0,6) -- (0,12) -- (6,12);           
\draw (6,12) -- (9,12) (12,12) -- (24,12) (18,0) -- (21,0);\draw[ultra thick]   (24,12) -- (30,12) -- (30,6);\draw[densely dotted] (7.5,12) arc (180:360:3 and 4);}$}; %tikzmath
%UPPER LEFT
\node (c) at (-60,18){$\left(\tikzmath[scale=\displscale]{\useasboundingbox (-2,-2) rectangle (32,14);\draw[thick, double]  (0,6) -- (0,0) -- (6,0);
\draw (6,0) -- (18,0) (21,0) -- (24,0) (9,12) -- (12,12);\draw[ultra thick]  (24,0) -- (30,0) -- (30,6);\draw[densely dotted] (16.5,0) arc (180:0:3 and 4);}%tikzmath
\right)'$};
%RIGHT
\node (d) at (50,0){$\tikzmath[scale=\displscale]{\useasboundingbox (-2,-2) rectangle (32,14);\draw[thick, double]  (0,6) -- (0,12) -- (6,12);
\draw (6,12) -- (9,12) (12,12) -- (24,12) (18,0) -- (21,0);\draw[ultra thick]   (24,12) -- (30,12) -- (30,6);\draw[densely dotted]  (19.5,0) -- (19.5,12);}$}; %tikzmath
%BOTTOM
\node (e) at (0,-35){$\tikzmath[scale=\displscale]{\useasboundingbox (-2,-2) rectangle (32,14);\draw[thick, double]  (0,6) -- (0,12) -- (6,12);
\draw (6,12) -- (9,12) (12,12) -- (24,12) (18,0) -- (21,0);\draw[ultra thick]  (24,12) -- (30,12) -- (30,6);}$}; %tikzmath
\foreach \source/\target/\position/\number/\sep in {b/a/north west/2/2,b/c/east/5/5,c/a/south/6/5,d/a/south west/4/2,e/b/north/1/5,e/d/north west/3/2}
{\draw[->] (\source) --node[anchor=\position, inner sep = \sep]{$\scriptstyle (\number)$} (\target);}
}%tikzmath
\end{equation*}

Because $D$ and $E$ are irreducible, the algebra \eqref{eq: BOTTOM of 12345} is a factor.  Using Lemma~\ref{lem: needed for dotted lines}, note that the algebra 
\!
$\tikzmath[scale=\textscale]{\useasboundingbox (8.5,-2) rectangle (32,14);\draw(12,12) -- (24,12) (18,0) -- (21,0);
\draw[ultra thick](24,12) -- (30,12) -- (30,6);\draw[dash pattern=on .4pt off .62pt](19.5,0) -- (19.5,12);}$
(the right connected component of the picture \eqref{eq: RIGHT of 12345}) is the commutant of a factor acting on a vacuum sector for $E$; it follows that \eqref{eq: RIGHT of 12345} is a factor.  More difficult is the fact that \eqref{eq: TOP of 12345} is a factor---that is the content of Corollary \ref{cor: The algebras ... are factors}, following from Lemma \ref{lem:commutants} below.  The algebra \eqref{eq: LOWER LEFT of 12345} is not a factor, but combining Lemma \ref{lem:equality of XxX matrices} below and Theorem~\ref{thm: semi-simplicity of DoE}, we will learn that it does have finite-dimensional center; let $n$ be the dimension of this center.

Let $\nu_1,\ldots,\nu_6$ be the matrices of statistical dimensions 
(Appendix~\ref{subsec:stat-dim+minimal-index})
of the various inclusions in the above diagram:\medskip
\[\medskip
\begin{split}
\nu_1\,:=\, \left\llbracket\,
\tikzmath[scale=\planscale]{\useasboundingbox (-2,-2) rectangle (32,14);\draw[thick, double]  (0,6) -- (0,12) -- (6,12);           
\draw (6,12) -- (9,12) (12,12) -- (24,12) (18,0) -- (21,0);\draw[ultra thick]   (24,12) -- (30,12) -- (30,6);\draw[dash pattern=on .4pt off .62pt] (7.5,12) arc (180:360:3 and 4);}   \,:\,
\tikzmath[scale=\planscale]{\useasboundingbox (-2,-2) rectangle (32,14);\draw[thick, double]  (0,6) -- (0,12) -- (6,12);
\draw (6,12) -- (9,12) (12,12) -- (24,12) (18,0) -- (21,0);\draw[ultra thick]  (24,12) -- (30,12) -- (30,6);}
\,\right\rrbracket\qquad&\,
\nu_2\,:=\, \left\llbracket\,
\tikzmath[scale=\planscale]{\useasboundingbox (-2,-2) rectangle (32,14);\draw[thick, double]  (0,6) -- (0,12) -- (6,12);
\draw (6,12) -- (9,12) (12,12) -- (24,12) (18,0) -- (21,0);\draw[ultra thick]  (24,12) -- (30,12) -- (30,6);\draw[dash pattern=on .4pt off .62pt]   (7.5,12) arc (180:360:3 and 4)(19.5,0) -- (19.5,12);}   \,:\,
\tikzmath[scale=\planscale]{\useasboundingbox (-2,-2) rectangle (32,14);\draw[thick, double]  (0,6) -- (0,12) -- (6,12);           
\draw (6,12) -- (9,12) (12,12) -- (24,12) (18,0) -- (21,0);\draw[ultra thick]   (24,12) -- (30,12) -- (30,6);\draw[dash pattern=on .4pt off .62pt] (7.5,12) arc (180:360:3 and 4);}
\,\right\rrbracket\\
\nu_3\,:=\, \left\llbracket\,
\tikzmath[scale=\planscale]{\useasboundingbox (-2,-2) rectangle (32,14);\draw[thick, double]  (0,6) -- (0,12) -- (6,12);
\draw (6,12) -- (9,12) (12,12) -- (24,12) (18,0) -- (21,0);\draw[ultra thick]   (24,12) -- (30,12) -- (30,6);\draw[dash pattern=on .4pt off .62pt]  (19.5,0) -- (19.5,12);}   \,:\,
\tikzmath[scale=\planscale]{\useasboundingbox (-2,-2) rectangle (32,14);\draw[thick, double]  (0,6) -- (0,12) -- (6,12);
\draw (6,12) -- (9,12) (12,12) -- (24,12) (18,0) -- (21,0);\draw[ultra thick]  (24,12) -- (30,12) -- (30,6);}
\,\right\rrbracket\qquad&\,
\nu_4\,:=\, \left\llbracket\,
\tikzmath[scale=\planscale]{\useasboundingbox (-2,-2) rectangle (32,14);\draw[thick, double]  (0,6) -- (0,12) -- (6,12);
\draw (6,12) -- (9,12) (12,12) -- (24,12) (18,0) -- (21,0);\draw[ultra thick]  (24,12) -- (30,12) -- (30,6);\draw[dash pattern=on .4pt off .62pt]   (7.5,12) arc (180:360:3 and 4)(19.5,0) -- (19.5,12);}   \,:\,
\tikzmath[scale=\planscale]{\useasboundingbox (-2,-2) rectangle (32,14);\draw[thick, double]  (0,6) -- (0,12) -- (6,12);
\draw (6,12) -- (9,12) (12,12) -- (24,12) (18,0) -- (21,0);\draw[ultra thick]   (24,12) -- (30,12) -- (30,6);\draw[dash pattern=on .4pt off .62pt]  (19.5,0) -- (19.5,12);}
\,\right\rrbracket\\
\nu_5:= \Big\llbracket\,
\left(\tikzmath[scale=\planscale]{\useasboundingbox (-2,-2) rectangle (32,14);\draw[thick, double]  (0,6) -- (0,0) -- (6,0);
\draw (6,0) -- (18,0) (21,0) -- (24,0) (9,12) -- (12,12);\draw[ultra thick]  (24,0) -- (30,0) -- (30,6);\draw[dash pattern=on .4pt off .62pt] (16.5,0) arc (180:0:3 and 4);}
\right)'   \,:\,
\tikzmath[scale=\planscale]{\useasboundingbox (-2,-2) rectangle (32,14);\draw[thick, double]  (0,6) -- (0,12) -- (6,12);           
\draw (6,12) -- (9,12) (12,12) -- (24,12) (18,0) -- (21,0);\draw[ultra thick]   (24,12) -- (30,12) -- (30,6);\draw[dash pattern=on .4pt off .62pt] (7.5,12) arc (180:360:3 and 4);}
\,\Big\rrbracket\quad&
\nu_6:= \Big\llbracket\,
\tikzmath[scale=\planscale]{\useasboundingbox (-2,-2) rectangle (32,14);\draw[thick, double]  (0,6) -- (0,12) -- (6,12);
\draw (6,12) -- (9,12) (12,12) -- (24,12) (18,0) -- (21,0);\draw[ultra thick]  (24,12) -- (30,12) -- (30,6);\draw[dash pattern=on .4pt off .62pt]   (7.5,12) arc (180:360:3 and 4)(19.5,0) -- (19.5,12);}   \,:\,
\left(\tikzmath[scale=\planscale]{\useasboundingbox (-2,-2) rectangle (32,14);\draw[thick, double]  (0,6) -- (0,0) -- (6,0);
\draw (6,0) -- (18,0) (21,0) -- (24,0) (9,12) -- (12,12);\draw[ultra thick]  (24,0) -- (30,0) -- (30,6);\draw[dash pattern=on .4pt off .62pt] (16.5,0) arc (180:0:3 and 4);}
\right)'\,\Big\rrbracket.
\end{split}
\]
Note that $\nu_3$ and $\nu_4$ are scalars, $\nu_1$ is a row vector, 
$\nu_2$ and $\nu_6$ are column vectors, and $\nu_5$ is an 
$n\times n$ matrix.
As the matrix of statistical dimensions is 
multiplicative \eqref{eq:matrix-of-stat-dim-for-AcBcC}, it follows that
$\nu_1\,\nu_2=\nu_3\,\nu_4$, and that $\nu_2=\nu_5\,\nu_6$.

We will need the following facts about $\nu_1,\ldots,\nu_6$:
\begin{enumerate}
\item There is an equality of $n \times n$ matrices
\[
\nu_5 \equiv
\Big\llbracket\, \left(\tikzmath[scale=\planscale]{\useasboundingbox (-2,-2) rectangle (32,14);\draw[thick, double]  (0,6) -- (0,0) -- (6,0);
\draw (6,0) -- (18,0) (21,0) -- (24,0) (9,12) -- (12,12);\draw[ultra thick]  (24,0) -- (30,0) -- (30,6);\draw[dash pattern=on .4pt off .62pt] (16.5,0) arc (180:0:3 and 4);}%tikzmath
\right)'   \,:\, \tikzmath[scale=\planscale]{\useasboundingbox (-2,-2) rectangle (32,14);\draw[thick, double]  (0,6) -- (0,12) -- (6,12);           
\draw (6,12) -- (9,12) (12,12) -- (24,12) (18,0) -- (21,0);\draw[ultra thick]   (24,12) -- (30,12) -- (30,6);\draw[dash pattern=on .4pt off .62pt] (7.5,12) arc (180:360:3 and 4);}%tikzmath
\,\Big\rrbracket
\;=\;
\Big\llbracket
\left(\tikzmath[scale=\planscale]{\useasboundingbox (-2,-2) rectangle (26,14);\draw[thick, double] (0,6)--(0,0)--(6,0); \draw (6,0)--(18,0);\draw[ultra thick] (18,0)--(24,0)--(24,6);}%tikzmath
\right)':
\tikzmath[scale=\planscale]{\useasboundingbox (-2,-2) rectangle (26,14);\draw[thick, double] (0,6) -- (0,12) -- (6,12);\draw (6,12) -- (18,12);\draw[ultra thick] (18,12) -- (24,12) -- (24,6);}%tikzmath
\Big\rrbracket.
\]
This is proven in Lemma \ref{lem:equality of XxX matrices} below.
\item The map \eqref{eq: definition of Omega} exhibits $L^2(\tikzmath[scale=\textscale]{\useasboundingbox (-2,0) rectangle (26,14);\draw[thick, double] (0,6) -- (0,12) -- (6,12);\draw (6,12) -- (18,12);\draw[ultra thick] (18,12) -- (24,12) -- (24,6);})$ as a sub-bimodule of \,$\tikzmath[scale=\textscale]{\fill[vacuumcolor] (0,0) rectangle (24,12);
\draw[double, thick](6,0) -- (0,0) -- (0,12) -- (6,12); \draw (6,0) -- (12,0) -- (12,12) -- (6,12)(18,0) -- (12,0)(12,12) -- (18,12);
\draw[ultra thick] (18,0) -- (24,0) -- (24,12) -- (18,12);}$\,. 
Using the additivity of statistical dimension~\eqref{eq:dim-and-sum},
we obtain the \emph{entrywise} matrix inequality
\[
\nu_5\,=\, 
\Big\llbracket
\left(\tikzmath[scale=\planscale]{\useasboundingbox (-2,-2) rectangle (26,14);\draw[thick, double] (0,6)--(0,0)--(6,0); \draw (6,0)--(18,0);\draw[ultra thick] (18,0)--(24,0)--(24,6);}%tikzmath
\right)':
\tikzmath[scale=\planscale]{\useasboundingbox (-2,-2) rectangle (26,14);\draw[thick, double] (0,6) -- (0,12) -- (6,12);\draw (6,12) -- (18,12);\draw[ultra thick] (18,12) -- (24,12) -- (24,6);}%tikzmath
\Big\rrbracket \,\ge\, \text{\bf\large 1}_n\,,
\]
where $\mathbf{1}_n$ denotes the identity matrix. 
\item In Corollary \ref{cor: nu_1=nu_6^t}, we will show that $\nu_1$ is the transpose of $\nu_6$.
\item In Corollary \ref{cor:muB=mu3=mu4},
we will show that $\nu_3=\sqrt\mu$ and $\nu_4=\sqrt\mu$, where $\mu=\mu(\calb)$ is the index of the conformal net $\calb$.
\end{enumerate}

Using these results and the fact that the matrices of statistical dimensions have only non-negative entries, we can now compute that
\[
\|\,\nu_2\|^2\,=\,\nu_2^T\,\nu_2\,=\,\nu_6^T\,\nu_5^T\,\nu_5\,\nu_6\,=\,\nu_1\,\nu_5^T\,\nu_5\,\nu_6\,\ge\,\nu_1\,\nu_5\,\nu_6\,=\,\nu_3\,\nu_4\,=\mu\,,
\]
with equality if and only if $\nu_5=\mathbf 1_n$ (here, $^T$ denotes transpose).
However, by 
applying~\eqref{eq:DIS-727}
to the algebras $M:=\hat\calb(J_2)$, $N:=\calb(J_2)$, 
and $A:=D(J_0)\vee\hat\calb(J_1)\vee E(J_3)$,
we obtain the reverse inequality 
\[
\|\nu_2\|\le \sqrt\mu.
\]
It follows that $\nu_5=\mathbf 1_n$, and 
%by \cite[\propbasicpropertiesofindex]{BDH(Dualizability+Index-of-subfactors)} 
the inclusion~\eqref{eq: our goal is to show that the inclusion is an isomorphism} 
is therefore an isomorphism.
\end{proof}

In Theorem \ref{thm:Haag-duality-composition-defects}, the defects $D$ and $E$ were assumed to be irreducible, but
the statement holds in general:

\begin{corollary}\label{cor:fiber-product-and-nets} 
Let ${}_\cala D_\calb$ and ${}_\calb E_\calc$ be defects.
If the conformal net $\calb$ has finite index,
then the algebra 
$D \big( \tikzmath[scale=\textscale]{\useasboundingbox (-2,0) rectangle (14,14); \draw[double, thick](0,6) -- (0,12) -- (6,12);\draw  (6,12) -- (12,12);}%tikzmath
\big)\vee E \big( \tikzmath[scale=\textscale]{\useasboundingbox (10,0) rectangle (26,14); \draw  (12,12) -- (18,12);\draw[ultra thick](18,12) -- (24,12) -- (24,6);}%tikzmath
\big)$
is the commutant of 
$D \big( \tikzmath[scale=\textscale]{\useasboundingbox (-2,-2) rectangle (14,12); \draw[double, thick](0,6) -- (0,0) -- (6,0);\draw  (6,0) -- (12,0);}%tikzmath
\big)\vee E \big( \tikzmath[scale=\textscale]{\useasboundingbox (10,-2) rectangle (26,12); \draw  (12,0) -- (18,0);\draw[ultra thick](18,0) -- (24,0) -- (24,6);}%tikzmath
\big)$
on $\tikzmath[scale=\textscale]{\fill[vacuumcolor] (0,0) rectangle (24,12);
\draw[double, thick](6,0) -- (0,0) -- (0,12) -- (6,12); \draw (6,0) -- (12,0) -- (12,12) -- (6,12)(18,0) -- (12,0)(12,12) -- (18,12);
\draw[ultra thick] (18,0) -- (24,0) -- (24,12) -- (18,12);}$\,.%\medskip
\end{corollary}

\begin{proof}
We need to show that $D(I_1)\vee E(I_2)=(D(I_3)\vee E(I_4))'$, where the intervals $I_1, I_2, I_3, I_4$ are as in \eqref{eq: I_1234}.
Disintegrating
\[
\textstyle D=\int^\oplus D_x \qquad\text{and}\qquad E=\int^\oplus E_y
\]
into irreducible defects, the Hilbert space
$\tikzmath[scale=\textscale]{\fill[vacuumcolor] (0,0) rectangle (24,12);
\draw[double, thick](6,0) -- (0,0) -- (0,12) -- (6,12); \draw (6,0) -- (12,0) -- (12,12) -- (6,12)(18,0) -- (12,0)(12,12) -- (18,12);
\draw[ultra thick] (18,0) -- (24,0) -- (24,12) -- (18,12);} = H_0(S_l,D)\boxtimes_{\calb(I)}H_0(S_r,E)$
decomposes correspondingly as $\iint^\oplus H_0(S_l,D_x)\boxtimes_{\calb(I)}H_0(S_r,E_y)$.
This induces direct integral decompositions 
\[
\begin{split}
D(I_1)\vee E(I_2)=&\textstyle \iint^\oplus D_x(I_1)\vee E_y(I_2),\\
D(I_3)\vee E(I_4)=&\textstyle \iint^\oplus D_x(I_3)\vee E_y(I_4),
\end{split}
\]
and therefore also
$(D(I_3)\vee E(I_4))'=\iint^\oplus(D_x(I_3)\vee E_y(I_4))' $,
where the commutant of $D_x(I_3)\vee E_y(I_4)$ is taken on $H_0(S_l,D_x)\boxtimes_{\calb(I)}H_0(S_r,E_y)$.
By Theorem~\ref{thm:Haag-duality-composition-defects}, 
we have $D_x(I_1)\vee E_y(I_2)=(D_x(I_3)\vee E_y(I_4))'$, 
and the result follows.
\end{proof}

\section{The dimension of the Haag inclusion} \label{sec:haaginclusion}

Recall the notion of the center of a defect from Section \ref{sec: examples of defects}: for a genuinely bicolored interval $I$,
the algebra $Z(D(I))$ is independent of $I$ (up to canonical isomorphism), and is denoted $Z(D)$.

From now on, the defects $D$ and $E$ are again assumed irreducible:

\begin{lemma} \label{lem:equality of XxX matrices}
Let $X$ %:=\mathrm{Spec}(Z(D\circledast_\calb E))$
 be the set of irreducible summands of ${D\circledast_\calb E}$. %\footnote{Here, we use ``Spec'' in the sense of algebraic geometry.} 
We have a canonical identification of centers
$
Z(D\circledast_\calb E)=
Z(\tikzmath[scale=\textscale]{\useasboundingbox (-2,-2) rectangle (32,12);\draw[thick, double] (0,6) -- (0,0) -- (6,0);\draw(6,0) -- (24,0);\draw[ultra thick] (24,0) -- (30,0) -- (30,6);})=
Z(\tikzmath[scale=\textscale]{\useasboundingbox (-2,0) rectangle (32,14);\draw[thick, double](0,6) -- (0,12) -- (6,12);\draw (6,12) -- (24,12);\draw[ultra thick](24,12) -- (30,12) -- (30,6);})=
Z(\tikzmath[scale=\textscale]{\useasboundingbox (-2,-2) rectangle (32,12);\draw[thick, double] (0,6) -- (0,0) -- (6,0);
\draw (6,0) -- (18,0) (21,0) -- (24,0) (9,10) -- (12,10);\draw[ultra thick] (24,0) -- (30,0) -- (30,6);\draw[dash pattern=on .4pt off .62pt](16.5,0) arc (180:0:3 and 4);})=
Z(\tikzmath[scale=\textscale]{\useasboundingbox (-2,0) rectangle (32,14);\draw[thick, double](0,6) -- (0,12) -- (6,12);
\draw (6,12) -- (9,12) (12,12) -- (24,12) (18,2) -- (21,2);\draw[ultra thick](24,12) -- (30,12) -- (30,6);\draw[dash pattern=on .4pt off .62pt] (7.5,12) arc (180:360:3 and 4);})
$,
and we have the following equality of $X\times X$ matrices:
\label{lem:mu=mu5}
\begin{equation*}
\Big\llbracket\Big(\tikzmath[scale=\planscale]
{\useasboundingbox (-2,-2) rectangle (32,14);\draw[thick, double] (0,6) -- (0,0) -- (6,0);\draw (6,0) -- (24,0);\draw[ultra thick] (24,0) -- (30,0) -- (30,6);} %tikzmath
\Big)' \; : \; \tikzmath[scale=\planscale]
{\useasboundingbox (-2,-2) rectangle (32,14);\draw[thick, double](0,6) -- (0,12) -- (6,12);\draw (6,12) -- (24,12);\draw[ultra thick](24,12) -- (30,12) -- (30,6);} %tikzmath
\Big\rrbracket \; = \; \Big\llbracket\Big(
\tikzmath[scale=\planscale]{\useasboundingbox (-2,-2) rectangle (32,14);\draw[thick, double] (0,6) -- (0,0) -- (6,0);
\draw (6,0) -- (18,0) (21,0) -- (24,0) (9,12) -- (12,12);\draw[ultra thick] (24,0) -- (30,0) -- (30,6);\draw[dash pattern=on .4pt off .62pt](16.5,0) arc (180:0:3 and 4);} %tikzmath
\Big)' \; : \;
\tikzmath[scale=\planscale]{\useasboundingbox (-2,-2) rectangle (32,14);\draw[thick, double](0,6) -- (0,12) -- (6,12);
\draw (6,12) -- (9,12) (12,12) -- (24,12) (18,0) -- (21,0);\draw[ultra thick](24,12) -- (30,12) -- (30,6);\draw[dash pattern=on .4pt off .62pt] (7.5,12) arc (180:360:3 and 4);} %tikzmath
\Big\rrbracket
\end{equation*}
\end{lemma}

\begin{proof}
Note that
$\llbracket A : B\rrbracket = 
  \llbracket A \,\bar{\ox}\, C : B \,\bar{\ox}\, C\rrbracket$ 
whenever $C$ is a factor~\eqref{eq:matrix-of-stat-dim-ox-C}.
The algebras $\calb\big([\tfrac76, \tfrac43]\big)$ and ${\calb\big([\tfrac76-\epsilon, \tfrac43+\epsilon]\big)'}$ are split on 
$H_0(S_r,\calb)$, and hence on any $\calb\big([\tfrac76-\epsilon, \tfrac43+\epsilon]\big)$-module.
Since
$\big(\tikzmath[scale=\textscale]{\useasboundingbox (-2,-2) rectangle (32,12);\draw[thick, double] (0,6) -- (0,0) -- (6,0);\draw(6,0) -- (24,0);\draw[ultra thick] (24,0) -- (30,0) -- (30,6); } %tikzmath
\big)'\subset \calb\big([\tfrac76-\epsilon, \tfrac43+\epsilon]\big)'$
on \,$\tikzmath[scale=\textscale]{\fill[vacuumcolor] (0,0) rectangle (24,12);
\draw[double, thick](6,0) -- (0,0) -- (0,12) -- (6,12); \draw (6,0) -- (12,0) -- (12,12) -- (6,12)(18,0) -- (12,0)(12,12) -- (18,12);
\draw[ultra thick] (18,0) -- (24,0) -- (24,12) -- (18,12);}$\,, it follows that
\[
\Big(\,\tikzmath[scale=\planscale]{\useasboundingbox (-2,-2) rectangle (32,14);\draw[thick, double] (0,6) -- (0,0) -- (6,0);\draw(6,0) -- (24,0);\draw[ultra thick] (24,0) -- (30,0) -- (30,6); } %tikzmath
\,\Big)' \; \vee \; \tikzmath[scale=\planscale]{\useasboundingbox (10,-2) rectangle (27,14);\draw(18,0) -- (21,0);} %tikzmath
=\,\, \Big(\,\tikzmath[scale=\planscale]{\useasboundingbox (-2,-2) rectangle (32,14);\draw[thick, double] (0,6) -- (0,0) -- (6,0);\draw(6,0) -- (24,0);\draw[ultra thick] (24,0) -- (30,0) -- (30,6);} %tikzmath
\,\Big)' \; \bar{\ox} \; \tikzmath[scale=\planscale]{\useasboundingbox (10,-2) rectangle (27,14);\draw(18,0) -- (21,0);}, %tikzmath
\]
where the line 
$\tikzmath[scale=\planscale]{\useasboundingbox (-2,-1) rectangle (5,5);\draw(0,0) -- (3,0);}$ stands for $\calb\big([\tfrac76, \tfrac43] \times \{0\}\big)$.
We conclude that
\begin{eqnarray} \label{eq:mu=mu5:eins}
\lefteqn{   \Big\llbracket\Big(
\tikzmath[scale=\planscale]{\useasboundingbox (-2,-2) rectangle (32,14);\draw[thick, double] (0,6) -- (0,0) -- (6,0);\draw(6,0) -- (24,0);\draw[ultra thick] (24,0) -- (30,0) -- (30,6); } %tikzmath
\Big)' \; : \;
\tikzmath[scale=\planscale]{\useasboundingbox (-2,-2) rectangle (32,14);\draw[thick, double](0,6) -- (0,12) -- (6,12);\draw (6,12) -- (24,12);\draw[ultra thick](24,12) -- (30,12) -- (30,6);} %tikzmath
\Big\rrbracket    }\\ \nonumber & = & \Big\llbracket \Big(
\tikzmath[scale=\planscale]{\useasboundingbox (-2,-2) rectangle (32,14);\draw[thick, double] (0,6) -- (0,0) -- (6,0);\draw(6,0) -- (24,0);\draw[ultra thick] (24,0) -- (30,0) -- (30,6);} %tikzmath
\Big)' \; \bar{\ox} \; \tikzmath[scale=\planscale]{\useasboundingbox (10,-2) rectangle (27,14);\draw(18,0) -- (21,0);} %tikzmath
\; : \; \tikzmath[scale=\planscale]{\useasboundingbox (-2,-2) rectangle (32,14);\draw[thick, double](0,6) -- (0,12) -- (6,12);\draw (6,12) -- (24,12);\draw[ultra thick](24,12) -- (30,12) -- (30,6);}%tikzmath
\; \bar{\ox} \; \tikzmath[scale=\planscale]{\useasboundingbox (10,-2) rectangle (27,14);\draw(18,0) -- (21,0);} %tikzmath
\Big\rrbracket  \\ \nonumber & = & \Big\llbracket  \Big(
\tikzmath[scale=\planscale]{\useasboundingbox (-2,-2) rectangle (32,14);\draw[thick, double] (0,6) -- (0,0) -- (6,0);\draw(6,0) -- (24,0);\draw[ultra thick] (24,0) -- (30,0) -- (30,6); } %tikzmath
\Big)' \; \vee \; \tikzmath[scale=\planscale]{\useasboundingbox (10,-2) rectangle (27,14);\draw(18,0) -- (21,0);} %tikzmath
\; : \; \tikzmath[scale=\planscale]{\useasboundingbox (-2,-2) rectangle (32,14);\draw[thick, double](0,6) -- (0,12) -- (6,12);\draw (6,12) -- (24,12);
\draw[ultra thick](24,12) -- (30,12) -- (30,6); \draw(18,0) -- (21,0);} %tikzmath
\Big\rrbracket  \\ \nonumber & = & \Big\llbracket \Big(
\tikzmath[scale=\planscale]{\useasboundingbox (-2,-2) rectangle (32,14);\draw[thick, double] (0,6) -- (0,0) -- (6,0);\draw(6,0) -- (24,0);\draw[ultra thick] (24,0) -- (30,0) -- (30,6);} %tikzmath
\; \cap \; \big(\tikzmath[scale=\planscale]{\useasboundingbox (10,-2) rectangle (27,14);\draw(18,0) -- (21,0);} %tikzmath
\big)' \Big)' \; : \;
\tikzmath[scale=\planscale]{\useasboundingbox (-2,-2) rectangle (32,14);\draw[thick, double](0,6) -- (0,12) -- (6,12);\draw (6,12) -- (24,12);
\draw[ultra thick](24,12) -- (30,12) -- (30,6);\draw(18,0) -- (21,0);} %tikzmath
\Big\rrbracket  \\ \nonumber & = & \Big\llbracket \Big(
\tikzmath[scale=\planscale]{\useasboundingbox (-2,-2) rectangle (32,14);\draw[thick, double](0,6) -- (0,0) -- (6,0);\draw (6,0) -- (18,0) (21,0) -- (24,0);
\draw[ultra thick] (24,0) -- (30,0) -- (30,6);\draw[dash pattern=on .4pt off .62pt](16.5,0) arc (180:0:3 and 4); } %tikzmath
\Big)' \; : \; \tikzmath[scale=\planscale]{\useasboundingbox (-2,-2) rectangle (32,14);\draw[thick, double](0,6) -- (0,12) -- (6,12);\draw (6,12) -- (24,12);
\draw[ultra thick](24,12) -- (30,12) -- (30,6); \draw(18,0) -- (21,0);} %tikzmath
\Big\rrbracket ;
\end{eqnarray}
the last equality follows by applying Lemma~\ref{lem:commutant-spacial-vee 1} with
$B := \tikzmath[scale=\textscale]
             {\useasboundingbox (-2,-2) rectangle (32,9);
              \draw[thick, double](0,6) -- (0,0) -- (6,0);
              \draw (6,0) -- (16,0) (22,0) -- (24,0);
              \draw[ultra thick] (24,0) -- (30,0) -- (30,6);
               } %tikzmath
$,
$A := \tikzmath[scale=\textscale]{\useasboundingbox (10,-2) rectangle (27,9);\draw(16,0) -- (23,0);} %tikzmath
$,
$A_0 := \tikzmath[scale=\textscale]{\useasboundingbox (10,-2) rectangle (27,9);\draw(18,0) -- (21,0);} %tikzmath
$, and $A \cap A_0' = \tikzmath[scale=\textscale]
     {\useasboundingbox (10,-2) rectangle (27,9);
           \draw (16,0) -- (18,0) (21,0) -- (23,0);
         \draw[dash pattern=on .4pt off .62pt]
              (16.5,0) arc (180:0:3 and 4); } %tikzmath
$.
Similarly, we have
\begin{eqnarray} \label{eq:mu=mu5:zwei}
\lefteqn{\Big\llbracket\Big(
\tikzmath[scale=\planscale]{\useasboundingbox (-2,-2) rectangle (32,14);\draw[thick, double](0,6) -- (0,12) -- (6,12);
\draw (6,12) -- (24,12) (18,0) -- (21,0);\draw[ultra thick](24,12) -- (30,12) -- (30,6);}%tikzmath
\Big)'\; : \;\tikzmath[scale=\planscale]{\useasboundingbox (-2,-2) rectangle (32,14);\draw[thick, double] (0,6) -- (0,0) -- (6,0);\draw (6,0) -- (18,0) (21,0) -- (24,0);
\draw[ultra thick] (24,0) -- (30,0) -- (30,6);\draw[dash pattern=on .4pt off .62pt] (16.5,0) arc (180:0:3 and 4);}%tikzmath
\Big\rrbracket } \\ \nonumber & = & \Big\llbracket \Big(
\tikzmath[scale=\planscale]{\useasboundingbox (-2,-2) rectangle (32,14);\draw[thick, double](0,6) -- (0,12) -- (6,12);
\draw (6,12) -- (24,12) (18,0) -- (21,0);\draw[ultra thick](24,12) -- (30,12) -- (30,6);}%tikzmath
\Big)' \; \bar{\ox} \; \tikzmath[scale=\planscale]{\useasboundingbox (5,-2) rectangle (18,14);\draw(9,12) -- (12,12);}%tikzmath 
\; : \; \tikzmath[scale=\planscale]{\useasboundingbox (-2,-2) rectangle (32,14);\draw[thick, double] (0,6) -- (0,0) -- (6,0);
\draw (6,0) -- (18,0) (21,0) -- (24,0);\draw[ultra thick] (24,0) -- (30,0) -- (30,6);\draw[dash pattern=on .4pt off .62pt] (16.5,0) arc (180:0:3 and 4);}%tikzmath
\; \bar{\ox} \; \tikzmath[scale=\planscale]{\useasboundingbox (5,-2) rectangle (18,14);\draw(9,12) -- (12,12);} %tikzmath
\,\,\Big\rrbracket \\ \nonumber & = & \Big\llbracket \Big(
\tikzmath[scale=\planscale]{\useasboundingbox (-2,-2) rectangle (32,14);\draw[thick, double](0,6) -- (0,12) -- (6,12);
\draw (6,12) -- (24,12) (18,0) -- (21,0);\draw[ultra thick](24,12) -- (30,12) -- (30,6);}%tikzmath
\Big)' \; \vee \; \tikzmath[scale=\planscale]{\useasboundingbox (5,-2) rectangle (18,14);\draw(9,12) -- (12,12);} %tikzmath
\; : \; \tikzmath[scale=\planscale] {\useasboundingbox (-2,-2) rectangle (32,14); \draw[thick, double] (0,6) -- (0,0) -- (6,0);\draw (6,0) -- (18,0) (21,0) -- (24,0) (9,12) -- (12,12);
\draw[ultra thick] (24,0) -- (30,0) -- (30,6);\draw[dash pattern=on .4pt off .62pt] (16.5,0) arc (180:0:3 and 4); }%tikzmath
\Big\rrbracket  \\ \nonumber & = & \Big\llbracket \Big(\tikzmath[scale=\planscale]{\useasboundingbox (-2,-2) rectangle (32,14);
\draw[thick, double](0,6) -- (0,12) -- (6,12);\draw (6,12) -- (24,12) (18,0) -- (21,0);\draw[ultra thick](24,12) -- (30,12) -- (30,6);} %tikzmath
\; \cap \; \big(\, \tikzmath[scale=\planscale] { \useasboundingbox (5,-2) rectangle (18,14);\draw(9,12) -- (12,12);}%tikzmath 
\big)' \Big)' \; : \; \tikzmath[scale=\planscale]{\useasboundingbox (-2,-2) rectangle (32,14);\draw[thick, double] (0,6) -- (0,0) -- (6,0);
\draw (6,0) -- (18,0) (21,0) -- (24,0) (9,12) -- (12,12);\draw[ultra thick] (24,0) -- (30,0) -- (30,6);\draw[dash pattern=on .4pt off .62pt] (16.5,0) arc (180:0:3 and 4); }%tikzmath
\Big\rrbracket   \\ \nonumber & = & \Big\llbracket \Big( \tikzmath[scale=\planscale]{\useasboundingbox (-2,-2) rectangle (32,14);\draw[thick, double](0,6) -- (0,12) -- (6,12);
\draw (6,12) -- (9,12) (12,12) -- (24,12) (18,0) -- (21,0); \draw[ultra thick](24,12) -- (30,12) -- (30,6);\draw[dash pattern=on .4pt off .62pt] (7.5,12) arc (180:360:3 and 4);} %tikzmath
\Big)' \; : \; \tikzmath[scale=\planscale] {\useasboundingbox (-2,-2) rectangle (32,14);\draw[thick, double] (0,6) -- (0,0) -- (6,0);\draw (6,0) -- (18,0) (21,0) -- (24,0) (9,12) -- (12,12);
\draw[ultra thick] (24,0) -- (30,0) -- (30,6);\draw[dash pattern=on .4pt off .62pt] (16.5,0) arc (180:0:3 and 4); }%tikzmath
\Big\rrbracket .
\end{eqnarray}
Taking commutants transposes the matrix of statistical
dimensions~\eqref{eq:matrix-of-stat-dim-commutants}.
%\cite[\corindexofcommutants]
%     {BDH(Dualizability+Index-of-subfactors)}.
Thus~\eqref{eq:mu=mu5:zwei} implies
\[
\Big\llbracket \Big(
\tikzmath[scale=\planscale]{\useasboundingbox (-2,-2) rectangle (32,14);\draw[thick, double](0,6) -- (0,0) -- (6,0);\draw (6,0) -- (18,0) (21,0) -- (24,0);
\draw[ultra thick] (24,0) -- (30,0) -- (30,6);\draw[dash pattern=on .4pt off .62pt](16.5,0) arc (180:0:3 and 4); } %tikzmath
\Big)' \; : \; \tikzmath[scale=\planscale]{\useasboundingbox (-2,-2) rectangle (32,14);\draw[thick, double](0,6) -- (0,12) -- (6,12);\draw (6,12) -- (24,12);
\draw[ultra thick](24,12) -- (30,12) -- (30,6); \draw(18,0) -- (21,0);} %tikzmath
\Big\rrbracket \;=\; \Big\llbracket\Big(
\tikzmath[scale=\planscale]{\useasboundingbox (-2,-2) rectangle (32,14);\draw[thick, double] (0,6) -- (0,0) -- (6,0);
\draw (6,0) -- (18,0) (21,0) -- (24,0) (9,12) -- (12,12);\draw[ultra thick] (24,0) -- (30,0) -- (30,6);\draw[dash pattern=on .4pt off .62pt](16.5,0) arc (180:0:3 and 4);} %tikzmath
\Big)' \; : \; \tikzmath[scale=\planscale]{\useasboundingbox (-2,-2) rectangle (32,14);\draw[thick, double](0,6) -- (0,12) -- (6,12);
\draw (6,12) -- (9,12) (12,12) -- (24,12) (18,0) -- (21,0);\draw[ultra thick](24,12) -- (30,12) -- (30,6);\draw[dash pattern=on .4pt off .62pt] (7.5,12) arc (180:360:3 and 4);} %tikzmath
\Big\rrbracket
\]
which, combined with~\eqref{eq:mu=mu5:eins}, proves the Lemma. 
\end{proof}

\section{The double bridge algebra is a factor}

Let $S_l$ and $S_r$ be as in \eqref{eq: figures of intervals 3} and now let $\tilde S_l:=\partial([0,\tfrac54]\times[0,1])$ and $\tilde S_r:={\partial([\tfrac54,2]\times[0,1])}$; give these circles the bicoloring
\begin{equation*}
\begin{split}
(S_l)_\circ := (S_l)_{x\le \frac12}\quad
(S_l)_\bullet := (S_l)_{x\ge \frac12}&\quad
(S_r)_\circ := (S_r)_{x\le \frac32}\quad
(S_r)_\bullet := (S_r)_{x\ge \frac32}\\
(\tilde S_l)_\circ := (\tilde S_l)_{x\le \frac12}\quad
(\tilde S_l)_\bullet := (\tilde S_l)_{x\ge \frac12}&\quad
(\tilde S_r)_\circ := (\tilde S_r)_{x\le \frac32}\quad
(\tilde S_r)_\bullet := (\tilde S_r)_{x\ge \frac32}.
\end{split}
\end{equation*}
%the same formulas as \eqref{eq: def colors of S and tildeS}
%(even tough their meaning in \eqref{eq: def colors of S and tildeS} is different).
Let $I:=S_l\cap S_r$ and $\tilde I:=\tilde S_l\cap \tilde S_r$.
Recall from~\eqref{eq:v_I} that for any conformal circle $S$ and any interval $I\subset S$,
the vacuum sector $H_0(S,\calb)$ is a unit for Connes fusion over $\calb(I)$.
By applying this fact twice, we can construct a non-canonical
isomorphism
\[
H_0(S_l,D)\circledast_{\calb(I)}H_0(S_r,E)\,\cong\, H_0(\tilde S_l,D)\circledast_{\calb(\tilde I)}H_0(\tilde S_r,E),
\]
equivariant with respect to the actions of $D(J)$ for every
$J\subset \tikzmath[scale=\textscale]{\useasboundingbox (-2,0) rectangle (26,12);\draw[thick, double](6,0)--(0,0)--(0,12)--(6,12);\draw(6,12)--(18,12);\draw[->](0,5.4) -- (0,5.3);}$
and $J\subset \tikzmath[scale=\textscale]{\useasboundingbox (-2,0) rectangle (26,12);\draw[thick, double](6,0)--(0,0)--(0,12)--(6,12);\draw(6,0)--(18,0);\draw[->](0,5.4) -- (0,5.3);}$,
and with respect to the actions of $E(J)$ for every
$J\subset \tikzmath[scale=\textscale]{\useasboundingbox (-2,0) rectangle (26,12);\draw[ultra thick](18,0)--(24,0)--(24,12)--(18,12);\draw(6,12)--(18,12);\draw[->, thick](24,6.8) -- (24,6.9);}$
and $J\subset \tikzmath[scale=\textscale]{\useasboundingbox (-2,0) rectangle (26,12);\draw[ultra thick](18,0)--(24,0)--(24,12)--(18,12);\draw(6,0)--(18,0);\draw[->, thick](24,6.8) -- (24,6.9);}$\,.

Recall the fiber product operation $\ast$ from Appendix~\ref{subsec:fusion+fiber-prod}.
Let
\[
\begin{split}
J:=[\tfrac76,\tfrac43]\times\{0,1\},\quad &K_l:=\partial^\sqsupset\big([\tfrac76,\tfrac54]\!\times\![0,1]\big),\quad K_r:=\partial^\sqsubset \big([\tfrac54,\tfrac43]\!\times\![0,1]\big),\\
K_l':=\partial^\sqsubset \big([0,&\tfrac76]\!\times\![0,1]\big),\qquad K_r':=\partial^\sqsupset \big([\tfrac43,2]\!\times\![0,1]\big),
\end{split}
\]
which we draw for convenience:
\[
J=\tikzmath[scale=\planscale]{\useasboundingbox (30,-2) rectangle (6,14);\draw (18,12)--(21,12)(18,0)--(21,0);}\,%tikzmath
,\quad K_l = \tikzmath[scale=\planscale]{\useasboundingbox (30,-2) rectangle (6,14);\draw (18,0) -- (19.5,0) -- (19.5,12) -- (18,12);}\,%tikzmath
,\quad K_r = \tikzmath[scale=\planscale]{\useasboundingbox (30,-2) rectangle (6,14);\draw (21,0) -- (19.5,0) -- (19.5,12) -- (21,12);}\,%tikzmath
,\quad K_l' = \tikzmath[scale=\planscale]{\useasboundingbox (30,-2) rectangle (6,14);\draw (18,0) -- (6,0) -- (6,12) -- (18,12);}\,%tikzmath
,\quad K_r'= \tikzmath[scale=\planscale]{\useasboundingbox (30,-2) rectangle (6,14);\draw (21,0) -- (30,0) -- (30,12) -- (21,12);}\,.%tikzmath
\]
We then have $\tilde S_l = K_l \cup K'_l$ and 
$\tilde S_r = K_r \cup K'_r$.
We use $H=H_0(\calb,\tilde S_l)$ and $K=H_0(\calb,\tilde S_r)$
in the definition \ref{def:fiber-product} of the fiber product $\calb(K_l)\ast_{\calb(\tilde I)}\calb(K_r) = \big(\calb(K_l)'\vee\calb(K_r)'\big)'$.
By Haag duality,
we have $\calb(K_l)'= \calb(K_l')$ and $\calb(K_r)' = \calb(K_r')$.
These algebras all act on  
$H_0(\calb,\tilde S_l) \boxtimes_{\calb(\tilde I)} 
     H_0(\calb,\tilde S_r)$, which can be identified with
$H_0(\calb; S_b)$ by~\eqref{eq:non-canonical-upsilon}.  
%\cite[\corvacuumvacuumvacuum]{BDH(nets)}.
Altogether we obtain
\[
\hat\calb(J)=\big(\calb(K_l')\vee\calb(K_r')\big)'=\big(\calb(K_l)'\vee\calb(K_r)'\big)'=\calb(K_l)\ast_{\calb(\tilde I)}\calb(K_r).
\]
We denote the above equation graphically by
\(
\tikzmath[scale=\textscale]{\useasboundingbox (16,-2) rectangle (23,14);\draw (18,12)--(21,12)(18,0)--(21,0);\draw[dash pattern=on .4pt off .62pt] (19.5,0) -- (19.5,12);}\,%tikzmath
=
\tikzmath[scale=\textscale]{\useasboundingbox (16,-2) rectangle (21.5,14);\draw (18,12) -- (19.5,12) -- (19.5,0) -- (18,0);}%tikzmath
\!\underset{\tikzmath[scale=.02]{\useasboundingbox (17.5,6) rectangle (21.5,14);\draw (19.5,0) -- (19.5,12);}%tikzmath
}{\ast}\!\tikzmath[scale=\textscale]{\useasboundingbox (17.5,-2) rectangle (23,14);\draw (21,12) -- (19.5,12) -- (19.5,0) -- (21,0);}%tikzmath
\)\,.
\vspace{.1cm}

\begin{lemma}  \label{lem:commutants}
We have the following equality of subalgebras of $\bfB\big(\tikzmath[scale=\textscale]{\fill[vacuumcolor] (0,0) rectangle (24,12);
\draw[double, thick](6,0) -- (0,0) -- (0,12) -- (6,12); \draw (6,0) -- (12,0) -- (12,12) -- (6,12)(18,0) -- (12,0)(12,12) -- (18,12);
\draw[ultra thick] (18,0) -- (24,0) -- (24,12) -- (18,12);}\big)$:
\begin{gather}\label{lem:commutants:line}
	\left(\,\tikzmath[scale=\planscale]{\useasboundingbox (-2,-2) rectangle (32,14);\draw[thick, double] (0,6) -- (0,0) -- (6,0);\draw (6,0) -- (18,0) (21,0) -- (24,0) (9,12) -- (12,12);
	\draw[ultra thick] (24,0) -- (30,0) -- (30,6);\draw[dash pattern=on .4pt off .62pt](10.5,0) -- (10.5,12);}%tikzmath
	\,\right)'\; = \;\tikzmath[scale=\planscale]{\useasboundingbox (-2,-2) rectangle (32,14);\draw[thick, double](0,6) -- (0,12) -- (6,12);\draw (6,12) -- (9,12) (12,12) -- (24,12) (18,0) -- (21,0);
	\draw[ultra thick](24,12) -- (30,12) -- (30,6);\draw[dash pattern=on .4pt off .62pt] (19.5,0) -- (19.5,12);}\,,%tikzmath
\\
%end{equation}
%\begin{equation}
\label{lem:commutants:line+halfcircle}
	\left(\,\tikzmath[scale=\planscale]{\useasboundingbox (-2,-2) rectangle (32,14);\draw[thick, double](0,6) -- (0,0) -- (6,0);
	\draw (6,0) -- (18,0) (21,0) -- (24,0) (9,12) -- (12,12);\draw[ultra thick] (24,0) -- (30,0) -- (30,6);}%tikzmath
	\,\right)' \; = \; \tikzmath[scale=\planscale]{\useasboundingbox (-2,-2) rectangle (32,14);\draw[thick, double](0,6) -- (0,12) -- (6,12);\draw (6,12) -- (9,12) (12,12) -- (24,12) (18,0) -- (21,0);
	\draw[ultra thick](24,12) -- (30,12) -- (30,6);\draw[dash pattern=on .4pt off .62pt] (7.5,12) arc (180:360:3 and 4)(19.5,0) -- (19.5,12);}\,.%tikzmath
\end{gather}
\end{lemma}

\begin{proof}
By Lemma \ref{lem: needed for dotted lines} and Proposition \ref{prop: [Haag duality for defects]}, respectively, we have
\begin{equation*}\medskip
\tikzmath[scale=\planscale]{ \useasboundingbox (-2,-2) rectangle (21.5,14);\draw[thick, double] (0,6) -- (0,0) -- (6,0);
\draw (6,0) -- (18,0)(9,12) -- (12,12);\draw[dash pattern=on .4pt off .62pt](10.5,0) -- (10.5,12); }%tikzmath
\; = \;\left(\tikzmath[scale=\planscale]{\useasboundingbox (-2,-2) rectangle (21.5,14);
\draw[thick, double](0,6) -- (0,12) -- (6,12);\draw (6,12) -- (9,12) (12,12) -- (19.5,12) -- (19.5,0) -- (18,0);}%tikzmath
\right)'\; \text{on}\,\, \;\tikzmath[scale=\planscale]{\useasboundingbox (-2,-2) rectangle (21.5,14);
\fill[vacuumcolor] (0,0) rectangle (19.5,12);\draw[thick, double] (6,12) -- (0,12) -- (0,0) -- (6,0);\draw (6,12) -- (19.5,12) -- (19.5,0) -- (6,0);}\,,%tikzmath 
\quad\,\,\, \text{and}\quad \;\tikzmath[scale=\planscale]{\useasboundingbox (17.5,-2) rectangle (32,14);\draw(21,0) -- (24,0);\draw[ultra thick] (24,0) -- (30,0) -- (30,6);}%tikzmath
\; = \; \left(\tikzmath[scale=\planscale]{\useasboundingbox (17.5,-2) rectangle (32,14);\draw (21,0) -- (19.5,0) -- (19.5,12) -- (24,12);\draw[ultra thick](24,12) -- (30,12) -- (30,6);}%tikzmath
\right)'\; \text{on}\,\, \;\tikzmath[scale=\planscale]{\useasboundingbox (17.5,-2) rectangle (32,14);\fill[vacuumcolor] (19.5,0) rectangle (30,12);
\draw(24,0) -- (19.5,0) -- (19.5,12) -- (24,12);\draw[ultra thick] (24,0) -- (30,0) -- (30,12) -- (24,12);}\,,%tikzmath
\end{equation*}
where $\tikzmath[scale=\textscale]{\useasboundingbox (-2,-2) rectangle (21.5,14);
\fill[vacuumcolor] (0,0) rectangle (19.5,12);\draw[thick, double] (6,12) -- (0,12) -- (0,0) -- (6,0);\draw (6,12) -- (19.5,12) -- (19.5,0) -- (6,0);}$
stands for $H_0(\tilde S_l,D)$, and $\tikzmath[scale=\textscale]{\useasboundingbox (17.5,-2) rectangle (32,14);\fill[vacuumcolor] (19.5,0) rectangle (30,12);
\draw(24,0) -- (19.5,0) -- (19.5,12) -- (24,12);\draw[ultra thick] (24,0) -- (30,0) -- (30,12) -- (24,12);}$ stands for $H_0(\tilde S_r,E)$.

We have the following sequence of equalities
\begin{equation*}
\begin{split}
\tikzmath[scale=\planscale]{\useasboundingbox (-2,-2) rectangle (32,14);\draw[thick, double](0,6) -- (0,12) -- (6,12);\draw (6,12) -- (9,12) (12,12) -- (24,12) (18,0) -- (21,0);
\draw[ultra thick](24,12) -- (30,12) -- (30,6);\draw[dash pattern=on .4pt off .62pt] (19.5,0) -- (19.5,12); }%tikzmath
\;&=\;\Big(\tikzmath[scale=\planscale]{\useasboundingbox (-2,-2) rectangle (20,14);\draw[thick, double](0,6) -- (0,12) -- (6,12);\draw (6,12) -- (9,12) (12,12) -- (18,12);}\Big)%tikzmath
\vee\Big(\,\,\tikzmath[scale=\planscale]{\useasboundingbox (16,-2) rectangle (23,14);\draw (18,12)--(21,12)(18,0)--(21,0);\draw[dash pattern=on .4pt off .62pt] (19.5,0) -- (19.5,12);}\,\,\Big)%tikzmath
\vee\Big(\tikzmath[scale=\planscale]{\useasboundingbox (19,-2) rectangle (32,14);\draw(21,12) -- (24,12);\draw[ultra thick](24,12) -- (30,12) -- (30,6);}\Big)\\%tikzmath
\;&=\;\Big(\tikzmath[scale=\planscale]{\useasboundingbox (-2,-2) rectangle (20,14);\draw[thick, double](0,6) -- (0,12) -- (6,12);\draw (6,12) -- (9,12) (12,12) -- (18,12);}\Big)%tikzmath
\vee\Big(\tikzmath[scale=\planscale]{\useasboundingbox (16,-2) rectangle (21.5,14);\draw (18,12) -- (19.5,12) -- (19.5,0) -- (18,0);}%tikzmath
\!\underset{\tikzmath[scale=.025]{\useasboundingbox (17.5,6) rectangle (21.5,14);\draw (19.5,0) -- (19.5,12);}%tikzmath
}{\ast}\!\tikzmath[scale=\planscale]{\useasboundingbox (17.5,-2) rectangle (23,14);\draw (21,12) -- (19.5,12) -- (19.5,0) -- (21,0);}\Big)%tikzmath
\vee\Big(\tikzmath[scale=\planscale]{\useasboundingbox (19,-2) rectangle (32,14);\draw(21,12) -- (24,12);\draw[ultra thick](24,12) -- (30,12) -- (30,6);}\Big)%tikzmath
\;=\;\tikzmath[scale=\planscale]{ \useasboundingbox (-2,-2) rectangle (21.5,14);\draw[thick, double](0,6) -- (0,12) -- (6,12);\draw (6,12) -- (9,12) (12,12) -- (19.5,12) -- (19.5,0) -- (18,0);}%tikzmath
\underset{\tikzmath[scale=.025]{\useasboundingbox (17.5,6) rectangle (21.5,14);\draw (19.5,0) -- (19.5,12);}%tikzmath
}{\ast}\tikzmath[scale=\planscale]{\useasboundingbox (17.5,-2) rectangle (32,14);\draw (21,0) -- (19.5,0) -- (19.5,12) -- (24,12);\draw[ultra thick](24,12) -- (30,12) -- (30,6);}\\%tikzmath
&\hspace{4.5cm}= \;\Big(\Big(\tikzmath[scale=\planscale]{ \useasboundingbox (-2,-2) rectangle (21.5,14);\draw[thick, double] (0,6) -- (0,0) -- (6,0);
\draw (6,0) -- (18,0)(9,12) -- (12,12);\draw[dash pattern=on .4pt off .62pt](10.5,0) -- (10.5,12);}\Big)%tikzmath
\vee\Big(\tikzmath[scale=\planscale]{\useasboundingbox (17.5,-2) rectangle (32,14);\draw(21,0) -- (24,0);\draw[ultra thick] (24,0) -- (30,0) -- (30,6);}\Big)\Big)'%tikzmath
=\,\Big(\tikzmath[scale=\planscale]{\useasboundingbox (-2,-2) rectangle (32,14);\draw[thick, double] (0,6) -- (0,0) -- (6,0);\draw (6,0) -- (18,0) (21,0) -- (24,0) (9,12) -- (12,12);
\draw[ultra thick] (24,0) -- (30,0) -- (30,6);\draw[dash pattern=on .4pt off .62pt](10.5,0) -- (10.5,12);}\Big)'\,.%tikzmath
\end{split}
\end{equation*}
Here the third equality uses Lemma~\ref{lem:add-to-fusion-of-algebras}.
By Lemma~\ref{lem: needed for dotted lines}, we also have
\begin{equation*}
\tikzmath[scale=\planscale]{\useasboundingbox (-2,-2) rectangle (21.5,14);\draw[thick, double](0,6) -- (0,12) -- (6,12);
\draw (6,12) -- (9,12)(12,12) -- (19.5,12) -- (19.5,0) -- (18,0);\draw[dash pattern=on .4pt off .62pt](7.5,12) arc (180:360:3 and 4);}%tikzmath
\; = \; \left(\tikzmath[scale=\planscale]{\useasboundingbox (-2,-2) rectangle (21.5,14);\draw[thick, double] (0,6) -- (0,0) -- (6,0);\draw (6,0) -- (18,0)(9,12) -- (12,12); }%tikzmath
\right)'\; \text{on}\,\, \;\tikzmath[scale=\planscale]{\useasboundingbox (-2,-2) rectangle (21.5,14);\fill[vacuumcolor] (0,0) rectangle (19.5,12);
\draw[thick, double] (6,12) -- (0,12) -- (0,0) -- (6,0);\draw (6,12) -- (19.5,12) -- (19.5,0) -- (6,0);}\,. %tikzmath
\end{equation*}
We therefore similarly have
\begin{equation*}
\begin{split}
\tikzmath[scale=\planscale]{\useasboundingbox (-2,-2) rectangle (32,14);\draw[dash pattern=on .4pt off .62pt](7.5,12) arc (180:360:3 and 4);\draw[thick, double](0,6) -- (0,12) -- (6,12);
\draw (6,12) -- (9,12) (12,12) -- (24,12) (18,0) -- (21,0);\draw[ultra thick](24,12) -- (30,12) -- (30,6);\draw[dash pattern=on .4pt off .62pt] (19.5,0) -- (19.5,12); }%tikzmath
\;&=\;\Big(\tikzmath[scale=\planscale]{\useasboundingbox (-2,-2) rectangle (20,14);\draw[dash pattern=on .4pt off .62pt](7.5,12) arc (180:360:3 and 4);
\draw[thick, double](0,6) -- (0,12) -- (6,12);\draw (6,12) -- (9,12) (12,12) -- (18,12);}\Big)%tikzmath
\vee\Big(\,\,\tikzmath[scale=\planscale]{\useasboundingbox (16,-2) rectangle (23,14);\draw (18,12)--(21,12)(18,0)--(21,0);\draw[dash pattern=on .4pt off .62pt] (19.5,0) -- (19.5,12);}\,\,\Big)%tikzmath
\vee\Big(\tikzmath[scale=\planscale]{\useasboundingbox (19,-2) rectangle (32,14);\draw(21,12) -- (24,12);\draw[ultra thick](24,12) -- (30,12) -- (30,6);}\Big)\\%tikzmath
\;&=\;\Big(\tikzmath[scale=\planscale]{\useasboundingbox (-2,-2) rectangle (20,14);\draw[dash pattern=on .4pt off .62pt](7.5,12) arc (180:360:3 and 4);
\draw[thick, double](0,6) -- (0,12) -- (6,12);\draw (6,12) -- (9,12) (12,12) -- (18,12);}\Big)%tikzmath
\vee\Big(\tikzmath[scale=\planscale]{\useasboundingbox (16,-2) rectangle (21.5,14);\draw (18,12) -- (19.5,12) -- (19.5,0) -- (18,0);}
\!\underset{\tikzmath[scale=.025]{\useasboundingbox (17.5,6) rectangle (21.5,14);\draw (19.5,0) -- (19.5,12);}%tikzmath
}{\ast}\!\tikzmath[scale=\planscale]{\useasboundingbox (17.5,-2) rectangle (23,14);\draw (21,12) -- (19.5,12) -- (19.5,0) -- (21,0);}\Big)%tikzmath
\vee\Big(\tikzmath[scale=\planscale]{\useasboundingbox (19,-2) rectangle (32,14);\draw(21,12) -- (24,12);\draw[ultra thick](24,12) -- (30,12) -- (30,6);}\Big)%tikzmath
\;=\;\tikzmath[scale=\planscale]{ \useasboundingbox (-2,-2) rectangle (21.5,14);\draw[dash pattern=on .4pt off .62pt](7.5,12) arc (180:360:3 and 4);
\draw[thick, double](0,6) -- (0,12) -- (6,12);\draw (6,12) -- (9,12) (12,12) -- (19.5,12) -- (19.5,0) -- (18,0);}%tikzmath
\underset{\tikzmath[scale=.025]{\useasboundingbox (17.5,6) rectangle (21.5,14);\draw (19.5,0) -- (19.5,12);}%tikzmath
}{\ast}\tikzmath[scale=\planscale]{\useasboundingbox (17.5,-2) rectangle (32,14);\draw (21,0) -- (19.5,0) -- (19.5,12) -- (24,12);\draw[ultra thick](24,12) -- (30,12) -- (30,6);}\\%tikzmath
&\hspace{4.5cm}= \;\Big(\Big(\tikzmath[scale=\planscale]{ \useasboundingbox (-2,-2) rectangle (21.5,14);\draw[thick, double] (0,6) -- (0,0) -- (6,0);
\draw (6,0) -- (18,0)(9,12) -- (12,12);}\Big)%tikzmath
\vee\Big(\tikzmath[scale=\planscale]{\useasboundingbox (17.5,-2) rectangle (32,14);\draw(21,0) -- (24,0);\draw[ultra thick] (24,0) -- (30,0) -- (30,6);}\Big)\Big)'%tikzmath
=\,\Big(\tikzmath[scale=\planscale]{\useasboundingbox (-2,-2) rectangle (32,14);\draw[thick, double] (0,6) -- (0,0) -- (6,0);\draw (6,0) -- (18,0) (21,0) -- (24,0) (9,12) -- (12,12);
\draw[ultra thick] (24,0) -- (30,0) -- (30,6);}\Big)'.%tikzmath
\end{split}
\end{equation*}
\end{proof}

\begin{corollary}\label{cor: The algebras ... are factors}
The algebra\, $ \tikzmath[scale=\textscale]{\useasboundingbox (-2,-2) rectangle (32,14);\draw[thick, double](0,6) -- (0,12) -- (6,12);\draw (6,12) -- (9,12) (12,12) -- (24,12) (18,0) -- (21,0);
\draw[ultra thick](24,12) -- (30,12) -- (30,6);\draw[dash pattern=on .4pt off .62pt] (7.5,12) arc (180:360:3 and 4)(19.5,0) -- (19.5,12);} $\, is a factor. \qed
\end{corollary}

\begin{corollary}\label{cor: nu_1=nu_6^t}
We have
\begin{equation}\label{eq: cor: nu_1=nu_6^t}
\left\llbracket\,
\tikzmath[scale=\planscale]{\useasboundingbox (-2,-2) rectangle (32,14);\draw[thick, double]  (0,6) -- (0,12) -- (6,12);           
\draw (6,12) -- (9,12) (12,12) -- (24,12) (18,0) -- (21,0);\draw[ultra thick]   (24,12) -- (30,12) -- (30,6);\draw[dash pattern=on .4pt off .62pt] (7.5,12) arc (180:360:3 and 4);}   \,:\,
\tikzmath[scale=\planscale]{\useasboundingbox (-2,-2) rectangle (32,14);\draw[thick, double]  (0,6) -- (0,12) -- (6,12);
\draw (6,12) -- (9,12) (12,12) -- (24,12) (18,0) -- (21,0);\draw[ultra thick]  (24,12) -- (30,12) -- (30,6);}
\,\right\rrbracket
\;=\;
\Big\llbracket\,
\tikzmath[scale=\planscale]{\useasboundingbox (-2,-2) rectangle (32,14);\draw[thick, double]  (0,6) -- (0,12) -- (6,12);
\draw (6,12) -- (9,12) (12,12) -- (24,12) (18,0) -- (21,0);\draw[ultra thick]  (24,12) -- (30,12) -- (30,6);\draw[dash pattern=on .4pt off .62pt]   (7.5,12) arc (180:360:3 and 4)(19.5,0) -- (19.5,12);}   \,:\,
\left(\tikzmath[scale=\planscale]{\useasboundingbox (-2,-2) rectangle (32,14);\draw[thick, double]  (0,6) -- (0,0) -- (6,0);
\draw (6,0) -- (18,0) (21,0) -- (24,0) (9,12) -- (12,12);\draw[ultra thick]  (24,0) -- (30,0) -- (30,6);\draw[dash pattern=on .4pt off .62pt] (16.5,0) arc (180:0:3 and 4);}
\right)'\,\Big\rrbracket^T
\end{equation}
\end{corollary}

\begin{proof}
By~\eqref{eq:matrix-of-stat-dim-commutants} and by Lemma~\ref{lem:commutants}, 
the right-hand side of \eqref{eq: cor: nu_1=nu_6^t} is equal to
\[
\Big\llbracket\,
\tikzmath[scale=\planscale]{\useasboundingbox (-2,-2) rectangle (32,14);\draw[thick, double]  (0,6) -- (0,0) -- (6,0);
\draw (6,0) -- (18,0) (21,0) -- (24,0) (9,12) -- (12,12);\draw[ultra thick]  (24,0) -- (30,0) -- (30,6);\draw[dash pattern=on .4pt off .62pt] (16.5,0) arc (180:0:3 and 4);}   \,:\,
\left(\tikzmath[scale=\planscale]{\useasboundingbox (-2,-2) rectangle (32,14);\draw[thick, double]  (0,6) -- (0,12) -- (6,12);
\draw (6,12) -- (9,12) (12,12) -- (24,12) (18,0) -- (21,0);\draw[ultra thick]  (24,12) -- (30,12) -- (30,6);\draw[dash pattern=on .4pt off .62pt]   (7.5,12) arc (180:360:3 and 4)(19.5,0) -- (19.5,12);}
\right)'\,\Big\rrbracket
\;=\;
\Big\llbracket\,
\tikzmath[scale=\planscale]{\useasboundingbox (-2,-2) rectangle (32,14);\draw[thick, double]  (0,6) -- (0,0) -- (6,0);
\draw (6,0) -- (18,0) (21,0) -- (24,0) (9,12) -- (12,12);\draw[ultra thick]  (24,0) -- (30,0) -- (30,6);\draw[dash pattern=on .4pt off .62pt] (16.5,0) arc (180:0:3 and 4);}   \,:\,
\tikzmath[scale=\planscale]{\useasboundingbox (-2,-2) rectangle (32,14);\draw[thick, double]  (0,6) -- (0,0) -- (6,0);
\draw (6,0) -- (18,0) (21,0) -- (24,0) (9,12) -- (12,12);\draw[ultra thick]  (24,0) -- (30,0) -- (30,6);}
\Big\rrbracket.
\]
The algebras
$\tikzmath[scale=\textscale]{\useasboundingbox (-2,-2) rectangle (32,14);\draw[thick, double]  (0,6) -- (0,0) -- (6,0);
\draw (6,0) -- (18,0) (21,0) -- (24,0) (9,12) -- (12,12);\draw[ultra thick]  (24,0) -- (30,0) -- (30,6);\draw[dash pattern=on .4pt off .62pt] (16.5,0) arc (180:0:3 and 4);}$
and
$\tikzmath[scale=\textscale]{\useasboundingbox (-2,-2) rectangle (32,14);\draw[thick, double]  (0,6) -- (0,0) -- (6,0);
\draw (6,0) -- (18,0) (21,0) -- (24,0) (9,12) -- (12,12);\draw[ultra thick]  (24,0) -- (30,0) -- (30,6);}$
are related to those on the left-hand side of \eqref{eq: cor: nu_1=nu_6^t} by the action of orientation reversing diffeomorphisms of the underlying 1-manifolds: these diffeomorphisms induce algebra isomorphisms 
$\tikzmath[scale=\textscale]{\useasboundingbox (-2,-2) rectangle (32,14);\draw[thick, double]  (0,6) -- (0,0) -- (6,0);
\draw (6,0) -- (18,0) (21,0) -- (24,0) (9,12) -- (12,12);\draw[ultra thick]  (24,0) -- (30,0) -- (30,6);\draw[dash pattern=on .4pt off .62pt] (16.5,0) arc (180:0:3 and 4);}
\,\cong\, \big(\tikzmath[scale=\textscale]{\useasboundingbox (-2,-2) rectangle (32,14);\draw[thick, double]  (0,6) -- (0,12) -- (6,12);           
\draw (6,12) -- (9,12) (12,12) -- (24,12) (18,0) -- (21,0);\draw[ultra thick]   (24,12) -- (30,12) -- (30,6);\draw[dash pattern=on .4pt off .62pt] (7.5,12) arc (180:360:3 and 4);}{\big)\!}^\op$
and $\tikzmath[scale=\textscale]{\useasboundingbox (-2,-2) rectangle (32,14);\draw[thick, double]  (0,6) -- (0,0) -- (6,0);
\draw (6,0) -- (18,0) (21,0) -- (24,0) (9,12) -- (12,12);\draw[ultra thick]  (24,0) -- (30,0) -- (30,6);}
\,\cong\, \big(\tikzmath[scale=\textscale]{\useasboundingbox (-2,-2) rectangle (32,14);\draw[thick, double]  (0,6) -- (0,12) -- (6,12);           
\draw (6,12) -- (9,12) (12,12) -- (24,12) (18,0) -- (21,0);\draw[ultra thick]   (24,12) -- (30,12) -- (30,6);}{\big)\!}^\op$.
The result now follows since $\llbracket A:B\rrbracket=\llbracket A^\op:B^\op\rrbracket$.
\end{proof}

\section{The dimension of the bridge inclusions}

\begin{lemma} \label{lem:index-is-sqrt-mu}
We have the following equalities of statistical dimensions:
\begin{equation*}
\begin{split}
\left\llbracket \tikzmath[scale=\planscale]{\useasboundingbox (4,-2) rectangle (32,14);\draw (12,12) -- (24,12) (18,0) -- (21,0) (6,12) -- (9,12);
\draw[ultra thick](24,12) -- (30,12) -- (30,6);\draw[dash pattern=on .4pt off .62pt] (19.5,0) -- (19.5,12);} %tikzmath
\; : \; \tikzmath[scale=\planscale]{\useasboundingbox (4,-2) rectangle (32,14);\draw (12,12) -- (24,12) (18,0) -- (21,0) (6,12) -- (9,12);\draw[ultra thick](24,12) -- (30,12) -- (30,6);} %tikzmath
\right\rrbracket   \!=\!  \left\llbracket \tikzmath[scale=\planscale]{\useasboundingbox (4,-2) rectangle (32,14);\draw (12,12) -- (24,12) (18,0) -- (21,0) (6,12) -- (9,12);
\draw[ultra thick](24,12) -- (30,12) -- (30,6);\draw[dash pattern=on .4pt off .62pt](7.5,12) arc (180:360:3 and 4);} %tikzmath
\; : \; \tikzmath[scale=\planscale]{\useasboundingbox (4,-2) rectangle (32,14);\draw (12,12) -- (24,12) (18,0) -- (21,0) (6,12) -- (9,12);\draw[ultra thick](24,12) -- (30,12) -- (30,6);} %tikzmath
\right\rrbracket \;\,= \hspace*{3cm}\\
\left\llbracket \tikzmath[scale=\planscale]{\useasboundingbox (4,-2) rectangle (32,14);\draw (12,12) -- (24,12) (18,0) -- (21,0) (6,12) -- (9,12);
\draw[ultra thick](24,12) -- (30,12) -- (30,6);\draw[dash pattern=on .4pt off .62pt](7.5,12) arc (180:360:3 and 4)(19.5,0) -- (19.5,12);} %tikzmath
\; : \; \tikzmath[scale=\planscale]{\useasboundingbox (4,-2) rectangle (32,14);\draw (12,12) -- (24,12) (18,0) -- (21,0) (6,12) -- (9,12);
\draw[ultra thick](24,12) -- (30,12) -- (30,6);\draw[dash pattern=on .4pt off .62pt](7.5,12) arc (180:360:3 and 4);} %tikzmath
\right\rrbracket  \!=\!  \left\llbracket \tikzmath[scale=\planscale]{\useasboundingbox (4,-2) rectangle (32,14);\draw (12,12) -- (24,12) (18,0) -- (21,0) (6,12) -- (9,12);
\draw[ultra thick](24,12) -- (30,12) -- (30,6);\draw[dash pattern=on .4pt off .62pt](7.5,12) arc (180:360:3 and 4)(19.5,0) -- (19.5,12);} %tikzmath
\; : \; \tikzmath[scale=\planscale]{\useasboundingbox (4,-2) rectangle (32,14);\draw (12,12) -- (24,12) (18,0) -- (21,0) (6,12) -- (9,12);
\draw[ultra thick](24,12) -- (30,12) -- (30,6);\draw[dash pattern=on .4pt off .62pt] (19.5,0) -- (19.5,12);} %tikzmath
\right\rrbracket\;\, =\, \sqrt{\mu(\calb)}.
\end{split}
\end{equation*}
\end{lemma}

\begin{proof}

We have
\begin{eqnarray*}
\left\llbracket\tikzmath[scale=\textscale]{\useasboundingbox (4,-2) rectangle (32,14);\draw (12,12) -- (24,12) (18,0) -- (21,0) (6,12) -- (9,12);
\draw[ultra thick](24,12) -- (30,12) -- (30,6);\draw[dash pattern=on .4pt off .62pt](7.5,12) arc (180:360:3 and 4)(19.5,0) -- (19.5,12);} %tikzmath
\; : \; \tikzmath[scale=\textscale]{\useasboundingbox (4,-2) rectangle (32,14);\draw (12,12) -- (24,12) (18,0) -- (21,0) (6,12) -- (9,12);
\draw[ultra thick](24,12) -- (30,12) -- (30,6);\draw[dash pattern=on .4pt off .62pt] (19.5,0) -- (19.5,12);} %tikzmath
\right\rrbracket & = & \left\llbracket \tikzmath[scale=\textscale]
{\useasboundingbox (-2,-2) rectangle (32,14);\draw (0,6) -- (0,12) -- (9,12) (12,12) -- (24,12)(18,0) -- (21,0);
\draw[ultra thick](24,12) -- (30,12) -- (30,6);\draw[dash pattern=on .4pt off .62pt](7.5,12) arc (180:360:3 and 4)(19.5,0) -- (19.5,12);} %tikzmath
\; : \; \tikzmath[scale=\textscale]{\useasboundingbox (-2,-2) rectangle (32,14);\draw (0,6) -- (0,12) -- (9,12) (12,12) -- (24,12)(18,0) -- (21,0);
\draw[ultra thick](24,12) -- (30,12) -- (30,6);\draw[dash pattern=on .4pt off .62pt] (19.5,0) -- (19.5,12);} %tikzmath
\right\rrbracket \\ & = & \left\llbracket \tikzmath[scale=\textscale]{\useasboundingbox (-2,-2) rectangle (32,14);
\draw (0,6) -- (0,0) -- (18,0) (21,0) -- (24,0)(9,12) -- (12,12);\draw[ultra thick](24,0) -- (30,0) -- (30,6);\draw[dash pattern=on .4pt off .62pt](10.5,0) -- (10.5,12);} %tikzmath
\; : \; \tikzmath[scale=\textscale]{\useasboundingbox (-2,-2) rectangle (32,14);\draw (0,6) -- (0,0) -- (18,0) (21,0) -- (24,0)(9,12) -- (12,12); \draw[ultra thick](24,0) -- (30,0) -- (30,6);} %tikzmath
\right\rrbracket \\ & = & \left\llbracket
\tikzmath[scale=\textscale]{\useasboundingbox (-2,-2) rectangle (32,14);\draw (0,6) -- (0,0) -- (18,0) (9,12) -- (12,12);\draw[dash pattern=on .4pt off .62pt](10.5,0) -- (10.5,12);} %tikzmath
\; : \; \tikzmath[scale=\textscale]{\useasboundingbox (-2,-2) rectangle (32,14);\draw (0,6) -- (0,0) -- (18,0) (9,12) -- (12,12);} %tikzmath
\right\rrbracket \; = \; \sqrt{\mu(\calb)},
\end{eqnarray*}
where the first equality is obtained by using an appropriate diffeomorphism,
the second one follows from~\eqref{eq:matrix-of-stat-dim-commutants}
and the special case of Lemma~\ref{lem:commutants} when 
$D$ is an identity defect, 
and the third one uses~\eqref{eq:matrix-of-stat-dim-ox-C}.

Let us introduce the auxiliary quantity
\begin{equation*}
\nu\;:=\;\left\llbracket\,\tikzmath[scale=\planscale]{\useasboundingbox (10,-2) rectangle (32,14);\draw (12,12) -- (24,12) (18,0) -- (21,0);
\draw[ultra thick](24,12) -- (30,12) -- (30,6);\draw[densely dotted] (19.5,0) -- (19.5,12);} %tikzmath
\;:\;\tikzmath[scale=\planscale]{\useasboundingbox (10,-2) rectangle (32,14);\draw (12,12) -- (24,12) (18,0) -- (21,0);
\draw[ultra thick](24,12) -- (30,12) -- (30,6);}%tikzmath
\,\right\rrbracket.
\end{equation*}
By Lemma \ref{lem: [ : ] < mu(A)}, we know that $\nu\le \sqrt{\mu(\calb)}$; in particular $\nu<\infty$.
By~\eqref{eq:matrix-of-stat-dim-ox-C}, we have
\[
\left\llbracket \tikzmath[scale=\planscale]{\useasboundingbox (4,-2) rectangle (32,14);\draw (12,12) -- (24,12) (18,0) -- (21,0) (6,12) -- (9,12);
\draw[ultra thick](24,12) -- (30,12) -- (30,6);\draw[dash pattern=on .4pt off .62pt] (19.5,0) -- (19.5,12);} %tikzmath
\; : \; \tikzmath[scale=\planscale]{\useasboundingbox (4,-2) rectangle (32,14);\draw (12,12) -- (24,12) (18,0) -- (21,0) (6,12) -- (9,12);\draw[ultra thick](24,12) -- (30,12) -- (30,6);} %tikzmath
\right\rrbracket \;=\; \left\llbracket\,\tikzmath[scale=\planscale]{\useasboundingbox (10,-2) rectangle (32,14);\draw (12,12) -- (24,12) (18,0) -- (21,0);
\draw[ultra thick](24,12) -- (30,12) -- (30,6);\draw[densely dotted] (19.5,0) -- (19.5,12);} %tikzmath
\;:\;\tikzmath[scale=\planscale]{\useasboundingbox (10,-2) rectangle (32,14);\draw (12,12) -- (24,12) (18,0) -- (21,0);
\draw[ultra thick](24,12) -- (30,12) -- (30,6);}%tikzmath
\,\right\rrbracket \;=\; \nu
\]
and
\[
\hspace{.15cm}\left\llbracket \tikzmath[scale=\planscale]{\useasboundingbox (4,-2) rectangle (32,14);\draw (12,12) -- (24,12) (18,0) -- (21,0) (6,12) -- (9,12);
\draw[ultra thick](24,12) -- (30,12) -- (30,6);\draw[dash pattern=on .4pt off .62pt](7.5,12) arc (180:360:3 and 4);} %tikzmath
\; : \; \tikzmath[scale=\planscale]{\useasboundingbox (4,-2) rectangle (32,14);\draw (12,12) -- (24,12) (18,0) -- (21,0) (6,12) -- (9,12);\draw[ultra thick](24,12) -- (30,12) -- (30,6);} %tikzmath
\right\rrbracket \,=\,\left\llbracket \tikzmath[scale=\planscale]{\useasboundingbox (4,-2) rectangle (32,14);\draw (12,12) -- (24,12)(6,12) -- (9,12);
\draw[ultra thick](24,12) -- (30,12) -- (30,6);\draw[dash pattern=on .4pt off .62pt](7.5,12) arc (180:360:3 and 4);} %tikzmath
\; : \; \tikzmath[scale=\planscale]{\useasboundingbox (4,-2) rectangle (32,14);\draw (12,12) -- (24,12)(6,12) -- (9,12);\draw[ultra thick](24,12) -- (30,12) -- (30,6);} %tikzmath
\right\rrbracket \,=\, \nu.
\]
Next observe that\smallskip
\[
\begin{split}
\left\llbracket \tikzmath[scale=\planscale]{\useasboundingbox (4,-2) rectangle (32,14);\draw (12,12) -- (24,12) (18,0) -- (21,0) (6,12) -- (9,12);
\draw[ultra thick](24,12) -- (30,12) -- (30,6);\draw[dash pattern=on .4pt off .62pt](7.5,12) arc (180:360:3 and 4)(19.5,0) -- (19.5,12);} %tikzmath
\; : \; \tikzmath[scale=\planscale]{\useasboundingbox (4,-2) rectangle (32,14);\draw (12,12) -- (24,12) (18,0) -- (21,0) (6,12) -- (9,12);
\draw[ultra thick](24,12) -- (30,12) -- (30,6);\draw[dash pattern=on .4pt off .62pt](7.5,12) arc (180:360:3 and 4);} %tikzmath
\right\rrbracket
\;=&\;
\left\llbracket\, \tikzmath[scale=\planscale]{\useasboundingbox (10,-2) rectangle (32,14);\draw (12,12) -- (24,12) (18,0) -- (21,0) (12,0) -- (15,0);
\draw[ultra thick](24,12) -- (30,12) -- (30,6);\draw[dash pattern=on .4pt off .62pt](13.5,0) -- (13.5,12)(19.5,0) -- (19.5,12);} %tikzmath
\; : \; \tikzmath[scale=\planscale]{\useasboundingbox (10,-2) rectangle (32,14);\draw (12,12) -- (24,12) (18,0) -- (21,0) (12,0) -- (15,0);
\draw[ultra thick](24,12) -- (30,12) -- (30,6);\draw[dash pattern=on .4pt off .62pt](13.5,0) -- (13.5,12);} %tikzmath
\,\right\rrbracket\\
\;=&\;
\left\llbracket\, \tikzmath[scale=\planscale]{\useasboundingbox (10,-2) rectangle (17,14);\draw(12,0) -- (15,0)(12,12) -- (15,12);
\draw[dash pattern=on .4pt off .62pt](13.5,0) -- (13.5,12);} %tikzmath
\vee \tikzmath[scale=\planscale]{\useasboundingbox (13,-2) rectangle (32,14);\draw (15,12) -- (24,12) (18,0) -- (21,0);
\draw[ultra thick](24,12) -- (30,12) -- (30,6);\draw[dash pattern=on .4pt off .62pt](19.5,0) -- (19.5,12);} %tikzmath
\; : \; \tikzmath[scale=\planscale]{\useasboundingbox (10,-2) rectangle (17,14);\draw(12,0) -- (15,0)(12,12) -- (15,12);
\draw[dash pattern=on .4pt off .62pt](13.5,0) -- (13.5,12);} %tikzmath
\vee \tikzmath[scale=\planscale]{\useasboundingbox (13,-2) rectangle (32,14);\draw (15,12) -- (24,12) (18,0) -- (21,0);
\draw[ultra thick](24,12) -- (30,12) -- (30,6);} %tikzmath
\,\right\rrbracket \\
\;=&\;\,
\Big\llbracket\,
\Big(\tikzmath[scale=\planscale]{\useasboundingbox (16,-2) rectangle (21.5,14);\draw (18,12) -- (19.5,12) -- (19.5,0) -- (18,0);}
\!\underset{\tikzmath[scale=.025]{\useasboundingbox (17.5,6) rectangle (21.5,14);\draw (19.5,0) -- (19.5,12);}%tikzmath
}{\ast}\!\tikzmath[scale=\planscale]{\useasboundingbox (17.5,-2) rectangle (23,14);\draw (21,12) -- (19.5,12) -- (19.5,0) -- (21,0);}\Big)%tikzmath
\vee \Big(\tikzmath[scale=\planscale]{\useasboundingbox (13,-2) rectangle (32,14);\draw (15,12) -- (24,12) (18,0) -- (21,0);
\draw[ultra thick](24,12) -- (30,12) -- (30,6);\draw[dash pattern=on .4pt off .62pt](19.5,0) -- (19.5,12);}\Big) %tikzmath
\, : \, 
\Big(\tikzmath[scale=\planscale]{\useasboundingbox (16,-2) rectangle (21.5,14);\draw (18,12) -- (19.5,12) -- (19.5,0) -- (18,0);}
\!\underset{\tikzmath[scale=.025]{\useasboundingbox (17.5,6) rectangle (21.5,14);\draw (19.5,0) -- (19.5,12);}%tikzmath
}{\ast}\!\tikzmath[scale=\planscale]{\useasboundingbox (17.5,-2) rectangle (23,14);\draw (21,12) -- (19.5,12) -- (19.5,0) -- (21,0);}\Big)%tikzmath
\vee \Big(\tikzmath[scale=\planscale]{\useasboundingbox (13,-2) rectangle (32,14);\draw (15,12) -- (24,12) (18,0) -- (21,0);
\draw[ultra thick](24,12) -- (30,12) -- (30,6);}\Big)\Big\rrbracket %tikzmath
\\
\;=&\;\,
\Big\llbracket\,
\tikzmath[scale=\planscale]{\useasboundingbox (16,-2) rectangle (21.5,14);\draw (18,12) -- (19.5,12) -- (19.5,0) -- (18,0);}
\!\underset{\tikzmath[scale=.025]{\useasboundingbox (17.5,6) rectangle (21.5,14);\draw (19.5,0) -- (19.5,12);}%tikzmath
}{\ast}\!\tikzmath[scale=\planscale]{\useasboundingbox (11.5,-2) rectangle (32,14);\draw (24,12) -- (13.5,12) -- (13.5,0) -- (15,0)(18,0) -- (21,0);
\draw[ultra thick](24,12) -- (30,12) -- (30,6);\draw[dash pattern=on .4pt off .62pt](19.5,0) -- (19.5,12);} %tikzmath
\; : \; 
\tikzmath[scale=\planscale]{\useasboundingbox (16,-2) rectangle (21.5,14);\draw (18,12) -- (19.5,12) -- (19.5,0) -- (18,0);}
\!\underset{\tikzmath[scale=.025]{\useasboundingbox (17.5,4) rectangle (21.5,14);\draw (19.5,0) -- (19.5,12);}%tikzmath
}{\ast}\!\tikzmath[scale=\planscale]{\useasboundingbox (11.5,-2) rectangle (32,14);\draw (24,12) -- (13.5,12) -- (13.5,0) -- (15,0)(18,0) -- (21,0);
\draw[ultra thick](24,12) -- (30,12) -- (30,6);} %tikzmath
\,\Big\rrbracket %tikzmath
\\
\;=&\;\,
\Big\llbracket\,
\tikzmath[scale=\planscale]{\useasboundingbox (11.5,-2) rectangle (32,14);\draw (24,12) -- (13.5,12) -- (13.5,0) -- (15,0)(18,0) -- (21,0);
\draw[ultra thick](24,12) -- (30,12) -- (30,6);\draw[dash pattern=on .4pt off .62pt](19.5,0) -- (19.5,12);} %tikzmath
\; : \; 
\tikzmath[scale=\planscale]{\useasboundingbox (11.5,-2) rectangle (32,14);\draw (24,12) -- (13.5,12) -- (13.5,0) -- (15,0)(18,0) -- (21,0);
\draw[ultra thick](24,12) -- (30,12) -- (30,6);} %tikzmath
\,\Big\rrbracket %tikzmath
\,=\, \nu,
\end{split}
\]
where the fourth and fifth equalities follow from Lemmas \ref{lem:add-to-fusion-of-algebras} and \ref{lem:[A * Bhat : A * B] = [Bhat : B]}, respectively.

We conclude the argument by noting that, by \eqref{eq:matrix-of-stat-dim-for-AcBcC},\smallskip
\[\smallskip
\left\llbracket \tikzmath[scale=\planscale]{\useasboundingbox (4,-2) rectangle (32,14);\draw (12,12) -- (24,12) (18,0) -- (21,0) (6,12) -- (9,12);
\draw[ultra thick](24,12) -- (30,12) -- (30,6);\draw[dash pattern=on .4pt off .62pt] (19.5,0) -- (19.5,12);} %tikzmath
\, : \tikzmath[scale=\planscale]{\useasboundingbox (4,-2) rectangle (32,14);\draw (12,12) -- (24,12) (18,0) -- (21,0) (6,12) -- (9,12);\draw[ultra thick](24,12) -- (30,12) -- (30,6);} %tikzmath
\right\rrbracket\left\llbracket \tikzmath[scale=\planscale]{\useasboundingbox (4,-2) rectangle (32,14);\draw (12,12) -- (24,12) (18,0) -- (21,0) (6,12) -- (9,12);
\draw[ultra thick](24,12) -- (30,12) -- (30,6);\draw[dash pattern=on .4pt off .62pt](7.5,12) arc (180:360:3 and 4)(19.5,0) -- (19.5,12);} %tikzmath
\, : \tikzmath[scale=\planscale]{\useasboundingbox (4,-2) rectangle (32,14);\draw (12,12) -- (24,12) (18,0) -- (21,0) (6,12) -- (9,12);
\draw[ultra thick](24,12) -- (30,12) -- (30,6);\draw[dash pattern=on .4pt off .62pt] (19.5,0) -- (19.5,12);} %tikzmath
\right\rrbracket
=
\left\llbracket \tikzmath[scale=\planscale]{\useasboundingbox (4,-2) rectangle (32,14);\draw (12,12) -- (24,12) (18,0) -- (21,0) (6,12) -- (9,12);
\draw[ultra thick](24,12) -- (30,12) -- (30,6);\draw[dash pattern=on .4pt off .62pt](7.5,12) arc (180:360:3 and 4);} %tikzmath
\, : \tikzmath[scale=\planscale]{\useasboundingbox (4,-2) rectangle (32,14);\draw (12,12) -- (24,12) (18,0) -- (21,0) (6,12) -- (9,12);\draw[ultra thick](24,12) -- (30,12) -- (30,6);} %tikzmath
\right\rrbracket
\left\llbracket \tikzmath[scale=\planscale]{\useasboundingbox (4,-2) rectangle (32,14);\draw (12,12) -- (24,12) (18,0) -- (21,0) (6,12) -- (9,12);
\draw[ultra thick](24,12) -- (30,12) -- (30,6);\draw[dash pattern=on .4pt off .62pt](7.5,12) arc (180:360:3 and 4)(19.5,0) -- (19.5,12);} %tikzmath
\, : \tikzmath[scale=\planscale]{\useasboundingbox (4,-2) rectangle (32,14);\draw (12,12) -- (24,12) (18,0) -- (21,0) (6,12) -- (9,12);
\draw[ultra thick](24,12) -- (30,12) -- (30,6);\draw[dash pattern=on .4pt off .62pt](7.5,12) arc (180:360:3 and 4);} %tikzmath
\right\rrbracket.
\]
In light of the above computations, that equation gives $\nu\sqrt{\mu(\calb)}=\nu^2$;
since $\nu$ is finite, we must have $\nu=\sqrt{\mu(\calb)}$, as required.
\end{proof}

As a corollary, we obtain the following improvement on Lemma \ref{lem: [ : ] < mu(A)}:

\begin{corollary}\label{cor: [[ : ]] = sqrt mu(A)}
We have
\begin{flalign*}
&&
\left\llbracket\,\tikzmath[scale=\planscale]{\useasboundingbox (10,-2) rectangle (32,14);\draw (12,12) -- (24,12) (18,0) -- (21,0);
\draw[ultra thick](24,12) -- (30,12) -- (30,6);\draw[dash pattern=on .4pt off .62pt] (19.5,0) -- (19.5,12);} %tikzmath
\;:\;\tikzmath[scale=\planscale]{\useasboundingbox (10,-2) rectangle (32,14);\draw (12,12) -- (24,12) (18,0) -- (21,0);
\draw[ultra thick](24,12) -- (30,12) -- (30,6);}%tikzmath
\,\right\rrbracket\;=\;\sqrt{\mu(\calb)}.&&\Box
\end{flalign*}
\end{corollary}

\begin{corollary}\label{cor:muB=mu3=mu4}
We have the following two equalities:
\begin{eqnarray*}
\left\llbracket \tikzmath[scale=\planscale]{\useasboundingbox (-2,-2) rectangle (32,14);\draw[thick, double](0,6) -- (0,12) -- (6,12);
\draw (6,12) -- (9,12) (12,12) -- (24,12) (18,0) -- (21,0);\draw[ultra thick](24,12) -- (30,12) -- (30,6);\draw[dash pattern=on .4pt off .62pt] (19.5,0) -- (19.5,12);} %tikzmath
\; : \; \tikzmath[scale=\planscale]{\useasboundingbox (-2,-2) rectangle (32,14);\draw[thick, double](0,6) -- (0,12) -- (6,12);
\draw (6,12) -- (9,12) (12,12) -- (24,12) (18,0) -- (21,0);\draw[ultra thick](24,12) -- (30,12) -- (30,6);} %tikzmath
\right\rrbracket & \; = \; & \sqrt{\mu(\calb)}\,,\\
\left\llbracket \tikzmath[scale=\planscale]{\useasboundingbox (-2,-2) rectangle (32,14);\draw[thick, double](0,6) -- (0,12) -- (6,12);\draw (6,12) -- (9,12) (12,12) -- (24,12) (18,0) -- (21,0);
\draw[ultra thick](24,12) -- (30,12) -- (30,6);\draw[dash pattern=on .4pt off .62pt] (7.5,12) arc (180:360:3 and 4)(19.5,0) -- (19.5,12);} %tikzmath
\; : \; \tikzmath[scale=\planscale] {\useasboundingbox (-2,-2) rectangle (32,14);\draw[thick, double](0,6) -- (0,12) -- (6,12);\draw (6,12) -- (9,12) (12,12) -- (24,12) (18,0) -- (21,0);
\draw[ultra thick](24,12) -- (30,12) -- (30,6);\draw[dash pattern=on .4pt off .62pt] (19.5,0) -- (19.5,12);} %tikzmath
\right\rrbracket & \; = \; & \sqrt{\mu(\calb)}\,. 
\end{eqnarray*}
\end{corollary}

\begin{proof}
The first equality follows immediately from Corollary \ref{cor: [[ : ]] = sqrt mu(A)}.
For the second equality, note that 
$\big\llbracket\, \tikzmath[scale=\textscale]{\useasboundingbox (-2,-2) rectangle (32,14);\draw[thick, double](0,6) -- (0,12) -- (6,12);\draw (6,12) -- (9,12) (12,12) -- (24,12) (18,0) -- (21,0);
\draw[ultra thick](24,12) -- (30,12) -- (30,6);\draw[dash pattern=on .4pt off .62pt] (7.5,12) arc (180:360:3 and 4)(19.5,0) -- (19.5,12);} %tikzmath
\, : \, \tikzmath[scale=\textscale] {\useasboundingbox (-2,-2) rectangle (32,14);\draw[thick, double](0,6) -- (0,12) -- (6,12);\draw (6,12) -- (9,12) (12,12) -- (24,12) (18,0) -- (21,0);
\draw[ultra thick](24,12) -- (30,12) -- (30,6);\draw[dash pattern=on .4pt off .62pt] (19.5,0) -- (19.5,12);} %tikzmath
\,\big\rrbracket \; = \; \big\llbracket\,\tikzmath[scale=\textscale]{\useasboundingbox (-2,-2) rectangle (32,14);\draw[thick, double] (0,6) -- (0,0) -- (6,0);\draw (6,0) -- (18,0) (21,0) -- (24,0) (9,12) -- (12,12);
\draw[ultra thick] (24,0) -- (30,0) -- (30,6);\draw[dash pattern=on .4pt off .62pt](10.5,0) -- (10.5,12);}%tikzmath
\, : \, \tikzmath[scale=\textscale]{\useasboundingbox (-2,-2) rectangle (32,14);\draw[thick, double] (0,6) -- (0,0) -- (6,0);\draw (6,0) -- (18,0) (21,0) -- (24,0) (9,12) -- (12,12);
\draw[ultra thick] (24,0) -- (30,0) -- (30,6);}%tikzmath
\,\big\rrbracket$ by Lemma \ref{lem:commutants}; the result follows by a version of Corollary \ref{cor: [[ : ]] = sqrt mu(A)} in which the roles of the nets $\cala$ and $\calc$ have been interchanged.
\end{proof}

\addtocontents{toc}{\protect\newpage}

\chapter{The $1 \boxtimes 1$-isomorphism}
  \label{sec:1-box-1}

We are now in a position to prove that the map $\Omega$ \eqref{eq: definition of Omega}, from the vacuum sector of the composition of two defects to the fusion of the vacuum sectors of the individual defects, is an isomorphism.
This isomorphism provides the modification\vspace{.2cm}
\begin{equation}\label{eq: 1x1iso}
\left(\tikzmath{\node (a) at (0,0) {$\cala$};\node (b) at (2,0) {$\calb$};\node (c) at (4,0) {$\calc$};
\draw[->] (a) to[out=45, in=135]node[above]{$\scriptstyle D$} (b);\draw[->] (a) to[out=-45, in=-135]node[below]{$\scriptstyle D$} (b);
\draw[->] (b) to[out=45, in=135]node[above]{$\scriptstyle E$} (c);\draw[->] (b) to[out=-45, in=-135]node[below]{$\scriptstyle E$} (c);
\node at (1,0) {$\Downarrow$};\node at (1.3,0) {$\scriptstyle 1_D$};\node at (3,0) {$\Downarrow$};\node at (3.3,0) {$\scriptstyle 1_E$};
}%tikzmath
\right)\quad\tikzmath{\useasboundingbox (-.4,-.3) rectangle (.4,.3);\node at (0,0) {$\Rrightarrow$};\node at (0,.35) {$\scriptstyle \Omega^{-1}_{D,E}$};} %tikzmath
\quad\left(\tikzmath{\node (a') at (0,0) {$\cala$};\node (c') at (3.8,0) {$\calc$};
\draw[->] (a') .. controls (1,.65) and (2.8,.65) .. node[above]{$\scriptstyle D\circledast_\calb E$} (c');
\draw[->] (a') .. controls (1,-.65) and (2.8,-.65) .. node[below]{$\scriptstyle D\circledast_\calb E$} (c');
\node at (1.9,0) {$\Downarrow$};\node at (2.55,0) {$\scriptstyle 1_{D\circledast_\calb E}$};}%tikzmath
\right)\vspace{.2cm}
\end{equation}
that one expects in any 3-category. More importantly, it also provides the foundation for our construction of the fundamental interchange modification
\[
\left(
\tikzmath[scale=1]{\node[inner sep=5] (a) at (0,0) {$\cala$};\node[inner sep=5] (b) at (2,0) {$\calb$};\node[inner sep=5] (bb) at (2.75,0) {$\calb$}; \node[inner sep=5] (c) at (4.75,0) {$\calc$};
\draw[->] (a) to[out=55, in=125]node[above]{$\scriptstyle D$} (b);\draw[->] (a) to[out=-55, in=-125]node[below]{$\scriptstyle P$} (b);
\draw[->] (a) -- node[fill=white, inner sep=.5]{$\scriptstyle F$} (b);\draw[->] (bb) -- node[fill=white, inner sep=.5]{$\scriptstyle G$} (c);
\draw[->] (bb) to[out=55, in=125]node[above]{$\scriptstyle E$} (c);\draw[->] (bb) to[out=-55, in=-125]node[below]{$\scriptstyle Q$} (c);
\node at (2.375,0) {$\circ$};
\node at (.95,.35) {$\scriptstyle\Downarrow$};\node at (1.15,.35) {$\scriptstyle H$};\node at (3.7,.35) {$\scriptstyle\Downarrow$};\node at (3.9,.35) {$\scriptstyle K$};
\node at (.95,-.35) {$\scriptstyle\Downarrow$};\node at (1.15,-.35) {$\scriptstyle L$};\node at (3.7,-.35) {$\scriptstyle\Downarrow$};\node at (3.9,-.35) {$\scriptstyle M$};
}
\right)
\quad\tikzmath{\useasboundingbox (-.4,-.3) rectangle (.4,.3);\node at (0,0) {$\Rrightarrow$};\node at (0,.35) {};} %tikzmath
\quad
\left(
\tikzmath[scale=1]{\node[inner sep=5] (a) at (0,0) {$\cala$};\node[inner sep=5] (b) at (2,0) {$\calb$};\node[inner sep=5] (c) at (4,0) {$\calc$};
\node[inner sep=5] (aa) at (0,-.75) {$\cala$};\node[inner sep=5] (bb) at (2,-.75) {$\calb$};\node[inner sep=5] (cc) at (4,-.75) {$\calc$}; 
\node at (2,-.375) {$\circ$};
\draw[->] (a) to[out=55, in=125]node[above]{$\scriptstyle D$} (b);\draw[->] (aa) to[out=-55, in=-125]node[below]{$\scriptstyle P$} (bb);
\draw[->] (a) -- node[fill=white, inner sep=.5]{$\scriptstyle F$} (b);
\draw[->] (aa) -- node[fill=white, inner sep=.5]{$\scriptstyle F$} (bb);
\draw[->] (b) -- node[fill=white, inner sep=.5]{$\scriptstyle G$} (c);
\draw[->] (bb) -- node[fill=white, inner sep=.5]{$\scriptstyle G$} (cc);
\draw[->] (b) to[out=55, in=125]node[above]{$\scriptstyle E$} (c);
\draw[->] (bb) to[out=-55, in=-125]node[below]{$\scriptstyle Q$} (cc);
\node at (.95,.35) {$\scriptstyle\Downarrow$};\node at (1.15,.35) {$\scriptstyle H$};\node at (2.95,.35) {$\scriptstyle\Downarrow$};\node at (3.15,.35) {$\scriptstyle K$};
\node at (.95,-1.1) {$\scriptstyle\Downarrow$};\node at (1.15,-1.1) {$\scriptstyle L$};\node at (2.95,-1.1) {$\scriptstyle\Downarrow$};\node at (3.15,-1.1) {$\scriptstyle M$};
}
\right)
\]
present in any 3-category; see Section~\ref{subsec:interchange}.

\section{The $1 \boxtimes 1$-map is an isomorphism}

Let $\cala$, $\calb$, and $\calc$ be conformal nets, always assumed to be irreducible, and let ${}_\cala D_\calb$ and ${}_\calb E_\calc$ be defects.
Assume furthermore that $\calb$ has finite index.
As before, we let $\tikzmath[scale=\textscale]{\fill[vacuumcolor]  (0,0) rectangle (24,12);\draw[thick, double]  (6,0) -- (0,0) -- (0,12) -- (6,12);
\draw (6,12) -- (18,12)  (6,0) -- (18,0);\draw[ultra thick] (18,0) -- (24,0) -- (24,12) -- (18,12);}$
denote the Hilbert space $L^2((D\circledast E)(S^1_\top))$, and we let
$\tikzmath[scale=\textscale] {\fill[vacuumcolor] (0,0) rectangle (24,12);\draw[thick, double]  (6,0) -- (0,0) -- (0,12) -- (6,12);
\draw (6,12) -- (12,12) -- (12,0) -- (6,0)(18,12) -- (12,12) -- (12,0) -- (18,0);\draw[ultra thick] (18,0) -- (24,0) -- (24,12) -- (18,12);}$
denote the fusion $L^2(D(S^1_\top))\boxtimes_{\calb(I)}L^2(E(S^1_\top))$, where $I$ is the middle vertical interval as in~\eqref{eq: figures of intervals 1}.

\begin{maintheorem}\label{thm: Omega is an iso}
Let $\cala$, $\calb$, and $\calc$ be irreducible conformal nets, and let ${}_\cala D_\calb$ and ${}_\calb E_\calc$ be defects.
The map
\begin{equation}\label{eq: Re: Omega}
\Omega_{D,E}\;:\;
\tikzmath[scale=\displscale]{\fill[vacuumcolor]  (0,0) rectangle (24,12);\draw[thick, double]  (6,0) -- (0,0) -- (0,12) -- (6,12);
\draw (6,12) -- (18,12)  (6,0) -- (18,0);\draw[ultra thick] (18,0) -- (24,0) -- (24,12) -- (18,12);} %tikzmath
\;\;\longrightarrow\;\;
\tikzmath[scale=\displscale] {\fill[vacuumcolor] (0,0) rectangle (24,12);\draw[thick, double]  (6,0) -- (0,0) -- (0,12) -- (6,12);
\draw (6,12) -- (12,12) -- (12,0) -- (6,0)(18,12) -- (12,12) -- (12,0) -- (18,0);\draw[ultra thick] (18,0) -- (24,0) -- (24,12) -- (18,12);} %tikzmath
\end{equation}
constructed in \eqref{eq: definition of Omega} is a unitary isomorphism.
\end{maintheorem}

\begin{proof}
Because the constructions of the source and target of $\Omega_{D,E}$ commute with direct integrals, we may assume without loss of generality that $D$ and $E$ are irreducible.
By construction the map $\Omega_{D,E}$ is an isometry.
The algebras 
\,$\tikzmath[scale=\textscale] {\useasboundingbox (-2,0) rectangle (26,14);\draw[thick, double] (0,6) -- (0,12) -- (6,12);\draw (6,12) -- (18,12);\draw[ultra thick] (18,12) -- (24,12) -- (24,6);}
=(D\circledast E)(S^1_\top)$\,
and \,$\tikzmath[scale=\textscale] {\useasboundingbox (-2,-2) rectangle (26,12);\draw[thick, double] (0,6) -- (0,0) -- (6,0);\draw (6,0) -- (18,0);\draw[ultra thick] (18,0) -- (24,0) -- (24,6);}
=(D\circledast E)(S^1_\bot)$\,
act faithfully on both sides of \eqref{eq: Re: Omega}.
They are certainly each other's commutants on
$\tikzmath[scale=\textscale]{\fill[vacuumcolor]  (0,0) rectangle (24,12);\draw[thick, double]  (6,0) -- (0,0) -- (0,12) -- (6,12);
\draw (6,12) -- (18,12)  (6,0) -- (18,0);\draw[ultra thick] (18,0) -- (24,0) -- (24,12) -- (18,12);}$\,, and the \linebreak$(D\circledast E)(S^1_\top)$\,--\,$(D\circledast E)(S^1_\bot)$-bimodule 
$\tikzmath[scale=\textscale]{\fill[vacuumcolor]  (0,0) rectangle (24,12);\draw[thick, double]  (6,0) -- (0,0) -- (0,12) -- (6,12);
\draw (6,12) -- (18,12)  (6,0) -- (18,0);\draw[ultra thick] (18,0) -- (24,0) -- (24,12) -- (18,12);}$\,
has the identity matrix as its matrix of statistical dimensions.
We already know (by Propositions~\ref{prop:G=L2} and~\ref{prop: local-fusion}) that $\Omega_{D,E}$ is an embedding.
By Theorem~\ref{thm:Haag-duality-composition-defects}, the algebras 
\,$\tikzmath[scale=\textscale] {\useasboundingbox (-2,0) rectangle (26,14);\draw[thick, double] (0,6) -- (0,12) -- (6,12);\draw (6,12) -- (18,12);\draw[ultra thick] (18,12) -- (24,12) -- (24,6);}$\,
and \,$\tikzmath[scale=\textscale] {\useasboundingbox (-2,-2) rectangle (26,12);\draw[thick, double] (0,6) -- (0,0) -- (6,0);\draw (6,0) -- (18,0);\draw[ultra thick] (18,0) -- (24,0) -- (24,6);}$\, 
are also each other's commutants on
$\tikzmath[scale=\textscale] {\fill[vacuumcolor] (0,0) rectangle (24,12);\draw[thick, double]  (6,0) -- (0,0) -- (0,12) -- (6,12);
\draw (6,12) -- (12,12) -- (12,0) -- (6,0)(18,12) -- (12,12) -- (12,0) -- (18,0);\draw[ultra thick] (18,0) -- (24,0) -- (24,12) -- (18,12);}$\,.
It follows (since statistical dimension is additive and every nonzero bimodule has nonzero statistical dimension) that the
$(D\circledast E)(S^1_\top)$\,--\,$(D\circledast E)(S^1_\bot)$-bimodule
$\tikzmath[scale=\textscale] {\fill[vacuumcolor] (0,0) rectangle (24,12);\draw[thick, double]  (6,0) -- (0,0) -- (0,12) -- (6,12);
\draw (6,12) -- (12,12) -- (12,0) -- (6,0)(18,12) -- (12,12) -- (12,0) -- (18,0);\draw[ultra thick] (18,0) -- (24,0) -- (24,12) -- (18,12);}$\,
also has the identity matrix as its matrix of statistical dimensions,
and that $\Omega_{D,E}$ is an isomorphism.
\end{proof}

Given the crucial importance of the ``$1$ times $1$ isomorphism'' $\Omega$,
we collect in one place the main ingredients used in its definition.
These are the unitary isomorphisms $\Psi$ from \eqref{eq: def of Psi} and $\Phi$ from \eqref{eq: Upsilon -- v_K -- Psi}; the $1 \boxtimes 1$-isomorphism is the composite $\Omega$ from \eqref{eq: definition of Omega}:
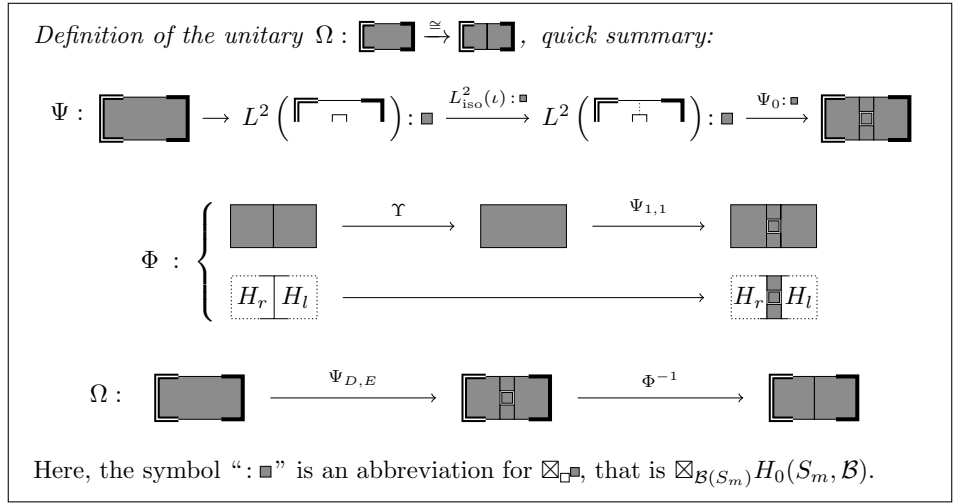
\label{THE RECAP OF OMEGA}
\[
\tikzmath{
\node[draw, inner ysep=6, inner xsep=10, scale=.95] at (0,0) {\parbox{12.5cm}{
{\it Definition of the unitary\, $\Omega:\,\tikzmath[scale=\textscale]{\fill[vacuumcolor]  (0,0) rectangle (24,12);\draw[thick, double]  (6,0) -- (0,0) -- (0,12) -- (6,12);
\draw (6,12) -- (18,12)  (6,0) -- (18,0);\draw[ultra thick] (18,0) -- (24,0) -- (24,12) -- (18,12);}\xrightarrow{\scriptscriptstyle\cong}
\tikzmath[scale=\textscale] {\fill[vacuumcolor] (0,0) rectangle (24,12);\draw[thick, double]  (6,0) -- (0,0) -- (0,12) -- (6,12);
\draw (6,12) -- (12,12) -- (12,0) -- (6,0)(18,12) -- (12,12) -- (12,0) -- (18,0);\draw[ultra thick] (18,0) -- (24,0) -- (24,12) -- (18,12);}$\,,\, quick summary:}\vspace{.1cm}
$$
\Psi:\!
\tikzmath{
\node[inner sep =5] (a) at (-4.8,0) {$
\tikzmath[scale=\displscale]{\fill[vacuumcolor]  (0,0) rectangle (24,12);\draw[thick, double]  (6,0) -- (0,0) -- (0,12) -- (6,12);
\draw (6,12) -- (18,12)  (6,0) -- (18,0);\draw[ultra thick] (18,0) -- (24,0) -- (24,12) -- (18,12);} %tikzmath
$};
\node[inner sep =5] (b) at (-2.1,0) {$L^2 \left(
\tikzmath[scale=\displscale]{\useasboundingbox (-2,0) rectangle (26,14);\draw[thick, double] (0,6) -- (0,12) -- (6,12);\draw (6,12) -- (18,12)(10,6) -- (10,8) -- (14,8) -- (14,6);
\draw[ultra thick] (18,12) -- (24,12) -- (24,6);} %tikzmath
\right)\! :\tikzmath[scale=\displscale]{\filldraw[fill=vacuumcolor](10.5,4.5) rectangle (13.5,7.5);} %tikzmath
$};
\node[inner sep =5] (c) at (2.1,0) {$L^2 \left(
\tikzmath[scale=\displscale]{\useasboundingbox (-2,0) rectangle (26,14);\draw[thick, double] (0,6) -- (0,12) -- (6,12);\draw (6,12) -- (18,12)(10,6) -- (10,8) -- (14,8) -- (14,6);
\draw[ultra thick] (18,12) -- (24,12) -- (24,6);\draw[densely dotted] (12,8) -- (12,12);} %tikzmath
\right)\! :\tikzmath[scale=\displscale]{\filldraw[fill=vacuumcolor](10.5,4.5) rectangle (13.5,7.5);} %tikzmath
$};
\node[inner sep =5] (d) at (5.3,0) {$
\tikzmath[scale=\displscale]
{\fill[vacuumcolor] (0,0) rectangle (10,12)(14,0) rectangle (24,12); \draw (6,12) -- (10,12) -- (10,0) -- (6,0) (18,12) -- (14,12) -- (14,0) -- (18,0);
\draw[thick, double] (6,12) -- (0,12) -- (0,0) -- (6,0);\draw[ultra thick] (18,12) -- (24,12) -- (24,0) -- (18,0);\fill[vacuumcolor](10,0) rectangle (14,4)(10,8) rectangle (14,12);
\draw (10,0) rectangle (14,4) (10,8) rectangle (14,12);\filldraw[fill=vacuumcolor](10.5,4.5) rectangle (13.5,7.5);} %tikzmath
$};
\draw [->] ($(a)!.5!(b)$) (a) --  (b);
\draw [->] ($(b)!.5!(c)$) node[above, yshift=-2]{$\scriptstyle L^2_\mathrm{iso}(\iota)\,:\,\tikzmath[scale=\textscale]{\filldraw[fill=vacuumcolor](10.5,4.5) rectangle (13.5,7.5);}$} (b) -- (c);
\draw [->] ($(c)!.5!(d)$) node[above, yshift=-2, xshift=10]{$\scriptstyle \Psi_0:\,\tikzmath[scale=\textscale]{\filldraw[fill=vacuumcolor](10.5,4.5) rectangle (13.5,7.5);}$} (c) --  (d);
} % tikzmath
\vspace{.4cm}$$
$$
\Phi\;:\; \left\{
\tikzmath{\useasboundingbox (-4.3,-1.4) rectangle (4.3,.4);
\node (a) at (-3.5,0) {$
\tikzmath[scale=\displscale] {\fill[vacuumcolor]  (0,0) rectangle (12,12) (12,0) rectangle (24,12);\draw (0,0) rectangle (12,12) (12,0) rectangle (24,12);} %tikzmath
$};
\node (b) at (0,0) {$
\tikzmath[scale=\displscale] {\fill[vacuumcolor] (0,0) rectangle  (24,12);\draw (0,0) rectangle  (24,12);} %tikzmath
$};
\node (c) at (3.5,0) {$
\tikzmath[scale=\displscale] {\fill[vacuumcolor] (0,0) rectangle (10,12) (14,0) rectangle (24,12) (10,0) rectangle (14,4) (10,8) rectangle (14,12) (10.5,4.5) rectangle (13.5,7.5);
\draw  (0,0) rectangle (10,12) (14,0) rectangle (24,12)(10,0) rectangle (14,4) (10,8) rectangle (14,12) (10.5,4.5) rectangle (13.5,7.5);} %tikzmath
$};
\node (d) at (-3.5,-1) {$
\tikzmath[scale=\displscale] {\draw (8,12) -- (16,12) (8,0) -- (16,0) (12,0) -- (12,12); \draw[ultra thin, dash pattern=on .5pt off 1pt]
(8,12) -- (0,12) -- (0,0) -- (8,0) (16,12) -- (24,12) -- (24,0) -- (16,0); \draw (6,6) node {$H_r$} (18,6) node {$H_l$};} %tikzmath
$};
\node (e) at (3.5,-1) {$
\tikzmath[scale=\displscale]
{\draw (8,12) -- (10,12) -- (10,0) -- (8,0) (16,12) -- (14,12) -- (14,0) -- (16,0);
\draw[ultra thin, dash pattern=on .5pt off 1pt] (8,12) -- (0,12) -- (0,0) -- (8,0) (16,12) -- (24,12) -- (24,0) -- (16,0); \draw (5,6) node {$H_r$} (19,6) node {$H_l$};
\fill[vacuumcolor] (10,0) rectangle (14,4) (10,8) rectangle (14,12) (10.5,4.5) rectangle (13.5,7.5);
\draw (10,0) rectangle (14,4)(10,8) rectangle (14,12) (10.5,4.5) rectangle (13.5,7.5);} %tikzmath 
$};
\draw [->] ($(a)!.5!(b)$) node[above, yshift=-2]{$\scriptstyle \Upsilon$} (a) --  (b);
\draw [->] ($(b)!.5!(c)$) node[above, yshift=-2] {$\scriptstyle \Psi_{1,1}$} (b) -- (c);
\draw [->] ($(d.east)+(-.25,0)$)  --  ($(e.west)+(.3,0)$);
} %tikzmath
\right.\vspace{.3cm}
$$
$$
\Omega:
\tikzmath{
\node (a) at (-4.3,0) {$
\tikzmath[scale=\displscale]{\fill[vacuumcolor]  (0,0) rectangle (24,12);\draw[thick, double]  (6,0) -- (0,0) -- (0,12) -- (6,12);
\draw (6,12) -- (18,12)  (6,0) -- (18,0);\draw[ultra thick] (18,0) -- (24,0) -- (24,12) -- (18,12);} %tikzmath
$};
\node (b) at (0,0) {$
\tikzmath[scale=\displscale]{\fill[vacuumcolor] (0,0) rectangle (10,12)(14,0) rectangle (24,12);\draw (6,12) -- (10,12) -- (10,0) -- (6,0)(18,12) -- (14,12) -- (14,0) -- (18,0);
\draw[thick, double] (6,12) -- (0,12) -- (0,0) -- (6,0);\draw[ultra thick]  (18,12) -- (24,12) -- (24,0) -- (18,0);\fill[vacuumcolor]  (10,0) rectangle (14,4)(10,8) rectangle (14,12);
\draw (10,0) rectangle (14,4) (10,8) rectangle (14,12);\fill[vacuumcolor]  (10.5,4.5) rectangle (13.5,7.5);\draw (10.5,4.5) rectangle (13.5,7.5);} %tikzmath
$};
\node (c) at (4.3,0) {$
\tikzmath[scale=\displscale] {\fill[vacuumcolor] (0,0) rectangle (24,12);\draw[thick, double]  (6,0) -- (0,0) -- (0,12) -- (6,12);
\draw (6,12) -- (12,12) -- (12,0) -- (6,0)(18,12) -- (12,12) -- (12,0) -- (18,0);\draw[ultra thick] (18,0) -- (24,0) -- (24,12) -- (18,12);} %tikzmath
$};
\draw [->] ($(a)!.5!(b)$) node[above, yshift=-2]{$\scriptstyle \Psi_{D,E}$} (a) --  (b);
\draw [->] ($(b)!.5!(c)$) node[above, yshift=-2]{$\scriptstyle \Phi^{-1}$} (b) -- (c);
} %tikzmath 
\vspace{.2cm}$$
Here, the symbol ``\,$:\tikzmath[scale=\displscale]{\filldraw[fill=vacuumcolor](10.5,4.5) rectangle (13.5,7.5);}$\,'' is an abbreviation for $\def\thinsquare{\tikzmath[scale=\textscale]{\draw[ultra thin] (0,0) rectangle (4,4);}}
\def\thinsqfill{\tikzmath[scale=\textscale]{\filldraw[ultra thin, fill=vacuumcolor] (0,0) rectangle (4,4);}}
\boxtimes_{\thinsquare}\thinsqfill$, that is $\boxtimes_{\calb(S_m)} H_0(S_m,\calb)$.
}};
}\medskip
\]

\subsection*{\hspace*{-18pt}The $1\boxtimes 1$ isomorphism is monoidal}  We now address the compatibility of the isomorphism $\Omega$ with the symmetric monoidal 
  structure on conformal nets and defects.
  Let  $_{\cala_ 1} (D_ 1)_{\calb_ 1}$, $_{\calb_ 1} (E_ 1)_{\calc_ 1}$,
  $_{\cala_ 2} (D_ 2)_{\calb_ 2}$ and $_{\calb_ 2} (E_ 2)_{\calc_ 2}$ be defects.
  We abbreviate $D_ 1 := D_ 1(S^1_{\top})$, $D_ 2 := D_ 2(S^1_{\top})$,
  $D_ 1 \otimes D_ 2 := (D_ 1 \otimes D_ 2)(S^1_{\top})$,
  $E_ 1 := E_ 1(S^1_{\top})$, $E_ 2 := E_ 2(S^1_{\top})$,
  $E_ 1 \otimes E_ 2 := (E_ 1 \otimes E_ 2)(S^1_{\top})$, 
  and $B_ 1 := \calb_ 1(I)$, $B_ 2 := \calb_ 2(I)$,
  $B_ 1 \otimes B_ 2 := (\calb_ 1 \otimes \calb_ 2)(I)$, where as before $I$ is the vertical unit interval.
  Combining the various isomorphisms from 
  Appendix~\ref{subsec:compatibility-ox}  
  we have canonical isomorphisms
  \begin{eqnarray}
    \label{eq:circledast-ox}
    \lefteqn{
     L^2(D_ 1 \circledast_{B_ 1} E_ 1)  \otimes  L^2(D_ 2 \circledast_{B_ 2} E_ 2)} \\
    & \hspace{ 26ex}  \cong  &
    L^2( (D_ 1 \otimes D_ 2) \circledast_{B_ 1  \otimes B_ 2} (E_ 1  \otimes E_ 2) ),
    \nonumber \\ 
    \label{eq:boxtimes-ox}
    \lefteqn{
    (L^2(D_ 1) \boxtimes_{B_ 1} L^2(E_ 1))  \otimes  
                   (L^2(D_ 2) \boxtimes_{B_ 2} L^2(E_ 2))} \\
    & \hspace{ 26ex} \cong  &
    L^2(D_ 1  \otimes D_ 2) \boxtimes_{B_ 1 \otimes B_ 2} 
              L^2(E_ 1  \otimes E_ 2). \nonumber 
  \end{eqnarray}

  \begin{proposition}
    \label{prop:omega-is-monodial}
    The unitary isomorphism $\Omega$ is compatible with the symmetric monoidal 
    structure on conformal nets and defects.
    More precisely, if $\calb_ 1$ and $\calb_ 2$ have finite index then 
    \begin{equation} \label{eq:omega-ox}
      {\eqref{eq:boxtimes-ox}}  \circ\, (\Omega_{D_ 1,E_ 1} \otimes \,\Omega_{D_ 2,E_ 2})  
      = \Omega_{D_ 1 \otimes D_ 2, E_ 1  \otimes E_ 2} \circ {\eqref{eq:circledast-ox}}.
    \end{equation}
  \end{proposition}

  \begin{proof}
    For every step in the construction of $\Omega$ there is an isomorphism analogous to~\eqref{eq:boxtimes-ox}
    or~\eqref{eq:circledast-ox}.
    It is a lengthy, but not difficult, exercise to check
    that for each step in the construction of $\Omega$ the
    corresponding version of equation~\eqref{eq:omega-ox}
    holds.
  \end{proof}

\begin{warning} \label{warning:not-natural!}
\emph{It appears that $\Omega$ is not a natural transformation!}
More precisely, there seem to exist irreducible defects 
${}_\cala D_\calb$, ${}_\calb E_\calc$, ${}_\cala F_\calb$, ${}_\calb G_\calc$
and finite natural transformations $\tau:D\to F$ and $\sigma:E\to G$ (see Definition \ref{def:finite defects maps}) for which the following diagram
fails to be commutative:
\smallskip
\begin{equation}\label{eq: the criminal}
\tikzmath{
\matrix [matrix of math nodes,column sep=1.3cm,row sep=7mm]
{ 
|(a)| H_0(D\circledast_\calb E) \pgfmatrixnextcell |(b)| H_0(D)\boxtimes_\calb H_0(E)\\ 
|(c)| H_0(F\circledast_\calb G) \pgfmatrixnextcell |(d)| H_0(F)\boxtimes_\calb H_0(G)\\ 
}; 
\draw[->] (a) -- node [above]	{$\scriptstyle \Omega_{D,E}$} (b);
\draw[->] (c) -- node [above]	{$\scriptstyle \Omega_{F,G}$} (d);
\draw[->] (a) -- node [left]		{$\scriptstyle H_0(\tau \circledast \sigma)$} (c);
\draw[->] (b) -- node [right]	{$\scriptstyle H_0(\tau)\,\boxtimes\, H_0(\sigma)$} (d);
}\smallskip
\end{equation}
Here, $H_0(\tau)$ is the value of the functor $L^2$ on the map $D(S^1_\top)\to F(S^1_\top)$
induced by $\tau$ (and similarly for $H_0(\sigma)$ and $H_0(\tau \circledast \sigma)$).
This problem can be blamed on the bad functorial properties of $L^2_\mathrm{iso}$ (used in the definition of $\Psi$).
However, $\Omega$ is still natural with respect to natural isomorphisms of defects.

There are two ways of dealing with the above situation:
{1.} Restrict to the groupoid part of $\CN_0$ and of $\CN_1$;
{2.} Replace the $L^2_\mathrm{iso}$ by $L^2$ in the definition of $\Psi$---the price to pay for this change is that $\Omega$ is then no longer unitary.  Both options seem to have shortcomings---in our exposition, we have opted for the first option.
One unfortunate consequence of the failure of commutativity 
of~\eqref{eq: the criminal} is that
given defects $D$, $E$, $F$, $G$ as above, and 
given dualizable sectors ${}_DH_F$ and ${}_EK_G$ with 
normalized duals\footnote{See equations \eqdualityequations \, 
  and \eqdualitynormalization \,
  in~\cite{BDH(Dualizability+Index-of-subfactors)}.}
$(\bar H,r_H,s_H)$ and $(\bar K,r_K,s_K)$
the horizontal composition $(\bar H\boxtimes_\calb \bar K, r, s)$ of those two normalized duals is not a normalized dual for $H\boxtimes_\calb K$.
(Here, the unit and counit $r$ and $s$ are given by the obvious formulas in terms of $r_H$, $r_K$, $s_H$, and $s_K$.)
\end{warning}

\section{The $1 \boxtimes 1$-isomorphism for an identity defect}\label{sec: Compatibility with left and right units}

We now show that, if one of the two defects in \eqref{eq: Re: Omega} is an identity defect,
then the $1 \boxtimes 1$-isomorphism admits a much simpler description, in terms of certain natural transformations
\begin{equation*}
\Upsilon^l \colon \tikzmath[scale=\displscale] 
{\fill[vacuumcolor] (0,0) rectangle (12,12);\draw (0,0) rectangle (12,12) (18,0) -- (12,0) -- (12,12) -- (18,12);\draw[dotted] (18,0) -- (24,0) -- (24,12) -- (18,12); } %tikzmath
\; \xrightarrow{\cong} \; \tikzmath[scale=\displscale] {\draw (18,0) -- (0,0) -- (0,12) -- (18,12); \draw[dotted](18,0) -- (24,0) -- (24,12) -- (18,12); } %tikzmath
\quad \text{and} \quad \Upsilon^r\colon
\tikzmath[scale=\displscale]{\fill[vacuumcolor] (12,0) rectangle (24,12);\draw[dotted](6,0) -- (0,0) -- (0,12) -- (6,12);\draw (6,12) -- (12,12) -- (12,0) -- (6,0)(12,0) rectangle (24,12);} %tikzmath
\; \xrightarrow{\cong} \; \tikzmath[scale=\displscale]{\draw[dotted] (6,0) -- (0,0) -- (0,12) -- (6,12);\draw (6,12) -- (24,12) -- (24,0) -- (6,0);} %tikzmath
\end{equation*}
that we describe below.  This section and the following Section~\ref{sec:quasi-identities} are concerned with the behavior of horizontal units in the 3-category of conformal nets; they are more technical and are not needed for the subsequent treatment of the fundamental interchange isomorphism in Section~\ref{subsec:interchange}.

\subsection*{\hspace*{-18pt}Transformations for fusion of vacuum sectors}

Let $S_l$, $S_r$, $S_b$, $I$, $j_l$, and $j_r$ be as in Section \ref{sec: Comparison between F and G}, and let $I_l:=j_l(I)$ and $I_r:=j_r(I)$.
We draw them here for convenience:
\begin{equation} \label{eq:circles-interval-involutions} 
\begin{split} 
S_l = \tikzmath[scale=\displscale]{\useasboundingbox (0,0) rectangle (24,12); \draw (0,0) rectangle (12,12);\draw[->] (5.6,12) -- (5.5,12);}& %tikzmath 
\; , \,\,\,\, 
S_r = \tikzmath[scale=\displscale] {\useasboundingbox (0,0) rectangle (24,12); \draw (12,0) rectangle (24,12); \draw[->] (17.6,12) -- (17.5,12);}\,, %tikzmath
\,\,\,\, 
S_b = \tikzmath[scale=\displscale] {\draw (0,0) rectangle (24,12); \draw[->] (11.6,12) -- (11.5,12);}\;,  %tikzmath
\,\,\,\, 
I \, = \tikzmath[scale=\displscale] {\useasboundingbox (0,0) rectangle (24,12);\draw (12,0) -- (12,12); \draw[->] (12,5.6) -- (12,5.5);}\;,  %tikzmath
\\ %\,\,\,\, 
j_l:\,\tikzmath[scale=\displscale]{\useasboundingbox (0,0) rectangle (24,12);\draw (-.3,0) rectangle (11.8,12);\draw[<->] (6.5,.5) to[in=170, out=70] (11,2);
\draw[<->] (1.5,1) to[in=175, out=50] (11,4);\draw[<->] (.3,6) -- (11,6);\draw[<->] (1.5,11) to[in=-175, out=-50] (11,8);\draw[<->] (6.5,11.5) to[in=-170, out=-70] (11,10);}&\;, %tikzmath
\,\,\,\, 
I_l \, = \tikzmath[scale=\displscale]{\useasboundingbox (0,0) rectangle (24,12); \draw (12,12) -- (0,12) -- (0,0) -- (12,0);\draw[->] (5.6,12) -- (5.5,12);}\,, %tikzmath 
\,\,\,\,\, 
j_r \,  :\,\,\, \tikzmath[scale=\displscale]{\useasboundingbox (0,0) rectangle (-24,12);\draw (.3,0) rectangle (-11.8,12);\draw[<->] (-6.5,.5) to[in=10, out=110] (-11,2);
\draw[<->] (-1.5,1) to[in=5, out=130] (-11,4);\draw[<->] (-.3,6) -- (-11,6);\draw[<->] (-1.5,11) to[in=-5, out=-130] (-11,8);\draw[<->] (-6.5,11.5) to[in=-10, out=-110] (-11,10);}\;, %tikzmath
\,\,\,\, 
I_r = \tikzmath[scale=\displscale] {\useasboundingbox (0,0) rectangle (24,12);\draw (12,12) -- (24,12) -- (24,0) -- (12,0); \draw[->] (17.6,12) -- (17.5,12);}\;. %tikzmath
\end{split}
\end{equation}
Recall that we equipped $S_b$ with a conformal structure that makes
$j_l|_{I_l}\cup \mathrm{Id_{I_r}} \colon S_b \to S_r$ and
$\mathrm{Id_{I_l}} \cup j_r|_{I_r} \colon S_b \to S_l$ conformal 
(and therefore smooth).

Fix a small number $\varepsilon$.
Then $\Upsilon^l$ and $\Upsilon^r$ are invertible natural transformations
$$
\modules{\cala(\partial^\sqsubset([1,3/2 - \e]\!\times\![0,1]))} \;\tworarrow\; \modules{\cala(\partial^\sqsubset([0,3/2 - \e]\!\times\![0,1]))}
$$
and
$$
\modules{\cala(\partial^\sqsupset([1/2+\e,1]\!\times\![0,1]))}\;\tworarrow\; \modules{\cala(\partial^\sqsupset([1/2+\e,2]\!\times\![0,1]))};
$$
that is,
\begin{equation} \label{eq:left+right-Upsilon}
\begin{split}
\Upsilon^l\;\;:\;\;\;
\modules{\cala\big(
\tikzmath[scale=\planscale]{\useasboundingbox (-2,0) rectangle (14,12);\draw (6,12) -- (0,12) -- (0,0) -- (6,0);\draw[->] (2.51,12) -- (2.5,12);} %tikzmath
\big)}\;\;&\tworarrow\,\;\;\modules{\cala(
\tikzmath[scale=\planscale]{\useasboundingbox (-2,0) rectangle (26,12);\draw (18,12) -- (0,12) -- (0,0) -- (18,0);\draw[->] (8.51,12) -- (8.5,12);} %tikzmath
)}
\\
\Upsilon^r\;\;:\;\;\;
\modules{\cala\big(
\tikzmath[scale=\planscale]{\useasboundingbox (-2,0) rectangle (14,12);\draw (6,12) -- (12,12) -- (12,0) -- (6,0);\draw[->] (8.51,12) -- (8.5,12);}
\big)}\;\;&\tworarrow\,\;\;\modules{\cala(
\tikzmath[scale=\planscale]{\useasboundingbox (-2,0) rectangle (26,12);\draw (6,12) -- (24,12) -- (24,0) -- (6,0);\draw[->] (14.51,12) -- (14.5,12);}
)}.\end{split}
\end{equation}
The natural transformation $\Upsilon^l$ goes from the functor
\(
H_0(S_r,\cala) \boxtimes_{\cala(I)} \!- 
\)
to the functor of restriction along $\cala(j_l|_{I_l}\cup \mathrm{Id}):\cala(
\tikzmath[scale=\textscale]{\useasboundingbox (-2,0) rectangle (26,12);\draw (18,12) -- (0,12) -- (0,0) -- (18,0);\draw[->] (8.51,12) -- (8.5,12);} %tikzmath
)\xrightarrow\cong\cala(
\tikzmath[scale=\textscale]{\useasboundingbox (-2,0) rectangle (14,12);\draw (6,12) -- (0,12) -- (0,0) -- (6,0);\draw[->] (2.51,12) -- (2.5,12);} %tikzmath
)$.
Its value on an $\cala(
\tikzmath[scale=\textscale]{\useasboundingbox (-2,0) rectangle (14,12);\draw (6,12) -- (0,12) -- (0,0) -- (6,0);\draw[->] (2.51,12) -- (2.5,12);} %tikzmath
)$-module $K$ is given by
\[
\Upsilon^l_K\,\,\colon\, H_0(S_l,\cala) \boxtimes_{\cala(I)} K \,\xrightarrow{\,\,\,w\otimes 1\,\,\,}\, 
L^2(\cala(I)) \boxtimes_{\cala(I)} K \,\cong\, K,
\]
where $w$ is the isomorphism obtained by composing the canonical identification
$v_{I_l}:H_0(S_l,\cala)\to L^2(\cala(I_l))$ 
from Appendix~\ref{subsec:vacuum-sector-net}
and the map $L^2(\cala(I_l))\to L^2(\cala(I))$ induced by $j_l:I_l\to I$.
Similarly, the natural transformation $\Upsilon^r$ goes from the functor
$-\boxtimes_{\cala(I)}H_0(S_r,\cala)$
to the functor of restriction along $\cala(\mathrm{Id} \cup j_r|_{I_r}):\cala(
\tikzmath[scale=\textscale]{\useasboundingbox (-2,0) rectangle (26,12);\draw (6,12) -- (24,12) -- (24,0) -- (6,0);\draw[->] (13.51,12) -- (13.5,12);}
)\xrightarrow\cong\cala(
\tikzmath[scale=\textscale]{\useasboundingbox (-2,0) rectangle (14,12);\draw (6,12) -- (12,12) -- (12,0) -- (6,0);\draw[->] (8.01,12) -- (8,12);}
)$. 
It is given by
\[
\Upsilon^r_K\,\,\colon\, K \boxtimes_{\cala(I)} H_0(S_r,\cala) \,\xrightarrow{\,\,\,1\otimes v\,\,\,}\, 
K \boxtimes_{\cala(I)} L^2(\cala(I)) \,\cong\, K,
\]
where $v=v_I$ is the identification $H_0(S_r,\cala)\cong L^2(\cala(I))$ 
from Appendix~\ref{subsec:vacuum-sector-net}.

The transformations $\Upsilon^l$ and $\Upsilon^r$ generalize the map $\Upsilon\!:\tikzmath[scale=\textscale] {\fill[vacuumcolor] (0,0) rectangle (24,12); \draw (0,0) rectangle (24,12) (12,0) -- (12,12);} %tikzmath
\to \tikzmath[scale=\textscale] {\filldraw[fill=vacuumcolor] (0,0) rectangle (24,12);}\,$ from \eqref{eq: Upsilon display} and \eqref{eq:Upsilon}:

\begin{lemma}\label{lem: Ups=Ups=Ups}
Let $\epsilon_l:= j_l|_I\cup \mathrm{Id}_{I_r}\in\Conf(S_r, S_b)$ and $\epsilon_r:= \mathrm{Id}_{I_l} \cup j_r|_I\in\Conf(S_l, S_b)$ be as in Section \ref{sec: Comparison between F and G}.
The two maps
\[
H_0(\epsilon_l,\cala)\circ\Upsilon^l_{H_0(S_r,\cala)}\,,\,\,\,\, H_0(\epsilon_r,\cala)\circ\Upsilon^r_{H_0(S_l,\cala)} \;\colon\;
\tikzmath[scale=\displscale] {\fill[vacuumcolor] (0,0) rectangle (24,12); \draw (0,0) rectangle (24,12) (12,0) -- (12,12);} %tikzmath
\;\; \xrightarrow{\,\,\cong\,\,} \;\; \tikzmath[scale=\displscale]{\filldraw[fill=vacuumcolor] (0,0) rectangle (24,12);} %tikzmath
\]
are equal to each other, and are equal to $\Upsilon$.
\end{lemma}

\begin{proof} 
The equality $\Upsilon=H_0(\epsilon_r)\circ\Upsilon^r_{H_0(S_l)}$ follows from the commutativity of the diagram
\[
\tikzmath{
\matrix [matrix of math nodes,column sep=1cm,row sep=4mm]
{|(a)| H_0(S_l)\underset{\cala(I)}\boxtimes H_0(S_r)
\pgfmatrixnextcell[.3cm] |(b)| L^2(\cala(I_l)) \underset{\cala(I)}\boxtimes L^2(\cala(I))\cong L^2(\cala(I_l))
\pgfmatrixnextcell |(c)| H_0(S_b) \\
|(d)| H_0(S_l)\underset{\cala(I)}\boxtimes H_0(S_r)
\pgfmatrixnextcell |(e)| H_0(S_l) \underset{\cala(I)}\boxtimes L^2(\cala(I))\,\cong\, H_0(S_l)
\pgfmatrixnextcell |(f)| H_0(S_b)
\\
};
\draw[->](c -| a.east)--node[above]{$\scriptstyle v_{I_l}\otimes v_I$}(c -| b.west);
\draw[->](c -| b.east)--node[above]{$\scriptstyle v_{I_l}^*$}(c);
\draw[->](f -| d.east)--node[above]{$\scriptstyle 1\otimes v\raisebox{-1.9pt}{$\scriptscriptstyle I$}$}(f -| e.west);
\draw[->](f -| e.east)--node[above]{$\scriptstyle H_0(\epsilon_r)$}(f);
\draw[double, double distance = 2](a)--(d);
\draw[line width = 2, white]($(a.south)+(0,0.01)$)--($(d.north)-(0,0.01)$);
\draw[double, double distance = 2](c)--(f);
\draw[line width = 2, white]($(c.south)+(0,0.01)$)--($(f.north)-(0,0.01)$);
\draw[->]($(b.south)+(1.7,.3)$) --node[left]{$\scriptstyle v_{I_l}^*$} ($(e.north)+(1.7,0)$);
}
\]
where the top row is $\Upsilon$ and the bottom row is $H_0(\epsilon_r,\cala)\circ\Upsilon^r_{H_0(S_l)}$.
The rightmost square commutes by the naturality of the maps $v_I$,
see~\eqref{eq:naturality-v_I}.

To see that $H_0(\epsilon_l)\circ\Upsilon^l_{H_0(S_r)}=H_0(\epsilon_r)\circ\Upsilon^r_{H_0(S_l)}$, one contemplates the diagram
\[
\tikzmath{
\node (a) at (.2,-.2) {$H_0(S_l)\underset{\cala(I)}\boxtimes H_0(S_r)$};
\node (c) at (5,-.2) {$H_0(S_l)\underset{\cala(I)}\boxtimes L^2(\cala(I))\,\cong\, H_0(S_l)$};
\node (d) at (5,1.2) {$L^2(\cala(I_l))\underset{\cala(I)}\boxtimes L^2(\cala(I))\cong L^2(\cala(I_l))$};
\node (e) at (5,2.6) {$L^2(\cala(I))\underset{\cala(I)}\boxtimes L^2(\cala(I))\cong L^2(\cala(I))$};
\node (f) at (.2,3.9) {$L^2(\cala(I_l))\!\underset{\cala(I)}\boxtimes\! H_0(S_r)$};
\node (g) at (6.4,3.9) {$L^2(\cala(I))\!\underset{\cala(I)}\boxtimes\! H_0(S_r)\cong H_0(S_r)$};
\node (h) at (10.5,-.1) {$H_0(S_b)$};
\node (h') at (10.5,4) {$H_0(S_b)$};

\draw[->](h -| a.east) --node[above]{$\scriptstyle 1\otimes v_I$} (h -| c.west);
\draw[->](a) --node[left, pos=.45]{$\scriptstyle v_{I_l}\otimes 1$} (f);
\draw[->](f) to[bend right = 10] node[left, pos=.7]{$\scriptstyle 1\otimes v_I$} ($(d.north)+(-2.85,0)$);
\draw[->](h' -| f.east) --node[above]{$\scriptstyle L^2(\cala(j_l))\otimes 1$} (h' -| g.west);
\draw[->]($(c.north)+(-.95,0)$) --node[left]{$\scriptstyle v_{I_l}\scriptscriptstyle\!\otimes 1$} ($(d.south)+(-.95,0)$);
\draw[->]($(d.north)+(-.95,0)$) --node[left, yshift=2]{$\scriptscriptstyle L^2(\cala(j_l))\otimes 1$} ($(e.south)+(-.95,0)$);
\draw[->]($(c.north)+(1.8,0)$) --node[left]{$\scriptstyle v_{I_l}$} ($(d.south)+(1.8,.2)$);
\draw[->]($(d.north)+(1.8,0)$) --node[left]{$\scriptstyle L^2(\cala(j_l))$} ($(e.south)+(1.8,.2)$);
\draw[<-]($(e.north)+(2.2,0)$) --node[right, yshift=-3]{$\scriptstyle v_I$} ($(g.south)+(1.4,.2)$);
\draw[<-]($(e.north)-(.5,0)$) -- node[left, yshift=3]{$\scriptstyle 1\otimes v_I$} ($(g.south)+(-1.3,.15)$);
\draw[<-]($(g.south)+(2,.2)$) to[bend left = 30] node[right, yshift=-3]{$\scriptstyle H_0(\tau)$} ($(c.north)+(2.5,.-.1)$);
\draw[->](h -| c.east)--node[above]{$\scriptstyle H_0(\epsilon_r)$}(h);
\draw[->](h' -| g.east)--node[above]{$\scriptstyle H_0(\epsilon_l)$}(h');
\draw[double, double distance = 2](h)--(h');\draw[line width = 2, white]($(h.north)-(0,0.01)$)--($(h'.south)+(0,0.01)$);
}
\]
where $\tau:=j_l|_{I_l}\cup j_r|_{I}=\epsilon_l^{-1}\circ\epsilon_r\in\Conf_+(S_l,S_r)$.
\end{proof}

\begin{lemma}\label{lem: associativity of Upsilon}
The map $\Upsilon$ satisfies the following version of associativity:
\begin{equation}\label{eq: Assoc of Upsilon}
\tikzmath{
\matrix [matrix of math nodes,column sep=2cm,row sep=8mm]
{ |(a)|\tikzmath[scale=\displscale]{\filldraw[fill=vacuumcolor]  (0,0) rectangle (36,12);\draw (12,0) -- (12,12)(24,0) -- (24,12);} %tikzmath
\pgfmatrixnextcell |(b)|\tikzmath[scale=\displscale]{\filldraw[fill=vacuumcolor]  (0,0) rectangle (36,12);\draw (12,0) -- (12,12);} %tikzmath
\\ |(c)|\tikzmath[scale=\displscale]{\filldraw[fill=vacuumcolor]  (0,0) rectangle (36,12);\draw(24,0) -- (24,12);} %tikzmath
\pgfmatrixnextcell |(d)|\tikzmath[scale=\displscale]{\filldraw[fill=vacuumcolor]  (0,0) rectangle (36,12);} %tikzmath
\\ }; 
\draw[->] (a) -- node [above]	{$\scriptstyle 1\otimes\Upsilon$} (b);
\draw[->] (c) -- node [above]	{$\scriptstyle H_0(\gamma_r,\cala)\circ\Upsilon^r_{\tikzmath[scale=.1]{\filldraw[fill=vacuumcolor] (0,0) rectangle (2,1);}}$} (d);
\draw[->] (a) -- node [left]		{$\scriptstyle \Upsilon\otimes 1$} (c);
\draw[->] (b) -- node [right]	{$\scriptstyle H_0(\gamma_l,\cala)\circ\Upsilon^l_{\tikzmath[scale=.1]{\filldraw[fill=vacuumcolor] (0,0) rectangle (2,1);}}$} (d);
}
\end{equation}
where $\gamma_l:\partial([1,3]\times[0,1])\to\partial([0,3]\times[0,1])$ and
$\gamma_r:\partial([0,2]\times[0,1])\to\partial([0,3]\times[0,1])$
are the maps
\[
\gamma_l:\tikzmath[scale=\displscale]{
\draw  (-.3,-.3) rectangle (36.3,12.3)(12.3,0.3) rectangle (35.7,11.7);
\draw[<-] (6,.5) to[in=170, out=70] (11,2);
\draw[<-] (1,1) to[in=175, out=50] (11,4);
\draw[<-] (.3,6) -- (11,6);
\draw[<-] (1,11) to[in=-175, out=-50] (11,8);
\draw[<-] (6,11.5) to[in=-170, out=-70] (11,10);
} %tikzmath
\qquad \gamma_r: \tikzmath[scale=\displscale]{
\draw  (-.3,-.3) rectangle (36.3,12.3)(0.3,0.3) rectangle (23.7,11.7);
\draw[<-] (30,.5) to[in=10, out=110] (25,2);
\draw[<-] (35,1) to[in=5, out=130] (25,4);
\draw[<-] (35.7,6) -- (25,6);
\draw[<-] (35,11) to[in=-5, out=-130] (25,8);
\draw[<-] (30,11.5) to[in=-10, out=-110] (25,10);
} %tikzmath
\]
given by 
$\gamma_l=j_l\cup \mathrm{Id}_{\,\tikzmath[scale=.1]{\draw (0,0) -- (2,0) -- (2,1) -- (0,1);}}$ and
$\gamma_r=\mathrm{Id}_{\,\tikzmath[scale=.1]{\draw (2,0) -- (0,0) -- (0,1) -- (2,1);}}\cup j_r^+$\,, 
and $j_r^+$ is obtained by conjugating $j_r$ by $(x,y)\mapsto(x+1,y)$.
\end{lemma}

\begin{proof}
Using Lemma \ref{lem: Ups=Ups=Ups} twice, we can expand \eqref{eq: Assoc of Upsilon} into the following diagram:
\[
\tikzmath{
\matrix [matrix of math nodes,column sep=2cm,row sep=8mm]
{ |(a)|\tikzmath[scale=\displscale]{\filldraw[fill=vacuumcolor]  (0,0) rectangle (36,12);\draw (12,0) -- (12,12)(24,0) -- (24,12);} %tikzmath
\pgfmatrixnextcell |(b)|\tikzmath[scale=\displscale]{\filldraw[fill=vacuumcolor]  (0,0) rectangle (24,12);\draw (12,0) -- (12,12);} %tikzmath
\pgfmatrixnextcell |(c)|\tikzmath[scale=\displscale]{\filldraw[fill=vacuumcolor]  (0,0) rectangle (36,12);\draw (12,0) -- (12,12);} %tikzmath
\\ |(d)|\tikzmath[scale=\displscale]{\filldraw[fill=vacuumcolor]  (0,0) rectangle (24,12);\draw(12,0) -- (12,12);} %tikzmath
\pgfmatrixnextcell |(e)|\tikzmath[scale=\displscale]{\filldraw[fill=vacuumcolor]  (0,0) rectangle (12,12);} %tikzmath
\pgfmatrixnextcell |(f)|\tikzmath[scale=\displscale]{\filldraw[fill=vacuumcolor]  (0,0) rectangle (24,12);} %tikzmath
\\ |(g)|\tikzmath[scale=\displscale]{\filldraw[fill=vacuumcolor]  (0,0) rectangle (36,12);\draw(24,0) -- (24,12);} %tikzmath
\pgfmatrixnextcell |(h)|\tikzmath[scale=\displscale]{\filldraw[fill=vacuumcolor]  (0,0) rectangle (24,12);} %tikzmath
\pgfmatrixnextcell |(i)|\tikzmath[scale=\displscale]{\filldraw[fill=vacuumcolor]  (0,0) rectangle (36,12);} %tikzmath
\\ }; 
\draw[->] (a) -- node [above]	{$\scriptstyle 1\otimes\Upsilon^r_{\tikzmath[scale=.1]{\filldraw[fill=vacuumcolor] (0,0) rectangle (1,1);}}\,=\, \Upsilon^r_{\tikzmath[scale=.1]{\filldraw[fill=vacuumcolor] (0,0) rectangle (2,1);\draw[-](1,0)--(1,1);}}$} (b); 
\draw[->] (b) -- node [above]	{$\scriptstyle 1\otimes H_0(\epsilon_r)$} (c);
\draw[->] (a) -- node[fill=white,inner sep=0,xshift=-2.5]		{$\scriptstyle \Upsilon^l_{\tikzmath[scale=.1]{\filldraw[fill=vacuumcolor] (0,0) rectangle (1,1);}}\otimes 1\,=\, \Upsilon^l_{\tikzmath[scale=.1]{\filldraw[fill=vacuumcolor] (0,0) rectangle (2,1);\draw[-](1,0)--(1,1);}}$} (d);
\draw[->] (b) -- node [left]		{$\scriptstyle \Upsilon^l_{\tikzmath[scale=.1]{\filldraw[fill=vacuumcolor] (0,0) rectangle (1,1);}}$} (e);
\draw[->] (c) -- node [right]		{$\scriptstyle \Upsilon^l_{\tikzmath[scale=.1]{\filldraw[fill=vacuumcolor] (0,0) rectangle (2,1);}}$} (f);
\draw[->] (d) -- node [above]	{$\scriptstyle \Upsilon^r_{\tikzmath[scale=.1]{\filldraw[fill=vacuumcolor] (0,0) rectangle (1,1);}}$} (e);
\draw[->] (e) -- node [above]	{$\scriptstyle H_0(\epsilon_r)$} (f);
\draw[->] (d) -- node [left]		{$\scriptstyle H_0(\epsilon_l)\otimes 1$} (g);
\draw[->] (e) -- node [left]		{$\scriptstyle H_0(\epsilon_l)$} (h);
\draw[->] (f) -- node [right]		{$\scriptstyle H_0(\gamma_l)$} (i);
\draw[->] (g) -- node [above]	{$\scriptstyle \Upsilon^r_{\tikzmath[scale=.1]{\filldraw[fill=vacuumcolor] (0,0) rectangle (2,1);}}$} (h);
\draw[->] (h) -- node [above]	{$\scriptstyle H_0(\gamma_r)$} (i);
}
\]
The lower right square commutes by the functoriality of $H_0$, 
see~\eqref{eq:vacuum-sector-functor}.
The remaining three squares commute by the fact that 
$\Upsilon^l$ and $\Upsilon^r$ are natural transformations.
\end{proof}

\subsection*{\hspace*{-18pt}The $(1_\mathrm{id} \boxtimes 1_D)$-isomorphism as a vacuum fusion transformation}

Let $\epsilon_l$ and $\epsilon_r$ be as above, 
and let $\epsilon_{l,\top}:S_{r,\top}\to S_{b,\top}$ and $\epsilon_{r,\top}:S_{l,\top}\to S_{b,\top}$ 
be their restrictions to the upper halves of $S_r$ and $S_l$, respectively.

\begin{lemma} \label{lem:omega-and-upsilon}
Let $\cala$ be a conformal net with finite index, and let ${}_\cala D_\calb$ be an irreducible defect.
Let $H_r:=H_0(S_r,D)$, where the circle $S_r$ is bicolored as in \eqref{eq: def colors of S and tildeS}.
Then the map 
$\Omega_{\mathrm{id}_\cala,D}: \tikzmath[scale=\textscale] {\fill[vacuumcolor] (0,0) rectangle  (24,12);\draw (0,0) rectangle  (24,12); \draw[ultra thick] (18,0) -- (24,0) -- (24,12) -- (18,12);} %tikzmath
\to \tikzmath[scale=\textscale] {\fill[vacuumcolor]  (0,0) rectangle (12,12) (12,0) rectangle (24,12);
\draw (0,0) rectangle (12,12) (12,0) rectangle (24,12);\draw[ultra thick] (18,0) -- (24,0) -- (24,12) -- (18,12);} %tikzmath
$
is the inverse of $L^2(D(\epsilon_{l,\top}))\circ\Upsilon^l_{\!H_r}$\!.

Similarly, assuming instead that $\calb$ has finite index, the map $\Omega_{D,\mathrm{id}_\calb}: 
\tikzmath[scale=\textscale] {\fill[vacuumcolor] (0,0) rectangle  (24,12);\draw (0,0) rectangle  (24,12);
\draw[ultra thick] (6,0) -- (24,0) -- (24,12) -- (6,12);} %tikzmath
\to \tikzmath[scale=\textscale] {\fill[vacuumcolor]  (0,0) rectangle (12,12) (12,0) rectangle (24,12);
\draw (0,0) rectangle (12,12);\draw[ultra thick] (6.3,0) -- (24,0) -- (24,12) -- (6.3,12)(12.3,0) -- (12.3,12);} %tikzmath
$ is the inverse of $L^2(D(\epsilon_{r,\top}))\circ\Upsilon^{\raisebox{1pt}{$\scriptstyle r$}}\!\!\!\!{}_{H_l},$ where $H_l := H_0(S_l,D)$.
\end{lemma}

\begin{proof}
We only treat the first equation
$\Omega^{-1}_{\mathrm{id}_\cala,D}=L^2(D(\epsilon_{l,\top}))\circ\Upsilon^l_{\!H_r}$.
We first prove it in the case when $D=\mathrm{id}_\cala$.
By definition, $\Omega_{\mathrm{id}_\cala, \mathrm{id}_\cala}$ is the composite
\begin{equation*}
\tikzmath[scale=\displscale]{\fill[vacuumcolor]  (0,0) rectangle (24,12);\draw  (6,0) -- (0,0) -- (0,12) -- (6,12);
\draw (6,12) -- (18,12)  (6,0) -- (18,0);\draw (18,0) -- (24,0) -- (24,12) -- (18,12);} %tikzmath
\;\;\xrightarrow{\,\,\Psi_{\mathrm{id}_\cala, \mathrm{id}_\cala\,\,}}\;\;
\tikzmath[scale=\displscale]{\fill[vacuumcolor] (0,0) rectangle (10,12)(14,0) rectangle (24,12);\draw (6,12) -- (10,12) -- (10,0) -- (6,0)(18,12) -- (14,12) -- (14,0) -- (18,0);
\draw (6,12) -- (0,12) -- (0,0) -- (6,0);\draw  (18,12) -- (24,12) -- (24,0) -- (18,0);\fill[vacuumcolor]  (10,0) rectangle (14,4)(10,8) rectangle (14,12);
\draw (10,0) rectangle (14,4) (10,8) rectangle (14,12);\fill[vacuumcolor]  (10.5,4.5) rectangle (13.5,7.5);\draw (10.5,4.5) rectangle (13.5,7.5);} %tikzmath
\;\;\xrightarrow{\Phi^{-1}_{H_0(\cala),H_0(\cala)}}\;\;
\tikzmath[scale=\displscale] {\fill[vacuumcolor] (0,0) rectangle (24,12);\draw  (6,0) -- (0,0) -- (0,12) -- (6,12);
\draw (6,12) -- (12,12) -- (12,0) -- (6,0)(18,12) -- (12,12) -- (12,0) -- (18,0);\draw (18,0) -- (24,0) -- (24,12) -- (18,12);}\,, %tikzmath
\end{equation*}
and $\Phi_{H_0(\cala),H_0(\cala)}$ is the composite 
\begin{equation*}
\tikzmath[scale=\displscale] {\fill[vacuumcolor]  (0,0) rectangle (12,12) (12,0) rectangle (24,12);\draw (0,0) rectangle (12,12) (12,0) rectangle (24,12);} %tikzmath
\;\;\xrightarrow{\hspace{5.55mm}\Upsilon\hspace{5.55mm}}\;\;
\tikzmath[scale=\displscale] {\fill[vacuumcolor] (0,0) rectangle  (24,12);\draw (0,0) rectangle  (24,12);} %tikzmath
\;\;\xrightarrow{\hspace{2.95mm}\Psi_{\id_\cala,\id_\cala}\hspace{2.95mm}}\;\;
\tikzmath[scale=\displscale] {\fill[vacuumcolor] (0,0) rectangle (10,12) (14,0) rectangle (24,12) (10,0) rectangle (14,4) (10,8) rectangle (14,12) (10.5,4.5) rectangle (13.5,7.5);
\draw  (0,0) rectangle (10,12) (14,0) rectangle (24,12)(10,0) rectangle (14,4) (10,8) rectangle (14,12) (10.5,4.5) rectangle (13.5,7.5);}\,. %tikzmath
\end{equation*}
It follows that $\Omega_{\mathrm{id}_\cala, \mathrm{id}_\cala}=\Upsilon^{-1}$ and we are done by Lemma~\ref{lem: Ups=Ups=Ups}.

We now treat the general case.
By Proposition \ref{prop: associativity of Omega} 
(with the defects $\mathrm{id}_\cala$, $\mathrm{id}_\cala$, and $D$),
we have the commutativity of the following diagram:
\begin{equation}\label{eq: square with four Omegas}
\hspace{.9cm}\tikzmath{
\matrix [matrix of math nodes,column sep=1.8cm,row sep=5mm]
{|(a)|\tikzmath[scale=\displscale]{\fill[vacuumcolor](0,0) rectangle (36,12);\draw(6,0)--(0,0)--(0,12)--(6,12);
\draw(6,12)--(18,12)(6,0)--(18,0);\draw(18,12)--(30,12)(18,0)--(30,0);\draw[ultra thick] (30,0)--(36,0)--(36,12)--(30,12);}%tikzmath
\pgfmatrixnextcell |(b)|\tikzmath[scale=\displscale]{\fill[vacuumcolor](0,0) rectangle (36,12);\draw(6,0)--(0,0)--(0,12)--(6,12);
\draw(6,12)--(18,12)(6,0)--(18,0);\draw(18,12)--(30,12)(18,0)--(30,0)(24,0)--(24,12);\draw[ultra thick] (30,0)--(36,0)--(36,12)--(30,12);}\\%tikzmath
|(c)|\tikzmath[scale=\displscale]{\fill[vacuumcolor](0,0) rectangle (36,12);\draw(6,0)--(0,0)--(0,12)--(6,12);
\draw(6,12)--(18,12)(6,0)--(18,0)(12,0)--(12,12);\draw(18,12)--(30,12)(18,0)--(30,0);\draw[ultra thick] (30,0)--(36,0)--(36,12)--(30,12);}%tikzmath
\pgfmatrixnextcell |(d)|\tikzmath[scale=\displscale]{\fill[vacuumcolor](0,0) rectangle (36,12);\draw(6,0)--(0,0)--(0,12)--(6,12);
\draw(6,12)--(18,12)(6,0)--(18,0)(12,0)--(12,12);\draw(18,12)--(30,12)(18,0)--(30,0)(24,0)--(24,12);\draw[ultra thick] (30,0)--(36,0)--(36,12)--(30,12);}\\%tikzmath
};
\draw[->](a)--node[above]{$\scriptstyle\Omega_{\mathrm{id}_\cala\circledast \mathrm{id}_\cala,D}$}(b);\draw[->](c)--node[above]{$\scriptstyle 1\otimes\Omega_{\mathrm{id}_\cala,D}$}(d);
\draw[->](a)--node[left]{$\scriptstyle\Omega_{\mathrm{id}_\cala,\mathrm{id}_\cala\circledast D}$}(c);
\draw[->](b)--node[right]{$\scriptstyle\Omega_{\mathrm{id}_\cala,\mathrm{id}_\cala}\otimes 1\;\;=\,\;\Upsilon^{-1}\otimes 1$}(d);}
\end{equation}
Consider the circles $S_l:=\partial( [0,1]\!\times\![0,1])$, 
$S_m:=\partial( [1,2]\!\times\![0,1])$,
$S_r:=\partial( [2,3]\!\times\![0,1])$,
$S_{lm}:=\partial( [0,2]\!\times\![0,1])$, 
$S_{mr}:=\partial( [1,3]\!\times\![0,1])$, 
$S_{lmr}:=\partial( [0,3]\!\times\![0,1])$, 
and the corresponding half-circles $S_{\alpha,\top}:=(S_\alpha)_{y\ge \frac12}$ for $\alpha\in\{l,m,r,lm,mr,lmr\}$.
Let $\varphi:[\frac12-\varepsilon,\frac12+\varepsilon]\to[\frac12-\varepsilon,\frac32+\varepsilon]$ and 
$\psi:[\frac32-2\varepsilon,\frac32-\varepsilon]\to[\frac32-2\varepsilon,\frac52-\varepsilon]$ be diffeomorphisms whose derivative is $1$ in a neighborhood of the boundary, where $\varepsilon$ is a fixed small number.
These extend to diffeomorphisms
\[
\begin{split}
\varphi_{lm}:S_{lm,\top}\to S_{lmr,\top}\,,&\qquad\qquad\varphi_l:S_{l,\top}\to S_{lm,\top}\,,\\
\psi_{lm}:S_{lm,\top}\to S_{lmr,\top}\,,&\qquad\qquad\psi_m:S_{m,\top}\to S_{mr,\top}\,,
\end{split}
\]
whose derivative is $1$ outside the domains of $\varphi$ and $\psi$, respectively.
Let also $\chi:=\psi_{lm}^{-1}\circ\varphi_{lm}$.
Note that  $\chi(x,y) = (x,y)$ for $x \geq \frac{3}{2}$.
{\def\squared#1{\tikz{\useasboundingbox (-.13,-.11) rectangle (.13,.12);\node[draw, inner sep = 1.5]{\tiny #1};}}
Since the construction of $\Omega$ is natural with respect to isomorphisms, the following squares \squared1 are commutative:
\begin{equation}\label{eq: diagram a bunch of Omega's}
\tikzmath{
\matrix [matrix of math nodes,column sep=1.8cm,row sep=6mm]
{\pgfmatrixnextcell |(aa)|\tikzmath[scale=\displscale]{\fill[vacuumcolor](12,0) rectangle (36,12);\draw(30,12)--(12,12)--(12,0)--(30,0);\draw[ultra thick] (30,0)--(36,0)--(36,12)--(30,12);}%tikzmath
\pgfmatrixnextcell |(bb)|
\tikzmath[scale=\displscale]{\fill[vacuumcolor](12,0) rectangle (36,12);\draw(30,12)--(12,12)--(12,0)--(30,0)(24,0)--(24,12);\draw[ultra thick] (30,0)--(36,0)--(36,12)--(30,12);}\\%tikzmath
|(aaa)|\tikzmath[scale=\displscale]{\fill[vacuumcolor](12,0) rectangle (36,12);\draw(30,12)--(12,12)--(12,0)--(30,0);\draw[ultra thick] (30,0)--(36,0)--(36,12)--(30,12);}%tikzmath
\pgfmatrixnextcell |(a)|\tikzmath[scale=\displscale]{\fill[vacuumcolor](0,0) rectangle (36,12);\draw(6,0)--(0,0)--(0,12)--(6,12);
\draw(6,12)--(18,12)(6,0)--(18,0);\draw(18,12)--(30,12)(18,0)--(30,0);\draw[ultra thick] (30,0)--(36,0)--(36,12)--(30,12);}%tikzmath
\pgfmatrixnextcell |(b)|\tikzmath[scale=\displscale]{\fill[vacuumcolor](0,0) rectangle (36,12);\draw(6,0)--(0,0)--(0,12)--(6,12);
\draw(6,12)--(18,12)(6,0)--(18,0);\draw(18,12)--(30,12)(18,0)--(30,0)(24,0)--(24,12);\draw[ultra thick] (30,0)--(36,0)--(36,12)--(30,12);}\\%tikzmath
|(bbb)| \tikzmath[scale=\displscale]{\fill[vacuumcolor](12,0) rectangle (36,12);\draw(30,12)--(12,12)--(12,0)--(30,0)(24,0)--(24,12);\draw[ultra thick] (30,0)--(36,0)--(36,12)--(30,12);}%tikzmath
\pgfmatrixnextcell 
|(c)|\tikzmath[scale=\displscale]{\fill[vacuumcolor](0,0) rectangle (36,12);\draw(6,0)--(0,0)--(0,12)--(6,12);
\draw(6,12)--(18,12)(6,0)--(18,0)(12,0)--(12,12);\draw(18,12)--(30,12)(18,0)--(30,0);\draw[ultra thick] (30,0)--(36,0)--(36,12)--(30,12);}%tikzmath
\pgfmatrixnextcell |(d)|\tikzmath[scale=\displscale]{\fill[vacuumcolor](0,0) rectangle (36,12);\draw(6,0)--(0,0)--(0,12)--(6,12);
\draw(6,12)--(18,12)(6,0)--(18,0)(12,0)--(12,12);\draw(18,12)--(30,12)(18,0)--(30,0)(24,0)--(24,12);\draw[ultra thick] (30,0)--(36,0)--(36,12)--(30,12);}\\%tikzmath
};
\draw[->](aaa)--node[left]{$\scriptstyle\Omega_{\mathrm{id}_\cala,D}$}(bbb);
\draw[->](aa)--node[above]{$\scriptstyle\Omega_{\mathrm{id}_\cala,D}$}(bb);
\draw[->](a)--node[below]{$\scriptstyle\Omega_{\mathrm{id}_\cala\circledast \mathrm{id}_\cala,D}$}(b);
\draw[->](c)--node[below]{$\scriptstyle 1\otimes\Omega_{\mathrm{id}_\cala,D}$}(d);
\draw[->](a)--node[fill=white, inner sep=0]{$\scriptstyle\Omega_{\mathrm{id}_\cala,\mathrm{id}_\cala\circledast D}$}(c);
\draw[->](b)--node[fill=white, inner sep=0]{\,\,\,$\scriptstyle\Upsilon^{-1}\otimes 1$}(d);
\draw[->](aa)--node[right]{$\scriptstyle L^2(D(\varphi_{lm}))$}(a);
\draw[->](bb)--node[right]{$\scriptstyle L^2(\cala(\varphi_{l}))\otimes 1$}(b);
\draw[->](aaa)--node[below]{$\scriptstyle L^2(D(\psi_{lm}))$}(a);
\draw[->](bbb)--node[below]{$\scriptstyle1\otimes L^2(D(\psi_{m}))$}(c);
\draw[<-, rounded corners=3](aaa) -- (aa-|aaa) -- node[above, pos=.45]{$\scriptstyle L^2(D(\chi))$}(aa);
\node at ($(aa)!0.5!(b)+(0.1,0)$) {\squared1};
\node at ($(aaa)!0.5!(c)+(-0.1,-0.1)$) {\squared1};
}
\end{equation}
The remaining two squares of \eqref{eq: diagram a bunch of Omega's} are commutative by \eqref{eq: square with four Omegas}, and by the definition of $\chi$.

We now consider the following diagram\medskip
\begin{equation}\label{eq: first big diagram -- no defect}
\tikzmath{
\node at (4.2,2) {\squared 2};
\node at (.7,3.15) {\squared 3};
\node at (.65,.35) {\squared 3};
\node at (8.5,1.2) {\squared 3};
\node (a) at (-2,3.5) {$\tikzmath[scale=\displscale]{\filldraw[fill=vacuumcolor](0,0) rectangle (12,12);}$};
\node (b) at (4,3.5) {$\tikzmath[scale=\displscale]{\filldraw[fill=vacuumcolor](0,0) rectangle (24,12);\draw(12,0)--(12,12);}$};
\node (c) at (9.5,3.5) {$\tikzmath[scale=\displscale]{\filldraw[fill=vacuumcolor](0,0) rectangle (36,12);\draw(24,0)--(24,12);}$};
\node (d) at (.7,2.5) {$\tikzmath[scale=\displscale]{\filldraw[fill=vacuumcolor](0,0) rectangle (24,12);}$};
\node (e) at (9.5,2) {$\tikzmath[scale=\displscale]{\filldraw[fill=vacuumcolor](0,0) rectangle (36,12);\draw(12,0)--(12,12)(24,0)--(24,12);}$};
\node (f) at (.7,1) {$\tikzmath[scale=\displscale]{\filldraw[fill=vacuumcolor](0,0) rectangle (24,12);}$};
\node (g) at (6,1) {$\tikzmath[scale=\displscale]{\filldraw[fill=vacuumcolor](0,0) rectangle (36,12);\draw(12,0)--(12,12);}$};
\node (h) at (-2,0) {$\tikzmath[scale=\displscale]{\filldraw[fill=vacuumcolor](0,0) rectangle (12,12);}$};
\node (i) at (2.8,0) {$\tikzmath[scale=\displscale]{\filldraw[fill=vacuumcolor](0,0) rectangle (24,12);\draw(12,0)--(12,12);}$};
\node (j) at (9.5,0) {$\scriptscriptstyle\tikzmath[scale=\displscale]{\filldraw[fill=vacuumcolor](0,0) rectangle (12,12);\filldraw[fill=vacuumcolor](15,0) rectangle (27,12);
\draw (12,.5) to[in=170, out=70] (15,4);\draw (12,3.5) to[in=175, out=50] (15,5);\draw (12,6) -- (15,6);\draw (12,8.5) to[in=-175, out=-50] (15,7);\draw (12,11.5) to[in=-170, out=-70] (15,8);}$};

\draw[->](a)--node[above]{$\scriptstyle (\Upsilon_{\tikzmath[scale=.1]{\filldraw[fill=vacuumcolor] (0,0) rectangle (1,1);}}^l)^{-1}$}(b);
\draw[->](b)--node[above]{$\scriptstyle L^2(\cala(\varphi_l))\otimes 1$}(c);
\draw[->](a)--node[below,xshift=-3,yshift=-1]{$\scriptstyle L^2(\cala(\epsilon_{l,\top}))$}(d);
\draw[->](c)--node[right]{$\scriptstyle\Upsilon^{-1}\otimes 1$}(e);
\draw[->](d)--node[below,xshift=3,yshift=1]{$\scriptstyle\Upsilon^{-1}$}(b);
\draw[->](d)--node[right]{$\scriptstyle L^2(\cala(\chi))$}(f);
\draw[->](a)--node[xshift=11,fill=white,inner sep=1]{$\scriptstyle L^2(\cala(\tau))$}(h);
\draw[->](g)--node[above,xshift=-3]{$\scriptstyle1\otimes\Upsilon^{-1}$}(e);
\draw[->](h)--node[above,xshift=-3,yshift=2]{$\scriptstyle L^2(\cala(\epsilon_{l,\top}))$}(f);
\draw[->](f)--node[above,xshift=3]{$\scriptstyle\Upsilon^{-1}$}(i);
\draw[->](h)--node[below]{$\scriptstyle (\Upsilon_{\tikzmath[scale=.1]{\filldraw[fill=vacuumcolor] (0,0) rectangle (1,1);}}^l)^{-1}$}(i);
\draw[->](i)--node[below]{$\scriptstyle 1\otimes L^2(\cala(\sigma))$}(j);
\draw[->](j)--node[right,yshift=2]{$\scriptstyle 1\otimes(\Upsilon_{\tikzmath[scale=.1]{\filldraw[fill=vacuumcolor] (0,0) rectangle (1,1);}}^l)^{-1}$}(e);
\draw[->](i)--node[above,xshift=-12,yshift=1]{$\scriptstyle 1\otimes L^2(\cala(\psi_m))$}(g);
\draw[->](j)--node[above,xshift=14,yshift=-1]{$\scriptstyle 1\otimes L^2(\cala(\epsilon_{l,\top}))$}(g);
}
\end{equation}
where $\tau=\epsilon_{l,\top}^{-1}\circ\chi\circ\epsilon_{l,\top}$, $\sigma=\epsilon_{l,\top}^{-1}\circ \psi_m$, the lower right corner 
stands for the fusion of $H_0(S_l)$ with $H_0(S_r)$ along $\cala(j_l^+|_{\{1\}\times[0,1]})$, and $j_l^+$ is %as in Lemma~\ref{lem: associativity of Upsilon}:
obtained by conjugating $j_l$ by $(x,y)\mapsto (x+1,y)$:
\[
j_l^+:\tikzmath[scale=\displscale]{
\draw (-12.5,0) rectangle (23.5,12) (-.5,0)--(-.5,12) (11.5,0)--(11.5,12);
\draw[<->] (6.5,.5) to[in=170, out=70] (11,2);\draw[<->] (1.5,1) to[in=175, out=50] (11,4);\draw[<->] (.3,6) -- (11,6);\draw[<->] (1.5,11) to[in=-175, out=-50] (11,8);\draw[<->] (6.5,11.5) to[in=-170, out=-70] (11,10);
} %tikzmath
\,,\qquad\qquad j_l^+|_{\{1\}\times[0,1]}:\tikzmath[scale=\displscale]{
\draw (-12.5,0) rectangle (-.5,12) (11.5,0) rectangle (23.5,12) (0,0)--(11,0)(0,12)--(11,12);
\draw[->] (1,1) to[in=175, out=50] (11,4);\draw[->] (.5,6) -- (11,6);\draw[->] (1,11) to[in=-175, out=-50] (11,8);
}\,. %tikzmath
\]
Using the identity $\Omega_{\mathrm{id}_\cala, \mathrm{id}_\cala}=\Upsilon^{-1}$ proved earlier, the case $D=\mathrm{id}_\cala$ of \eqref{eq: diagram a bunch of Omega's}
implies the commutativity of \squared2\,.
\!The triangles \squared3 commute by virtue of Lemma~\ref{lem: Ups=Ups=Ups},
and so the whole diagram \eqref{eq: first big diagram -- no defect} is commutative.

Let $\hat\tau,\hat\sigma\in\Diff(\partial[0,1]^2)$ be the symmetric extensions of $\tau$ and $\sigma$, so that $\hat\tau|_{S^1_\top}=\tau$ and $\hat\sigma|_{S^1_\top}=\sigma$.
Note that by definition, both $\hat\tau$ and $\hat\sigma$ commute with $(x,y)\mapsto (x,1-y)$.
From the fact that $\chi(x,y) = (x,y)$ for $x \geq \frac{3}{2}$, it follows that
$\hat\sigma(x,y) = \hat\tau(x,y) = (x,y)$ for $x \geq \frac{1}{2}$.
%Therefore, by Lemma~\ref{lem: canonical quantization of the symmetric diffeomorphism}
Let $u\in\cala(\tikzmath[scale=\textscale]{\useasboundingbox (-2,0) rectangle (14,12);\draw (6,12) -- (0,12) -- (0,0) -- (6,0);})$ and $v\in\cala(\tikzmath[scale=\textscale]{\useasboundingbox (-2,0) rectangle (14,12);\draw (6,12) -- (0,12) -- (0,0) -- (6,0);})$ 
be the canonical quantizations, as in the discussion following equation \eqref{eq: music sign},
of the symmetric diffeomorphisms $\hat\tau$ and $\hat\sigma$.
By the definition of these quantizations, we have $L^2(\cala(\tau))=\pi(u)$ and $L^2(\cala(\sigma))=\pi(v)$, where $\pi$ is the action of $\cala(
\tikzmath[scale=\textscale]{\useasboundingbox (-2,0) rectangle (14,12);\draw (6,12) -- (0,12) -- (0,0) -- (6,0);} %tikzmath
)$ on $\tikzmath[scale=\textscale]{\filldraw[fill=vacuumcolor](0,0) rectangle (12,12);}=H_0(\cala)$.

We now consider the following diagram of natural transformations between functors from 
$\cala(
\tikzmath[scale=\textscale]{\useasboundingbox (-2,0) rectangle (14,12);\draw (6,12) -- (0,12) -- (0,0) -- (6,0);} %tikzmath
)$-modules to Hilbert spaces:
\begin{equation}\label{eq:diagram of natural transformations}
\tikzmath{
\node (a) at (0,1.5) {$\tikzmath[scale=\displscale]{\draw(6,0) -- (0,0) -- (0,12) -- (6,12);\draw[dotted] (6,0) -- (12,0) -- (12,12) -- (6,12);}$};
\node (b) at (4,1.5) {$\tikzmath[scale=\displscale]{\filldraw[fill=vacuumcolor](0,0) rectangle (12,12);\draw(12,0)--(18,0)(12,12)--(18,12);\draw[dotted] (18,0) -- (24,0) -- (24,12) -- (18,12);}$};
\node (c) at (9,1.5) {$\tikzmath[scale=\displscale]{\filldraw[fill=vacuumcolor](0,0) rectangle (24,12);\draw(24,0)--(30,0)(24,12)--(30,12);\draw[dotted] (30,0) -- (36,0) -- (36,12) -- (30,12);}$};
\node (e) at (9,0) {$\tikzmath[scale=\displscale]{\filldraw[fill=vacuumcolor](0,0) rectangle (24,12);\draw(24,0)--(30,0)(24,12)--(30,12);\draw[dotted] (30,0) -- (36,0) -- (36,12) -- (30,12);\draw(12,0)--(12,12);}$};
\node (h) at (0,0) {$\tikzmath[scale=\displscale]{\draw(6,0) -- (0,0) -- (0,12) -- (6,12);\draw[dotted] (6,0) -- (12,0) -- (12,12) -- (6,12);}$};
\node (i) at (2.5,0) {$\tikzmath[scale=\displscale]{\filldraw[fill=vacuumcolor](0,0) rectangle (12,12);\draw(12,0)--(18,0)(12,12)--(18,12);\draw[dotted] (18,0) -- (24,0) -- (24,12) -- (18,12);}$};
\node (j) at (5.5,0) {$\tikzmath[scale=\displscale]{\filldraw[fill=vacuumcolor](0,0) rectangle (12,12);\draw(21,0) -- (15,0) -- (15,12) -- (21,12);\draw[dotted] (21,0) -- (28,0) -- (28,12) -- (21,12);\draw (12,.5) to[in=170, out=70] (15,4);\draw (12,3.5) to[in=175, out=50] (15,5);\draw (12,6) -- (15,6);\draw (12,8.5) to[in=-175, out=-50] (15,7);\draw (12,11.5) to[in=-170, out=-70] (15,8);
}$};

\draw[->](a)--node[above]{$\scriptstyle (\Upsilon^l)^{-1}$}(b);
\draw[->](b)--node[above]{$\scriptstyle L^2(\cala(\varphi_l))\otimes 1$}(c);
\draw[->](c)--node[right]{$\scriptstyle\Upsilon^{-1}\otimes 1$}(e);
\draw[->](a)--node[left]{$\scriptstyle u$}(h);
\draw[->](h)--node[above]{$\scriptstyle (\Upsilon^l)^{-1}$}(i);
\draw[->](i)--node[above]{$\scriptstyle 1\otimes v$}(j);
\draw[->](j)--node[above]{$\scriptstyle 1\otimes(\Upsilon^l)^{-1}$}(e);
}
\end{equation}
When evaluated on $H_0(\cala)$, the above diagram commutes by \eqref{eq: first big diagram -- no defect}. 
Therefore, by Lemma~\ref{lem: NT between module categories}, since $H_0(\cala)$ is a faithful $\cala(
\tikzmath[scale=\textscale]{\useasboundingbox (-2,0) rectangle (14,12);\draw (6,12) -- (0,12) -- (0,0) -- (6,0);} %tikzmath
)$-module, the diagram \eqref{eq:diagram of natural transformations} commutes 
regardless of the module one evaluates it on.
We now consider the following variant of diagram \eqref{eq: first big diagram -- no defect}:
\[
\tikzmath{
\node at (4.2,2) {\squared 4};
\node at (.7,3.15) {\squared 5};
\node at (.65,.35) {\squared 5};
\node at (8.5,1.2) {\squared 5};
\node (a) at (-2,3.5) {$\tikzmath[scale=\displscale]{\filldraw[fill=vacuumcolor](0,0) rectangle (12,12);\draw[ultra thick] (6,0) -- (12,0) -- (12,12) -- (6,12);}$};
\node (b) at (4,3.5) {$\tikzmath[scale=\displscale]{\filldraw[fill=vacuumcolor](0,0) rectangle (24,12);\draw(12,0)--(12,12);\draw[ultra thick] (18,0) -- (24,0) -- (24,12) -- (18,12);}$};
\node (c) at (9.5,3.5) {$\tikzmath[scale=\displscale]{\filldraw[fill=vacuumcolor](0,0) rectangle (36,12);\draw(24,0)--(24,12);\draw[ultra thick] (30,0) -- (36,0) -- (36,12) -- (30,12);}$};
\node (d) at (.7,2.5) {$\tikzmath[scale=\displscale]{\filldraw[fill=vacuumcolor](0,0) rectangle (24,12);\draw[ultra thick] (18,0) -- (24,0) -- (24,12) -- (18,12);}$};
\node (e) at (9.5,2) {$\tikzmath[scale=\displscale]{\filldraw[fill=vacuumcolor](0,0) rectangle (36,12);\draw(12,0)--(12,12)(24,0)--(24,12);\draw[ultra thick] (30,0) -- (36,0) -- (36,12) -- (30,12);}$};
\node (f) at (.7,1) {$\tikzmath[scale=\displscale]{\filldraw[fill=vacuumcolor](0,0) rectangle (24,12);\draw[ultra thick] (18,0) -- (24,0) -- (24,12) -- (18,12);}$};
\node (g) at (6,1) {$\tikzmath[scale=\displscale]{\filldraw[fill=vacuumcolor](0,0) rectangle (36,12);\draw(12,0)--(12,12);\draw[ultra thick] (30,0) -- (36,0) -- (36,12) -- (30,12);}$};
\node (h) at (-2,0) {$\tikzmath[scale=\displscale]{\filldraw[fill=vacuumcolor](0,0) rectangle (12,12);\draw[ultra thick] (6,0) -- (12,0) -- (12,12) -- (6,12);}$};
\node (i) at (2.8,0) {$\tikzmath[scale=\displscale]{\filldraw[fill=vacuumcolor](0,0) rectangle (24,12);\draw(12,0)--(12,12);\draw[ultra thick] (18,0) -- (24,0) -- (24,12) -- (18,12);}$};
\node (j) at (9.5,0) {$\tikzmath[scale=\displscale]{\filldraw[fill=vacuumcolor](0,0) rectangle (12,12);\filldraw[fill=vacuumcolor](15,0) rectangle (27,12);
\draw (12,.5) to[in=170, out=70] (15,4);\draw (12,3.5) to[in=175, out=50] (15,5);\draw (12,6) -- (15,6);\draw (12,8.5) to[in=-175, out=-50] (15,7);\draw (12,11.5) to[in=-170, out=-70] (15,8);
\draw[ultra thick] (21,0) -- (27,0) -- (27,12) -- (21,12);}$};

\draw[->](a)--node[above]{$\scriptstyle (\Upsilon_{\tikzmath[scale=.1]{\filldraw[fill=vacuumcolor] (0,0) rectangle (1,1);\draw[-][thick] (.5,0) -- (1,0) -- (1,1) -- (.5,1);}}^l)^{-1}$}(b);
\draw[->](b)--node[above]{$\scriptstyle L^2( \cala(\varphi_l))\otimes 1$}(c);
\draw[->](a)--node[below,xshift=-3,yshift=-1]{$\scriptstyle L^2( D(\epsilon_{l,\top}))$}(d);
\draw[->](c)--node[right]{$\scriptstyle\Upsilon^{-1}\otimes 1$}(e);
\draw[->](d)--node[below,xshift=5]{$\scriptstyle\Omega_{\id_\cala,D}$}(b);
\draw[->](d)--node[right]{$\scriptstyle L^2( D(\chi))$}(f);
\draw[->](a)--node[xshift=11,fill=white,inner sep=1]{$\scriptstyle L^2( D(\tau))$}(h);
\draw[->](g)--node[above,xshift=-3]{$\scriptstyle1\otimes\Omega_{\id_\cala,D}$}(e);
\draw[->](h)--node[above,xshift=-3,yshift=2]{$\scriptstyle L^2( D(\epsilon_{l,\top}))$}(f);
\draw[->](f)--node[above,xshift=5]{$\scriptstyle\Omega_{\id_\cala,D}$}(i);
\draw[->](h)--node[below]{$\scriptstyle (\Upsilon_{\tikzmath[scale=.1]{\filldraw[fill=vacuumcolor] (0,0) rectangle (1,1);\draw[-][thick] (.5,0) -- (1,0) -- (1,1) -- (.5,1);}}^l)^{-1}$}(i);
\draw[->](i)--node[below]{$\scriptstyle 1\otimes L^2( D(\sigma))$}(j);
\draw[->](j)--node[right,yshift=2]{$\scriptstyle 1\otimes(\Upsilon_{\tikzmath[scale=.1]{\filldraw[fill=vacuumcolor] (0,0) rectangle (1,1);\draw[-][thick] (.5,0) -- (1,0) -- (1,1) -- (.5,1);}}^l)^{-1}$}(e);
\draw[->](i)--node[above,xshift=-12,yshift=1]{$\scriptstyle 1\otimes L^2( D(\psi_m))$}(g);
\draw[->](j)--node[above,xshift=14,yshift=-2]{$\scriptstyle 1\otimes L^2( D(\epsilon_{l,\top}^{-1}))$}(g);
}
\]
Our goal is to show is that the triangles \squared5 are commutative.
Since $D$ is irreducible, there exists an invertible complex number $\lambda$ such that 
$$\Omega_{\mathrm{id}_\cala,D}\circ L^2(D(\epsilon_{l,\top}))= \lambda\, (\Upsilon_{\tikzmath[scale=.15]{\filldraw[fill=vacuumcolor] (0,0) rectangle (1,1);\draw[-][thick] (.5,0) -- (1,0) -- (1,1) -- (.5,1);}}^l)^{-1}.$$
The 7-gon \squared4 is simply \eqref{eq: diagram a bunch of Omega's}, and it is therefore commutative.
The triangle \squared5 occurs two times with a given orientation, and once with the opposite orientation:
the outer 7-gon therefore commutes up to a factor of $\lambda$.
But using Lemma \ref{lem: canonical quantization of the symmetric diffeomorphism} note that outer 7-gon is an instance of \eqref{eq:diagram of natural transformations}, and is therefore commutative.
It follows that $\lambda=1$.
} %ends the { that started before the def of \squared
\end{proof}

%==================================================================
\section{Unitors for horizontal fusion of sectors}
  \label{sec:quasi-identities}

In this section, we will introduce variants of the transformations $\Upsilon^l$ and $\Upsilon^r$
that function as unitors for horizontal fusion and that will be more convenient than $\Upsilon^l$ and $\Upsilon^r$ for verifying (in \cite{BDH(3-category)}) that conformal nets form a 3-category
(more precisely, an internal dicategory in the 2-category of symmetric monoidal categories \cite[Def. 3.3]{Douglas-Henriques(Internal-bicategories)}; see Footnote \textsuperscript{\ref{footn:dicategory}}).

We will again be using the circles $S_l$, $S_r$, $S_b$, the intervals $I$, $I_l$, $I_r$,
and the involutions $j_l \in \Conf_-(S_l)$, $j_r \in \Conf_-(S_r)$ from~\eqref{eq:circles-interval-involutions}.
Let $\alpha_l:=j_l|_{I_l} \cup \id_{ I_r} \colon S_b \to S_r$ and $\alpha_r:=\id_{I_l} \cup j_r|_{I} \colon S_b \to S_l$ be
the diffeomorphisms used in the definition of $\Upsilon^l$ and $\Upsilon^r$---their inverses appeared in Lemma~\ref{lem: Ups=Ups=Ups} under the names $\epsilon_l$ and $\epsilon_r$.

The restriction $\alpha_l|_{\{0\} \x [0,1]}$ is not the map $(0,y)\mapsto(1,y)$ and,
as a consequence, the way $\Upsilon^l$ interacts with horizontal fusion is somewhat complicated to describe.
(see the lower right corner of \eqref{eq: first big diagram -- no defect} and the vertical arrow from it);
a similar story holds for the restriction of $\alpha_r$ to $\{2\} \x [0,1]$.
Our next goal is to introduce diffeomorphisms $\beta_l\colon S_b \to S_r$ and $\beta_r\colon S_b \to S_l$ to replace $\alpha_l$ and $\alpha_r$,
so that the corresponding variants of $\Upsilon^l$ and $\Upsilon^r$ avoid this complication with horizontal fusion.

Pick intervals $I^+_l \subset S_b$ and $I^+_r  \subset S_b$ that are sightly larger than $I_l$ and $I_r$, and
diffeomorphisms $\beta_l$ and $\beta_r$ that satisfy
\[
\begin{split}
\beta_l(0,y) = (1,y),\,\quad \beta_l\circ j=j\circ \beta_l\,&,\quad \beta_l|_{S_b\setminus I_l^+}=\mathrm{id}\\
\beta_r(2,y) = (1,y),\quad \beta_r\circ j=j\circ \beta_r&,\quad \beta_r|_{S_b\setminus I_r^+}=\mathrm{id},
\end{split}
\]
where $j(x,y)=(x,1-y)$:
\begin{equation*} %\label{eq: baba}
\begin{split}
\alpha_l \, : \, 
\tikzmath[scale=\displscale]{\draw(-.5,-.3) rectangle (24.3,12.3)(12.3,0.3) rectangle (23.7,11.7);\draw[->] (6,.5) to[in=170, out=70] (11,2);\draw[->] (1,1) to[in=175, out=50] (11,4);
\draw[->] (.7,6) -- (11,6);\draw[->] (1,11) to[in=-175, out=-50] (11,8);\draw[->] (6,11.5) to[in=-170, out=-70] (11,10);} \,\hspace{.1mm} ,\hspace{.3mm} %tikzmath
\qquad \beta_l \, :& \, 
\tikzmath[scale=\displscale]{\draw(-.5,-.3) rectangle (24.3,12.3)(12.3,0.3) rectangle (23.7,11.7);\draw[->] (.7,2) -- (11,2);\draw[->] (.7,4) -- (11,4);\draw[->] (.7,6) -- (11,6);
\draw[->] (.7,8) -- (11,8);\draw[->] (.7,10) -- (11,10);} \, , %tikzmath
\qquad I_l \, = \, \tikzmath[scale=\displscale]{\useasboundingbox (-2,0) rectangle (26,12);\draw (12,12) -- (0,12) -- (0,0) -- (12,0);} %tikzmath
\subset \,I^+_l\, = \, \tikzmath[scale=\displscale]{\useasboundingbox (-2,0) rectangle (26,12);\draw (14,12) -- (0,12) -- (0,0) -- (14,0);} %tikzmath
\\ \alpha_r \, : \, 
\tikzmath[scale=\displscale]{\draw(-.3,-.3) rectangle (24.5,12.3) (0.3,0.3) rectangle (11.7,11.7);\draw[->] (18,.5) to[in=10, out=110] (13,2);\draw[->] (23,1) to[in=5, out=130] (13,4);
\draw[->] (23.3,6) -- (13,6);\draw[->] (23,11) to[in=-5, out=-130] (13,8);\draw[->] (18,11.5) to[in=-10, out=-110] (13,10);} \, , %tikzmath
\qquad \beta_r \, :& \, 
\tikzmath[scale=\displscale]{\draw(-.3,-.3) rectangle (24.5,12.3) (0.3,0.3) rectangle (11.7,11.7);\draw[->] (23.3,2) -- (13,2);\draw[->] (23.3,4) -- (13,4);\draw[->] (23.3,6) -- (13,6);
\draw[->] (23.3,8) -- (13,8);\draw[->] (23.3,10) -- (13,10);} \, , %tikzmath
\qquad I_r \, = \, \tikzmath[scale=\displscale]{\useasboundingbox (-2,0) rectangle (26,12);\draw (12,12) -- (24,12) -- (24,0) -- (12,0);} %tikzmath
\subset \, I^+_r \, = \, \tikzmath[scale=\displscale]{\useasboundingbox (-2,0) rectangle (26,12);\draw (10,12) -- (24,12) -- (24,0) -- (10,0);}%tikzmath
\,. \end{split}
\end{equation*}
The illustrations of $\beta_l$ and $\beta_r$ are somewhat rough, in that they do not record the nontrivial contraction on the horizontal segments of $I^+_l$ and $I^+_r$.

The composition $\beta_l \circ \alpha_l^{-1} \colon S_r \to S_r$ is symmetric with respect to $j$, and restricts to the identity on the complement of $\alpha_l(I_l^+)$. 
As explained in the discussion preceding Lemma~\ref{lem: canonical quantization of the symmetric diffeomorphism},
there is a canonical implementation $u_l$ of the diffeomorphism $\beta_l \circ \alpha_l^{-1}$ on $H_0(S_r)$.
Similarly, there is a canonical implementation $u_r$ of $\beta_r \circ \alpha_r^{-1}$ on $H_0(S_l)$.

Given a $D$-$E$-sector $H$ between $\cala$-$\calb$-defects $D$ and $E$,
pulling back $H$ along $\cala({\alpha_l})$ produces an $(\id_\cala \circledast D)$-$(\id_\cala \circledast E)$-sector; cf. Lemma~\ref{lem: unitor} and the subsequent discussion.
This operation is a functor $\alpha_l^*:\CN_2 \to \CN_2$.
Similarly, we have functors $\beta_l^*$, $\alpha_r^*$, $\beta_r^*:\CN_2 \to \CN_2$. 
Multiplication by $u_l$ and $u_r$ then provide natural isomorphisms
\[
U^l:\alpha_l^*\cong \beta_l^*:\CN_2 \tworarrow \CN_2\,,\qquad\,\,\,\,\,
U^r:\alpha_r^*\cong \beta_r^*:\CN_2 \tworarrow \CN_2
\]
between the functors $\alpha_l^*$ and $\beta_l^*$, and between the functors $\alpha_r^*$ and $\beta_r^*$.
Recall that $\Upsilon^l$ and $\Upsilon^r$ are natural isomorphisms $H_0(\cala) \boxtimes_\cala - \cong {\alpha_l^*}$ and
$- \boxtimes_\calb H_0(\calb) \cong {\alpha_r^*}$ of functors from $\CN_2$ to $\CN_2$.

\begin{definition} \label{def:quasi-id}
The left and right \emph{unitors} are the natural transformations
$\hat\Upsilon^l:=U^l \circ \Upsilon^l \colon H_0(\cala) \boxtimes_\cala - \to \beta_l^*$
and $\hat\Upsilon^r:=U^r \circ \Upsilon^r \colon -\boxtimes_\calb H_0(\calb)  \to\beta_r^*$.
We draw them as
\[
\hat \Upsilon^l_H \colon \, \tikzmath[scale=\displscale]{\fill[fill=vacuumcolor](0,0) rectangle (12,12);\fill[fill=spacecolor](12,0) rectangle (24,12);\draw (0,0) rectangle (12,12)(18,12) -- (12,12)(12,0) -- (18,0);
\draw (18,6) node {$H$};\draw[ultra thick] (18,12) -- (24,12) -- (24,0) -- (18,0);} %tikzmath
\, \to \,  \tikzmath[scale=\displscale]{\draw (-1.5,-1.5) rectangle (25.5,13.5); \fill[fill=spacecolor](12,0) rectangle (24,12);\draw (18,6) node {$H$};\draw(18,12) -- (12,12) -- (12,0) -- (18,0);
\draw[ultra thick] (18,12) -- (24,12) -- (24,0) -- (18,0); \draw[->] (.3,1) -- (11,1);\draw[->] (.3,3.5) -- (11,3.5);\draw[->] (.3,6) -- (11,6);\draw[->] (.3,8.5) -- (11,8.5);\draw[->] (.3,11) -- (11,11);} %tikzmath
\qquad\quad\hat \Upsilon^r_H \colon \,\tikzmath[scale=\displscale]{\fill[fill=vacuumcolor](24,0) rectangle (36,12);\fill[fill=spacecolor](12,0) rectangle (24,12);\draw (18,0) -- (12,0)-- (12,12) -- (18,12);
\draw[ultra thick] (24,0) rectangle (36,12) (18,12) -- (24,12)(18,0) -- (24,0); \draw (18,6) node {$H$};} %tikzmath
\, \to \, \tikzmath[scale=\displscale]{\draw (10.5,-1.5) rectangle (37.5,13.5); \fill[fill=spacecolor](12,0) rectangle (24,12);\draw (18,0) -- (12,0)-- (12,12) -- (18,12);
\draw[ultra thick] (18,12) -- (24,12) -- (24,0) -- (18,0);\draw[<-] (25,1) -- (35.7,1);\draw[<-] (25,3.5) -- (35.7,3.5);
\draw[<-] (25,6) -- (35.7,6);\draw[<-] (25,8.5) -- (35.7,8.5);\draw[<-] (25,11) -- (35.7,11);\draw (18,6) node {$H$};} \; . %tikzmath
\]
(Of course, the precise transformations $\hat\Upsilon^l$ and $\hat\Upsilon^r$ depend on our unspecified choices of $\beta_l$ and $\beta_r$.)
\end{definition}
  
We record some properties of the unitors that will be used in our later paper~\cite{BDH(3-category)}.
There is the following compatibility with the $1 \boxtimes 1$-isomorphism $\Omega$: 
  
\begin{lemma}
\label{lem:quasi-id-and-1x1-left}
Let $\calb$ be a conformal net with finite index, and let
$_\cala D_\calb$ and $_\calb E_\calc$ be defects. The unitors commute with the $1 \boxtimes 1$-isomorphism in the sense that
\begin{eqnarray*}
\beta_l^*\,\Omega_{D,E}\circ \hat \Upsilon^l_{H_0(D \circledast_\calb E)}  & = & \hat \Upsilon^l_{H_0(D) \boxtimes_\calb H_0(E)} \circ \big(1_{H_0(\cala)} \boxtimes_\cala \Omega_{D,E}\big),\\
\beta_r^*\,\Omega_{D,E}\circ \hat \Upsilon^r_{H_0(D \circledast_\calb E)} & = & \hat \Upsilon^r_{H_0(D) \boxtimes_\calb H_0(E)} \circ \big(\Omega_{D,E} \boxtimes_\calc1_{H_0(\calc)}\big). 
\end{eqnarray*}
That is, the two diagrams
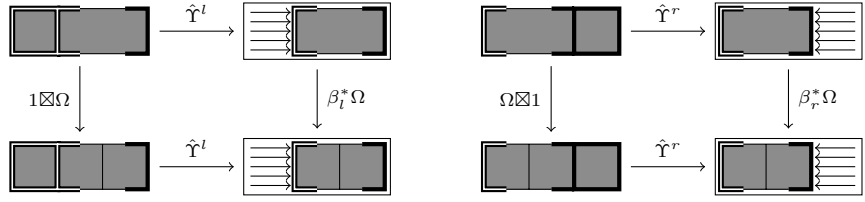
\begin{equation*}
\tikzmath{\matrix [matrix of math nodes,column sep=1cm,row sep=8mm]
{ |(a)| \tikzmath[scale=\displscale]{\fill[fill=vacuumcolor](0,0) rectangle (36,12);\draw[thick, double] (0,0) rectangle (12,12) (18,12) -- (12,12)(12,0) -- (18,0);
\draw (18,12) -- (30,12)(18,0) -- (30,0);\draw[ultra thick] (30,12) -- (36,12) -- (36,0) -- (30,0); } %tikzmath
\pgfmatrixnextcell |(b)| \tikzmath[scale=\displscale]{\draw (-1.5,-1.5) rectangle (37.5,13.5); \fill[fill=vacuumcolor](12,0) rectangle (36,12);
\draw[thick, double] (18,12) -- (12,12) -- (12,0) -- (18,0);\draw (18,12) -- (30,12)(18,0) -- (30,0);\draw[ultra thick] (30,12) -- (36,12) -- (36,0) -- (30,0);
\draw[->] (.3,1) -- (11,1);\draw[->] (.3,3.5) -- (11,3.5);\draw[->] (.3,6) -- (11,6);\draw[->] (.3,8.5) -- (11,8.5);\draw[->] (.3,11) -- (11,11);} %tikzmath
\\  |(c)| \tikzmath[scale=\displscale]{\fill[fill=vacuumcolor](0,0) rectangle (36,12); \draw[thick, double] (0,0) rectangle (12,12)(18,12) -- (12,12)(12,0) -- (18,0);
\draw (18,12) -- (30,12) (24,12) -- (24,0)(18,0) -- (30,0);\draw[ultra thick] (30,12) -- (36,12) -- (36,0) -- (30,0);} %tikzmath
\pgfmatrixnextcell |(d)| \tikzmath[scale=\displscale]{\fill[fill=vacuumcolor](12,0) rectangle (36,12);\draw[thick, double] (18,12) -- (12,12) -- (12,0) -- (18,0);
\draw (18,12) -- (30,12) (24,12) -- (24,0)(18,0) -- (30,0);\draw[ultra thick] (30,12) -- (36,12) -- (36,0) -- (30,0); \draw (-1.5,-1.5) rectangle (37.5,13.5); 
\draw[->] (.3,1) -- (11,1);\draw[->] (.3,3.5) -- (11,3.5);\draw[->] (.3,6) -- (11,6);\draw[->] (.3,8.5) -- (11,8.5);\draw[->] (.3,11) -- (11,11);} %tikzmath
\\ }; %matrix 
\draw[->] (a) -- node [above]{$\scriptstyle \hat\Upsilon^l$} (b); \draw[->] (c) -- node [above]{$\scriptstyle \hat\Upsilon^l$} (d);
\draw[->] (a) -- node [left]{$\scriptstyle 1 \boxtimes \Omega$} (c);\draw[->] (b) -- node [right] {$\scriptstyle \beta_l^* \Omega $} (d);}%tikzmath
\qquad \tikzmath{ \matrix [matrix of math nodes,column sep=1cm,row sep=8mm]
{ |(a)| \tikzmath[scale=\displscale]{\fill[fill=vacuumcolor](0,0) rectangle (36,12);\draw[thick, double] (6,12) -- (0,12) -- (0,0) -- (6,0);\draw (6,12) -- (18,12)(6,0) -- (18,0);
\draw[ultra thick] (24,0) rectangle (36,12) (18,12) -- (24,12)(18,0) -- (24,0); } %tikzmath
\pgfmatrixnextcell |(b)| \tikzmath[scale=\displscale]{\draw (-1.5,-1.5) rectangle (37.5,13.5); \fill[fill=vacuumcolor](0,0) rectangle (24,12);
\draw[thick, double] (6,12) -- (0,12) -- (0,0) -- (6,0);\draw (6,12) -- (18,12)(18,0) -- (6,0);\draw[ultra thick] (18,12) -- (24,12) -- (24,0) -- (18,0);
\draw[<-] (25,1) -- (35.7,1);\draw[<-] (25,3.5) -- (35.7,3.5);\draw[<-] (25,6) -- (35.7,6);\draw[<-] (25,8.5) -- (35.7,8.5);\draw[<-] (25,11) -- (35.7,11);} %tikzmath
\\ |(c)| \tikzmath[scale=\displscale]{\fill[fill=vacuumcolor](0,0) rectangle (36,12); \draw[thick, double] (6,12) -- (0,12) -- (0,0) -- (6,0);
\draw (6,12) -- (18,12) (12,0) -- (12,12) (6,0) -- (18,0); \draw[ultra thick] (24,0) rectangle (36,12) (18,12) -- (24,12)(18,0) -- (24,0); } %tikzmath
\pgfmatrixnextcell |(d)| \tikzmath[scale=\displscale]{ \draw (-1.5,-1.5) rectangle (37.5,13.5); \fill[fill=vacuumcolor](0,0) rectangle (24,12);
\draw[thick, double] (6,12) -- (0,12) -- (0,0) -- (6,0);\draw (6,12) -- (18,12) (12,0) -- (12,12)(18,0) -- (6,0);\draw[ultra thick] (18,12) -- (24,12) -- (24,0) -- (18,0);
\draw[<-] (25,1) -- (35.7,1);\draw[<-] (25,3.5) -- (35.7,3.5);\draw[<-] (25,6) -- (35.7,6);\draw[<-] (25,8.5) -- (35.7,8.5);\draw[<-] (25,11) -- (35.7,11);} %tikzmath
\\}; %matrix 
\draw[->] (a) -- node [above]{$\scriptstyle \hat\Upsilon^r$} (b); \draw[->] (c) -- node [above]{$\scriptstyle \hat\Upsilon^r$} (d);
\draw[->] (a) -- node [left] {$\scriptstyle \Omega \boxtimes 1$} (c); \draw[->] (b) -- node [right]	 {$\scriptstyle \beta_r^* \Omega $} (d);}%tikzmath 
\end{equation*}
commute.
\end{lemma}
  
\begin{proof}
These equalities follow from the naturality of $\hat \Upsilon^l$ and $\hat \Upsilon^r$, specifically by viewing them as a natural transformations
\begin{equation*} 
\begin{split}
\hat \Upsilon^l\;\;:\;\;\; \modules{\cala\big(
\tikzmath[scale=\planscale]{\useasboundingbox (-2,0) rectangle (14,12);\draw (6,12) -- (0,12) -- (0,0) -- (6,0);\draw[->] (2.51,12) -- (2.5,12);} %tikzmath
\big)}\;\;&\tworarrow\,\;\;\modules{\cala(
\tikzmath[scale=\planscale]{\useasboundingbox (-2,0) rectangle (26,12);\draw (18,12) -- (0,12) -- (0,0) -- (18,0);\draw[->] (8.51,12) -- (8.5,12);} %tikzmath
)}
\\
\hat \Upsilon^r\;\;:\;\;\; \modules{\calc\big(
\tikzmath[scale=\planscale]{\useasboundingbox (-2,0) rectangle (14,12);\draw (6,12) -- (12,12) -- (12,0) -- (6,0);\draw[->] (8.51,12) -- (8.5,12);}
\big)}\;\;&\tworarrow\,\;\;\,\modules{\calc(
\tikzmath[scale=\planscale]{\useasboundingbox (-2,0) rectangle (26,12);\draw (6,12) -- (24,12) -- (24,0) -- (6,0);\draw[->] (14.51,12) -- (14.5,12);}
)}.\qedhere\end{split}
\end{equation*}
\end{proof}

Let $\alpha_{l,\top},\beta_{l,\top}:S_{b,\top}\to S_{r,\top}$ and $\alpha_{r,\top}, \beta_{r,\top}:S_{b,\top}\to S_{l,\top}$ be the restrictions 
of $\alpha_l$, $\beta_l$, $\alpha_r$, $\beta_r$ to the upper half of $S_b$.

\begin{lemma} \label{lem:omega-and-hat-upsilon}
Let ${}_\cala D_\calb$ be a defect, color the circles $S_r$ and $S_l$ as in~\eqref{eq: def colors of S and tildeS}, and let $H_r:=H_0(S_r,D)$ and $H_l:=H_0(S_l,D)$.

Then the map $\hat \Upsilon^l_{H_r} \circ \Omega_{\mathrm{id}_\cala,D}:
\tikzmath[scale=\textscale]{\fill[vacuumcolor] (0,0) rectangle(24,12);\draw (0,0) rectangle(24,12); \draw[ultra thick] (18,0) -- (24,0) -- (24,12) -- (18,12);} %tikzmath
\to \tikzmath[scale=\textscale]{\draw (-1.5,-1.5) rectangle (25.5,13.5); \fill[fill=vacuumcolor](12,0) rectangle (24,12);
\draw(18,12) -- (12,12) -- (12,0) -- (18,0);\draw[ultra thick] (18,12) -- (24,12) -- (24,0) -- (18,0);\draw[->] (.3,1) -- (11,1);
\draw[->] (.3,3.5) -- (11,3.5);\draw[->] (.3,6) -- (11,6);\draw[->] (.3,8.5) -- (11,8.5);\draw[->] (.3,11) -- (11,11);} %tikzmath
$
coincides with $L^2(D(\beta_{l,\top}))$,
and $\hat\Upsilon^r_{H_l}\circ \Omega_{D,\mathrm{id}_\calb}: 
\tikzmath[scale=\textscale]{
\fill[vacuumcolor] (0,0) rectangle(24,12);\draw (0,0) rectangle(24,12);\draw[ultra thick] (6,0) -- (24,0) -- (24,12) -- (6,12);} %tikzmath
\to \tikzmath[scale=\textscale]{\draw (10.5,-1.5) rectangle (37.5,13.5); \fill[fill=vacuumcolor](12,0) rectangle (24,12);
\draw (18,0) -- (12,0)-- (12,12) -- (18,12);\draw[ultra thick] (18,12) -- (24,12) -- (24,0) -- (18,0);
\draw[<-] (25,1) -- (35.7,1);\draw[<-] (25,3.5) -- (35.7,3.5);\draw[<-] (25,6) -- (35.7,6);\draw[<-] (25,8.5) -- (35.7,8.5);\draw[<-] (25,11) -- (35.7,11);} %tikzmath
$ 
coincides with $L^2(D(\beta_{r,\top}))$.
\end{lemma}

\begin{proof}
From Lemma~\ref{lem:omega-and-upsilon} (recall that $\epsilon_l=\alpha_l^{-1}$ and $\epsilon_r=\alpha_r^{-1}$), we know that
\[
\Upsilon^l_{H_r} \circ \Omega_{\mathrm{id}_\cala,D}=L^2(D(\alpha_{l,\top}))\,\,,\,\qquad
\Upsilon^r_{H_l}\circ \Omega_{D,\mathrm{id}_\calb}=L^2(D(\alpha_{r,\top})).
\]
We also have
$U^l_{H_r}=L^2(D(\beta_{l,\top} \circ \alpha_{l,\top}^{-1}))$ and $U^r_{H_l}=L^2(D(\beta_{r,\top} \circ \alpha_{r,\top}^{-1}))$ by Lemma~\ref{lem: canonical quantization of the symmetric diffeomorphism}.
Finally recall that $\hat\Upsilon^l=U^l \circ \Upsilon^l$ and $\hat\Upsilon^r=U^r \circ \Upsilon^r$.
We therefore have
\[
\begin{split}
\hat \Upsilon^l_{H_r} \circ \Omega_{\mathrm{id}_\cala,D}
&= U^l_{H_r} \circ \Upsilon^l_{H_r} \circ \Omega_{\mathrm{id}_\cala,D} \\
&= L^2(D( \beta_{l,\top} \circ \alpha^{-1}_{l,\top})) \circ L^2(D(\alpha_{l,\top}))
= L^2(D( \beta_{l,\top}))
\end{split}
\]
and
\[
\begin{split}
\hat \Upsilon^r_{H_l} \circ \Omega_{D,\mathrm{id}_\calb}
&= U^r_{H_l} \circ \Upsilon^r_{H_l} \circ \Omega_{D,\mathrm{id}_\calb} \\
&= L^2(D( \beta_{r,\top} \circ \alpha^{-1}_{r,\top})) \circ L^2(D(\alpha_{r,\top}))
= L^2(D( \beta_{r,\top})).\qedhere
\end{split}
\]
%\begin{equation*}
%\hat \Upsilon^l_{H_r} = L^2(D( \beta_{l,\top} \circ \alpha^{-1}_{l,\top})) \circ \Upsilon^l_{H_r}
%\,\qquad
%\hat \Upsilon^r_{H_l} = L^2(D( \beta_{r,\top} \circ \alpha^{-1}_{r,\top})) \circ \Upsilon^r_{H_l}
%\end{equation*}
%and the result follows.
\end{proof}

Combining this lemma with the factorization of $L^2(D(\beta_{l,\top}))$ and $L^2(D(\beta_{r,\top}))$ from~\eqref{eq:L2(f)-factorization},
we obtain the commutativity of the following two diagrams:
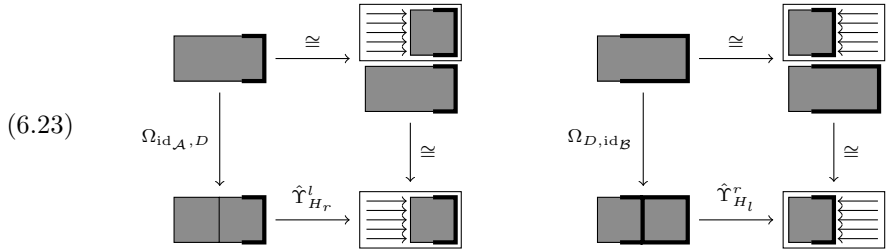
\begin{equation}  \label{eq:lemma-M} %has a number because we want to cite it in cn-proof for lemma M
\tikzmath{ \matrix [matrix of math nodes,column sep=1cm,row sep=8mm]
{|(a)| \tikzmath[scale=\displscale]{\fill[fill=vacuumcolor](0,0) rectangle (24,12);\draw(18,12) -- (0,12) -- (0,0) -- (18,0);
\draw[ultra thick] (18,12) -- (24,12) -- (24,0) -- (18,0);} %tikzmath
\pgfmatrixnextcell |(b)|\tikzmath[scale=\displscale]{\begin{scope} [yshift=7cm]
\draw (-1.5,-1.5) rectangle (25.5,13.5); \fill[fill=vacuumcolor](12,0) rectangle (24,12);\draw(18,12) -- (12,12) -- (12,0) -- (18,0);\draw[ultra thick] (18,12) -- (24,12) -- (24,0) -- (18,0);
\draw[->] (.3,1) -- (11,1); \draw[->] (.3,3.5) -- (11,3.5); \draw[->] (.3,6) -- (11,6); \draw[->] (.3,8.5) -- (11,8.5); \draw[->] (.3,11) -- (11,11);\end{scope}
\begin{scope} [yshift=-8cm]\fill[fill=vacuumcolor](0,0) rectangle (24,12);\draw(18,12) -- (0,12) -- (0,0) -- (18,0);\draw[ultra thick] (18,12) -- (24,12) -- (24,0) -- (18,0);\end{scope}} %tikzmath
\\ |(c)| \tikzmath[scale=\displscale]{\fill[fill=vacuumcolor](0,0) rectangle (24,12);\draw (0,0) rectangle (12,12)(18,12) -- (12,12)(12,0) -- (18,0);
\draw[ultra thick] (18,12) -- (24,12) -- (24,0) -- (18,0);} %tikzmath
\pgfmatrixnextcell |(d)| \tikzmath[scale=\displscale]{\draw (-1.5,-1.5) rectangle (25.5,13.5); \fill[fill=vacuumcolor](12,0) rectangle (24,12);
\draw(18,12) -- (12,12) -- (12,0) -- (18,0);\draw[ultra thick] (18,12) -- (24,12) -- (24,0) -- (18,0);\draw[->] (.3,1) -- (11,1);
\draw[->] (.3,3.5) -- (11,3.5);\draw[->] (.3,6) -- (11,6);\draw[->] (.3,8.5) -- (11,8.5);\draw[->] (.3,11) -- (11,11);} %tikzmath
\\ }; %matrix 
\draw[->] (a) -- node [above] {$\scriptstyle \cong$} (b);\draw[->] (c) -- node [above] {$\scriptstyle \hat\Upsilon^l_{H_r}$} (d);
\draw[->] (a) -- node [left] {$\scriptstyle \Omega_{\id_\cala,D}$} (c);\draw[->] (b) -- node [right] {$\scriptstyle \cong$} (d);}%tikzmath
\qquad \quad
\tikzmath{\matrix [matrix of math nodes,column sep=1cm,row sep=8mm]
{|(a)|\tikzmath[scale=\displscale]{\fill[fill=vacuumcolor](12,0) rectangle (36,12);\draw (18,0) -- (12,0)-- (12,12) -- (18,12);\draw[ultra thick] (18,12) -- (36,12) -- (36,0) -- (18,0); }%tikzmath
\pgfmatrixnextcell |(b)|\tikzmath[scale=\displscale]{\begin{scope} [yshift=7cm]\draw (10.5,-1.5) rectangle (37.5,13.5); \fill[fill=vacuumcolor](12,0) rectangle (24,12);\draw (18,0) -- (12,0)-- (12,12) -- (18,12);\draw[ultra thick] (18,12) -- (24,12) -- (24,0) -- (18,0);\draw[<-] (25,1) -- (35.7,1);\draw[<-] (25,3.5) -- (35.7,3.5);\draw[<-] (25,6) -- (35.7,6);\draw[<-] (25,8.5) -- (35.7,8.5);\draw[<-] (25,11) -- (35.7,11);\end{scope}\begin{scope} [yshift=-8cm]\fill[fill=vacuumcolor](12,0) rectangle (36,12);\draw (18,0) -- (12,0)-- (12,12) -- (18,12);\draw[ultra thick] (18,12) -- (36,12) -- (36,0) -- (18,0); \end{scope}} %tikzmath
\\ |(c)| \tikzmath[scale=\displscale]{\fill[fill=vacuumcolor](12,0) rectangle (36,12);\draw (18,0) -- (12,0)-- (12,12) -- (18,12);
\draw[ultra thick] (24,0) rectangle (36,12) (18,12) -- (24,12)(18,0) -- (24,0); }%tikzmath
\pgfmatrixnextcell|(d)|\tikzmath[scale=\displscale]{\draw (10.5,-1.5) rectangle (37.5,13.5); \fill[fill=vacuumcolor](12,0) rectangle (24,12);\draw (18,0) -- (12,0)-- (12,12) -- (18,12);\draw[ultra thick] (18,12) -- (24,12) -- (24,0) -- (18,0);\draw[<-] (25,1) -- (35.7,1);\draw[<-] (25,3.5) -- (35.7,3.5);\draw[<-] (25,6) -- (35.7,6);\draw[<-] (25,8.5) -- (35.7,8.5);\draw[<-] (25,11) -- (35.7,11);} %tikzmath
\\ }; %matrix 
\draw[->] (a) -- node [above]{$\scriptstyle \cong$} (b);\draw[->] (c) -- node [above]{$\scriptstyle \hat\Upsilon^r_{H_l}$} (d);
\draw[->] (a) -- node [left]{$\scriptstyle \Omega_{D,\id_{\calb}}$} (c);\draw[->] (b) -- node [right]{$\scriptstyle \cong$} (d);} %tikzmath 
\end{equation}
Here, the stacked pictures in the upper right-hand corners signify the Connes  fusion along the algebras 
$D(\tikzmath[scale=\textscale]{\draw (0,0) -- (18,0); \draw[ultra thick] (18,0) -- (24,0);})$ and 
$D(\tikzmath[scale=\textscale]{\draw (0,0) -- (6,0); \draw[ultra thick] (6,0) -- (24,0);})$ respectively, and we have suppressed
the isomorphism~\eqref{eq:fusion-over-A-or-B}.

\section{The interchange isomorphism}
\label{subsec:interchange}

In a 2-category, the {\em interchange law} says that the two ways of evaluating the diagram
\[
\tikzmath[scale=1.7]{\node[inner sep=5] (a) at (0,0) {$\cala$};\node[inner sep=5] (b) at (2,0) {$\calb$};\node[inner sep=5] (c) at (4,0) {$\calc$};
\draw[->] (a) to[out=55, in=125]node[above]{$\scriptstyle D$} (b);\draw[->] (a) to[out=-55, in=-125]node[below]{$\scriptstyle P$} (b);
\draw[->] (a) -- node[fill=white, inner sep=.5]{$\scriptstyle F$} (b);\draw[->] (b) -- node[fill=white, inner sep=.5]{$\scriptstyle G$} (c);
\draw[->] (b) to[out=55, in=125]node[above]{$\scriptstyle E$} (c);\draw[->] (b) to[out=-55, in=-125]node[below]{$\scriptstyle Q$} (c);
\node at (.95,.3) {$\Downarrow$};\node at (1.15,.3) {$ H$};\node at (2.95,.3) {$\Downarrow$};\node at (3.15,.3) {$ K$};
\node at (.95,-.3) {$\Downarrow$};\node at (1.15,-.3) {$ L$};\node at (2.95,-.3) {$\Downarrow$};\node at (3.15,-.3) {$ M$};
}%tikzmath
\]
are equal to each other:
if one first performs the two vertical compositions 
and then composes horizontally, or one first composes horizontally and then vertically, 
one should obtain the same result.
In our case, in a 3-category, the two ways of fusing four sectors
\begin{equation}\label{eq:   4squares}
\def\h{10}
\tikzmath[scale=.09]{
\useasboundingbox (-5.5,-23) rectangle (35,17);
\fill[spacecolor](0,0) rectangle(12,12);
\draw[thick, double] (6,0) -- (0,0) -- (0,12) -- (6,12);
\draw(6,12) -- (12,12) -- (12,0) -- (6,0); 
\draw (-3,6) node {$\cala$}(6,15) node {$D$}(6,-3) node {$F$}(6,6)node {$H$};\pgftransformxshift{\h}\draw(15,6) node {$\calb$};\pgftransformxshift{\h}
\pgftransformxshift{515}     \fill[spacecolor](0,0) rectangle(12,12);\draw (6,0) -- (0,0) -- (0,12) -- (6,12);
\draw[ultra thick](6,12) -- (12,12) -- (12,0) -- (6,0); 
\draw (15,6) node {$\calc$}(6,15) node {$E$}(6,6)node {$K$};
\pgftransformyshift{-513}     \fill[spacecolor](0,0) rectangle(12,12);
\draw (6,0) -- (0,0) -- (0,12) -- (6,12);
\draw[ultra thick](6,12) -- (12,12) -- (12,0) -- (6,0); 
\draw (-3,6) (15,6) node {$\calc$}(6,15) node {$G$}(6,-3) node {$Q$}(6,6)node {$M$};
\pgftransformxshift{-\h}\pgftransformxshift{-\h}
\pgftransformxshift{-515}     \fill[spacecolor](0,0) rectangle(12,12);
\draw[thick, double] (6,0) -- (0,0) -- (0,12) -- (6,12);\draw(6,12) -- (12,12) -- (12,0) -- (6,0); 
\draw (-3,6) node {$\cala$}(6,-3) node {$P$}(6,6)node {$L$};\pgftransformxshift{\h}\draw(15,6) node {$\calb$};}\,\,,
\end{equation}
namely
\(
\tikzmath[scale=\squarescale]
{
\fill[spacecolor] (0,3) -- (0,15) -- (12,15) -- (12,3);
\draw[thick, double] (6,3) -- (0,3) -- (0,15) -- (6,15);
\draw (6,15) -- (12,15) -- (12,3) -- (6,3); 
\fill[spacecolor] (0,-15) -- (0,-3) -- (12,-3) -- (12,-15);
\draw[thick, double] (6,-15) -- (0,-15) -- (0,-3) -- (6,-3);
\draw (6,-3) -- (12,-3) -- (12,-15) -- (6,-15); 
\draw (6,9)node {$H$} (6,-9)node {$L$}; 
\draw[<->] (-.6,9) to[out=180,in=180] (-.6,-9); 
\draw[<->] (-.6,6) to[out=200,in=160] (-.6,-6); 
\draw[<->] (-.6,2.5) to[out=230,in=130] (-.6,-2.5); 
\draw[<->] (2.7,2.5) to[out=255,in=105] (2.7,-2.5); 
\draw[<->] (12.2,9) to[out=0,in=0] (12.2,-9); 
\draw[<->] (12.2,6) to[out=-20,in=20] (12.2,-6); 
\draw[<->] (12.2,2.7) to[out=-50,in=50] (12.2,-2.7); 
\draw[<->] (9.3,2.7) to[out=-75,in=75] (9.3,-2.7); 
\draw[<->] (6,2.5) -- (6,-2.5); 
\draw[stealth-stealth] (12.4,15) -- (27.9,15);
\draw[stealth-stealth] (12.4,11) -- (27.9,11);
\draw[stealth-stealth] (16.1,7) -- (24.2,7);
\draw[stealth-stealth] (17.6,2.35) -- (22.7,2.35);
\draw[stealth-stealth] (12.4,-15) -- (27.9,-15);
\draw[stealth-stealth] (12.4,-11) -- (27.9,-11);
\draw[stealth-stealth] (16.1,-7) -- (24.2,-7);
\draw[stealth-stealth] (17.6,-2.35) -- (22.7,-2.35); 
\pgftransformxshift{805} 
\fill[spacecolor] (0,3) -- (0,15) -- (12,15) -- (12,3);
\draw (6,3) -- (0,3) -- (0,15) -- (6,15);
\draw[ultra thick](6,15) -- (12,15) -- (12,3) -- (6,3); 
\fill[spacecolor] (0,-15) -- (0,-3) -- (12,-3) -- (12,-15);
\draw (6,-15) -- (0,-15) -- (0,-3) -- (6,-3);
\draw[ultra thick](6,-3) -- (12,-3) -- (12,-15) -- (6,-15); 
\draw (6,9)node {$K$} (6,-9)node {$M$}; 
\draw[<->] (-.2,9) to[out=180,in=180] (-.2,-9); 
\draw[<->] (-.2,6) to[out=200,in=160] (-.2,-6); 
\draw[<->] (-.2,2.7) to[out=230,in=130] (-.2,-2.7); 
\draw[<->] (2.7,2.7) to[out=255,in=105] (2.7,-2.7); 
\draw[<->] (12.5,9) to[out=0,in=0] (12.5,-9); 
\draw[<->] (12.5,6) to[out=-20,in=20] (12.5,-6); 
\draw[<->] (12.5,2.6) to[out=-50,in=50] (12.5,-2.6);
\draw[<->] (9.3,2.6) to[out=-75,in=75] (9.3,-2.6); 
\draw[<->] (6,2.5) -- (6,-2.5);} %tikzmath
\) versus \(
\tikzmath[scale=\squarescale]{ 
\fill[spacecolor] (0,3) -- (0,15) -- (12,15) -- (12,3);
\draw[thick, double] (6,3) -- (0,3) -- (0,15) -- (6,15);
\draw (6,15) -- (12,15) -- (12,3) -- (6,3); 
\fill[spacecolor] (0,-15) -- (0,-3) -- (12,-3) -- (12,-15);
\draw[thick, double] (6,-15) -- (0,-15) -- (0,-3) -- (6,-3);
\draw (6,-3) -- (12,-3) -- (12,-15) -- (6,-15); 
\draw (6,9)node {$H$} (6,-9)node {$L$};
\draw[<->] (-.6,9) to[out=180,in=180] (-.6,-9); 
\draw[<->] (-.6,6) to[out=200,in=160] (-.6,-6); 
\draw[<->] (-.6,2.5) to[out=230,in=130] (-.6,-2.5); 
\draw[<->] (2.7,2.5) to[out=255,in=105] (2.7,-2.5); 
\draw[stealth-stealth] (12.4,15) -- (20.7,15);
\draw[stealth-stealth] (12.4,12.15) -- (20.7,12.15);
\draw[stealth-stealth] (12.4,9.3) -- (20.7,9.3);
\draw[stealth-stealth] (12.4,6.4) -- (20.7,6.4);
\draw[stealth-stealth] (12.4,3.5) -- (20.7,3.5);
\draw[stealth-stealth] (12.4,-15) -- (20.7,-15);
\draw[stealth-stealth] (12.4,-12.15) -- (20.7,-12.15);
\draw[stealth-stealth] (12.4,-9.3) -- (20.7,-9.3);
\draw[stealth-stealth] (12.4,-6.4) -- (20.7,-6.4);
\draw[stealth-stealth] (12.4,-3.5) -- (20.7,-3.5);  
\pgftransformxshift{15}\draw[<->] (6,2.8) -- (6,-2.8); 
\pgftransformxshift{15}\draw[<->] (9.5,2.8) -- (9.5,-2.8); 
\pgftransformxshift{15}\draw[<->] (13,2.8) -- (13,-2.8); 
\pgftransformxshift{15}\draw[<->] (16.55,2.8) -- (16.55,-2.8); 
\pgftransformxshift{15}\draw[<->] (20.1,2.8) -- (20.1,-2.8); 
\pgftransformxshift{15}\draw[<->] (23.6,2.8) -- (23.6,-2.8); 
\pgftransformxshift{510} 
\fill[spacecolor] (0,3) -- (0,15) -- (12,15) -- (12,3);
\draw (6,3) -- (0,3) -- (0,15) -- (6,15);
\draw[ultra thick](6,15) -- (12,15) -- (12,3) -- (6,3); 
\fill[spacecolor] (0,-15) -- (0,-3) -- (12,-3) -- (12,-15);
\draw (6,-15) -- (0,-15) -- (0,-3) -- (6,-3);
\draw[ultra thick](6,-3) -- (12,-3) -- (12,-15) -- (6,-15); 
\draw (6,9)node {$K$} (6,-9)node {$M$}; 
\draw[<->] (12.5,9) to[out=0,in=0] (12.5,-9); 
\draw[<->] (12.5,6) to[out=-20,in=20] (12.5,-6); 
\draw[<->] (12.5,2.6) to[out=-50,in=50] (12.5,-2.6); 
\draw[<->] (9.3,2.6) to[out=-75,in=75] (9.3,-2.6);} %tikzmath
\),\medskip\\
should yield the same answer up to natural isomorphism. 
In other words, we need a unitary natural transformation\footnote{Here, as in the $1\boxtimes 1$-isomorphism $\Omega$, we restrict to the groupoid parts of $\CN_1$ and $\CN_0$.}
\begin{equation} \label{Exc}
\big( \CN_2\times_{ \CN_1^f}  \CN_2\big)\times_{ \CN_0^f}
\big( \CN_2\times_{ \CN_1^f}  \CN_2\big)\,\tworarrow\,  \CN_2
\end{equation}
between the functors
$\mathsf{fusion_h}\circ (\mathsf{fusion_v}\times \mathsf{fusion_v})$ and $\mathsf{fusion_v}\circ (\mathsf{fusion_h}\times \mathsf{fusion_h})\circ \tau$,
where
\[
\begin{split}
\tau\,\,\,:\,\,\,\big( \CN_2\times_{ \CN_1^f} & \CN_2\big)\times_{ \CN_0^f}
\big( \CN_2\times_{ \CN_1^f}  \CN_2\big)\\
\to
\big( \CN_2&\times_{ \CN_0^f}  \CN_2\big)\times_{ \CN_1^f\times_{ \CN_0^f}  \CN_1^f}
\big( \CN_2\times_{ \CN_0^f}  \CN_2\big)
\end{split}
\]
is the isomorphism that exchanges the two middle factors.

More concretely, given sectors ${}_DH_F$, ${}_EK_G$, ${}_FL_P$, ${}_GM_Q$ as in (\ref{eq:   4squares}),
we are looking for a unitary isomorphism
\begin{equation}\label{eq: interchange law -- concretely}
\big(H\boxtimes_F L\big)\boxtimes_{\calb} \big(K\boxtimes_G M\big)
\stackrel{\scriptscriptstyle \cong}{\longrightarrow} 
\big(H\boxtimes_{\calb} K\big)\boxtimes_{F\circledast_\calb G} \big(L\boxtimes_{\calb} M\big)
\end{equation}
of $D\circledast_\calb E$\,-\,$P\circledast_\calb Q$\,-sectors.

We may view $(H, K, L, M)$ as an object of the category
\[
\mathsf C:=\quad
\tikzmath{
\node[anchor=east] at (0,1) {$
F\Big(\,\,\tikzmath[scale=\displscale]
{\draw[thick, double] (6,0) -- (0,0) -- (0,6);\draw  (12,12) -- (12,0) -- (6,0); }%tikzmath      
\,\,\Big)\text{-modules}\,\,\,\times\,\,\,
G\Big(\,\,
\tikzmath[scale=\displscale]
{\draw (6,0) -- (0,0) -- (0,12);\draw[ultra thick]  (12,6) -- (12,0) -- (6,0); }%tikzmath      
\,\,\Big)\text{-modules}$};
\node[anchor=east] at (0,0) {$ \times\,\,\,\,\,
F\Big(\,\,\tikzmath[scale=\displscale]
{\draw[thick, double] (0,6) -- (0,12) -- (6,12);\draw  (6,12) -- (12,12) -- (12,0); }%tikzmath      
\,\,\Big)\text{-modules}\,\,\,\times\,\,\,
G\Big(\,\,\tikzmath[scale=\displscale]
{\draw (6,12) -- (0,12) -- (0,0);\draw[ultra thick]  (12,6) -- (12,12) -- (6,12); }%tikzmath      
\,\,\Big)\text{-modules}$};
}\,\,\,.
\]
The forgetful functor $( \CN_2\times_{\{F\}}  \CN_2)\times ( \CN_2\times_{\{G\}}  \CN_2)\to \mathsf C$ is faithful.
In order to construct the natural transformation \eqref{Exc}, it is therefore enough
to produce corresponding natural transformations
\begin{equation}\label{CtwoHS}
\mathsf C\hspace{.15cm} \tworarrow\hspace{.2cm} \text{Hilbert Spaces}
\end{equation}
for every $F$ and $G$.
The fact that \eqref{eq: interchange law -- concretely} 
intertwines the actions of 
$D(\,\tikzmath[scale=\textscale]
{\useasboundingbox(0,3) rectangle (12,12);
\draw[thick, double] (6,12) -- (0,12) -- (0,6);\draw  (12,12) -- (6,12);}\,)$, %tikzmath 
$E(\,\tikzmath[scale=\textscale]
{\useasboundingbox(0,3) rectangle (12,12);
\draw[ultra thick] (6,12) -- (12,12) -- (12,6);\draw  (0,12) -- (6,12);}\,)$ %tikzmath 
$P(\,\tikzmath[scale=\textscale]
{\useasboundingbox(0,0) rectangle (12,9);
\draw[thick, double] (6,0) -- (0,0) -- (0,6);\draw  (12,0) -- (6,0);}\,)$, and %tikzmath 
$Q(\,\tikzmath[scale=\textscale]
{\useasboundingbox(0,0) rectangle (12,9);
\draw[ultra thick] (6,0) -- (12,0) -- (12,6);\draw  (0,0) -- (6,0);}\,)$, %tikzmath 
i.e., that it is a morphism of $D\circledast_\calb E$\,-\,$P\circledast_\calb Q$\,-sectors,
will then follow from the naturality of \eqref{CtwoHS}.

The object $\big(H_0(F),\,H_0(G),\,H_0(F),\,H_0(G)\big)\in\mathsf C$
consists of faithful modules; the obvious analog of Lemma \ref{lem: Mod(A1) x Mod(A2) --> Mod(B)} (itself a generalization of \ref{lem: NT between module categories}) applies,
and so it is enough to construct the natural transformation (\ref{CtwoHS}) on this object.
Using the isomorphisms \eqref{eq: Re: Omega} and (\ref{eabd}), the required transformation is given by
\[
\begin{split}
\big(H_0&(F)\boxtimes_F H_0(F)\big)\boxtimes_{\calb} \big(H_0(G)\boxtimes_G H_0(G)\big)\xrightarrow{\scriptscriptstyle \cong}\,\,
H_0(F)\boxtimes_{\calb} H_0(G)\\
\xrightarrow{\,\Omega^{-1}}\,\,&
H_0\big(F\circledast_\calb G\big)
\xrightarrow{\scriptscriptstyle \cong}\,\,
H_0\big(F\circledast_\calb G\big)\boxtimes_{F\circledast_\calb G} H_0\big(F\circledast_\calb G\big)\\
\xrightarrow{\Omega \,\boxtimes\, \Omega}\,&
\big(H_0(F)\boxtimes_{\calb} H_0(G)\big)\boxtimes_{F\circledast_\calb G} \big(H_0(F)\boxtimes_{\calb} H_0(G)\big).
\end{split}
\]
Using the compatibility of $\Omega$ with the monoidal structure (Proposition~\ref{prop:omega-is-monodial}),
the same can be deduced for the interchange isomorphism: \eqref{Exc} is a monoidal natural transformation.

 %\ignore{  %ignore3

%---------appendix----------
%---------appendix----------
%---------appendix----------
\appendix 

\renewcommand{\thesection}{\Alph{chapter}.{\Roman{section}}}

%---------appendix----------
%---------appendix----------
%---------appendix----------

\chapter{Components for the 3-category of conformal nets}
%%\addtocontents{toc}{\SkipTocEntry}

The purpose of our next paper \cite{BDH(3-category)} is to construct the symmetric monoidal 3-category of conformal nets.
More precisely, we will construct an internal dicategory object $(\mathsf C_0,\mathsf C_1,\mathsf C_2)$ in the 2-category of symmetric monoidal categories \cite[Def. 3.3]{Douglas-Henriques(Internal-bicategories)}.\footnote{\label{footn:dicategory} Following \cite{Douglas-Henriques(Internal-bicategories)}, a dicategory is a bicategory where the associators $(fg)h\Rightarrow f(gh)$ are strict (i.e., they are identity 2-morphisms) but the unitors $1f\Rightarrow f$ and $f1\Rightarrow f$ are not necessarily strict.}
%\AB{Why do we use $\mathsf C$ and not $\CN$ here?}
In this book, we have developed the essential ingredients of that 3-category.
These ingredients are:

\begin{list}{$\bullet$}{\itemsep=1ex \leftmargin=3ex \labelsep=1.5ex \topsep=2ex}
\item A symmetric monoidal category $\mathsf C_0$ whose objects are the conformal nets with finite index, and whose morphisms are the isomorphisms between them.
\item A symmetric monoidal category $\mathsf C_1$ whose objects are the defects between conformal nets of finite index, and whose morphisms are the isomorphisms.
\item A symmetric monoidal category $\mathsf C_2$ whose objects are sectors (between defects between conformal nets of finite index), and whose morphisms are those homomorphisms of 
sectors that cover isomorphisms of defects and of conformal nets.
\item These come with source and target functors
$\mathsf{s},\mathsf{t}:\mathsf C_1\to \mathsf C_0$ and $\mathsf{s},\mathsf{t}:\mathsf C_2\to \mathsf C_1$
subject to the identities $\mathsf{s}\circ\mathsf{s}=\mathsf{s}\circ\mathsf{t}$ and $\mathsf{t}\circ\mathsf{s}=\mathsf{t}\circ\mathsf{t}$.
\item A symmetric monoidal functor $\mathsf{composition}: \mathsf C_1 \times_{\mathsf C_0} \mathsf C_1 \to \mathsf C_1$
that describes the composition (or fusion) of defects \eqref{eq: composition -- this time between the correct categories}.
That the composition of defects exists is the content of Theorem \ref{thm:fusion-of-defects-is-defect}.
\item A symmetric monoidal functor
$\mathsf{fusion_h}: \mathsf C_2 \times_{\mathsf C_0} \mathsf C_2 \to \mathsf C_2$
providing the horizontal composition of sectors \eqref{eq: The functor fusion_h}.
\item A symmetric monoidal functor $\mathsf{fusion_v}: \mathsf C_2 \times_{\mathsf C_1} \mathsf C_2 \to \mathsf C_2$
providing the vertical composition of sectors \eqref{eq: functor of vertical fusion}.
\item Symmetric monoidal functors $\mathsf{identity}:\mathsf C_0 \to \mathsf C_1$ and $\mathsf{identity_v}:\mathsf C_1 \to \mathsf C_2$ providing the identity defects \eqref{eq: identity functor} and identity sectors \eqref{eq: functor identity_v}.
\item A monoidal natural transformation
$\mathsf{associator}\colon \mathsf C_1 \times_{\mathsf C_0} \mathsf C_1\times_{\mathsf C_0} \mathsf C_1\, \tworarrow\, \mathsf C_1$
that is an associator for $\mathsf{composition}$ \eqref{eq: associator natural transformation}.
\item A monoidal natural transformation
$\mathsf{associator_h}:\mathsf C_2\times_{\mathsf C_0}\mathsf C_2\times_{\mathsf C_0}\mathsf C_2\,\tworarrow\, \mathsf C_2$
that is an associator for $\mathsf{fusion_h}$ \eqref{eq: nt associator_h}
\item A monoidal natural transformation
$\mathsf{associator_v}:\mathsf C_2\times_{\mathsf C_1}\mathsf C_2\times_{\mathsf C_1}\mathsf C_2\,\tworarrow\, \mathsf C_2$
that is an associator for $\mathsf{fusion_v}$ \eqref{eq: nt associator_v}.
\item Two monoidal natural transformations
$\mathsf{unitor}_\mathsf{t}, \mathsf{unitor}_\mathsf{b}:\mathsf C_2\tworarrow \mathsf C_2$
that relate $\mathsf{fusion_v}$ and $\mathsf{identity_v}$ \eqref{eq: `top' and `bottom' identity nt}.
\item The\vspace{-.5ex} coherences for $\mathsf{composition}$ and $\mathsf{identity}$ are ``weak'':
instead of natural transformations $\mathsf C_1\tworarrow \mathsf C_1$, we have four
functors $\mathsf{unitor}_\mathsf{tl},\mathsf{unitor}_\mathsf{tr},\mathsf{unitor}_\mathsf{bl},\mathsf{unitor}_\mathsf{br}:\mathsf C_1 \to \mathsf C_2$ (\ref{eq: unitor_tl}, \ref{eq: unitor_tl+}).
This weakness, which is an intrinsic feature of conformal nets and defects, is what forces us to use the notion of internal dicategory \cite[Def.\!~3.3]{Douglas-Henriques(Internal-bicategories)} instead of the simpler notion of internal 2-category \cite[Def.\!~3.1]{Douglas-Henriques(Internal-bicategories)}.
\item The coherence between $\mathsf{composition}$ and $\mathsf{identity_v}$ is a monoidal natural\vspace{-.5ex} transformation $\mathsf C_1\times_{\mathsf C_0} \mathsf C_1\, \tworarrow\, \mathsf C_2$ (\ref{eq: 1x1iso}, \ref{eq: Re: Omega}). The difficult construction of this coherence forces us to restrict the morphisms in the category $\mathsf C_1$ of defects to be isomorphisms.  We have called this natural transformation the ``$1 \boxtimes 1$-isomorphism'' because its domain is a Connes fusion of two identity sectors.
\item Finally, the fundamental interchange isomorphism, a coherence between $\mathsf{fusion_h}$ and $\mathsf{fusion_v}$, is a monoidal natural transformation 
\[
\big(\mathsf C_2\times_{\mathsf C_1} \mathsf C_2\big)\times_{\mathsf C_0}
\big(\mathsf C_2\times_{\mathsf C_1} \mathsf C_2\big)\,\tworarrow\, \mathsf C_2
\quad \eqref{Exc}.
\]
Its definition relies crucially on the $1 \boxtimes 1$-isomorphism.
%\item A compatibility
%  between the $1 \boxtimes 1$-isomorphism and the unit map for
%  identity defects (Lemma~\ref{lem:omega-and-upsilon}).
%\item Section~\ref{sec:quasi-identities} contains the construction of two further
%  structure maps concerning units that will be
%  useful in the sequel~\cite{BDH(3-category)}. 
\end{list}

\chapter{Von Neumann algebras}
\label{app:vN-algebras}

    Given a Hilbert space $H$, we let $\bfB(H)$ denote its algebra of 
    bounded operators. 
    The ultraweak topology on $\bfB(H)$ is the topology of pointwise 
    convergence with respect to the pairing with its predual, 
    the trace class operators.

  \begin{definition}
    A von Neumann algebra, is a topological *-algebra (without any compatibility between the topology and the algebra structure) that is embeddable 
    as closed subalgebra of $\bfB(H)$ with respect to the ultraweak topology.
  \end{definition}

  The spatial tensor product $A_1\bar\otimes A_2$ of von Neumann algebras 
  $A_i\subset \bfB(H_i)$ is the closure in $\bfB(H_1\otimes H_2)$ of 
  their algebraic tensor product $A_1\otimes_\alg A_2$.

  \begin{definition}
    Let $A$ be a von Neumann algebra.
    A left (right) $A$-module is a Hilbert space $H$ equipped with a 
    continuous homomorphism from $A$ (respectively $A^\op$) to $\bfB(H)$.
    We will use the notation ${}_AH$ (respectively $H_A$) to denote the 
    fact that $H$ is a left (right) $A$-module.
  \end{definition}
 
We now concisely recapitulate those aspects of  \cite{BDH(Dualizability+Index-of-subfactors), BDH(nets), BDH(modularity)} that are used in the present book,
along with some other general facts about von Neumann algebras.
For further details, we refer the reader
to \cite[\S 2 and \S 6]{BDH(Dualizability+Index-of-subfactors)} for Section \ref{subsec:Haagerup-L^2},
to \cite[\S 3]{BDH(Dualizability+Index-of-subfactors)} for Section \ref{subsec:Connes-fusion},
to \cite[Appendix A]{BDH(modularity)} for Section \ref{subsec:cyclic-fusion}, 
to \cite[\S1.3]{BDH(nets)} for Section \ref{subsec:fusion+fiber-prod}, 
to \cite[\S 4]{BDH(Dualizability+Index-of-subfactors)} for Section \ref{subsec:dualizability}, 
and to \cite[\S 5]{BDH(Dualizability+Index-of-subfactors)} for Section \ref{subsec:stat-dim+minimal-index}.

\section{The Haagerup $L^2$-space}
    \label{subsec:Haagerup-L^2}
%\addtocontents{toc}{\SkipTocEntry}
%    (See~\cite[\S 2]{BDH(Dualizability+Index-of-subfactors)} for further details.) 
    A faithful left module $H$ for a von Neumann algebra $A$ is called a 
    \emph{standard form} if it comes equipped with 
    an antilinear isometric involution $J$ and a selfdual cone $P\subset H$ 
    subject to the properties
    \begin{enumerate}
	\item $J A J = A'$ on $H$,
	\item $J c J = c^*$ for all $c \in Z(A)$,
	\item $J \xi = \xi$ for all $\xi \in P$,
	\item $a J a J (P) \subseteq P$ for all $a \in A$
    \end{enumerate}
    where $A'$ denotes the commutant of $A$.
    The operator $J$ is called the \emph{modular conjugation}.
    The standard form is an $A$--$A$-bimodule, 
    with right action $\xi a := J a^* J \xi$.
    It is  unique up to unique unitary 
    isomorphism~\cite{Haagerup(1975standard-form)}.
    
    The space of continuous linear functionals $A\to \IC$ forms a Banach 
    space $A_*=L^1(A)$ called the predual of $A$.
    It comes with a positive cone 
    $L^1_+(A):=\{\phi\in A_*\,|\,\phi(x)\ge0\,\, \forall x\in A_+\}$ 
    and two commuting $A$-actions given by $(a\phi b)(x):=\phi(bxa)$.
    Given a von Neumann algebra $A$ there is a canonical
    construction of a standard form 
    for $A$~\cite{Kosaki(PhD-thesis)}.
    It is the completion of 
    $$\bigoplus_{\phi\in L^1_+(A)} \IC\textstyle\sqrt{\phi}$$
    with respect to some pre-inner product, and is denoted $L^2(A)$.
    The positive cone in $L^2A$ is given by 
    $L^2_+(A):=\{\sqrt\phi\,\,|\,\phi\in L^1_+(A)\}$.
    The modular conjugation $J_A$ 
    sends $\lambda\sqrt{\phi}$ to $\bar\lambda\sqrt{\phi}$ for 
    $\lambda\in\IC$.

    If $f \colon A \to B$ is an isomorphism,
    then we write $L^2(f) \colon L^2(A) \xrightarrow{\cong} L^2(B)$
    for the induced unitary isomorphism.
    The standard form is in fact functorial with respect to a bigger class
    of maps; see Section~\ref{subsec:dualizability}. 

    %\begin{remark} \label{rem: L^2(A^op)}
      %There is an isomorphism $L^2(A) \cong L^2(A^\op)$ under which the 
      %left action of $A$ on $L^2A$
      %corresponds to the right action of $A^\op$ on $L^2(A^\op)$, 
      %and the right action of $A$ on $L^2A$
      %corresponds to the left action of $A^\op$ on $L^2(A^\op)$.
    %\end{remark}

\section{Connes fusion} \label{subsec:Connes-fusion}
%\addtocontents{toc}{\SkipTocEntry}
%   (See~\cite[\S 3]{BDH(Dualizability+Index-of-subfactors)} for further details.)
   \begin{definition}
     Given two modules $H_A$ and ${}_AK$ over a von Neumann algebra $A$,
     their Connes fusion $H\boxtimes_A K$ is the 
     completion of
     \begin{equation}\label{eq:def of CFus}
         \mathrm{Hom}\big(L^2(A)_A,H_A\big)\otimes_A L^2(A)\otimes_A 
                \mathrm{Hom}\big({}_AL^2(A),{}_AK\big)
     \end{equation}
     with respect to the inner product
     $\big\langle\phi_1\otimes \xi_1\otimes 
        \psi_1,\,\phi_2\otimes \xi_2\otimes \psi_2\big\rangle:=
      \big\langle(\phi_2^*\phi_1)\xi_1(\psi_1\psi_2^*),\xi_2\big\rangle$ \cite{Connes(Geometrie-non-commutative), 
       Sauvageot(Sur-le-produit-tensoriel-relatif), 
       Wassermann(Operator-algebras-and-conformal-field-theory)}.
     Here, we have written the action of $\psi_i$ on the right, 
     which means that $\psi_1\psi_2^*$ stands for the composite 
     $L^2(A)\xrightarrow{\psi_1} K\xrightarrow{\psi_2^*} L^2(A)$.
   \end{definition}

   The $L^2$ space is a unit for Connes fusion in the sense that there are 
   canonical unitary isomorphisms 
   \begin{equation}\label{eq: unitality of CFusion}
     {}_A L^2(A) \boxtimes_A H \cong {}_A H \qquad \text{and} 
          \qquad H\boxtimes_A L^2(A)_A \cong H_A
   \end{equation}
   defined by $\phi \otimes \xi \otimes \psi \mapsto (\phi \xi) \psi$
   and $\phi \otimes \xi \otimes \psi \mapsto \phi (\xi \psi)$.
   If $f \colon A \to B$ is an isomorphism of von Neumann algebras,
   $H_A$ and ${}_B K$ are modules, then 
   \begin{equation} \label{eq:fusion-over-A-or-B}
      (H_A)_{f^{-1}} \boxtimes_B {}_B K  \cong  H_A \boxtimes_A {}_f({}_B K)
   \end{equation}
   via $\phi \otimes \xi \otimes \psi \mapsto 
   (\phi\circ L^2(f)) \otimes L^2(f)^{-1}(\xi) \otimes (\psi \circ L^2(f))$.
   Here the indices ${}_{f^{-1}}$ and ${}_f$ indicate restrictions
   of actions along the isomorphisms $f$ and $f^{-1}$. 
   Using~\eqref{eq: unitality of CFusion}  $L^2(f)$ can be expressed as
   \begin{equation}
     \label{eq:L2(f)-factorization}
     L^2(A) \cong L^2(A)_{f^{-1}} \boxtimes_B L^2(B) \cong 
        L^2(A) \boxtimes_A {}_f L^2(B) \cong L^2(B). 
   \end{equation}
   
\section{Cyclic fusion}
    \label{subsec:cyclic-fusion}
%\addtocontents{toc}{\SkipTocEntry}
  
  Let $n\ge 2$ be some number.
  For $i\in\{1,\ldots,n\}$, let $A_i$ be a von Neumann algebras, 
  and let $H_i$ be 
  $A_i^{\phantom{\op}}\!\!\!\bar\otimes\, A_{i+1}^\op$-modules (cyclic numbering).
  Then for each $i,j\in\{1,\ldots,n\}$, we can form the fusion of
  $H_i\boxtimes_{A_{i+1}}\ldots\boxtimes_{A_{j-1}}H_{j-1}$
  (cyclic numbering) with
  $H_j\boxtimes_{A_{j+1}}\ldots\boxtimes_{A_{i-1}}H_{i-1}$
  over the algebra $A_i^\op\,\bar\otimes\, A_j^{\phantom{\op}}\!\!\!$.
  The Hilbert space
  \[
   \big(H_i\boxtimes_{A_{i+1}}\ldots\boxtimes_{A_{j-1}}H_{j-1}\big)
   \underset{A_i^\op\,\bar\otimes\, A_j^{\phantom{\op}}\!\!\!}\boxtimes
   \big(H_j\boxtimes_{A_{j+1}}\ldots\boxtimes_{A_{i-1}}H_{i-1}\big)
  \]
  is independent, up to canonical unitary isomorphism, 
  of the choices of $i$ and $j$~\cite[Appendix A]{BDH(modularity)}.
%  \ABcomm{Add precise reference.}
  We call the above Hilbert space the cyclic fusion of the $H_i$'s, 
  and denote it by
  \begin{equation*}
   \label{eq:cyc-fus}
    \tikzmath{
        \node (a) at (0,0) 
         {$H_1\,\boxtimes_{A_2}\cdots\,\boxtimes_{A_n}
            \!H_n\,\,\boxtimes_{A_1}$};
        \def\dd{.4}
        \def\ll{.35}
        \def\rr{.25}
        \draw[dashed, rounded corners = 5] 
          (a.east) -- ++(\rr,0) -- ++(0,-\dd) -- 
              ($(a.west) + (-\ll,-\dd)$) -- +(0,\dd) -- (a.west);
    } \quad .
  \end{equation*}

   %\begin{lemma}
   %  [{\cite[Lemma~2.6]{Haagerup(1975standard-form)}}]
   %   \label{lem: L^2(pAp) = pL^2(A)p}
   %  Given any projection $p\in A$, there is a canonical unitary 
   %  isomorphism $L^2(pAp)\cong p(L^2A)p$
   %  sending $\sqrt\phi\in L^2(pAp)$ to $\sqrt{\phi\circ E}$, 
   %  where $E(a)=pap$.
   %\end{lemma}

\section{Fusion and fiber product of von Neumann algebras}
    \label{subsec:fusion+fiber-prod}
%\addtocontents{toc}{\SkipTocEntry}
  
  \begin{definition}\label{def: Fusion of vN alg}
    Let $A\leftarrow C^\op$, $C\to B$ be two homomorphisms between von 
    Neumann algebras, and let ${}_AH$ and ${}_BK$ be faithful modules.
    Viewing $H$ as a right $C$-module, we may form the Connes 
    fusion $H\boxtimes_C K$.
    One then defines the fusion of $A$ and $B$ over $C$ as 
   \begin{equation}\label{eq:def fusion of algebras}
      A \circledast_C B := (A\cap {C^\op\hspace{.2mm}}'\hspace{.2mm})
           \vee(C'\cap B)\,\subset\, \bfB(H\boxtimes_C K),
   \end{equation}
   where the commutants of $C^\op$ and $C$ are taken in $H$ and $K$, 
   respectively. %, and algebra generated by $A^\op\cap C'$ and $B\cap C'$ 
   %is computed in $\bfB(H\boxtimes_C K)$.
  \end{definition}

  The fusion is independent, up to canonical isomorphism, of the choice of modules $H$ and 
  $K$~\cite[\propAoastCBindependentofHandK]{BDH(nets)}.
  If those modules are not faithful, 
  then there is still an action, albeit non-faithful, 
  of $A \circledast_C B$ on 
  $H\boxtimes_C K$~\cite[\lemextendstoACB]{BDH(nets)}.
  %The operation $\circledast$ is compatible with spatial tensor 
  %product in the sense that given algebras 
  %$A_1$, $B_1$, $C_1$, $A_2$, $B_2$, $C_2$ and homomorphisms 
  %$A_1\leftarrow C_1^\op$, $C_1\to B_1$, $A_2\leftarrow C_2^\op$, 
  %$C_2\to B_2$,
  %there is a canonical isomorphism
  %\begin{equation}\label{eq: compatibility between circledast and bartimes}
  % (A_1\,\bar\otimes\,A_2) \circledast_{C_1\,\bar\otimes\,C_2} 
  %        (B_1\,\bar\otimes\,B_2) \cong
  % (A_1 \circledast_{C_1} B_1)\,\bar\otimes\,(A_2 \circledast_{C_2} B_2).
  %\end{equation}
  Note that the operation $\circledast$ is not associative~\cite[\warnpoorformalproperties]{BDH(nets)}. 
  
  The fiber product of von Neumann algebras (introduced 
  in~\cite{Sauvageot(1985Produits-tensoriels)} when $C$ is abelian, and in~\cite[Def.~10.2.4]{Timmermann(Invitation-2-quantum-groups)} in general)
  is better behaved:
  
  \begin{definition} 
    \label{def:fiber-product}
    In the situation of Definition~\ref{def: Fusion of vN alg},
    the fiber product of $A$ and $B$ over $C$ is given by
    \[
        A \ast_C B := (A'\otimes_\alg B')',
    \]
    where the commutants $A'$ and $B'$ are taken in $\bfB(H)$ and 
    $\bfB(K)$ respectively, while the last one is taken in
    $\bfB(H\boxtimes_C K)$. 
  \end{definition}

  The fiber product is independent of the choice of modules $H$ and $K$ and there is an associator $A \ast_C (B \ast_D E) \to  (A \ast_C B) \ast_D E$
  that satisfies the pentagon identity~\cite[Prop.~10.2.8]{Timmermann(Invitation-2-quantum-groups)}.

  If $C = \IC$, then $A \ast_\IC B = A \circledast_\IC B$
  is the spatial tensor product $A \,\bar \ox\, B$ of von Neumann algebras.
  
\section{Compatibility with tensor products}
      \label{subsec:compatibility-ox}
%\addtocontents{toc}{\SkipTocEntry}
  There is a canonical isomorphism~\cite{Miura-Tomiyama(1984),
                                     Schmitt-Wittstock(1982)} 
  \begin{equation*}
        L^2(A) \otimes L^2(B)  \cong L^2(A \,\bar \ox\, B)
  \end{equation*}
  that sends $\sqrt{\phi} \,\otimes \sqrt{\psi}$ to 
  $\sqrt{\phi \otimes \psi}$. 
  This isomorphism provides a natural compatibility between Connes fusion and
  tensor products,
  \begin{equation*}
    (H_1 \boxtimes_A H_2) \otimes (K_1 \boxtimes_B K_2) 
      \cong (H_1 \otimes K_1) \boxtimes_{A \bar \ox B} (H_2 \otimes K_2).
  \end{equation*}
  This isomorphism can then be used to construct
  natural compatibility isomorphisms between the 
  spatial tensor product and the fusion, respectively the fiber product,
  of von Neumann algebras:
  \begin{eqnarray*}
    (A_1 \circledast_{C_1} B_1) \,\bar \ox\, (A_2 \circledast_{C_2} B_2) 
     & \cong & (A_1 \,\bar \ox\, A_2) \circledast_{C_1 \bar \ox C_2} 
           (B_1 \,\bar \ox\, B_2), \\ 
    (A_1 \ast_{C_1} B_1) \,\bar \ox\, (A_2 \ast_{C_2} B_2) 
     & \cong & (A_1 \,\bar \ox\, A_2) \ast_{C_1 \bar \ox C_2} (B_1 \,\bar \ox\, B_2).
  \end{eqnarray*}
  Here, those isomorphisms also rely on the equation
  $(A \,\bar \ox\, B)' = A' \,\bar \ox\, B'$ (Tomita's commutator theorem \cite[Thm.~12.3]{Takesaki:TomitasTheory}).
  
\section{Dualizability} \label{subsec:dualizability}
%\addtocontents{toc}{\SkipTocEntry}
%  (See~\cite[\S 4]{BDH(Dualizability+Index-of-subfactors)} for further details.)
  A von Neumann algebra whose center is $\IC$ is called a \emph{factor}.
  Von Neumann algebras with finite-dimensional center are finite direct sums of factors.
  \begin{definition}\label{def:dual}
    Let $A$ and $B$ be von Neumann algebras with finite-dimensional center. Given an $A$--$B$-bimodule $H$, we say that a 
    $B$--$A$-bimodule $\bar H$ is dual to $H$ if it comes equipped with maps
    \begin{equation}\label{eq:duality maps}
       R \colon {}_AL^2(A)_A \rightarrow {}_AH\boxtimes_B \bar H_A\qquad\quad
       S \colon {}_BL^2(B)_B \rightarrow  {}_B\bar H\boxtimes_A H_B
    \end{equation}
    subject to the duality equations $(R^*\otimes 1)(1\otimes S)=1$, 
    $(S^*\otimes 1)(1\otimes R)=1$, and to
    the normalization ${R^*(x\otimes 1)R} = {S^*(1\otimes x)S}$ for all 
    $x\in \mathrm{End}({}_AH_B)$.
    A bimodule whose dual module exists is called \emph{dualizable}.
  \end{definition}

  If ${}_AH_B$ is a dualizable bimodule, 
  then its dual bimodule is well defined up to canonical unitary 
  isomorphism~\cite[Thm. 4.22]{BDH(Dualizability+Index-of-subfactors)}.
  Moreover, the dual bimodule is canonically isomorphic to the complex 
  conjugate Hilbert space $\overline H$, with the actions 
  $b\bar \xi a:= \overline{a^* \xi b^*}$~\cite[Cor. 6.12]
         {BDH(Dualizability+Index-of-subfactors)}.

  A homomorphism $f \colon A \to B$ between von Neumann algebras with 
  finite-dim\-ensional center is said to be \emph{finite} if the associated 
  bimodule ${}_A L^2(B)_B$ is dualizable.
  If $f \colon A \to B$ is a finite homomorphism, then there is an induced map 
  $L^2(f) \colon L^2(A) \to L^2(B)$, and we have 
  $L^2(f \circ g) = L^2(f) \circ L^2(g)$.
  In other words,  Haagerup's $L^2$-space is functorial with  respect to finite
  homomorphisms~\cite{BDH(Dualizability+Index-of-subfactors)}.
  The map $L^2(f)$ is bounded and $A$-$A$-bilinear,
  but usually not isometric.  

  %\Begin{lemma}[{\cite[Lemma 4.6]{BDH(Dualizability+Index-of-subfactors)}}]
  % \label{lem: characterization of duals}
  % Let ${}_AH_B$ and ${}_BK_A$ be dualizable irreducible bimodules. Then
  % \[
  %   \mathrm{Hom}_{A,A}\big(H\boxtimes_BK,L^2(A)\big)=\begin{cases}
  %      \IC&\text{if ${}_BK_A\cong{}_B\bar H_A$}\\
  %      0&\text{otherwise}\\
  %   \end{cases}
  % \]
  %\end{lemma}

  % \begin{lemma}[{\cite[Lemma 4.10]{BDH(Dualizability+Index-of-subfactors)}}]
  %     \label{lem: endolgebra dim<oo}
  %    If ${}_AH_B$ is a dualizable bimodule, then its algebra of 
  %    $A$-$B$-bilinear endomorphisms is finite-dimensional.
  % \end{lemma}

\section{Statistical dimension and minimal index}
     \label{subsec:stat-dim+minimal-index}
%\addtocontents{toc}{\SkipTocEntry}
%    (See~\cite[\S 5]{BDH(Dualizability+Index-of-subfactors)} for further details.)
    The statistical dimension of a dualizable bimodule ${}_AH_B$ between factors is given by
    \begin{equation*}
      \qquad\dim ({}_A H_B) := R^*R = S^* S\,\in\,\IR_{\ge 0}
    \end{equation*}
    where $R$ and $S$ are as in \eqref{eq:duality maps}.
    For non-dualizable bimodules, one declares $\dim({}_AH_B)$ to be $\infty$.
    If $A = \oplus A_i$ and $B = \oplus B_j$ are finite direct sums
    of factors, then we can 
    decompose $H = \oplus H_{ij}$ as a direct sum of $A_i$--$B_j$-bimodules
    and define the matrix-valued statistical dimension
    \[
      \dim({}_AH_B)_{ij} := \dim({}_{A_i}{H_{ij}}\,{}_{B_j}).
    \]
    This matrix-valued dimension is additive with respect to addition of modules
    and multiplicative with respect to Connes fusion~\cite[\S 5]{BDH(Dualizability+Index-of-subfactors)}:
    \begin{gather}
      \label{eq:dim-and-sum}
      \dim ({}_AH_B \oplus {}_AK_B ) = \dim ({}_AH_B) + \dim ( {}_AK_B )
\\
      \label{eq:dim-and-Connes-fusion}
      \dim ({}_AH_B \boxtimes_B {}_BK_C ) 
             = \dim ({}_AH_B) \cdot \dim ( {}_BK_C ).
    \end{gather} 
   Given a finite homomorphism $f:A\to B$ between von Neumann 
    algebras with 
    finite-dimensional center, we let 
    \begin{equation*}
       \llbracket B : A\rrbracket:=\dim({}_AL^2B_B)
    \end{equation*} 
    denote the matrix of statistical dimensions of ${}_AL^2B_B$.
    If ${}_AH_B$ is a bimodule where $B$ acts faithfully, then
    by~\cite[\leminddim]{BDH(Dualizability+Index-of-subfactors)}
    \begin{equation}
      \label{eq:dim-on-H-is-dim-on-L2}
      \dim({}_AH_B) = \llbracket B':A\rrbracket
    \end{equation}
    where $B'$ is the commutant of $B$ on $H$.
    If $A$ also acts faithfully, then
    \begin{equation}
      \label{eq:matrix-of-stat-dim-commutants}
      \llbracket B : A\rrbracket = \llbracket A' : B' \rrbracket^T,
    \end{equation}
    where $^T$ denotes the transposed 
    matrix~\cite[\corindexofcommutants]{BDH(Dualizability+Index-of-subfactors)}.
    The minimal index $[B:A]$ of an inclusion of factors 
    ${\iota \colon A \to B}$ is the square of the statistical 
    dimension of ${}_AL^2B_B$ \cite[Def.~5.10]{BDH(Dualizability+Index-of-subfactors)}.
    For inclusions $A \subseteq B \subseteq C$ of  von Neumann algebras
    with finite-dimensional center we have
    \begin{equation}
        \label{eq:matrix-of-stat-dim-for-AcBcC}
        \llbracket C : A\rrbracket = \llbracket B : A\rrbracket
            \cdot \llbracket C : B \rrbracket,
    \end{equation}
    by~\cite[\eqpropertiesofmatrixofstatdim]
             {BDH(Dualizability+Index-of-subfactors)}.
    Moreover, the inclusion map $A\to B$ is an isomorphism if and only if
    $\llbracket B : A\rrbracket$ is a permutation matrix~\cite[\S 5]{BDH(Dualizability+Index-of-subfactors)}.    
    If $C$ is a factor, then
    \begin{equation}
      \label{eq:matrix-of-stat-dim-ox-C}
      \llbracket B \,\bar \ox\, C : A \,\bar\ox\, C \rrbracket
          = \llbracket B : A\rrbracket. 
    \end{equation}
    
    We recall two further results~\cite[\corbeforefinal, \corfinal]{BDH(Dualizability+Index-of-subfactors)} that are crucial for the proof of Theorem~\ref{thm:Haag-duality-composition-defects}.
Let $\big\| . \big\|_2$ stand for the $l^2$-norm of a vector.
    Let $N\subset M\subset  \bfB(H)$ be factors such that the inclusion $N\subset M$ has finite index.
    If $M \subset A \subset \bfB(H)$ is such that 
    one of the two relative commutants $N'\cap A$ or $M'\cap A$ 
    is a factor and the other has finite-dimensional center,
    then
    \begin{equation}
      \label{eq:DIS-726}
        \big\| \llbracket N'\cap A : M'\cap A\rrbracket 
             \big\|_2\, \,\,\le\,\, \llbracket M:N\rrbracket.
    \end{equation}
    Similarly, if $N \subset M \subset \bfB(H)$ are factors with $N \subset M$ of finite index, and $A\subset M'\subset \bfB(H)$ is
    such that one of the two algebras $N\vee A$ or $M\vee A$ is a factor
    and the other has finite-dimensional center, then
    \begin{equation}
      \label{eq:DIS-727}
       \big\| \llbracket M\vee A:N\vee A \rrbracket \big\|_2\, \,\,\le\,\, 
             \llbracket M:N\rrbracket.
    \end{equation}

  %As a corollary of Lemma \ref{lem: endolgebra dim<oo}, we have:

  %\begin{lemma}[{\cite[Lemma 5.15]{BDH(Dualizability+Index-of-subfactors)}}]
  %  \label{lem: rel comm is finite dim}
  %  Let $\iota:A\to B$ be a finite homomorphism between factors.
  %  Then the relative commutant of $\iota(A)$ in $B$ is finite-dimensional.
  %\end{lemma}

\section{Functors between module categories}
\label{subsec: Functors between module categories}
%\addtocontents{toc}{\SkipTocEntry}

The first lemma below is well known 
(\cite[Rem.~2.1.3.~(iii)]{Jones-Sunder(Intro-to-subfactors)}).
It is the main distinguishing feature of the representation 
theory of von Neumann algebras.
Here, $\ell^2$ stands for $\ell^2(\mathbb N)$ (or possibly $\ell^2(X)$ for a set $X$ of sufficiently large cardinality if the Hilbert spaces we deal with are not separable).

\begin{lemma}
  \label{lem:modules-are-summands}
  Let $A$ be a von Neumann algebra and let $H$ and $K$ be 
  faithful left $A$-modules.
  Then $H\otimes \ell^2\cong K\otimes \ell^2$.
  In particular, any $A$-module is isomorphic to a direct summand of 
  $H\otimes \ell^2$.
\end{lemma}

%If the Hilbert spaces $H$ and $K$ are separable, 
%then $\ell^2$ can be taken to mean $\ell^2(\mathbb N)$.
%Otherwise, the lemma is true for $\ell^2=\ell^2(X)$, with $X$ a set of sufficiently large cardinality.

  Let $A$ and $B$ be  von Neumann algebras. 
  We call a functor 
  $F:\modules{A} \,\to\,\, \modules{B}$
  \emph{normal} if it is continuous with respect to the 
  ultra-weak topology on hom-spaces,
  preserves adjoints $F(f^*) = F(f)^*$, and is additive in the following 
  sense: for $A$-modules $M_i$ the  map
  $\oplus F(\iota_i) \colon \oplus F(M_i) \to F( \oplus M_i)$
  induced by the inclusions $\iota_i \colon M_i \to \oplus_k M_k$
  is a unitary isomorphism.
Such functors are uniquely determined by their value on a single faithful $A$-module:

\begin{lemma}\label{lem: Functors between module categories}
Let $A$ and $B$ be von Neumann algebras.
Let $M$ be a faithful $A$-module, let $N$ be an arbitrary $B$-module, and
let
\[
F_1:\mathrm{End}_A(M)\to \mathrm{End}_B(N)
\]
be a morphism of von Neumann algebras.
Then the assignment $F(M):=N$, $F(f):=F_1(f)$
extends uniquely (up to unique unitary isomorphism) to a normal functor $F$ 
from the category of $A$-modules to the category of $B$-modules.
\end{lemma}

\begin{proof}
We prove existence and leave uniqueness to the reader.
Given an $A$-module $H$, by Lemma \ref{lem:modules-are-summands} 
we may pick an isomorphism
\begin{equation}\label{shds}
H\,\cong\,\mathrm{im}\big(p:M\otimes \ell^2\to M\otimes \ell^2\big)
\end{equation}
of $H$ with the image of a projection $p\in \mathrm{End}_A(M)\,\bar\otimes\,\bfB(\ell^2)$.
We can then define 
\begin{equation*}
F(H)\,:=\,\mathrm{im}\big((F_1\otimes \mathrm{Id}_{\ell^2})(p):N\otimes \ell^2\to N\otimes \ell^2\big),
\end{equation*}

For morphisms, if $H\cong\mathrm{im}(p)$ and $K\cong\mathrm{im}(q)$ are $A$-modules given as above, then the image under $F$ of an $A$-linear map $r:H\to K$
is the unique map $F(r):F(H)\to F(K)$ for which the composite
\[
N\otimes \ell^2\twoheadrightarrow F(H)\xrightarrow{F(r)} F(K)\hookrightarrow N\otimes \ell^2
\]
is the image under $F_1\otimes \mathrm{Id}_{\ell^2}$ of the map
$M\otimes \ell^2\twoheadrightarrow H\xrightarrow{\,r\,} K\hookrightarrow M\otimes \ell^2.$
\end{proof}

A similar result holds for natural transformations.

\begin{lemma}\label{lem: NT between module categories}
Let
\(
F,G:\modules{A} \to \modules{B}
\)
be two normal functors and let $M$ be a faithful $A$-module.
Then, in order to uniquely define a natural transformation $a:F\to G$, it is enough to specify its
value on $M$, and to check that for each $r\in\mathrm{End}_A(M)$, the diagram
\[
%\xymatrix{
%F(M)\ar[r]^{F(r)}\ar[d]_{a_M}&F(M)\ar[d]^{a_M}\\
%G(M)\ar[r]_{G(r)}&G(M)&
%}
\tikzmath{
\matrix [matrix of math nodes,column sep=1cm,row sep=5mm]
{ 
|(a)| F(M) \pgfmatrixnextcell |(b)| F(M)\\ 
|(c)| G(M) \pgfmatrixnextcell |(d)| G(M)\\ 
}; 
\draw[->] (a) -- node [above]	{$\scriptstyle F(r)$} (b);
\draw[->] (c) -- node [above]	{$\scriptstyle G(r)$} (d);
\draw[->] (a) -- node [left]		{$\scriptstyle a_{M}$} (c);
\draw[->] (b) -- node [right]	{$\scriptstyle a_{M}$} (d);
}
\]
commutes.
\end{lemma}

\begin{proof}
Given an $A$-module $H$ along with an isomorphism \eqref{shds},
one uses the natural inclusion $F(H)\subset F(M\otimes \ell^2)\cong F(M)\otimes \ell^2$ to
define
\[
a_H := (a_M\otimes \mathrm{Id}_{\ell^2})|_{F(H)}.
\]
This prescription is independent of the choice of isomorphism.
\end{proof}

\section{The split property}
  \label{subsec:split-property}
%\addtocontents{toc}{\SkipTocEntry}

\begin{definition} \label{def:split}
Given two commuting von Neumann algebras $A$ and $B$ acting on a Hilbert space $H$, we say that
$A$ and $B$ are {\it split on $H$} if the natural map $A \otimes_\alg B\to \bfB(H)$ extends to a homomorphism $A \,\bar\otimes\, B\to \bfB(H)$. %!!not always an isomorphism $A \,\bar\otimes\, B\to A\vee B!!$
We also say that an inclusion $A_0\hookrightarrow A$ is split if there exists a (equivalently, for any) faithful $A$-module $H$
such that $A_0$ and $A'$ are split on $H$.
\end{definition}
 
\begin{lemma}
\label{lem:commutant-spacial-vee 1}
Let $A_0 \subseteq A$ be von Neumann algebras acting faithfully on a Hilbert space $H$, and let $B \subseteq A'$ be an algebra that commutes with $A$.  If the inclusion $A_0\hookrightarrow A$ is split, then we have
	\begin{equation}\label{eq: lemma appendix A2}
	B \vee ( A \cap A_0') = (B \vee A) \cap A_0'.
	\end{equation}
\end{lemma}

\begin{proof}
Consider $H \otimes L^2(A_0)$ as an $A' \,\barox\, A_0$-module, where $A'$ acts on the first factor and $A_0$ acts on the second factor.
Both $H$ and $H \otimes L^2(A_0)$ are faithful $A' \,\barox\, A_0$-modules.
 
So we may pick an $A' \,\barox\, A_0$-module isomorphism between $H\otimes \ell^2$ and $H \otimes L^2(A_0)\otimes \ell^2$.
Let $K:=L^2(A_0)\otimes \ell^2$, so we have $H\otimes \ell^2\cong H\otimes K$.
Under this identification, the subalgebra
\[
\big(B \vee ( A \cap A_0')\big)\,\bar\otimes\, \bfB(\ell^2) = (B\,\bar\otimes\, 1) \vee \big( (A\,\bar\otimes\, \bfB(\ell^2)) \cap (A_0\,\bar\otimes\, 1)'\big)
\]
of $\bfB(H\otimes \ell^2)$ corresponds to
\[
(B\,\bar\otimes\, 1) \vee \big( (A\,\bar\otimes\, \bfB(K)) \cap (1\,\bar\otimes\, A_0)'\big) = (B\,\bar\otimes\, 1) \vee (A\,\bar\otimes\, A_0') = (B\vee A) \,\bar\otimes\, A_0'
\]
in $\bfB(H\otimes K)$. Similarly,
\[
\big((B \vee A) \cap A_0'\big)\,\bar\otimes\, \bfB(\ell^2) = \big((B\,\bar\otimes\, 1) \vee (A\,\bar\otimes\, \bfB(\ell^2))\big) \cap (A_0\,\bar\otimes\, 1)'
\]
corresponds to
\[
\big((B\,\bar\otimes\, 1) \vee (A\bar\otimes \bfB(K))\big) \cap (1\,\bar\otimes\, A_0)' = \big((B \vee A) \bar\otimes \bfB(K)\big) \cap \big(\bfB(H)\bar\otimes A_0'\big)
= (B\vee A) \,\bar\otimes\, A_0'.
\]
The algebras \eqref{eq: lemma appendix A2} agree after tensoring with $\bfB(\ell^2)$, so they are equal.
\end{proof}

\begin{lemma}
\label{lem:commutant-spacial-vee 2}
Let $A\subset \bfB(H)$ be a factor, and let $A_0$ be a subalgebra of $A$.
If the inclusion $A_0\hookrightarrow A$ is split, then
	\begin{equation}\label{eq: lemma appendix A2 (bis)}
	( A_0 \vee A')\cap A = A_0.
	\end{equation}
\end{lemma}

\begin{proof}
As in the previous lemma, we pick an isomorphism $H\otimes \ell^2\cong H\otimes K$ of  
$A' \,\barox\, A_0$-modules, where $K=L^2(A_0)\otimes \ell^2$.
  Under that isomorphism, the algebras $A_0\,\bar\otimes\, 1$ and
  \[
  \big(( A_0 \vee A')\cap A\big)\,\bar\otimes\, 1 
  = \big((A_0\,\bar\otimes\, 1) \vee (A'\,\bar\otimes\, 1)\big)\cap (A\,\bar\otimes\, \bfB(\ell^2))
  \]
  correspond to $1\,\bar\otimes\, A_0$ and
  \[
  \big((1\,\bar\otimes\, A_0) \vee (A'\,\bar\otimes\, 1)\big)\cap (A\,\bar\otimes\, \bfB(K))
  =  (A'\,\bar\otimes\, A_0) \cap (A\,\bar\otimes\, \bfB(K)) = 1\,\bar\otimes\, A_0.
  \]
  Since their images in $\bfB(H\otimes K)$ agree, the two algebras \eqref{eq: lemma appendix A2 (bis)} are equal.
\end{proof}

Recall the fiber product operation $\ast$ from Definition~\ref{def:fiber-product}.

\begin{lemma}\label{lem:[A * Bhat : A * B] = [Bhat : B]}
Let $A$, $B$, and $C$ be factors, and let $A^\op\leftarrow C \to B$ be homomorphisms.
Let $B\subset \hat B$ be a subfactor of finite index.
Assuming that the inclusion $C\to A^\op$ is split, then $A\ast_C \hat B$ and $A\ast_C B$ are factors, and we have
\begin{equation}\label{eq:[A * Bhat : A * B] = [Bhat : B]}
[A\ast_C \hat B:A\ast_C B] = [\hat B : B].
\end{equation}
\end{lemma}
\begin{proof}
Let $H$ be a faithful $A$-module and $K$ a faithful $\hat B$-module.
Let $A'$ be the commutant of $A$ on $H$, and let $B'$ and $\hat B'$ be the commutants of $B$ and $\hat B$ on $K$.
Finally, let $C'$ be the commutant of $C$ on $K$, and let $ {}'\hspace{-.3mm}C$ be the commutant of $C^\op$ on $H$.

Since the inclusion of $C$ into $A^\op$ is split, so is the inclusion $A'\hookrightarrow  {}'\hspace{-.3mm}C$.
The algebra $C'$ is $ {}'\hspace{-.3mm}C$'s commutant on $H\boxtimes_C K$, and so
$A'$ and $C'$ are split on $H\boxtimes_C K$.
Finally, $B'$ and $\hat B'$ being subalgebras of $C'$, we conclude that $A'$ and $B'$, and also $A'$ and $\hat B'$, are split on $H\boxtimes_C K$.
It follows that the algebras
\[
\begin{split}
(A\ast_C B)' = A' \vee B' = A' \,\bar\otimes\, B'\quad\text{and}\quad
(A\ast_C \hat B)' = A' \vee \hat B' = A' \,\bar\otimes\, \hat B'\\
\end{split}
\]
are factors, and thus so are $A\ast_C B$ and $A\ast_C \hat B$.
Finally, we have
\[
[A\ast_C \hat B:A\ast_C B] = [(A\ast_C B)' : (A\ast_C \hat B)'] =[A' \,\bar\otimes\, B' : A' \,\bar\otimes\, \hat B'] = [B' : \hat B'] = [\hat B : B]
\]
by~\eqref{eq:matrix-of-stat-dim-commutants} and~\eqref{eq:matrix-of-stat-dim-ox-C}.
\end{proof}

\begin{lemma} \label{lem:add-to-fusion-of-algebras}
  Let $A_0$ and $A_1$ be commuting subalgebras of $\bfB(H)$, and let $B_0$ and $B_1$ be commuting subalgebras of $\bfB(K)$.
  Let $C^{op} \to A_0$ and $C \to B_0$ be injective homomorphisms.
  If $C^{op}$ and $A_0'$ are split on $H$, then we have
  \begin{equation}\label{eq: (A_1 vee A_0) circledast_C (B_0 vee B_1)}
  A_1 \vee (A_0 \ast_C B_0) \vee B_1 =  (A_1 \vee A_0) \ast_C (B_0 \vee B_1) 
  \end{equation}
  on $H \boxtimes_C K$.
\end{lemma}

\begin{proof}
As in the proof of the previous lemma, the algebras $A_0'$ and $C'$ are split on $H \boxtimes_C K$. 
%  Since $A_0'$ and $C$ are split on $H$, the actions of $A_0'$ and $C'$ on $H \boxtimes_C K$ induce an action of $A_0' \,\barox\, C'$.
  In particular, the actions of $A_0'$ on $H$ and of $B_0'$ on $K$ induce an action of $A_0' \,\barox\, B_0'$ on $H \boxtimes_C K$.

  Consider $H \otimes K$ as an $A_0' \,\barox\, B_0'$-module, where $A_0'$ acts on $H$  and $B_0'$ acts on $K$.
  Since this is a faithful module, we can find an $A_0'\,\bar\otimes\, B_0'$-linear isometry
  \[
  H \boxtimes_C K \hookrightarrow H \otimes K \otimes  \ell^2.
  \]
  Let $p \in \bfB(H \otimes K \otimes \ell^2)$ be its range projection.
  Under the induced isomorphism
  \[
  \alpha:\bfB(H \boxtimes_C K) \,\xrightarrow{\scriptscriptstyle \cong}\, p \Big( \bfB(H) \,\barox\, \bfB(K) \,\barox\, \bfB(\ell^2) \Big) p,
  \]
we have
\[
  \begin{split}
  \alpha(A_0') = \big(\,   A_0' \,\barox\, \IC \,\barox\, \IC   \,\big)  p,&  \qquad\quad 
  \alpha(B_0') = \big(\,   \IC \,\barox\, B_0' \,\barox\, \IC   \,\big)  p,  \\
  \alpha(A_1)  = \big(\,   A_1 \,\barox\, \IC \,\barox\, \IC   \,\big)  p,&  \qquad\quad
  \alpha(B_1)  = \big(\,  \IC \,\barox\, B_1 \,\barox\, \IC    \,\big) p.
\end{split}
\]
Recalling the definition $A_0 \ast_C B_0:=(A_0'\vee B_0')'$, we then see that\smallskip
\[
\hspace{.32cm}\alpha(A_0 \ast_C B_0) = p \big(\, A_0 \,\barox\, B_0 \,\barox\, \bfB(\ell^2)  \,\big)  p.
\]
Similarly, 
\(
  \alpha ((A_1 \vee A_0) \ast_C (B_0 \vee B_1)) = p \big((A_1 \vee A_0) \,\barox\, (B_0 \vee B_1) \,\barox\, \bfB(\ell^2)\big)p,
\)
and equation \eqref{eq: (A_1 vee A_0) circledast_C (B_0 vee B_1)} follows since 
\[
  (A_1 \barox \IC \barox \IC ) \vee 
  (A_0 \barox B_0  \barox \bfB(\ell^2)) \vee 
  (\IC \barox B_1 \barox \IC)         
   = (A_1 \vee A_0) \barox (B_0 \vee B_1) \barox \bfB(\ell^2).\qedhere
\]
\end{proof}

\section{Two-sided fusion on $L^2$-spaces}
\label{subsec:symmetric-fusion-on-standard-forms}
%\addtocontents{toc}{\SkipTocEntry}

Let $M$ be a von Neumann algebra, and let $M_0$ and $A$ be two commuting subalgebras such that $M_0 \vee A = M$.
Let $H_A$ be a faithful right $A$-module, and let $B$ be its commutant, acting on $H$ on the left.
Then $H$ is naturally a $B$--$A$-bimodule, and its conjugate $\overline H$ is an $A$--$B$-bimodule.
Consider the Hilbert space
\[
\widehat H \,:=\, H \boxtimes_A L^2(M) \boxtimes_A \overline H,
\]
which is a completion of $\hom(L^2A_A,H_A) \otimes_A L^2(M) \otimes_A \hom({}_AL^2A,{}_A\overline H)$.

Let us denote by $J_A$ and $J_M$ the modular conjugations on $L^2A$ and $L^2M$.
There is an antilinear involution $\widehat J: \widehat H \to \widehat H$ given by
\[
\widehat J\,\big(\varphi \otimes \xi \otimes \psi\big)\,=\, \bar\psi \otimes J_M(\xi) \otimes \bar\varphi,
\]
where $\xi\in L^2M$ is a vector, and
for $\varphi\in \hom(L^2A_A,H_A)$ and $\psi\in\hom({}_AL^2A,{}_A\overline H)$, 
the maps
\[
\bar\varphi\in \hom({}_AL^2A,{}_A\bar H),\qquad \bar\psi\in\hom(L^2A_A,H_A)
\]
are given by $\bar\varphi=I\circ \varphi \circ J_A$ and $\bar\psi=I\circ \psi\circ J_A$, where $I$ is the identity map between $H$ and $\overline H$.
There are natural left and right actions of $B$ on $\widehat H$ coming from its actions on $H$ and $\overline H$.
Moreover, the left and right actions of $M$ on $L^2(M)$ induce actions of $M_0$ on $\widehat H$.
The left and right actions of $M_0$ and $B$ are interchanged (up to a star) by $\widehat J$,
and so the algebra $\,\widehat {\!M\!}\, := M_0 \vee B$ generated by them in their left action
on $\widehat H$ is isomorphic to the algebra generated by them in their right action on $\widehat H$.
From this discussion, we see that $\widehat H$ is an $\,\widehat {\!M\!}\,$--$\,\widehat {\!M\!}\,$-bimodule with
an involution $\widehat J$ that satisfies $\widehat J(a\xi b)=b^* \widehat J(\xi) a^*$.

\begin{proposition} \label{prop:hat-M}
In the above situation, there is a canonical positive cone $\widehat P$ in $\widehat H := H \boxtimes_A L^2M \boxtimes_A \overline H$ such
that $(\widehat H , \widehat J, \widehat P)$ is a standard form for $\,\widehat {\!M\!}\,=M_0\vee B$.
\end{proposition}

\noindent
In the following proof, as in Section \ref{subsec: Functors between module categories}, 
$\ell^2$ stands for $\ell^2(\mathbb N)$, or perhaps $\ell^2(X)$ for a set $X$ of sufficiently large cardinality.
If $H$ admits a cyclic vector for $A$ then we can replace $\ell^2$ by $\mathbb C$ everywhere, and the proof simplifies.

\begin{proof}  
Pick an %not surjective!
$A$-linear isometry $u \colon H_A \to \ell^2 \otimes L^2(A)_A$ (Lemma \ref{lem:modules-are-summands}) and let
\[
\bar u:=(1\otimes J_A)\circ u\circ I: {}_A\overline H\to  \overline{\ell^2}\otimes L^2(A)\cong {}_AL^2(A)\otimes \overline{\ell^2}.
\] 
The endomorphism algebra of $\ell^2 \otimes L^2(A)_A$ can be identified with $\bfB(\ell^2) \,\barox\, A$.
In particular, the range projection $p := uu^*$ is in $\bfB(\ell^2) \,\barox\, A$.

Let us define $M_1:= \bfB(\ell^2) \,\barox\, M$, with associated standard form $(L^2M_1,J_{M_1},P_{M_1})$
and let $q:= p\, J_{M_1} p\, J_{M_1}\in\bfB(L^2M_1)$ or, equivalently, $q(\xi):= p\,\xi\, p$.
Composing $u \boxtimes \id_{L^2(M)} \boxtimes\, \bar u$ with the obvious identifications
$
(\ell^2\otimes L^2A ) \boxtimes_A L^2M \boxtimes_A  (\, L^2A\otimes \overline{\displaystyle \ell^2}\, )
\,\cong\, \ell^2\otimes L^2M \otimes \overline{\displaystyle\ell^2} \,\cong\, L^2M_1,
$
we get an isometry
\[
\begin{split}       v \,:\, 
\widehat H = \,&H \boxtimes_A L^2M \boxtimes_A \overline H \rightarrow  L^2M_1\\
v\big(\varphi\,\otimes \,&\xi \otimes \psi\big)=(u\, \varphi)\cdot \xi\cdot( \bar u\, \psi)
\end{split}
\]
with range projection $vv^*=q$.
%\marginpar{.
%\\ $\varphi\in \hom(L^2A_A,H_A)$\smallskip\\  $\psi\in\hom({}_AL^2A,{}_A\bar H)$\smallskip\\
%$\bar\varphi\in \hom({}_AL^2A,{}_A\bar H)$\smallskip\\ $\bar\psi\in\hom(L^2A_A,H_A)$\smallskip\\
%$\mbox{$u\in\hom(H_A, \ell^2 \otimes L^2(A)_A)$}$\smallskip\\ $\mbox{$\bar u\in\hom({}_A\bar H, {}_AL^2(A)\otimes \overline{\ell^2})$}$.}
The resulting isomorphism $\widehat H\cong q(L^2M_1)$ intertwines $\widehat J$ and $q J_{M_1}$, as can be seen from the commutativity of the following diagram:
\[
\tikzmath{
\matrix [matrix of math nodes,column sep=1cm,row sep=5mm]
{ 
|(a)| \varphi\otimes \xi \otimes \psi				\pgfmatrixnextcell |(b)| \bar\psi\otimes J_M(\xi)\otimes \bar \varphi\\
										\pgfmatrixnextcell |(d)| (u\,\bar\psi)\cdot J_M(\xi) \cdot (\bar u\,\bar\varphi)\\
|(c)| (u\, \varphi)\cdot \xi\cdot( \bar u\, \psi)		\pgfmatrixnextcell |(e)| \,(\bar u\,\psi)^* \cdot J_M(\xi)\cdot (u\,\varphi)^*\\
}; 
\draw[->] (a) -- node [above]	{$\scriptstyle \widehat J$} (b);
\draw[->] (c) -- node [above]	{$\scriptstyle J_{M_1}$} (e);
\draw[->] (a) -- node [left]		{$\scriptstyle v$} (c);
\draw[->] (b) -- node [right]	{$\scriptstyle v$} (d);
\draw[double, double distance = 1.5] (d) -- (e);
\foreach \n/\loc/\x/\y in {a/east/0/.07, a/south/.07/0, b/south/.07/0, c/east/0/.07}
{\draw ($(\n.\loc) + (\x,\y)$) -- ($(\n.\loc) - (\x,\y)$);}
}
\]
Here, the last equality holds because the preimage of
$u\,\bar\psi$ under the left action map $\ell^2\otimes A\to \hom(L^2A_A,\ell^2 \otimes L^2A_A)$ agrees with the 
preimage of $(\bar u\,\psi)^*$ under the right action map $\ell^2\otimes A\to \hom({}_AL^2A\otimes\overline{\ell^2},  {}_AL^2A)$,
and the preimage of $\bar u\,\bar\varphi$ under the right action map agrees with the preimage of $(u\,\varphi)^*$ under the left action map.

  Recall that $B$ is the commutant of $A$ on $H$.
  In its action on $L^2M_1=\ell^2\otimes L^2M\otimes \overline{\ell^2}$, we have $B\equiv vBv^*=q(\bfB(\ell^2)\,\barox\,A)q$, and so it follows that
\begin{equation}\label{eq: Mhat is a corner}
  \begin{split}
  \,\widehat{\!M\!}\,\equiv v\,\,\widehat{\!M\!}\,\,v^*\,
  &= v(M_0\vee B)v^*\\
  &= qM_0\vee q(\bfB(\ell^2)\,\barox\,A)q\\
  &=q\big(\bfB(\ell^2)\,\barox\,(M_0\vee A)\big)q
  =q\big(\bfB(\ell^2)\,\barox\,M\big)q
  = q\,  M_1\, q.
  \end{split}
\end{equation}\pagebreak

\noindent  Now by~\cite[Lem.~2.6]{Haagerup(1975standard-form)}, we know that $\big(q(L^2M_1), q J_{M_1},q(P_{M_1})\big)$ is a standard form for $q M_1 q$.
  Therefore, by letting $\widehat P := v^{-1} (q (P_{M_1}))$, we have that $(\widehat H, \widehat J, \widehat P)$ is a standard form for $\,\widehat{\!M\!}\,$.
\end{proof}

\begin{lemma}\label{lem: factoriality of alg with dotted lines}
Let $M$, $M_0$, $A$, $\,\widehat {\!M\!}\,$ be as in the previous lemma.
%above. Let $H_A$ be a non-zero $A$-module, and let $\,\widehat {\!M\!}\,$ be the von Neumann algebra generated by the left actions of
%$M_0$ and $A'$ on $H\boxtimes_A L^2(M)$.
Then if $M$ is a factor, so is $\,\widehat {\!M\!}\,$.
\end{lemma}

\begin{proof}
We have seen in \eqref{eq: Mhat is a corner} that $\,\widehat {\!M\!}\,=q\,  M_1\, q=q(\bfB(\ell^2) \,\barox\, M)q$.
The result follows since corners of factors are factors.
\end{proof}

The isomorphism constructed by Proposition~\ref{prop:hat-M} satisfies the following version of associativity.
Let $M=M_{0}\vee A_1\vee A_2$ be a von Neumann algebra, where $M_{0}$, $A_1$, and $A_2$ are commuting subalgebras of $M$.
Let $H_i$ be faithful right $A_i$-modules, and let $B_i$ be their commutants.
Then we can form the Hilbert spaces
\[
\widehat H_1:= H_1\boxtimes_{A_1} L^2M \boxtimes_{A_1} \overline H_1
\quad\text{and}\quad
\widehat H_2:= H_2\boxtimes_{A_2} L^2M \boxtimes_{A_2} \overline H_2
\]
on which the algebras $\,\widehat {\!M\!}_1\,:=M_0\vee B_1\vee A_2$ and $\,\widehat {\!M\!}_2\,:=M_0\vee A_1\vee B_2$ act.
By Proposition~\ref{prop:hat-M}, we have canonical isomorphisms $\widehat H_1\cong L^2 \,\widehat {\!M\!}_1$ and $\widehat H_2\cong L^2 \,\widehat {\!M\!}_2$.

Furthermore, we can form the Hilbert spaces
\def\doublehatH {\widehat{\phantom{\big ||}}\hspace{-.285cm}{\widehat H}}
\def\doublehatM {\,\widehat{\phantom{\big ||}}\hspace{-.24cm}{\widehat {\!M\!}}\,}
\[
\doublehatH_1:= H_2\boxtimes_{A_2} L^2\,\widehat {\!M\!}_1 \boxtimes_{A_2} \overline H_2
\quad\text{and}\quad
\doublehatH_2:= H_1\boxtimes_{A_1} L^2\,\widehat {\!M\!}_2 \boxtimes_{A_1} \overline H_1,
\]
on which the algebra $\doublehatM:=M_{0}\vee B_1\vee B_2$ acts.
Again by Proposition~\ref{prop:hat-M}, we then have canonical isomorphisms\, $\doublehatH_1\cong L^2 \doublehatM \,\cong\,\, \doublehatH_2$.

\begin{proposition}\label{prop: associativity of Psi (Appendix)}
In the above situation, the following diagram is commutative:
\begin{equation}\label{eq: hH,H...}
\quad
\tikzmath{
\node (f) at (0,0) {$H_2\underset{A_2}\boxtimes L^2 \,\widehat {\!M\!}_1\underset{A_2}\boxtimes\overline H_2$};
\node (g) at (2,0.11) {$=\,\doublehatH_1\phantom{=}$};
\node (e) at (-3.3,0) {$H_2\underset{A_2}\boxtimes \widehat H_1\underset{A_2}\boxtimes\overline H_2$};
\node (d) at (2,1.36) {$L^2 \doublehatM$};
\node (c) at (-3.3,2.5) {$H_1\underset{A_1}\boxtimes \widehat H_2\underset{A_1}\boxtimes\overline H_1$};
\node (a) at (0,2.5) {$H_1\underset{A_1}\boxtimes L^2 \,\widehat {\!M\!}_2\underset{A_1}\boxtimes\overline H_1$};
\node (b) at (2,2.61) {$=\,\doublehatH_2\phantom{=}$};
\draw[->] ($(c.east)+(0,0.05)$) -- ($(a.west)+(0,0.05)$);
\draw[->] (b) -- (d);
\draw[->] (g) -- (d);
\draw[->] ($(e.east)+(0,0.05)$) -- ($(f.west)+(0,0.05)$);
\draw ($(c.south)+(0,0.1)$) --node[left] {$\scriptstyle\cong$} (e);
} % tikzmath
\end{equation}
\end{proposition}

\begin{proof}
Let $\ell_1$ and $\ell_2$ be two copies of $\ell^2$.
Pick isometries $u_i \colon (H_i)_{A_i} \hookrightarrow (\ell_i \otimes L^2A_i)_{A_i}$,
so as to identify $\widehat H_1$ with $L^2(p_1(\bfB(\ell_1)\,\barox\, M)p_1)$,
and $\widehat H_2$ with $L^2(p_2(M\,\barox\, \bfB(\ell_2))p_2)$,
for $p_i:=u_iu_i^*$.
Here, we have $p_1\in \bfB(\ell_1)\,\barox\, M$ and $p_2\in M\,\barox\, \bfB(\ell_2)$.
Let us also define the projections $q_1$ on $L^2(\bfB(\ell_1)\,\barox\, M)\cong\ell_1\otimes \overline{\ell_1}\otimes L^2M$ and $q_2$ on $L^2(M \,\barox\, \bfB(\ell_2))\cong L^2M\otimes \ell_2\otimes \overline{\ell_2}$ by $q_i(\xi)=p_i\xi p_i$.
\pagebreak

Given the above notations, consider the following diagram:

{
\def \bo#1 {\underset{A_#1}\boxtimes}
\def \n#1 #2 #3 {\node[inner sep = 1] (#1) at (#2) {$\scriptstyle #3$};}
\def \Bla {\bfB(\ell_1)}
\def \Blb {\bfB(\ell_2)}
\def \B {\big}
\def \b#1 {{\textstyle #1}}
\[
\tikzmath{
\n a {8.3,7.8} {L^2\b( (p_1\ox 1)(1\ox p_2)\b( \Bla\barox M\barox \Blb\b) (1\ox p_2)(p_1\ox 1)\b) }
\n b {8.7,6.6} {L^2\b( \Bla\barox\b( p_2(M\barox \Blb)p_2\b) \b) }
\n c {1,7.8} {H_1 \bo1 L^2\b( p_2(M\barox \Blb)p_2\b) \bo1 H_1}
\n d {4.5,7.1} {\ell_1\otimes \bar\ell_1\otimes L^2\b( p_2(M\barox \Blb)p_2\b) }
\n e {1,5.8} {H_1\bo1 \b( q_2(L^2M\otimes \ell_2\otimes\bar\ell_2)\b) \bo1 \bar H_1}
\n f {2.5,6.5} {H_1\bo1 L^2\b( M\barox \Blb\b) \bo1 \bar H_1}
\n g {6.4,5.8} {\ell_1\otimes \bar\ell_1\otimes L^2\b( M\barox \Blb\b) }
\n h {9,3.9} {L^2\b( \Bla\barox M\barox \Blb\b) }
\n i {1,4.3} {H_1\bo1 \big(H_2\bo2 L^2M\bo2 \bar H_2\big)\bo1 \bar H_1}
\n j {4.3,5.1} {H_1\bo1 \big(L^2M\otimes \ell_2\otimes\bar\ell_2\big)\bo1 \bar H_1}
\n k {6.5,4.3} {\ell_1\otimes \bar\ell_1\otimes \b( L^2M\otimes \ell_2\otimes\bar\ell_2\b) }
\n l {1,3.5} {H_2\bo2 \big(H_1\bo1 L^2M\bo1 \bar H_1\big)\bo2 \bar H_2}
\n m {4.3,2.7} {H_2\bo2 \big(\ell_1\otimes \bar\ell_1\otimes L^2M\big)\bo2 \bar H_2}
\n n {6.5,3.5} {\b( \ell_1\otimes \bar\ell_1\otimes L^2M\b) \otimes \ell_2\otimes\bar\ell_2}
\n o {1,2} {H_2\bo2 \b( q_1(\ell_1\otimes \bar\ell_1\otimes L^2M)\b) \bo2 \bar H_2}
\n p {2.5,1.3} {H_2\bo2 L^2\b( \Bla\barox M\b) \bo2 \bar H_2}
\n q {6.4,2} {L^2\b( \Bla\barox M\b) \otimes \ell_2\otimes \bar\ell_2}
\n r {1,0} {H_2 \bo2 L^2\b( p_1(\Bla\barox M)p_1\b) \bo2 H_2}
\n s {4.5,0.7} {L^2\b( p_1(\Bla\barox M)p_1\b) \otimes \ell_2\otimes\bar\ell_2}
%\n t {8.7,2.2} {L^2\b( \Bl\barox M\barox \Bl\b) }
\n u {8.7,1.2} {L^2\b( \b( p_1( \Bla\barox M)p_1\b) \barox\Blb\b) }
\n v {8.3,0} {L^2\b( (1\ox p_2)(p_1\ox 1)\b( \Bla\barox M\barox \Blb\b) (p_1\ox 1)(1\ox p_2)\b) }
\foreach \source/\target in {c/a, d/b, h/g, j/f, k/g, n/k, p/m, q/n,  h/q, u/s, v/r}
{\draw (\source) -- (\target);}
\foreach \source/\target in {a/b, c/d, c/f, f/g, e/j, d/g,  i/j, j/k,  l/m, m/n, o/m, p/q, r/p, s/q, r/s, v/u}
{\draw[->] (\source) -- (\target);}
\draw ($(v.north)+(2.7,0)$) --node[below, fill=white, pos=.54]{$\scriptstyle L^2 \,\widehat{\phantom{\textstyle t}}\hspace{-.16cm}{\widehat {\!M\!}}\,$} ($(a.south)+(2.7,0)$);
%\draw ($(t.north)+(.3,0)$) -- ($(h.south)+(.3,0)$);
\draw[->] ($(b.south)+(.05,0)$) -- ($(h.north)+(.15,0)$);
\draw[->] ($(u.north)+(.05,0)$) -- ($(h.south)+(.15,0)$);
\foreach \source/\target in {r/o, o/l, l/i, i/e, e/c}
{\draw ($(\source.north)-(.5,0)$) -- ($(\target.south)-(.5,0)$);}
%\foreach \p in {a,b,c,d,e,f,g,h,i,j,k,l,m,n,o,p,q,r,s,u,v}{\node[red, scale=.5, fill=white] at (\p) {\p};}
}
\]
\medskip

\noindent Here, arrows denote inclusions and lines denote isomorphisms.
One recognizes \eqref{eq: hH,H...} as the outside of the above diagram, and each one of the interior cells commutes for obvious reasons.
}
\end{proof}

\chapter{Conformal nets}
\label{app:nets}

\section{Axioms for conformal nets} \label{subsec:defnets}
%\addtocontents{toc}{\SkipTocEntry}

  Let $\VN$ be the category whose objects are von Neumann 
  algebras with separable preduals, 
  and whose morphisms are $\IC$-linear homomorphisms 
  and $\IC$-linear antihomomorphisms.
  A \emph{net} is a covariant functor 
  $\cala \colon \INT \to \VN$ taking 
  orientation-preserving embeddings to injective homomorphisms and
  orientation-reversing embeddings to injective antihomomorphisms.
  We call a net \emph{continuous} if for any intervals $I$ and $J$, 
  the  map 
  $\mathrm{Hom}_{\INT}(I,J)\to \mathrm{Hom}_{\VN}(\cala(I),\cala(J))$,
  $\varphi \mapsto \cala(\varphi)$ is continuous 
  for the $\mathcal C^\infty$ topology on $\mathrm{Hom}_{\INT}(I,J)$ and 
  Haagerup's $u$-topology on 
  $\mathrm{Hom}_{\VN}(\cala(I),\cala(J))$~\cite[Appendix]{BDH(nets)}.
  Given a subinterval $I \subseteq K$, we will 
  often not distinguish between $\cala(I)$ and its image in $\cala(K)$.
  
  A \emph{conformal net} is a continuous net $\cala$ subject to the 
  following conditions.
  Here, $I$ and $J$ are subintervals of an interval $K$: 
  \begin{enumerate}
  \item \emph{Locality:} If $I \subset K$ and $J\subset K$ have disjoint interiors, then 
     $\cala(I)$ and $\cala(J)$ are commuting subalgebras of $\cala(K)$.
  \item \emph{Strong additivity:} If $K = I \cup J$, then $\cala(K)$ 
     is generated as a von Neumann algebra by the two subalgebras: 
     $\cala(K) = \cala(I) \vee \cala(J)$.
  \item \emph{Split property:} 
    If $I\subset K$ and $J\subset K$ are disjoint, then the map from the algebraic tensor 
    product $\cala(I) \ox_{\alg} \cala(J) \to \cala(K)$ extends to a map 
    from the spatial tensor product
    $\cala(I) \, \bar{\ox} \, \cala(J) \to \cala(K)$.
  \item \emph{Inner covariance:} 
    If $\varphi\in\Diff_+(I)$ restricts to the identity in a neighborhood of 
    $\partial I$, then $\cala(\varphi)$ is an inner automorphism 
    of $\cala(I)$. 
    (A unitary $u \in \cala(I)$ with $\Ad(u) = \cala(\varphi)$
    is said to \emph{implement} $\varphi$.)
  \item %\label{def:conformal-net:vacuum} 
    \emph{Vacuum sector:} 
    Suppose that $J \subsetneq I$ contains the boundary point
    $p \in \dd I$, and let $\bar{J}$ denote $J$ with the reversed 
    orientation; $\cala(J)$ acts on $L^2(\cala(I))$ 
    via the left action of $\cala(I)$, and 
    $\cala(\bar{J}) \cong \cala(J)^\op$ acts on $L^2(\cala(I))$ 
    via the right action of $\cala(I)$.
    In this case, we require that the action of
    $\cala(J) \ox_{\alg} \cala( \bar{J} )$ on $L^2(\cala(I))$
    extends to an action of $\cala(J \cup_p \bar{J})$:
    \begin{equation*} %\label{eq: Vaccum sector axiom for nets}
      \qquad\tikzmath{
            \matrix [matrix of math nodes,column sep=1cm,row sep=5mm]
                { 
                       |(a)| \cala(J) \ox_{\alg} \cala( \bar{J} ) 
                       \pgfmatrixnextcell |(b)| \bfB(L^2\cala(I))\\ 
                       |(c)| \cala(J \cup_p \bar{J}) \\ 
                }; 
             \draw[->] (a) -- (b);
             \draw[->] (a) -- (c);
             \draw[->,dashed] (c) -- (b);
       } %tikzmath
    \end{equation*} 
    Here, $J \cup_p \bar{J}$ is equipped with any smooth structure extending 
    the given smooth structures on $J$ and $\bar J$, and for which 
    the orientation-reversing involution that exchanges $J$ and $\bar{J}$ 
    is smooth.
  \end{enumerate}
  
  A conformal net $\cala$ is called \emph{irreducible} if the 
  algebras $\cala(I)$ are factors.
  As discussed in Correction~\ref{corr:no-nets-disintegration}, contrary to our claim in~\cite[\subseccentraldecomposition, Eq~1.42]{BDH(nets)}, we do not know whether an arbitrary conformal net 
 decomposes as a direct integral of irreducible ones.  A conformal net is called \emph{semisimple} if it is a finite direct sum of irreducible conformal nets.  We denote by $\CN_0$ the symmetric monoidal category of semisimple conformal nets and their natural transformations.    The tensor product of nets $\cala$ and $\calb$ is defined using
  the spatial tensor product of von Neumann algebras:
  $(\cala  \otimes \calb)(I) := \cala(I) \,\bar \ox\, \calb(I)$.  A natural transformation $\tau: \cala \to \calb$ between semisimple conformal nets is called \emph{finite} if for all intervals $I$, the map $\tau_I : \cala(I) \to \calb(I)$ is a finite homomorphism (Appendix~\ref{subsec:dualizability}).  
%We will denote by $\CN_0^f \subset \CN_0$ the symmetric monoidal category of semisimple conformal nets all of whose irreducible summands have finite index (Appendix~\ref{subsec:finite-nets}), with their finite natural transformations. 
  
\section{The vacuum sector}
%\addtocontents{toc}{\SkipTocEntry}

    \label{subsec:vacuum-sector-net}
  A \emph{conformal circle} $S$ is a circle $S$ together with a diffeomorphism
  $S \to S^1$ that is specified up to a (not-necessarily orientation preserving) M\"obius transformation of 
  $S^1$~\cite[\defconformalcircle]{BDH(nets)};
  here, $S^1$ denotes the standard circle $\{z\in \mathbb C\,:\,|z|=1\}$.
  The set of conformal maps $S \to S'$ is denoted by $\Conf(S,S')$.
  If $S$ and $S'$ are oriented, then we denote by $\Conf_+(S,S')$ and $\Conf_-(S,S')$ the subsets of orientation preserving and orientation reversing maps.
  From now on, all our circles are implicitly assumed to be oriented.
  
  For a conformal net $\cala$ there is a 
  functor~\cite[\thmVaccumSector]{BDH(nets)}
  \begin{equation} \label{eq:vacuum-sector-functor}
    S \mapsto H_0(S,\cala)
  \end{equation}
  from the category of oriented conformal circles to the category
  of Hilbert spaces.
  It sends orientation preserving conformal maps
  to unitary operators and orientation reversing conformal maps
  to anti-unitary operators. %\footnote{The adjoint $u^*$ of an antilinear operator 
 %    $u$ is defined by 
 %  $\langle u\xi,\eta\rangle=\overline{\langle\xi,u^*\eta\rangle}$.
 % If $u^* = u^{-1}$ then $u$ is an anti-unitary.} operators.    - No need to explain, this is a standard notion
  The Hilbert space $H_0(S,\cala)$ is called 
  the \emph{vacuum sector of $\cala$ on $S$}, and comes equipped with compatible actions of 
  the algebras $\cala(I)$ for any subinterval $I$ of $S$.
  
  For $\varphi \in \Conf(S,S')$, the operator $H_0(\varphi,\cala)$ implements the diffeomorphism $\varphi$, that is:
  \begin{equation*}
      \begin{split}
        \cala(\varphi) (a)
          =  H_0(\varphi,\cala)\, a\, H_0(\varphi,\cala)^* & 
           \quad \text{if} \; \varphi \in \Conf_+(S,S')  \\ 
        \cala(\varphi) (a) = 
         H_0(\varphi,\cala)\, a^* H_0(\varphi,\cala)^*   & 
           \quad \text{if} \; \varphi \in \Conf_-(S,S')
       \end{split}
  \end{equation*}   
  for any $I\subset S$ and $a \in \cala(I)$.

  Moreover, for every interval $I\subset S$, there is a canonical unitary identification 
  \begin{equation}
    \label{eq:v_I}
    v_I \colon H_0(S,\cala) \to L^2(\cala(I)).
  \end{equation}
  These unitaries are such that for $\varphi \in \Conf_+(S,S')$ and $\psi \in \Conf_-(S,S')$, the diagrams
  \begin{equation} \label{eq:naturality-v_I}
    \tikzmath{
       \matrix [matrix of math nodes,column sep=1.4cm,row sep=7mm]
          { 
              |(a)| H_0(S,\cala) \pgfmatrixnextcell 
              |(c)| L^2(\cala(I))\\ 
              |(b)| H_0(S',\cala) \pgfmatrixnextcell 
              |(d)| L^2(\cala(\varphi(I)))\\ 
          }; 
          \draw[->] (a) -- node [left, xshift=-1]	
              {$\scriptstyle H_0(\varphi,\cala)$} (b);
          \draw[->] (c) -- node [left, xshift=1]	
              {$\scriptstyle L^2(\cala(\varphi))$} (d);
          \draw[->] (a) -- node [above]	{$\scriptstyle v_{I}$} (c);
          \draw[->] (b) -- node [above]	{$\scriptstyle v_{\varphi(I)}$} (d);
          }%\quad
    \tikzmath{
        \matrix [matrix of math nodes,column sep=1.4cm,row sep=7mm]
           {       
               |(a)| H_0(S,\cala) \pgfmatrixnextcell 
               |(c)| L^2(\cala(I))\\ 
               |(b)| H_0(S',\cala) \pgfmatrixnextcell 
               |(d)| L^2(\cala(\psi(I')))\\ 
           }; 
        \draw[->] (a) -- node [left, xshift=-1]	
            {$\scriptstyle H_0(\psi,\cala)$} (b);
        \draw[->] (c) -- node [left, xshift=1]	
            {$\scriptstyle L^2(\cala(\psi j))\circ J$} (d);
        \draw[->] (a) -- node [above]	{$\scriptstyle v_{I}$} (c);
        \draw[->] (b) -- node [above]	{$\scriptstyle v_{\psi (I')}$} (d);
      }
  \end{equation}
  commute,
  where $J$ is the modular conjugation on $L^2(\cala(I))$,
  $j \in \Conf_-(S)$ is the involution that fixes $\dd I$, and $I'=j(I)$ is the closure of $S\setminus I$.
  Taking $\psi := j$ in the second diagram, we recover the modular conjugation as
  $J = v_I H_0(j_I,\cala) v_I^*$.
  
  If $S$ is a circle without a conformal structure, then it is still
  possible to define $H_0(S,\cala)$ as $L^2(\cala(I))$ of some interval $I\subset S$,
  but this only defines $H_0(S,\cala)$ up to non-canonical
  unitary isomorphism~\cite[\defnoncanonicalvacuum]{BDH(nets)}.  
  We will sometimes abbreviate $H_0(S,\cala)$ by $H_0(S)$.

  \begin{proposition}{\rm (Haag duality for conformal nets \cite[\propHaagduality]{BDH(nets)})}\label{prop: [Haag duality for defects]-duality-nets}
     Let $\cala$ be a conformal net, and let $S$ be a circle.
     Then for any $I\subset S$, the algebra $\cala(I')$ is the 
     commutant of $\cala(I)$ on $H_0(S,\cala)$.
           
     If $J\subset K$ are intervals such that $J^c$, the closure of 
     $K\setminus J$, is itself an interval, then the relative commutant of 
     $\cala(J)$ in $\cala(K)$ is $\cala(J^c)$.
  \end{proposition}

\section{Gluing vacuum sectors}
%\addtocontents{toc}{\SkipTocEntry}  
      \label{subsec:glueing}
     Consider a theta-graph $\Theta$, and let $S_1$, $S_2$, $S_3$ be its three circle subgraphs with orientations as drawn below:
     \begin{equation*} %\label{eq: Theta-graph}
\Theta:\, \tikzmath[scale=.055]{ \draw (60:14) arc (60:300:14); \draw (14,0) +(120:14) arc (120:-120:14); \draw (0,0) +(60:14) -- +(300:14);
%\node at (4,10) {$\scriptstyle p$}; \node at (4,-10) {$\scriptstyle q$};
}%tikzmath
\,\,\,\,\,\,\qquad \tikzmath[scale=.033]{ \useasboundingbox (-14,-30) rectangle (28,24); \draw[line width=.7] (60:14) arc (60:300:14); \draw[densely dotted] (14,0) +(120:14) arc (120:-120:14);
\draw[line width=.7] (0,0) +(60:14) -- +(300:14); \node at (8,-24) {$S_1$}; \draw[->] (90:14) ++ (-.5,0) -- +(-.1,0); }%tikzmath
\,\,,\quad \tikzmath[scale=.033]{ \useasboundingbox (-14,-30) rectangle (28,24); \draw[densely dotted] (60:14) arc (60:300:14); \draw[line width=.7] (14,0) +(120:14) arc (120:-120:14);
\draw[line width=.7] (0,0) +(60:14) -- +(300:14); \node at (8,-24) {$S_2$}; \draw[->] (14,14) ++ (-.5,0) -- +(-.1,0); }%tikzmath
\,\,,\quad \tikzmath[scale=.033]{ \useasboundingbox (-14,-30) rectangle (28,24); \draw[line width=.7] (60:14) arc (60:300:14); \draw[line width=.7] (14,0) +(120:14) arc (120:-120:14);
\draw[densely dotted] (0,0) +(60:14) -- +(300:14); \node at (8,-24) {$S_3$}; \draw[->] (14,14) ++ (-.5,0) -- +(-.1,0); }%tikzmath
\,\,.\end{equation*}
  (\hspace{.2mm}Elsewhere in this book, we more often depict circles as squares:
  \begin{equation*}\label{eq: Theta-graph-square}
  \Theta:
   \tikzmath[scale=0.11] { 
      \useasboundingbox (-3,-3) rectangle (30,15);
      \draw (0,0) rectangle (24,12) (12,0) -- (12,12);
%      \node at (10,10) {$\scriptstyle p$};    \node at (10,2) {$\scriptstyle q$};
   } %tikzmath
   \quad
   \tikzmath[scale=\displscale] { 
      \useasboundingbox (-5,-5) rectangle (27,17);
      \draw[line width=.7] (0,0) rectangle (12,12) ;
      \draw[densely dotted] (12,0) -- (24,0) -- (24,12) -- (12,12);
      \node at (12.5,-6) {$S_1$};
   }  , %tikzmath
   \quad
   \tikzmath[scale=\displscale] { 
      \useasboundingbox (-5,-5) rectangle (27,17);
      \draw[line width=.7] (12,0) rectangle (24,12) ;
      \draw[densely dotted] (12,0) -- (0,0) -- (0,12) -- (12,12);
      \node at (12.5,-6) {$S_2$};
   }  ,%tikzmath
   \quad
    \tikzmath[scale=\displscale] { 
      \useasboundingbox (-5,-5) rectangle (27,17);
      \draw[line width=.7] (0,0) rectangle (24,12);
      \draw[densely dotted] (12,0) --  (12,12);
      \node at (12.5,-6) {$S_3$};
   } .)%tikzmath
  \end{equation*}
  We equip $\Theta$ with a `smooth structure' in the sense of \cite[Def. 1.4]{BDH(modularity)} and let
%   $S_1$, $S_2$, $S_3$ with smooth structures for which there exists 
%  an action of the symmetric group $\mathfrak S_3$ on $\Theta$
%  that fixes $p$ and $q$, permutes the three circles, and such 
%  that $\pi|_{S_a}$ is smooth for every $\pi\in \mathfrak S_3$ and 
%  $a\in \{1,2,3\}$.
  \[
   I:=S_1\cap S_2,\quad K:=S_1\cap S_3,\quad L:=S_2\cap S_3.
  \]
  Let us give $K$ the orientation coming from $S_1$, and let us give $I$ and $L$ the orientations coming from $S_2$. 
  
  Then given a conformal net $\cala$, there is a non-canonical isomorphism~
     \cite[\corvacuumvacuumvacuum]{BDH(nets)}   
  \begin{equation}
       \label{eq:non-canonical-upsilon}
       H_0(S_1,\cala) \boxtimes_{\cala(I)} H_0(S_2,\cala) 
        \,\,\cong\,\, H_0(S_3,\cala),
  \end{equation}
  compatible with the actions of $\cala(K)$ and $\cala(L)$.
  Moreover, in the presence of suitable conformal structures,
  this isomorphism can be constructed canonically:
  equip $S_1$ and $S_2$ with conformal structures, and let $j_1 \in \Conf_-(S_1)$, $j_2 \in \Conf_-(S_2)$ be the unique involutions fixing $\dd I$.
  Then there is a unique conformal structure on $S_3$ for which 
  $j_2|_I\cup \mathrm{Id}_{K}:S_1\to S_3$ and $j_1|_I\cup \mathrm{Id}_{L}:S_2\to S_3$ are conformal.
  We can then use~\eqref{eq:v_I} to obtain the canonical isomorphism~\cite[\corConformalversionofvacuumvacuumvacuum]{BDH(nets)}   
  \begin{equation}
       \label{eq:Upsilon}
    \begin{split}
      \Upsilon \colon H_0(S_1,\cala) \boxtimes_{\cala(I)} H_0(S_2,\cala)
           \xrightarrow{v_K\boxtimes\hspace{.3mm} 
                 v_{I}} \,\,& 
         L^2(\cala(K))\boxtimes_{\cala(I)}L^2(\cala(I))\\
           \xrightarrow{\cong}\,\, &
          L^2(\cala(K)) \xrightarrow{v_{K}^*}H_0(S_3,\cala).
    \end{split}
  \end{equation}

\section{Finite-index conformal nets}
%\addtocontents{toc}{\SkipTocEntry}
      \label{subsec:finite-nets}
  Let $S$ be a circle, split into four intervals 
  $I_1$, $I_2$, $I_3$, $I_4$ as follows:
%  such that each $I_i$ intersects $I_{i+1}$ (cyclic numbering) in a single point:
  \begin{equation} \label{eq:circle-4-intervals}
        \tikzmath[scale=.07]
             { \useasboundingbox (-20,-22) rectangle (20,22);
                  \draw (0,0) circle (15);
                  \foreach \thet in {45,135, ..., 315} 
                      { \draw[thick] (\thet:14) -- (\thet:16);}
                  \foreach \ii in {1, ..., 4} 
                      {  \node at (\ii*90-90:19) {$I_\ii$}; }}
  \end{equation} 
  Given an irreducible conformal net $\cala$,
  the algebras $\cala(I_1\cup I_3)= \cala(I_1)\,\bar\otimes\,\cala(I_3)$ 
  and $\cala(I_2\cup I_4)= \cala(I_2)\,\bar\otimes\,\cala(I_4)$ act on
  $H_0(S,\cala)$ and commute with each other.
  The \emph{index} $\mu(A)$ of %the conformal net 
  $\cala$ is then defined to be minimal index (see Appendix~\ref{subsec:stat-dim+minimal-index}) of the inclusion 
  $\cala(I_1\cup I_3)\subseteq \cala(I_2\cup I_4)'$ \cite{Kawahigashi-Longo-Mueger(2001multi-interval), Xu(Jones-Wassermann-subfactors)}:
  \begin{equation*}
           \mu(\cala) := [\cala(I_2\cup I_4)':\cala(I_1\cup I_3)],
  \end{equation*}
  where the commutant is taken on $H_0(S,\cala)$.  

\section{Sectors and the Hilbert space of the annulus}
%\addtocontents{toc}{\SkipTocEntry}
     \label{subsec:sectors+KLM}
  
  Let $\cala$ be an irreducible conformal net and let $S$ be a 
  circle (always oriented).
  An $S$-sector of $\cala$ is a Hilbert space $H$ together 
  with homomorphisms
  \[
  \rho_I \colon \cala(I) \to \bfB(H),\quad I\subset S
  \]
  subject to the compatibility condition $\rho_I|_{\cala(J)}=\rho_J$ 
  whenever $J \subset I$.
  
  Let us write $\Delta$ for the set of isomorphism classes
  of irreducible $S$-sectors of $\cala$.
  The vacuum sector discussed before is an example of a sector and
  we write $0$ for the corresponding element of $\Delta$.
  As all circles are diffeomorphic, $\Delta$ does not depend on the
  choice of circle $S$.
  There is an involution $\lambda \mapsto \bar \lambda$ on $\Delta$
  given by sending an $S$-sector to its pull back  along
  an orientation reversing diffeomorphism of $S$, as defined in \cite[(1.13)]{BDH(nets)}.
  For $\lambda \in \Delta$, we write $H_\lambda(S)$ 
  for a representative of $\lambda$ as an $S$-sector.
  Of course, $H_\lambda(S)$ is only determined up to
  non-canonical isomorphism.

  Let $S_l$ be a circle, decomposed into four intervals $I_1, \ldots, I_4$
  as in~\eqref{eq:circle-4-intervals}, 
  and let $S_r$ be another circle, similarly decomposed into four 
  intervals $I_5, \ldots, I_8$.
  Let $\varphi \colon I_5\to I_1$ and $\psi \colon I_7\to I_3$ 
  be orientation-reversing diffeomorphisms.  
  These diffeomorphisms equip $H_0(S_l)$ with the structure of a right 
  $\cala(I_5)\,\bar\otimes\,\cala(I_7)$-module.
  We are interested in the Hilbert space
  \[
     H_{\Sigma}\,\,:=\,\,\,\,H_0(S_l)\!
        \underset{\cala(I_5)\bar\otimes\cala(I_7)}\boxtimes\! H_0(S_r)
       \cong 
       \tikzmath{\node (A) at (0,0) 
         {$H_0(S_l) \boxtimes_{\cala(I_5)} H_0(S_r) \boxtimes_{\cala(I_3)}$};
       \def\dd{.3}\def\ll{.1}\def\rr{.15}
       \draw[dashed, rounded corners = 5] 
           (A.east) -- ++(\rr,0) -- ++(0,-\dd) -- 
          ($(A.west) + (-\ll,-\dd)$) -- +(0,\dd) -- (A.west);}
  \]
  This space is associated to the annulus 
  $\Sigma=\ID_l\cup_{I_5\cup I_7} \ID_r$, where $\ID_l$ and $\ID_r$ are 
  disks bounding $S_l$ and $S_r$.
  (As $H_0(S_l)$ and $H_0(S_r)$ are only determined up to
  non-canonical isometric isomorphism, the same Hilbert space $H_\Sigma$ is, at this point, also only determined up to non-canonical isometric isomorphism.)
  Let $S_b:=I_2\cup I_8$ and $S_m:=I_4\cup I_6$ be the two boundary 
  circles of this annulus.
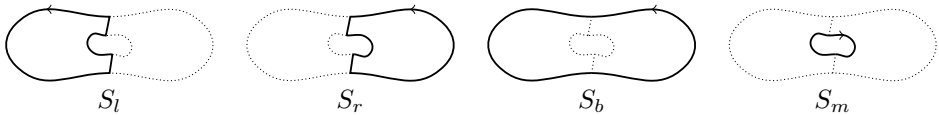
\[
\tikzmath[scale=.25]{\coordinate (a) at (0,0);\coordinate (b) at (.15,1);\coordinate (c) at (-.2,2);\coordinate (d) at (0,3);\coordinate (e) at (-5,.4);\coordinate (f) at (-5,2.6);\coordinate (g) at (-1,1.2);\coordinate (h) at (-1,2);\coordinate (a') at (d);\coordinate (b') at (c);\coordinate (c') at (b);\coordinate (d') at (a);\begin{scope}[yshift = 85, rotate= 180]\coordinate (e') at (-5,.4);\coordinate (f') at (-5,2.6);\coordinate (g') at (-1,1.2);\coordinate (h') at (-1,2);\end{scope}
\draw[line width=.7] (b) to (a) to [out = 180, in = -45, looseness=1.1] (e) 
	to [out = -45 + 180, in = 225, looseness=1.1] (f) to [out = 225 + 180, in = 180, looseness=1.1] node (a1) [pos = .37] {}
	(d) to [looseness=0] (c) to [out = 180, in = 45, looseness=1.1] 
	(h) to [out = 45 + 180, in = -225, looseness=1.1] (g) to [out = -225 + 180, in = 180, looseness=1.1] (b);
\draw[densely dotted] (a') to [out = 0, in = -45 + 180, looseness=1.1]node (a1') [pos = .6] {} (e') to [out = -45, in = 225 + 180, looseness=1.1]  
	(f') to [out = 225, in = 0, looseness=1.1] (d');
\draw[densely dotted] (c') to [out = 0, in = 45 + 180, looseness=1.1] (h') to [out = 45, in = -225 + 180, looseness=1.1]
	(g') to [out = -225, in = 0, looseness=1.1] node (b1) [pos = .37] {} (b');
\node at (0,-1.5) {$S_l$};
\draw[->] (a1.center) -- ++ (180:0.01);
} % tikzmath
\quad\tikzmath[scale=.25]{\coordinate (a) at (0,0);\coordinate (b) at (.15,1);\coordinate (c) at (-.2,2);\coordinate (d) at (0,3);\coordinate (e) at (-5,.4);\coordinate (f) at (-5,2.6);\coordinate (g) at (-1,1.2);\coordinate (h) at (-1,2);\coordinate (a') at (d);\coordinate (b') at (c);\coordinate (c') at (b);\coordinate (d') at (a);\begin{scope}[yshift = 85, rotate= 180]\coordinate (e') at (-5,.4);\coordinate (f') at (-5,2.6);\coordinate (g') at (-1,1.2);\coordinate (h') at (-1,2);\end{scope}
\draw[densely dotted] (a) to [out = 180, in = -45, looseness=1.1] (e) 
	to [out = -45 + 180, in = 225, looseness=1.1] (f) to [out = 225 + 180, in = 180, looseness=1.1] node (a1) [pos = .37] {} (d);
\draw[densely dotted] (c) to [out = 180, in = 45, looseness=1.1] 
	(h) to [out = 45 + 180, in = -225, looseness=1.1] (g) to [out = -225 + 180, in = 180, looseness=1.1] (b);
\draw[line width=.7] (b') to (a') to [out = 0, in = -45 + 180, looseness=1.1]node (a1') [pos = .6] {} (e') to [out = -45, in = 225 + 180, looseness=1.1]  
	(f') to [out = 225, in = 0, looseness=1.1] (d') to [looseness=0] (c') to [out = 0, in = 45 + 180, looseness=1.1] (h') to [out = 45, in = -225 + 180, looseness=1.1]
	(g') to [out = -225, in = 0, looseness=1.1] node (b1) [pos = .37] {} (b');
\node at (0,-1.5) {$S_r$};
\draw[->] (a1'.center) -- ++ (180:0.01);
} % tikzmath
\quad\tikzmath[scale=.25]{\coordinate (a) at (0,0);\coordinate (b) at (.15,1);\coordinate (c) at (-.2,2);\coordinate (d) at (0,3);\coordinate (e) at (-5,.4);\coordinate (f) at (-5,2.6);\coordinate (g) at (-1,1.2);\coordinate (h) at (-1,2);\coordinate (a') at (d);\coordinate (b') at (c);\coordinate (c') at (b);\coordinate (d') at (a);\begin{scope}[yshift = 85, rotate= 180]\coordinate (e') at (-5,.4);\coordinate (f') at (-5,2.6);\coordinate (g') at (-1,1.2);\coordinate (h') at (-1,2);\end{scope}
\draw[line width=.7] (a) to [out = 180, in = -45, looseness=1.1] (e) 
	to [out = -45 + 180, in = 225, looseness=1.1] (f) to [out = 225 + 180, in = 180, looseness=1.1] node (a1) [pos = .37] {} (d);
\draw[densely dotted] (c) to [out = 180, in = 45, looseness=1.1] 
	(h) to [out = 45 + 180, in = -225, looseness=1.1] (g) to [out = -225 + 180, in = 180, looseness=1.1] (b);
\draw[line width=.7] (a') to [out = 0, in = -45 + 180, looseness=1.1]node (a1') [pos = .6] {} (e') to [out = -45, in = 225 + 180, looseness=1.1]  
	(f') to [out = 225, in = 0, looseness=1.1] (d');
\draw[densely dotted] (c') to [out = 0, in = 45 + 180, looseness=1.1] (h') to [out = 45, in = -225 + 180, looseness=1.1]
	(g') to [out = -225, in = 0, looseness=1.1] node (b1) [pos = .37] {} (b');
\draw[densely dotted] (b') to (a')  (d') to (c');
\node at (0,-1.5) {$S_b$};
\draw[->] (a1'.center) -- ++ (180:0.01);
} % tikzmath
\quad\tikzmath[scale=.25]{\coordinate (a) at (0,0);\coordinate (b) at (.15,1);\coordinate (c) at (-.2,2);\coordinate (d) at (0,3);\coordinate (e) at (-5,.4);\coordinate (f) at (-5,2.6);\coordinate (g) at (-1,1.2);\coordinate (h) at (-1,2);\coordinate (a') at (d);\coordinate (b') at (c);\coordinate (c') at (b);\coordinate (d') at (a);\begin{scope}[yshift = 85, rotate= 180]\coordinate (e') at (-5,.4);\coordinate (f') at (-5,2.6);\coordinate (g') at (-1,1.2);\coordinate (h') at (-1,2);\end{scope}
\draw[densely dotted] (a) to [out = 180, in = -45, looseness=1.1] (e) 
	to [out = -45 + 180, in = 225, looseness=1.1] (f) to [out = 225 + 180, in = 180, looseness=1.1] node (a1) [pos = .37] {} (d);
\draw[line width=.7] (c) to [out = 180, in = 45, looseness=1.1] 
	(h) to [out = 45 + 180, in = -225, looseness=1.1] (g) to [out = -225 + 180, in = 180, looseness=1.1] (b);
\draw[densely dotted] (a') to [out = 0, in = -45 + 180, looseness=1.1]node (a1') [pos = .6] {} (e') to [out = -45, in = 225 + 180, looseness=1.1]  
	(f') to [out = 225, in = 0, looseness=1.1] (d');
\draw[line width=.7] (c') to [out = 0, in = 45 + 180, looseness=1.1] (h') to [out = 45, in = -225 + 180, looseness=1.1]
	(g') to [out = -225, in = 0, looseness=1.1] node (b1) [pos = .37] {} (b');
\draw[densely dotted] (b') to (a')  (d') to (c');
\node at (0,-1.5) {$S_m$};
\draw[->] (b1.center) -- ++ (0:0.01);
} % tikzmath
\]

  The Hilbert space $H_{\Sigma}$ is an $S_m$-$S_b$-sector, which means that it is equipped with compatible actions of the 
  algebras $\cala(J)$ associated to all subintervals of $S_m$ and $S_b$~\cite[\secHilbertspaceforannulus]{BDH(nets)}.
  
  We finish by stating the computation of the annular Hilbert space, which, formulated in a different language, is due to~\cite{Kawahigashi-Longo-Mueger(2001multi-interval)}:

  \begin{theorem}[{\cite[\thmKLM, \thmKLMallirreduciblesectorsarefinite]{BDH(nets)}}]
    \label{thm:KLM}
    If the conformal net $\cala$ has finite index, then the set $\Delta$ is finite, and there is a unitary 
    isomorphism of $S_m$-$S_l$-sectors
    \begin{equation*}
      H_{\Sigma} \cong 
        \bigoplus_{\lambda\in\Delta} H_\lambda(S_m) \otimes H_{\bar \lambda}(S_b).
    \end{equation*}
  \end{theorem}
  
\section{Extension of conformal nets to all $1$-manifolds}
%\addtocontents{toc}{\SkipTocEntry}
      \label{subsec:cala(1-mfd)}
  
  A priori, the only manifolds on which a conformal net $\cala \colon \INT \to \VN$ can be evaluated are intervals.
  However, the functor $\cala$ can be extended, in a canonical way, to the larger category $\mathsf{1MAN}$ of compact oriented $1$-manifolds~\cite[\thmextendAtohatA]{BDH(modularity)}.
  We denote the extension $\mathsf{1MAN} \to \VN$ by the same letter $\cala$.

  For $S$ a circle, %the construction of $\cala(S)$ depends on the Hilbert space $H_\Sigma$ associated to the annulus
  %$\Sigma := S \x [0,1]$---see~\cite[~(1.6) and Thm. 1.38]{BDH(modularity)} for a refinement of the construction
  %presented in Appendix \ref{subsec:sectors+KLM} above.
  the algebra $\cala(S)$ is defined to be the subalgebra of $\bfB(H_\Sigma)$ generated by $\cala(I \x \{ 0 \})$ for all $I\subset S$,
  where $S$ is one of the two boundaries of the annulus $\Sigma := S \x [0,1]$.

  \begin{theorem}[{\cite[\thmcomputebfBnet]{BDH(modularity)}}]
    \label{thm:cala(S)}
    Let $\cala$ be a conformal net with finite index and let $S$ be a circle.
    Then there is a canonical isomorphism
    \begin{equation}\label{eq: A(S)=(+)} 
        \cala(S) \,\cong\, \bigoplus_{\lambda\in\Delta} \bfB(H_\lambda(S)).
    \end{equation}
  \end{theorem}

  Note that even though $H_\lambda(S)$ is only defined up to non-canonical isomorphism,
  its algebra of bounded operators is defined up to canonical isomorphism.
  It therefore makes sense for the isomorphism \eqref{eq: A(S)=(+)} to be canonical.
  \newpage
%}

\chapter{Diagram of dependencies}%
\label{app:depend}
\thispagestyle{empty}

\newgeometry{top=2cm,bottom=0cm}
\vspace*{1cm}
%\begin{sideways}
%\begin{minipage}{9in}
\[\hspace{-1cm}
\tikzmath[xscale=1.3,yscale=2.5]{
\useasboundingbox (-10,9) rectangle (1,1);
\pgftransformrotate{90}
\node[rotate=90] (T142) at (2,10)  {\parbox{3cm}{ \center Thm 1.44 \\ \center (Fusion of defects)}};
\node[rotate=90] (T62) at (4.8,10)  {\parbox{2cm}{ \center Thm 6.2 \\ \center ($1 \boxtimes 1 = 1$)}};
\node[rotate=90] (T52) at (6.7,8.4)  {\parbox{4.7cm}{ \center Thm 5.2 \\ \center (Defect fusion Haag duality)}};
\node[rotate=90] (6D) at (6.9,10)  {\parbox{4.7cm}{ \center Eq (6.25) \\ \center (Interchange isomorphism)}};
\node[rotate=90] (P429) at (4.6,8.3)  {\parbox{2cm}{ \center Prop 4.29 \\ \center 
$\tikzmath[scale=\textscale]
     { \draw (8,12) -- (16,12) (8,0) -- (16,0) (12,0) -- (12,12); \draw[ultra thin, dash pattern=on .5pt off 1pt]
         (8,12) -- (0,12) -- (0,0) -- (8,0) (16,12) -- (24,12) -- (24,0) -- (16,0);
     }
    =  
     \tikzmath[scale=\textscale]
     { \draw (8,12) -- (10,12) -- (10,0) -- (8,0) (16,12) -- (14,12) -- (14,0) -- (16,0);
       \draw[ultra thin, dash pattern=on .5pt off 1pt] (8,12) -- (0,12) -- (0,0) -- (8,0) (16,12) -- (24,12) -- (24,0) -- (16,0);
       \fill[vacuumcolor] (10,0) rectangle (14,4) (10,8) rectangle (14,12) (10.5,4.5) rectangle (13.5,7.5);
       \draw (10,0) rectangle (14,4)  (10,8) rectangle (14,12) (10.5,4.5) rectangle (13.5,7.5);
     }$
}};
\node[rotate=90] (L421) at (3.3,7.75)  {\parbox{2cm}{ \center Lma 4.21 \\ \center
$\tikzmath[scale=\textscale]{\fill[vacuumcolor]  (0,0) rectangle (24,12);\draw  (6,0) -- (0,0) -- (0,12) -- (6,12);
\draw (6,12) -- (18,12)  (6,0) -- (18,0);\draw (18,0) -- (24,0) -- (24,12) -- (18,12);} 
=
\tikzmath[scale=\textscale]{\fill[vacuumcolor] (0,0) rectangle (10,12)(14,0) rectangle (24,12);\draw (6,12) -- (10,12) -- (10,0) -- (6,0)(18,12) -- (14,12) -- (14,0) -- (18,0);
\draw (6,12) -- (0,12) -- (0,0) -- (6,0);\draw  (18,12) -- (24,12) -- (24,0) -- (18,0);\fill[vacuumcolor]  (10,0) rectangle (14,4)(10,8) rectangle (14,12);
\draw (10,0) rectangle (14,4) (10,8) rectangle (14,12);\fill[vacuumcolor]  (10.5,4.5) rectangle (13.5,7.5);\draw (10.5,4.5) rectangle (13.5,7.5);}$ 
}};
\node[rotate=90] (P418) at (2,6.9)  {\parbox{2.5cm}{ \center Prop 4.18 \\ \center
$\tikzmath[scale=\textscale]{\fill[vacuumcolor]  (0,0) rectangle (24,12);\draw[thick, double]  (6,0) -- (0,0) -- (0,12) -- (6,12);
\draw (6,12) -- (18,12)  (6,0) -- (18,0);\draw[ultra thick] (18,0) -- (24,0) -- (24,12) -- (18,12);} 
\hookrightarrow
\tikzmath[scale=\textscale]{\fill[vacuumcolor] (0,0) rectangle (10,12)(14,0) rectangle (24,12);\draw (6,12) -- (10,12) -- (10,0) -- (6,0)(18,12) -- (14,12) -- (14,0) -- (18,0);
\draw[thick, double] (6,12) -- (0,12) -- (0,0) -- (6,0);\draw[ultra thick]  (18,12) -- (24,12) -- (24,0) -- (18,0);\fill[vacuumcolor]  (10,0) rectangle (14,4)(10,8) rectangle (14,12);
\draw (10,0) rectangle (14,4) (10,8) rectangle (14,12);\fill[vacuumcolor]  (10.5,4.5) rectangle (13.5,7.5);\draw (10.5,4.5) rectangle (13.5,7.5);}$ 
}};
\node[rotate=90] (Z) at (3.3,9)  {\parbox{2.5cm}{ \center Eq (4.31) \\ \center
$\tikzmath[scale=\textscale]{\fill[vacuumcolor]  (0,0) rectangle (24,12);\draw[thick, double]  (6,0) -- (0,0) -- (0,12) -- (6,12);
\draw (6,12) -- (18,12)  (6,0) -- (18,0);\draw[ultra thick] (18,0) -- (24,0) -- (24,12) -- (18,12);} 
\hookrightarrow
\tikzmath[scale=\textscale]{\fill[vacuumcolor] (0,0) rectangle (24,12);
\draw[double, thick](6,0) -- (0,0) -- (0,12) -- (6,12); \draw (6,0) -- (12,0) -- (12,12) -- (6,12)(18,0) -- (12,0)(12,12) -- (18,12);
\draw[ultra thick] (18,0) -- (24,0) -- (24,12) -- (18,12);}$ 
}};
\draw[->] (P418) -- (Z);\draw[->] (P429) -- (Z);\draw[<-] (T52.west)+(.1,.1) to[bend right=14] (Z);
\draw[->] (T62) -- (6D);
\node[rotate=90] (T151) at (8.6,8.4)  {\parbox{1.9cm}{ \center Thm 1.53 \\ \center $\circledast = \ast$}};
\node[rotate=90] (X) at (9,9.75)  {\parbox{2.5cm}{ \center Eq (1.55) \\ \center (Associativity of fusion)}};
\draw[->] (T151) -- (X);
\node[rotate=90] (T411) at (1,5.75) {\parbox{3cm}{ \center Thm 4.11 \\ \center
$L^2 \!\left( \tikzmath[scale=\textscale]
               { \useasboundingbox (-2,0) rectangle (26,14);
                 \draw[thick, double] (0,6) -- (0,12) -- (6,12);
                 \draw (6,12) -- (18,12) (10,6) -- (10,8) -- (14,8) -- (14,6);
                 \draw[ultra thick] (18,12) -- (24,12) -- (24,6); 
                 \draw[densely dotted] (12,8) -- (12,12); 
               } %tikzmath
        \right)
    =
      \tikzmath[scale=\textscale]
      {    \fill[vacuumcolor] (0,0) rectangle (10,12)
                             (14,0) rectangle (24,12); 
           \draw (6,12) -- (10,12) -- (10,0) -- (6,0) 
                 (18,12) -- (14,12) -- (14,0) -- (18,0);
           \draw[thick, double] (6,12) -- (0,12) -- (0,0) -- (6,0);
           \draw[ultra thick] (18,12) -- (24,12) -- (24,0) -- (18,0);
           \fill[vacuumcolor]  (10,0) rectangle (14,4)  
                              (10,8) rectangle (14,12);
           \draw (10,0) rectangle (14,4) (10,8) rectangle (14,12);
      }$
}};
\node[rotate=90] (P44) at (3.3,6.3) {\parbox{2.5cm}{ \center Prop 4.4 \\ \center
$\exists \; \tikzmath[scale=\textscale]
     { \draw (8,12) -- (16,12) (8,0) -- (16,0) (12,0) -- (12,12); \draw[ultra thin, dash pattern=on .5pt off 1pt]
         (8,12) -- (0,12) -- (0,0) -- (8,0) (16,12) -- (24,12) -- (24,0) -- (16,0);
     }
    \cong  
     \tikzmath[scale=\textscale]
     { \draw (8,12) -- (10,12) -- (10,0) -- (8,0) (16,12) -- (14,12) -- (14,0) -- (16,0);
       \draw[ultra thin, dash pattern=on .5pt off 1pt] (8,12) -- (0,12) -- (0,0) -- (8,0) (16,12) -- (24,12) -- (24,0) -- (16,0);
       \fill[vacuumcolor] (10,0) rectangle (14,4) (10,8) rectangle (14,12) (10.5,4.5) rectangle (13.5,7.5);
       \draw (10,0) rectangle (14,4)  (10,8) rectangle (14,12) (10.5,4.5) rectangle (13.5,7.5);
     }$
}};
\node[rotate=90] (TC8) at (3.3,4.8)  {\parbox{2.5cm}{ \center Thm C.8 \\ \center
$\tikzmath[scale=\textscale]
{\draw[fill=vacuumcolor] (0,0) circle (8);
\draw[fill=white] (0,0) circle (3);}
= 
\bigoplus\tikzmath[scale=\textscale]
{\draw[fill=spacecolor] (0,0) circle (3);\useasboundingbox;\node[scale=.8] at (0,-6) {$\scriptstyle \lambda$};}
\,\raisebox{1pt}{$\scriptscriptstyle\otimes$}\,
\tikzmath[scale=\textscale]
{\draw[fill=spacecolor] (0,0) circle (8);\node[scale=.8] at (0,0) {$\scriptstyle \bar\lambda$};}
$}};

\node[rotate=90] (C416) at (2,4.5)  {\parbox{2.5cm}{ \center Cor 4.16 \\ \center
$\tikzmath[scale=\textscale]
	{\draw[thick, double] (0,6) -- (0,12) -- (6,12);
	\draw (6,12) -- (18,12) (10,6) -- (10,8) -- (14,8) -- (14,6);
	\draw[ultra thick] (18,12) -- (24,12) -- (24,6);} \,\overset{\textrm{\tiny{fin}}}{\hookrightarrow}\,
\tikzmath[scale=\textscale]
	{\draw[thick, double] (0,6) -- (0,12) -- (6,12);
	\draw (6,12) -- (18,12) (10,6) -- (10,8) -- (14,8) -- (14,6);
	\draw[dash pattern=on .4pt off .62pt] (12,8) -- (12,12);
	\draw[ultra thick] (18,12) -- (24,12) -- (24,6);}$
	}};
\node[rotate=90] (C37) at (1,3) {\parbox{3cm}{\center Cor 3.7 \\ \center (Fused defect \\ \center is semisimple)}};
\node[rotate=90] (T36) at (3,2.125) {\parbox{3cm}{\center Thm 3.6 \\ \center(Fused algebra \\ \center is semisimple)}};
\node[rotate=90] (L316) at (3,0.525) {\parbox{2.5cm}{ \center Lma 3.17 \\ \center
$\tikzmath[scale=\textscale]
	{\draw[fill=vacuumcolor,line width=0] (0,0) rectangle (12,12); %\draw[ultra thick] (6,0) -- (12,0) -- (12,12) -- (6,12);
	\draw (6,-1.5) -- (-1.5,-1.5) -- (-1.5,13.5) -- (6,13.5);
	\draw[ultra thick] (6,13.5) -- (12,13.5) (13.5,12) -- (13.5,0) (12,-1.5) -- (6,-1.5);
	}$
finite
%$\tikzmath[scale=\textscale]
%	{\draw[fill=vacuumcolor] (0,0) rectangle (12,12); \draw[ultra thick] %(6,0) -- (11,0) (12,1) -- (12,11) (11,12) -- (6,12);}$
}};
\node[rotate=90] (L314) at (3,-1) {\parbox{2.5cm}{ \center Lma 3.15 \\ \center
$\tikzmath[scale=.015]
{\useasboundingbox (-15,-7) rectangle (15,7);
\draw[ultra thick] (-90:15) arc (-90:-4:15) (4:15) arc (4:90:15);
\draw (90:15) arc (90:127:15) (135:15) arc (135:176:15) (-90:15) arc (-90:-127:15) (-135:15) arc (-135:-176:15);}$ finite
}};
\node[rotate=90] (L519) at (5,0.5) {\parbox{7cm}{ \center Lma 5.19 \\ \center
$
\left\llbracket \tikzmath[scale=\textscale]{\draw (12,12) -- (24,12) (18,0) -- (21,0) (6,12) -- (9,12);
\draw[ultra thick](24,12) -- (30,12) -- (30,6);\draw[dash pattern=on .4pt off .62pt] (19.5,0) -- (19.5,12);} %tikzmath
 :  \tikzmath[scale=\textscale]{\draw (12,12) -- (24,12) (18,0) -- (21,0) (6,12) -- (9,12);\draw[ultra thick](24,12) -- (30,12) -- (30,6);} %tikzmath
\right\rrbracket   \!=\!  \left\llbracket \tikzmath[scale=\textscale]{\draw (12,12) -- (24,12) (18,0) -- (21,0) (6,12) -- (9,12);
\draw[ultra thick](24,12) -- (30,12) -- (30,6);\draw[dash pattern=on .4pt off .62pt](7.5,12) arc (180:360:3 and 4);} %tikzmath
 :  \tikzmath[scale=\textscale]{\draw (12,12) -- (24,12) (18,0) -- (21,0) (6,12) -- (9,12);\draw[ultra thick](24,12) -- (30,12) -- (30,6);} %tikzmath
\right\rrbracket =  \sqrt{\mu(\calb)}
$
\\ \center
$
\left\llbracket \tikzmath[scale=\textscale]{\draw (12,12) -- (24,12) (18,0) -- (21,0) (6,12) -- (9,12);
\draw[ultra thick](24,12) -- (30,12) -- (30,6);\draw[dash pattern=on .4pt off .62pt](7.5,12) arc (180:360:3 and 4)(19.5,0) -- (19.5,12);} %tikzmath
 :  \tikzmath[scale=\textscale]{\draw (12,12) -- (24,12) (18,0) -- (21,0) (6,12) -- (9,12);
\draw[ultra thick](24,12) -- (30,12) -- (30,6);\draw[dash pattern=on .4pt off .62pt](7.5,12) arc (180:360:3 and 4);} %tikzmath
\right\rrbracket  \!=\!  \left\llbracket \tikzmath[scale=\textscale]{\draw (12,12) -- (24,12) (18,0) -- (21,0) (6,12) -- (9,12);
\draw[ultra thick](24,12) -- (30,12) -- (30,6);\draw[dash pattern=on .4pt off .62pt](7.5,12) arc (180:360:3 and 4)(19.5,0) -- (19.5,12);} %tikzmath
 :  \tikzmath[scale=\textscale]{\draw (12,12) -- (24,12) (18,0) -- (21,0) (6,12) -- (9,12);
\draw[ultra thick](24,12) -- (30,12) -- (30,6);\draw[dash pattern=on .4pt off .62pt] (19.5,0) -- (19.5,12);} %tikzmath
\right\rrbracket = \sqrt{\mu(\calb)}
$
}};
\node[rotate=90] (C520) at (5,2.25) {\parbox{4cm}{ \center Cor 5.20 \\ \center
$
\left\llbracket\,\tikzmath[scale=\textscale]{\draw (12,12) -- (24,12) (18,0) -- (21,0);
\draw[ultra thick](24,12) -- (30,12) -- (30,6);\draw[dash pattern=on .4pt off .62pt] (19.5,0) -- (19.5,12);} %tikzmath
: \tikzmath[scale=\textscale]{\draw (12,12) -- (24,12) (18,0) -- (21,0);
\draw[ultra thick](24,12) -- (30,12) -- (30,6);}%tikzmath
\,\right\rrbracket\;=\;\sqrt{\mu(\calb)}
$
}};
\node[rotate=90] (C521) at (5.5,4) {\parbox{5cm}{ \center Cor 5.21 \\ \center
$
\left\llbracket \tikzmath[scale=\textscale]{\draw[thick, double](0,6) -- (0,12) -- (6,12);
\draw (6,12) -- (9,12) (12,12) -- (24,12) (18,0) -- (21,0);\draw[ultra thick](24,12) -- (30,12) -- (30,6);\draw[dash pattern=on .4pt off .62pt] (19.5,0) -- (19.5,12);} %tikzmath
 :  \tikzmath[scale=\textscale]{\draw[thick, double](0,6) -- (0,12) -- (6,12);
\draw (6,12) -- (9,12) (12,12) -- (24,12) (18,0) -- (21,0);\draw[ultra thick](24,12) -- (30,12) -- (30,6);} %tikzmath
\right\rrbracket   =   \sqrt{\mu(\calb)}
$
\\ \center
$
\left\llbracket \tikzmath[scale=\textscale]{\draw[thick, double](0,6) -- (0,12) -- (6,12);\draw (6,12) -- (9,12) (12,12) -- (24,12) (18,0) -- (21,0);
\draw[ultra thick](24,12) -- (30,12) -- (30,6);\draw[dash pattern=on .4pt off .62pt] (7.5,12) arc (180:360:3 and 4)(19.5,0) -- (19.5,12);} %tikzmath
 :  \tikzmath[scale=\textscale] {\draw[thick, double](0,6) -- (0,12) -- (6,12);\draw (6,12) -- (9,12) (12,12) -- (24,12) (18,0) -- (21,0);
\draw[ultra thick](24,12) -- (30,12) -- (30,6);\draw[dash pattern=on .4pt off .62pt] (19.5,0) -- (19.5,12);} %tikzmath
\right\rrbracket  =  \sqrt{\mu(\calb)}
$}};
\node[rotate=90] (L513) at (8,-1) {\parbox{3.5cm}{ \center Lma 5.13 \\ \center
$
	\left(\,\tikzmath[scale=\textscale]{\draw[thick, double] (0,6) -- (0,0) -- (6,0);\draw (6,0) -- (18,0) (21,0) -- (24,0) (9,12) -- (12,12);
	\draw[ultra thick] (24,0) -- (30,0) -- (30,6);\draw[dash pattern=on .4pt off .62pt](10.5,0) -- (10.5,12);}%tikzmath
	\,\right)'\; = \;\tikzmath[scale=\textscale]{\draw[thick, double](0,6) -- (0,12) -- (6,12);\draw (6,12) -- (9,12) (12,12) -- (24,12) (18,0) -- (21,0);
	\draw[ultra thick](24,12) -- (30,12) -- (30,6);\draw[dash pattern=on .4pt off .62pt] (19.5,0) -- (19.5,12);}%tikzmath
$ \\ \center
$
	\left(\,\tikzmath[scale=\textscale]{\draw[thick, double](0,6) -- (0,0) -- (6,0);
	\draw (6,0) -- (18,0) (21,0) -- (24,0) (9,12) -- (12,12);\draw[ultra thick] (24,0) -- (30,0) -- (30,6);}%tikzmath
	\,\right)' \; = \; \tikzmath[scale=\textscale]{\draw[thick, double](0,6) -- (0,12) -- (6,12);\draw (6,12) -- (9,12) (12,12) -- (24,12) (18,0) -- (21,0);
	\draw[ultra thick](24,12) -- (30,12) -- (30,6);\draw[dash pattern=on .4pt off .62pt] (7.5,12) arc (180:360:3 and 4)(19.5,0) -- (19.5,12);}%tikzmath
$
}};
\node[rotate=90] (C517) at (7.35,4) {\parbox{5cm}{ \center Cor 5.17 \\ \center
$
\llbracket\,
\tikzmath[scale=\textscale]{\draw[thick, double]  (0,6) -- (0,12) -- (6,12);
\draw (6,12) -- (9,12) (12,12) -- (24,12) (18,0) -- (21,0);\draw[ultra thick]  (24,12) -- (30,12) -- (30,6);\draw[dash pattern=on .4pt off .62pt]   (7.5,12) arc (180:360:3 and 4)(19.5,0) -- (19.5,12);}   \,:\,
\left(\tikzmath[scale=\textscale]{\draw[thick, double]  (0,6) -- (0,0) -- (6,0);
\draw (6,0) -- (18,0) (21,0) -- (24,0) (9,12) -- (12,12);\draw[ultra thick]  (24,0) -- (30,0) -- (30,6);\draw[dash pattern=on .4pt off .62pt] (16.5,0) arc (180:0:3 and 4);}
\right)'\,\rrbracket^T
$
\\ \center
$
\;\;=
\llbracket\,
\tikzmath[scale=\textscale]{\draw[thick, double]  (0,6) -- (0,12) -- (6,12);           
\draw (6,12) -- (9,12) (12,12) -- (24,12) (18,0) -- (21,0);\draw[ultra thick]   (24,12) -- (30,12) -- (30,6);\draw[dash pattern=on .4pt off .62pt] (7.5,12) arc (180:360:3 and 4);}   \,:\,
\tikzmath[scale=\textscale]{\draw[thick, double]  (0,6) -- (0,12) -- (6,12);
\draw (6,12) -- (9,12) (12,12) -- (24,12) (18,0) -- (21,0);\draw[ultra thick]  (24,12) -- (30,12) -- (30,6);}
\,\rrbracket
$
}};
\node[rotate=90] (C516) at (8.75,4) {\parbox{2.5cm}{ \center Cor 5.16 \\ \center
$ \tikzmath[scale=\textscale]{\useasboundingbox (-2,-2) rectangle (32,14);\draw[thick, double](0,6) -- (0,12) -- (6,12);\draw (6,12) -- (9,12) (12,12) -- (24,12) (18,0) -- (21,0);
\draw[ultra thick](24,12) -- (30,12) -- (30,6);\draw[dash pattern=on .4pt off .62pt] (7.5,12) arc (180:360:3 and 4)(19.5,0) -- (19.5,12);} $ factor
}};
\node[rotate=90] (L510) at (9,6.5) {\parbox{4cm}{ \center Lma 5.10 \\ \center
$\llbracket \left(\tikzmath[scale=\textscale]
{\draw[thick, double] (0,6) -- (0,0) -- (6,0);\draw (6,0) -- (24,0);\draw[ultra thick] (24,0) -- (30,0) -- (30,6);} %tikzmath
 \right)'  :  \tikzmath[scale=\textscale]
{\draw[thick, double](0,6) -- (0,12) -- (6,12);\draw (6,12) -- (24,12);\draw[ultra thick](24,12) -- (30,12) -- (30,6);} %tikzmath
\,\rrbracket =$ \\ \center
$  \;\;\;\llbracket \left(
\tikzmath[scale=\textscale]{\draw[thick, double] (0,6) -- (0,0) -- (6,0);
\draw (6,0) -- (18,0) (21,0) -- (24,0) (9,12) -- (12,12);\draw[ultra thick] (24,0) -- (30,0) -- (30,6);\draw[dash pattern=on .4pt off .62pt](16.5,0) arc (180:0:3 and 4);} %tikzmath
 \right)'  : 
\tikzmath[scale=\textscale]{\draw[thick, double](0,6) -- (0,12) -- (6,12);
\draw (6,12) -- (9,12) (12,12) -- (24,12) (18,0) -- (21,0);\draw[ultra thick](24,12) -- (30,12) -- (30,6);\draw[dash pattern=on .4pt off .62pt] (7.5,12) arc (180:360:3 and 4);} %tikzmath
\,\rrbracket
$
}};

\draw[->] (T142) to[bend right=35](C37);

\draw[->] (T52)--(T62);
\draw[->] (T52)--(T151);
\draw[->] (T411)--(P418);
\draw[->] (P418)--(T142);
\draw[->] (P418)--(L421);
%\draw[->] (P418) to[bend left=25](T62);
\draw[->] (L421)--(P429);
%\draw[->] (P429)--(T62);
\draw[->] (Z)--(T62);
\draw[->] (P44)--(L421);
\draw[->] (TC8)--(P44);

\draw[->] (C416)--(P418);
\draw[->] (T36)--(T52);
\draw[->] (C521)--(T52);
\draw[->] (C517)--(T52);
\draw[->] (C516)--(T52);
\draw[->] (L510)--(T52);
\draw[->] (T36)--(C416);
\draw[->] (T36)--(C37);
\draw[->] (L316)--(T36);
\draw[->] (L314)--(L316);
\draw[->] (L314)--(L519);
\draw[->] (L519)--(C520);
\draw[->] (C520)--(C521);
\draw[->] (L513)--(L519);
\draw[->] (L513)--(C521);
\draw[->] (L513)--(C517);
\draw[->] (L513)--(C516);

}
\]
%\end{minipage}
%\end{sideways}

\restoregeometry

\backmatter

\bibliographystyle{amsalpha}
\bibliography{db-cn3-fix}

\printindex

\end{document}